\documentclass[a4paper,bibliography=totoc,index=totoc]{scrbook}

\usepackage[german,british]{babel}
\usepackage[T1]{fontenc}
\usepackage[utf8]{inputenc}
\usepackage[autostyle=true]{csquotes}
\usepackage[german]{isodate}

\usepackage{lmodern}

\usepackage{amsmath}
\usepackage{amssymb}
\usepackage{amsthm}
\usepackage{bm}

\numberwithin{equation}{chapter}
\numberwithin{figure}{chapter}
\numberwithin{table}{chapter}
\newtheorem{definition}{Definition}[chapter]
\newtheorem{lemma}[definition]{Lemma}
\newtheorem{proposition}[definition]{Proposition}
\newtheorem{theorem}[definition]{Theorem}
\newtheorem{corollary}[definition]{Corollary}
\newtheorem{example}[definition]{Example}
\newtheorem{remark}[definition]{Remark}

\DeclareMathOperator{\spn}{span}

\DeclareMathOperator{\indicator}{\chi}
\DeclareMathOperator{\supp}{supp}
\DeclareMathOperator{\interior}{int}

\newcommand{\B}{\mathbb{B}}
\newcommand{\I}{\mathcal{I}}
\newcommand{\N}{\mathbb{N}}
\newcommand{\Z}{\mathbb{Z}}
\newcommand{\R}{\mathbb{R}}
\newcommand{\C}{\mathbb{C}}

\newcommand{\ul}[1]{\underline{#1}}
\newcommand{\ol}[1]{\overline{#1}}

\usepackage{accsupp}
\newcommand{\noncopynumber}[1]{%
    \BeginAccSupp{method=escape,ActualText={}}
    #1
    \EndAccSupp{}
}

\usepackage{graphicx}
\usepackage{placeins}

\usepackage{tikzexternal}

\tikzsetexternalprefix{figures/}
\usepackage{babelbib}
\usepackage[plain,chapter]{algorithm}
\usepackage[algo2e]{algorithm2e}
\usepackage{caption}
\usepackage{cases}
\usepackage{enumerate}
\usepackage{float}
\usepackage{here}
\usepackage{listingsutf8}
\usepackage{longtable}
\usepackage{subcaption}
\usepackage{url}
\usepackage{verse}

\DeclareCaptionLabelFormat{algorithm}{%
  Algorithm~#2\autodot
}
\usepackage{nameref}
\PassOptionsToPackage{hidelinks}{hyperref}
\usepackage{hyperref}

\begin{document}
\selectlanguage{german}
\frontmatter
\begin{titlepage}
	\raggedright
	\null
	\includegraphics[width=0.25\textwidth]{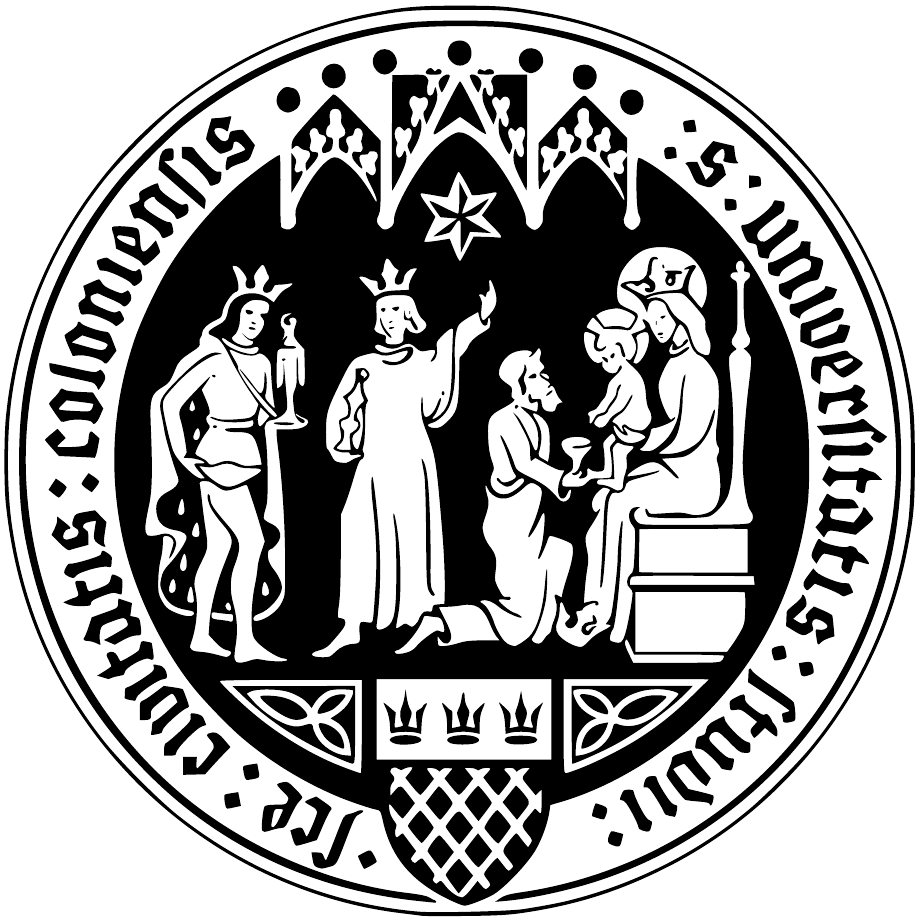}\par\vspace{1cm}
	{\scshape\LARGE Universität zu Köln\par}\vspace{0.3cm}
	{\Large Mathematisches Institut\\}\par\vspace{1.5cm}
	{\scshape\Large Masterarbeit\par}
	\vspace{1.5cm}
	{\huge\bfseries Modern Methods for Signal Analysis: Empirical Mode
	Decomposition Theory and Hybrid Operator-Based Methods Using B-Splines\par}
	\vspace{1.5cm}
	{\Large\itshape Laslo Hunhold\par}
	\vfill
	{\em Erstgutachterin:}\par
	Prof.~Dr.~Angela \textsc{Kunoth}\\[3mm]
	{\em Zweitgutachter:}\par
	Dr. Boqiang \textsc{Huang}
	\vfill
	{\large\printdate{2019-05-27}, überarbeitet am \printdate{2020-03-11}}
\end{titlepage}
\selectlanguage{british}
\chapter{Preface}
Signal analysis is as diverse as the data it tries to comprehend. The
empirical mode decomposition is no exception to this rule and the great
attention it has received over the years with numerous applications in
many fields has always been overshadowed by the introduction of more
and more increasingly powerful but also heuristical approaches.
\par
Since I have been first roughly introduced to the topic by Prof. Dr.
Angela \textsc{Kunoth}, who has published in the field and managed
to spark my interest, in 2015 I was always willing to further understand
and advance it. With other endeavours in the meantime I was given
the chance to work on this topic in the course of my master's thesis.
\par
First of all I would like to thank Prof. Dr. Angela \textsc{Kunoth} for
introducing me to and supporting and entrusting me with
this fascinating and complex topic. I would also like to thank Dr. Boqiang
\textsc{Huang} for his support with his deep insight as a researcher into
the field, his patience and the in-depth discussions.
Last but not least, I would like to thank my family for their unwavering
support and encouragement.
\vspace{0.5cm}
\par\noindent
\emph{Wesseling, Germany} \hfill Laslo Hunhold
\newline
\emph{May 2019}
\vfill
\begin{verse}
	When I heard the learn’d astronomer,\\
	When the proofs, the figures, were ranged in columns before me,\\
	When I was shown the charts and diagrams, to add, divide, and measure\\
	\vin them,\\
	When I sitting heard the astronomer where he lectured with much\\
	\vin applause in the lecture-room,\\
	How soon unaccountable I became tired and sick,\\
	Till rising and gliding out I wander’d off by myself,\\
	In the mystical moist night-air, and from time to time,\\
	Look’d up in perfect silence at the stars.
\end{verse}
\nopagebreak{%
	\raggedleft\footnotesize
	Walt \textsc{Whitman} (1819--1892) \hspace{1cm}
	\par
}
\vfill
\tableofcontents
\mainmatter
\chapter{Introduction}\label{ch:introduction}
This thesis examines the empirical mode decomposition (EMD), a method for
decomposing multicomponent signals, from a modern, both theoretical and 
practical, perspective. The motivation is to further formalize the concept
and develop new methods to approach it numerically.
\paragraph{Multicomponent Signal Decomposition}
A signal is a time- or space-dependent univariate function $s(t)$ that
carries information about the properties of a phenomenon in the form of
variations of an observed quantity over space or time. For instance,
time-varying signals can be financial or audio data, and space-varying
signals can be images or maps.
\par
Naturally, due to the complexity of reality, it is impossible to find a
quantity that only contains information about the phenomenon you are
interested in. Instead, it will contain information about multiple
phenomena simultaneously. To give an example, if you observe bat calls
using a powerful ultrasonic microphone outside at night, you will also
pick up a lot of environment noise (birds, wind, cars, airplanes,
et cetera) that is mixed with your bat calls. As humans
we are good at filtering out audible noise intuitively due to the anatomy
of our ears and function of our brains, which can for example be observed during a
conversation at an event with loud music or background noise.
The computer lacks such intuition. It is in our interest to
quantify this separation process so that it can be applied
to larger and more general problems automatically.
\par
One way to approach this is to consider the concept of \enquote{frequency},
the rate of change of an oscillation over time or space. The signal
is considered as a (weighted) sum of oscillations of different frequencies,
a so-called \enquote{multicomponent signal}, where each summand is called a
\enquote{component} (see Figure~\ref{fig:signal_example} for an example).
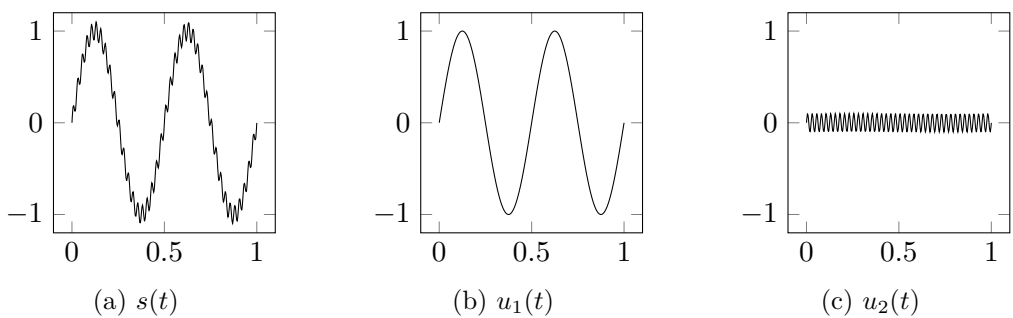
\begin{figure}[htbp]
	\centering
	\begin{subfigure}[c]{0.32\textwidth}
		\centering
		\begin{tikzpicture}
			\begin{axis}[
				name=plot1,
				height=4.5cm,width=4.5cm,
				ymin=-1.2,
				ymax=1.2,
				domain=0:1,
			]
				\addplot[mark=none,smooth,samples=200] (x,
					{sin(deg(4 * pi * x)) + 0.1 * sin(deg(80 * pi * x))});
			\end{axis}
		\end{tikzpicture}
		\subcaption{$s(t)$}
	\end{subfigure}
	\begin{subfigure}[c]{0.32\textwidth}
		\centering
		\begin{tikzpicture}
			\begin{axis}[
				name=plot1,
				height=4.5cm,width=4.5cm,
				ymin=-1.2,
				ymax=1.2,
				domain=0:1,
			]
				\addplot[mark=none,smooth,samples=200] (x,
					{sin(deg(4 * pi * x))});
			\end{axis}
		\end{tikzpicture}
		\subcaption{$u_1(t)$}
	\end{subfigure}
	\begin{subfigure}[c]{0.32\textwidth}
		\centering
		\begin{tikzpicture}
			\begin{axis}[
				name=plot1,
				height=4.5cm,width=4.5cm,
				ymin=-1.2,
				ymax=1.2,
				domain=0:1,
			]
				\addplot[mark=none,smooth,samples=200] (x,
					{0.1 * sin(deg(80 * pi * x))});
			\end{axis}
		\end{tikzpicture}
		\subcaption{$u_2(t)$}
	\end{subfigure}	
	\caption{%
		An example for a signal $s(t)$ that is an additive composite of
		a low-frequency oscillation $u_1(t)$ and a high-frequency oscillation
		$u_2(t)$.
	}
	\label{fig:signal_example}
\end{figure}
A single component does not necessarily correspond
to the phenomenon we are interested in but this decomposition into
components quantifies the signal and a subset of these components as a
whole can convey the information we need. This process is called
\enquote{signal decomposition}.
\par
Reconsidering our bat call example, we can easily discard all components
corresponding to frequency ranges that are above or below the frequency
ranges of bat calls. More problematic are the frequency ranges of the
bat calls themselves and how to decompose them usefully.
To approach this issue, we will as follows introduce the three most
popular signal decomposition methods. The last one, the
Empirical Mode Decomposition (EMD), will be the main focus of this thesis.
\subparagraph{\textsc{Fourier} Transform}
First proposed in 1822 by Jean-Baptiste-Joseph \textsc{Fourier}
(see \cite{f22}), the \emph{\textsc{Fourier} transform} is the most 
well-known method in this context. It is based on the observation
that for every $1$-periodic function $\Phi(t)$ (which means that for
all $t \in \R$ it holds $\Phi(t) = \Phi(t + 1)$) and $j \in \Z$ we can find
$c_j \in \C$ such that
\begin{equation}\label{eq:fourier-series}
	\Phi(t) = \sum_{j \in \Z} c_j \cdot \exp(2\pi ijt).
\end{equation}
This is due to the fact that, roughly spoken,
$\{ t \mapsto \exp(2\pi i j t) \mid j \in \Z \}$ is an orthonormal basis
of the \textsc{Hilbert} space (a real or complex vector space with an inner
product that is a complete metric space in regard to the norm induced by the inner
product) of square integrable $1$-periodic functions.
Equation~(\ref{eq:fourier-series}) is called the \enquote{\textsc{Fourier} series} of
$\Phi(t)$ and the coefficients $c_j$ are calculated as
\begin{equation}\label{eq:fourier-coefficients}
	c_j := \int_{-\frac{1}{2}}^{\frac{1}{2}} \Phi(t) \cdot \exp(-2\pi i j t)
	\, \mathrm{d}t.
\end{equation}
The mapping $j \mapsto c_j$ is called the \enquote{\textsc{Fourier} transform}
of $\Phi(t)$ and the parameter $j$ corresponds to the frequency. The higher
the $j$, the faster the $\exp(2\pi ijt)$ term oscillates over time $t$, resulting in
a higher-frequency oscillation. If we cover $j \in \Z$ we obtain a complete
coverage of low and high frequencies. In Figure~\ref{fig:fourier-series} you can
see an example of how a finite \textsc{Fourier} series composes a 1-periodic
function.
\par
We can immediately see that this series is a $c_j$-weighted sum
of oscillations $\exp(2\pi ijt)$, a property we desired based on the
observations in the previous paragraph on multicomponent signals.
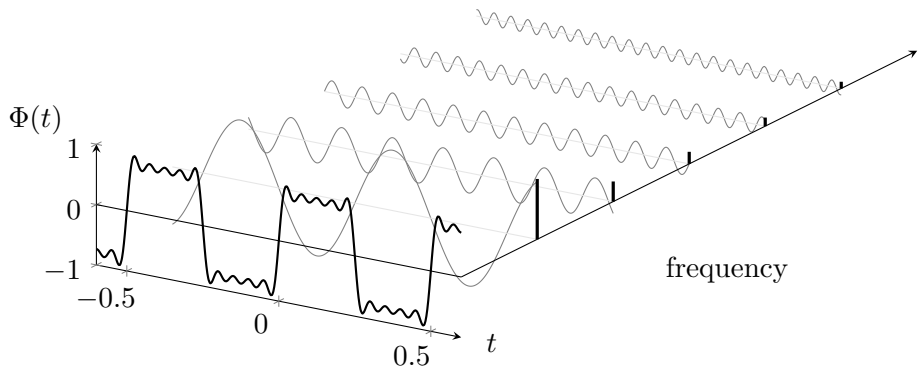
\begin{figure}[htbp]
	\centering
	\begin{tikzpicture}
		\begin{axis}[
			x = {(4cm, -0.8cm)},
			y = {(1cm,  0.5cm)},
			z = {(0cm,  0.8cm)},
			ylabel=frequency,
			xlabel near ticks,
			ylabel near ticks,
			ylabel style={rotate=-90},
			axis x line=left,
			axis y line=left,
			axis z line=left,
			every outer y axis line/.append style={white},
			xtick={-0.5,0,0.5},
			ytick=\empty,
			ztick={-1,0,1},
			domain=-0.6:0.6,
			ymax=6,
			samples y=0,
			clip mode=individual,
		]
			\def\sinsum{0};
			\pgfplotsinvokeforeach{5,4,...,1}{
				\xdef\sinval{(1 / ((#1 - 0.5) * 2)) * sin(deg(2 * pi * ((#1 - 0.5) * 4) * x))};
				\addplot3[mark=none,smooth,samples=200,black!50] (x, #1, \sinval);
				\draw [black!10] (axis cs:-0.6,#1,0) -- (axis cs:0.6,#1,0);
				\draw [very thick] (axis cs:0.6, #1, 0) -- (axis cs:0.6, #1, {(1 / ((#1 - 0.5) * 2))});
				\xdef\sinsum{\sinsum + \sinval};
			}
			\addplot3[mark=none,smooth,samples=200,thick] (x, 0, \sinsum);
			\draw (axis cs:-0.6,0,0) -- (axis cs:0.6,0,0);
			\draw [->,>=stealth] (axis cs:0.6,0,0) -- (axis cs:0.6,6,0);
			\node at (axis cs: 0.7,0,-1) {$t$};
			\node at (axis cs: -0.8,0,1.2) {$\Phi(t)$};
		\end{axis}
	\end{tikzpicture}
    \caption{%
    	Visualization of the \textsc{Fourier} series composition of a $1$-periodic
    	function $\Phi(t)$ that is additively composed of $5$ oscillatory terms.
    }
    \label{fig:fourier-series}
\end{figure}
We have to note here, though, that $\Phi(t)$ is $1$-periodic, which a signal
$s(t)$ is not in general. To extend the \textsc{Fourier} transform to non-periodic
signals, we first note that for any $T > 0$ a $T$-periodic function $\tilde{\Phi}(t)$
can be transformed into a $1$-periodic function $\Phi(t)$ via
$\Phi(t) := \tilde{\Phi}(t / T)$. If we apply this transformation to the above
expression and let $T \to \infty$ we obtain the general \textsc{Fourier} expression
of a non-periodic signal $s(t)$ as
\begin{equation}\label{eq:fourier-expression}
	s(t) = \int_{-\infty}^{\infty} (\mathcal{F}s)(f) \cdot \exp(2\pi i ft) \, \mathrm{d}f
\end{equation}
with the \textsc{Fourier} transform
\begin{equation}\label{eq:fourier-transform}
	(\mathcal{F}s)(f) := \int_{-\infty}^{\infty} s(t) \cdot \exp(-2\pi i ft) \, \mathrm{d}t.
\end{equation}
The parameter $f$ of the \textsc{Fourier} transform $(\mathcal{F}s)(f)$ corresponds to a
continuous form of the $j$ we have seen earlier. During the \textsc{Fourier} transform,
each \enquote{frequency} $f$'s share is averaged over the entire interval the signal
is defined on. In turn, if the signal's frequency composition varies across this
timeframe the \textsc{Fourier}
transform is unable to reflect these changes, and we can say that it is only suitable
for \emph{stationary signals}, which are signals whose frequency compositions do not
change much over time.
\par
One may approach this problem by reducing the area the \textsc{Fourier} transform covers.
This is done by applying a so-called \enquote{window function} to the signal that zeros out
all of the signal except on a compact interval. The window function is varied by employing a
base window function (e.g.\ a \textsc{Gauss} function) that is \enquote{moved} to multiple
parts of the time domain until it has been fully covered and all sub-intervals are analyzed.
This method is called the Short-Time \textsc{Fourier} Transform (STFT).
The \textsc{Küpfmüller} uncertainty principle states that it is impossible to both clearly
localize a signal in both the time and frequency domain (see \cite[VII.47 (29a)]{kk00}).
This shows that the downside of the STFT is that with increasing tightness of the
time-interval that is studied the frequency becomes more and more uncertain.
\par
In summary, on the one hand, the classical \textsc{Fourier} transform localizes the signal
perfectly in the frequency domain, but has the worst possible time resolution. On the other
hand, applying window functions presents the limits of signal analysis and leads to
bad time-resolution for high frequencies, because the time-window is made arbitrarily small.
For further reading on \textsc{Fourier} analysis one may consult \cite{k88}.
\subparagraph{Wavelet Transform}
The \emph{wavelet transform} based on the groundwork by
Alfréd \textsc{Haar} in 1910 (see \cite{h10}) is closely related to
the STFT and makes use of by now so-called \enquote{wavelet} functions that
are \enquote{better-behaved} as window functions than the ones used for STFT.
\enquote{Better-behaved} in this context means
providing better frequency-resolution for shorter time-intervals and
better time-resolution for high frequency bands.
\par
The fundamental idea is to consider the \textsc{Hilbert} space (a
real or complex vector space with an inner product that is a complete
metric space in regard to the norm induced by the inner product)
$(\mathcal{L}^2(\R), \langle,\rangle)$ of square-integrable functions
with the standard inner product $\langle,\rangle \colon
\mathcal{L}^2(\R) \times \mathcal{L}^2(\R) \to \mathcal{L}^2(\R)$ defined as
\begin{equation}\label{eq:standard_inner_product}
	\langle f,g \rangle := \int_{-\infty}^\infty f(t) \ol{g(t)} \, \mathrm{d}t
\end{equation}
and find an orthonormal basis (which means that the inner product of two distinct
basis elements is zero and one for two same basis elements) for it.
In the context of the \textsc{Fourier} transform, we noted previously that the set
$\{ t \mapsto \exp(2\pi i jt) \mid j \in \Z \}$  was an orthonormal basis of
the Hilbert space of the $1$-periodic square-integrable
functions. However, our interest here is to find basis functions with compact support
(which means that they are zero everywhere except on a compact interval).
An additional limitation with wavelets, in terms of an orthonormal basis, is that we, just like previously with the window functions
for STFT, consider one basic function we in this context call
\enquote{mother wavelet} $\psi(t)$ that is moved and transformed across
the time interval to generate all other basis functions. The transformations are so-called
\enquote{dyadic translations} and \enquote{dilations} and are parametrized
by $j,k \in \Z$, yielding a family of functions defined as
\begin{equation}\label{eq:wavelet-basis}
	\psi_{j,k}(t) := 2^{\frac{j}{2}} \cdot \psi(2^j \cdot t - k).
\end{equation}
If the mother wavelet $\psi(t)$ can be used to construct a \textsc{Hilbert}
basis as described above, we call it an \emph{orthonormal wavelet}. Then we
can express any signal $s(t)$ with $(j,k) \in \Z^2$ and $c_{j,k} \in \R$ as
\begin{equation}\label{eq:wavelet-series}
	s(t) = \sum_{(j,k) \in \Z^2} c_{j,k} \cdot \psi_{j,k}(t)
\end{equation}
with the wavelet coefficients
\begin{equation}\label{eq:wavelet-coefficients}
	c_{j,k} = \langle s,\psi_{j,k} \rangle.
\end{equation}
The advantage of this separation becomes apparent when we consider
that the parameters $j$ and $k$ play special roles: The parameter $j$
corresponds to the frequency (dyadic dilation), whereas $k$ corresponds to the
location (dyadic translation). When we group the sum by frequency, we obtain
\begin{equation}\label{eq:wavelet-series-decomp}
	s(t) = \sum_{j \in \Z} \sum_{k \in \Z} c_{j,k} \cdot \psi_{j,k}(t)
	=:  \sum_{j \in \Z} g_j(t),
\end{equation}
effectively yielding a separation of the signal into functions
$g_j(t)$ reflecting the relative share of the frequency respective to
$j$ within the signal over time.
\par
When discussing the wavelet transform one has to observe that
the choice of the mother wavelet $\psi(t)$ is
neither canonical nor domain-specific. Additionally, the range of the
parameters $j$ and $k$ have to be determined in advance or adaptively,
making it necessary to apply some kind of preprocessing to the signal. The
dyadic translations (parametrized by $k$) impose the same grid-density
across the entire interval, as $\psi_{j,k}$ is linear in $k \in \Z$.
This is problematic when a small timeframe of the input signal has high oscillations
that might require a high local time resolution.
Still, as opposed to the \textsc{Fourier} transform, the wavelet
transform can be used for non-stationary signals (those that exhibit
changes of their frequency-composition over time) as well.
For further reading on the wavelet transform, one may
consult \cite{dau92} and \cite{m09}.
\subparagraph{Empirical Mode Decomposition}
The method focused on in this thesis is the
\emph{empirical mode decomposition} (EMD) proposed in 1998 by
\textsc{Huang} et aliī (see \cite{hsl+98}) that has gained lots of
attention since then. In contrast to the classic
\textsc{Fourier} and Wavelet transforms that depend on a predefined
finite subset of a \textsc{Hilbert} basis to match a certain frequency
range, the EMD is an iterative data-adaptive process. This means that
there needs to be no preprocessing and it adapts to the incoming data
as it analyzes it. In contrast to classical versions of the \textsc{Fourier}
and wavelet transforms, it also does not require the input data to be
regularly aligned on a grid.
\par
This method works as follows: The signal $s(t)$ is additively adaptively separated
into $S$ so-called \enquote{intrinsic mode functions} (IMFs) $u_i(t)$,
which each more or less correspond to the signal components laid out
earlier, and a residual $r_{S+1}$ that remains from the signal after
the $S$ extraction steps. The crucial difference compared to the previous methods is
that the IMFs are allowed to slowly vary in frequency and intensity over
time, whereas previously we had functions that were more or less fixed
in the frequency- and time-domains. A formal definition of IMFs is given
in \ref{sec:intrinsic_mode_functions}.
\par
Assuming we have determined the IMFs, we obtain the signal representation
\begin{equation}\label{eq:emd-representation}
	s(t) = \sum_{i=1}^{S} u_i(t) + r_{S+1}(t)
\end{equation}
and set requirements that are satisfied in the ideal case. Each IMF
shall have the form
\begin{equation*}
	u_i(t) = a_i(t) \cdot \cos(\phi_i(t)),
\end{equation*}
with a so-called \enquote{instantaneous amplitude} $a_i(t)$ and
\enquote{instantaneous phase} $\phi_i(t)$. The derivative $\phi'_i(t)$
of the instantaneous phase $\phi_i(t)$ is the frequency. Thus, this
form allows the IMF to be both variable in amplitude and frequency.
We require for all $t \in \R$
the natural conditions $a_i(t) \ge 0$ and $\phi'_i(t) > 0$. These are
necessary given negative amplitudes or frequencies are not physically meaningful.
We also want $a_i(t)$ and $\phi_i'(t)$ to be
slowly varying, which will be formally laid out later. The residual
shall at best be monotonic or have at most one local maximum or minimum, which
of course is dependent on how many extraction steps $S$ were taken.
\par
Separating a signal $s(t)$ into IMFs is called \enquote{sifting} (see
Figure~\ref{fig:sifting}). This is a multi-step-process, but each step
is more or less independent from the others. A single step extracts one
IMF from the signal, subtracts the IMF from the signal and returns the result
as the so-called \enquote{residual}, which is then again processed as a
new input signal in the next step. For this reason, we will, as follows and
in this thesis, almost exclusively focus on a single step of the sifting process.
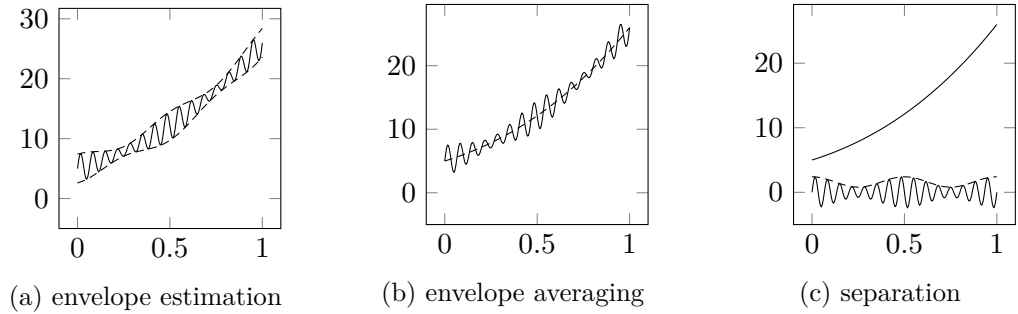
\begin{figure}[htbp]
	\centering
	\begin{subfigure}[c]{0.32\textwidth}
		\centering
		\begin{tikzpicture}
			\begin{axis}[
				name=plot1,
				height=4.5cm,width=4.5cm,
				ymin=-5,
			]
				\addplot [domain=0:1, samples=200]
					{2 + 3 * (1 + x)^3 +
					0.8 * (2 + cos(deg(4 * pi * x))) * sin(deg(30 * pi * x))};
				\addplot [domain=0:1, samples=200, densely dashed]
					{2 + 3 * (1 + x)^3 +
					0.8 * (2 + cos(deg(4 * pi * x))) * 1};
				\addplot [domain=0:1, samples=200, densely dashed]
					{2 + 3 * (1 + x)^3 +
					0.8 * (2 + cos(deg(4 * pi * x))) * (-1)};
			\end{axis}
		\end{tikzpicture}
		\subcaption{envelope estimation}
	\end{subfigure}
	\begin{subfigure}[c]{0.32\textwidth}
		\centering
		\begin{tikzpicture}
			\begin{axis}[
				name=plot1,
				height=4.5cm,width=4.5cm,
				ymin=-5,
			]
				\addplot [domain=0:1, samples=200, densely dashed]
					{2 + 3 * (1 + x)^3};
				\addplot [domain=0:1, samples=200]
					{2 + 3 * (1 + x)^3 +
					0.8 * (2 + cos(deg(4 * pi * x))) * sin(deg(30 * pi * x))};
			\end{axis}
		\end{tikzpicture}
		\subcaption{envelope averaging}
	\end{subfigure}
	\begin{subfigure}[c]{0.32\textwidth}
		\centering
		\begin{tikzpicture}
			\begin{axis}[
				name=plot1,
				height=4.5cm,width=4.5cm,
				ymin=-5,
			]
				\addplot [domain=0:1, samples=200]
					{2 + 3 * (1 + x)^3};
				\addplot [domain=0:1, samples=200]
					{0.8 * (2 + cos(deg(4 * pi * x))) * sin(deg(30 * pi * x))};
				\addplot [domain=0:1, samples=200, densely dashed]
					{0.8 * (2 + cos(deg(4 * pi * x))) * 1};
			\end{axis}
		\end{tikzpicture}
		\subcaption{separation}
	\end{subfigure}
	\caption{%
		Visualization of the EMD sifting process of a $1$-component
		signal.
	}
	\label{fig:sifting}
\end{figure}
The classic procedure for a sifting step laid out in \cite{hsl+98} is to
determine lower and upper
envelopes $\ul{a}(t)$ and $\ol{a}(t)$ of the signal $s(t)$, usually by
interpolating local maxima and minima.
The mean $\frac{1}{2} \cdot (\ol{a} + \ul{a}) =: r(t)$ is defined as the
residual $r(t)$ for the next decomposition step and the difference
$s(t) - r(t) =: u(t)$ between signal and residual is the desired intrinsic
mode function. As a side-result, one obtains the amplitude $a(t)$ of the
intrinsic mode function $u(t) := a(t) \cdot \cos(\phi(t))$ by the difference
$\ol{a}(t) - r(t) := a(t)$ between the upper envelope and the residual.
\par
The next (optional, depending on the application) step is to do a so-called \enquote{spectral analysis} of the
extracted IMFs, which means that for a given IMF $u(t)$ the instantaneous
amplitude $a(t)$ and phase/frequency $\phi(t)$/$\phi'(t)$ are extracted.
As the name implies, the EMD is an empirical method. Unfortunately, this
procedure has up to now a relatively weak theoretical footing
compared to the strong theory behind the \textsc{Fourier} and wavelet transforms.
\par
The big disadvantage of the EMD representation in
Equation~(\ref{eq:emd-representation}) is that it is not unique, making it
difficult to formulate theoretical assessments. Without providing more
conditions, the uniqueness guarantee is impossible to give. However,
the big strength of the EMD is that both the instantaneous amplitude and phase can
have arbitrary form within the bounds of physical meaningfulness and
slow variation in Equation~(\ref{eq:emd-representation}). This means that
a single IMF can \enquote{track} a subcomponent of a singal over the time-
and frequency-domain even if this subcomponent happens to change in intensity
or frequency and this rate of change falls within the previously set bounds.
\par
When considering the decomposition provided by the \textsc{Fourier}
transform in Equation~(\ref{eq:fourier-transform}), we see that it has a
constant \enquote{amplitude} $(\mathcal{F}(f))(f)$ for each oscillation
term $\exp(2\pi ift)$ belonging to the fixed \enquote{frequency} $f$,
given the \textsc{Fourier} transform
provides no time-resolution. The wavelet transform decomposition in
Equation~(\ref{eq:wavelet-series-decomp}) improves upon this problem by
having a decomposition into \enquote{frequency share} functions $g_k(t)$.
However, it is unable to reflect the condition when a signal component
leaves the frequency-band relating to $k$ without further post-processing
of some kind. The empirical mode decomposition with its flexible instantaneous
amplitude and phase for each intrinsic mode function is able to flexibly reflect
both changes in amplitude and frequency over time.
\paragraph{Operator-Based Signal Separation and Null-Space-Pursuit}
The operator-based signal separation (OSS) was proposed in 2008 by
\textsc{Peng} et aliī (see \cite{ph08} and \cite{ph10}) as an idea for
a more formal foundation for the empirical mode decomposition.
The basic concept is centered around the idea of an \enquote{adaptive operator},
which we will explain with an example as follows.
\par
Consider you have a function $h(t)$ and you only know that it is
of the form $h(t) := \cos(\phi(t))$. The function $\phi(t)$ is not
known and it is your goal to determine it.
Let us consider the first and second derivatives of $h(t)$.
We obtain with the chain rule that $h'(t) = -\phi'(t) \cdot \sin(\phi(t))$ and
$h''(t) = -\phi''(t) \cdot \sin(\phi(t)) - {(\phi'(t))}^2 \cdot \cos(\phi(t))$.
If we then define a differential operator $\mathcal{D}_{\tilde{\phi}}$ with respect
to the input function $h(t)$ as
\begin{equation}\label{eq:diffop-example}
	\left(\mathcal{D}_{\tilde{\phi}}h\right)(t) := h''(t) -
		\frac{\tilde{\phi}''(t)}{\tilde{\phi}'(t)} \cdot h'(t) +
		{(\tilde{\phi}'(t))}^2 \cdot h(t),
\end{equation}
it follows directly, because the terms cancel each other out, that
\begin{equation*}
	\mathcal{D}_{\phi}h \equiv 0.
\end{equation*}
The differential operator is defined in terms of the
parameter $\tilde{\phi}(t)$.
If we manage to choose it as $\phi(t)$, the \enquote{hidden} function
within the cosine-term of $h(t)$, the operator applied to $h(t)$ vanishes.
In other words, we can say that then $h(t)$ is in the \enquote{null-space} of the
operator. The search for the correct parameter to annihilate the operator applied
to $h(t)$ can consequently be called \enquote{null-space-pursuit}.
\par
When we reconsider the EMD signal representation from
Equation~(\ref{eq:emd-representation}), it becomes clear that this
operator-based approach can be
used to process IMFs $u(t)$ in some fashion. The IMFs are of the form
$u(t) := a(t) \cdot \cos(\phi(t))$ in regard to their instantaneous
amplitude $a(t)$ and phase $\phi(t)$, but both are not known.
For the purpose of spectral analysis, i.e.\ determining these factors,
we use the adaptive differential
operator that is parametrized by $\tilde{a}(t)$ and $\tilde{\phi}(t)$ and
annihilate the IMF $u(t)$ when the parameters are tuned to $a(t)$ and $\phi(t)$,
effectively yielding us the previously unknown instantaneous amplitude
and frequency.
\par
This however is not enough. To completely formally express the EMD, another important aspect is to
grasp the IMF extraction itself. In each step, we split the input
signal $s(t)$ into an IMF $u(t)$ and a
residual $r(t)$. A canonical extraction condition is to demand that we
extract as much as possible from the input signal, namely, that the residual
$r(t) = s(t) - u(t)$ is \enquote{minimal}, in a sense that is to be made precise.
\par
Combining both ideas, a single EMD extraction step for an input signal $s(t)$
can be expressed as a regularized optimization problem.
We consider the residual $r(t)$ and use a fitting differential operator
$\mathcal{D}_{(\tilde{a},\tilde{\phi})}$ (see \cite[Equation~(3)]{ph10}) with
some real parameter $\lambda > 0$. The function-minimization-terms are put into
norms so they yield a cost-function with values in $\R$ as
\begin{equation*}
	(r,a,\phi) = \arg\min_{\tilde{r},\tilde{a},\tilde{\phi}}
		\left\{ {\left\| \mathcal{D}_{(\tilde{a},\tilde{\phi})}
		(s-\tilde{r}) \right\|}_2^2 +
		\lambda \cdot {\| \tilde{r} \|}_2^2 \right\}.
\end{equation*}
This problem can be reformulated in terms of an IMF $u$ with
instantaneous amplitude $a$ and phase $\phi$ as
\begin{equation}\label{eq:oss-optimization}
	(u,a,\phi) = \arg\min_{\tilde{u},\tilde{a},\tilde{\phi}}
		\left\{ {\left\| \mathcal{D}_{(\tilde{a},\tilde{\phi})}
		\tilde{u} \right\|}_2^2 +
		\lambda \cdot {\| s - \tilde{u} \|}_2^2 \right\}.
\end{equation}
The first term in Equation~(\ref{eq:oss-optimization}) ensures that
the resulting function $u$ is an IMF and enables us to perform a
spectral analysis in terms of $a$ and $\phi$, because it strives
to annihilate the operator whose parameters we are tuning to.
The second term ensures, as previously
discussed, maximum extraction from the signal $s$, i.e.\ a minimal residual.
This is referred to in \cite{ph10} as the \enquote{greedy} approach.
Note that there are other ways to formulate an extraction condition other than
the minimization of $r(t) = s(t) - u(t)$ within a norm. If we assume
that our residual is reasonably smooth, the greedy approach is a valid approach
compared to other approaches considering higher-order differentiation
of the residual within the norm.
\par
The approach of the operator-based signal separation
with the null-space-pursuit provides an elegant formalization of the EMD,
combining both the sifting and spectral analysis into one optimization
problem. The previous difficulty that the separation of the signal
into an IMF and residual directly relies on the spectral analysis of
said IMF \enquote{in-situ} is elegantly solved by weaving the
spectral analysis in form of an adaptive operator into the extraction
process itself.
\par
An open question is the choice of such an adaptive operator and how
well it enforces the IMF conditions, given the one in \cite[Equation~(3)]{ph10}
is not unique. To give an example, let us
consider the differential operator $\frac{\partial^3}{\partial t^3}$
as an example for a differential operator to \enquote{match} (i.e.\ annihilate)
quadratic functions.
Quadratic functions are in its null-space, which means that it is
suitable for a null-space-pursuit to enforce quadratic functions.
However, linear and constant functions are also in its null-space
and will subsequently be also \enquote{matched}. The same problem,
though much harder to grasp, might be present for IMF-annihilating
adaptive operators.
\paragraph{EMD Optimization Problem}
Taking a step back from the formalized EMD by \cite{ph10} in
Equation~(\ref{eq:oss-optimization}), this thesis proposes to take
a new look at the EMD as a constrained optimization problem
for each step.
The author calls this the EMD optimization problem (EMDOP) and
investigates this in Chapter~\ref{ch:emd_analysis}.
It considers the operator-based method as a form of regularization
over the set of IMFs (see \ref{sec:intrinsic_mode_functions}) and generalizes it. The extraction condition is
that $r(t) = s(t) - u(t)$ shall be minimal (maximum extraction, minimal residual), yielding the
optimization problem
\begin{align*}
	\min_{u} \quad &
		{\| s - u \|}_2^2 \\
	\text{s.t.} \quad & u \text{\ IMF}.
\end{align*}
This optimization problem corresponds to one single step of the EMD
and is later generalized to arbitrary \enquote{cost functions} other
than ${\| s - u \|}_2^2$.
The EMDOP will be the main focus of the theoretical groundwork of this
thesis in Chapter~\ref{ch:emd_analysis}. It provides new results
for OSS/NSP and other similar regularization-based EMD-schemes.
As previously stated, the only path toward an EMD-algorithm yielding
unique results is to add more conditions to the extracted IMF.
One path is to add more regularization terms which has consecutively
been done in the analysis in \cite{ph08} and \cite{ph10}. Another way is
to add more constraints to the EMDOP. This thesis considers the latter
approach, as more regularization terms weaken the theoretical foundation
of the EMD method even more. Another reason for the latter approach is that regularization terms in the cost functions are in fact there
to enforce some kind of condition on the extracted IMF, so it only makes
sense to avoid this indirect route and directly state these conditions.
\paragraph{B-Splines}
Introduced by Isaac Jacob \textsc{Schoenberg} in 1946 (see \cite{sch46a}
and \cite{sch46b}), B-splines (\enquote{Basis-splines}) have become an
integral part of numerical analysis due to their very useful theoretical and
practical properties as basis functions for a space of piecewise polynomial
functions. In the course of this thesis, we will make use of these properties.
\par
The objects of interest in the presented signal analysis
methods are functions, not scalars. Looking at the EMD, for instance, we have
the functions describing the signal $s(t)$: IMF $u(t)$, residual $r(t)$,
instantaneous amplitude $a(t)$ and instantaneous phase $\phi(t)$.
Many publications, despite dealing with functions, express their algorithms
in terms of discrete samples. The author considers this to be a problem as
the process of fitting a function to samples opens up a new set of problems,
for instance over- or underfitting the data in some way. This problem
is discussed in \cite{di95} and, for reasons of scope due to the complexity
of sampling theory, is left out in this thesis.
\par
However, to explain this briefly,
when analyzing a signal you are mostly interested in a certain frequency band. Frequencies
above or below that are considered as \enquote{noise}. When taking a step back,
though, there is really no such thing as \enquote{noise}, given this \enquote{noise}
is just a signal with frequencies we are not interested in. A strictly sample-based
algorithm needs to be careful, given that samples can contain such oscillations
depending on the sampling rate. When working with continuous signals, we allow
such high oscillations but don't make ourselves dependent on the sampling rate.
\par
To avoid such problems with samples and discrete signals, B-splines are used to model smooth functions in a discretized (such that they are machine-representable) way in this thesis and all outside inputs considered to be functions rather than
samples.
\paragraph{Goal of this Thesis}
The goal of this thesis is to take a both theoretical
and practical look at the empirical mode decomposition.
We want to answer the question how to classify the previously introduced
OSS/NSP method (see Equation~(\ref{eq:oss-optimization})) within
the aforementioned strictly theoretical newly introduced
EMD model EMDOP. After theoretical assessments, the canonical objective
is to make use of the theoretical results to develop new EMD methods
with regard to sifting and spectral analysis by employing the OSS/NSP method.
\paragraph{Structure of this Thesis}
Following Chapter~\ref{ch:b-splines}, which
introduces B-splines, Chapters \ref{ch:emd_analysis}
and \ref{ch:oss} contain the main theoretical results of
this thesis.
Chapter~\ref{ch:emd_analysis} analyzes the empirical
mode decomposition by first formalizing the
aforementioned EMD optimization problem in Sections
\ref{sec:intrinsic_mode_functions},
\ref{sec:cost_functions} and \ref{sec:general_opt}
and proving it to be \textsc{Slater}-regular
in Section~\ref{sec:regularity}.
Chapter~\ref{ch:oss} motivates the foundation of the
operator-based signal-separation method and analyzes the
operator-based analysis of intrinsic mode functions.
\par
Using the results obtained in Chapters \ref{ch:emd_analysis}
and \ref{ch:oss}, Chapter \ref{ch:hobm} proposes
an EMD approach that is a hybrid of classic and
modern methods. In the course of this construction,
a new \enquote{iterative slope} envelope estimation algorithm
is motivated, presented and evaluated in
Section~\ref{sec:envelope_estimation}.
This yields the final hybrid operator-based EMD
method in Section~\ref{sec:emd_algo}.
As another coproduct, the \enquote{ETHOS} toolbox
is presented and documented in Section~\ref{sec:ethos_toolbox}.
\chapter{B-Splines}\label{ch:b-splines}
The central objects of interest in the empirical mode decomposition are
functions. Our interpolated signals, intrinsic mode functions and
instantaneous amplitudes and frequencies are all time-variant quantities
and, thus, a good model for them is of high importance.
\par
A priority we can note is that however we model functions, they should
be easy to represent numerically. Another key aspect of interest, as we
will make heavy use of it later, is the ability to evaluate the functions
and their derivatives easily and quickly. The approach of many
publications is to directly work with samples and use in-situ-approximated
derivatives. However, this makes it difficult to formalize the process
and distinguish between sampling errors and weaknesses in the process itself.
Thus, even though the classic EMD algorithm presented in \cite{hsl+98}
works with discrete datapoints, the main interest should be to strive to
understand why such heuristics work and when and not mix the problem with
one related to sampling theory. In other words: Oversampling should not
affect the result and continuous rather than discrete signals help us
mitigate this issue.
\par
In general the basis function approach is that one considers a finite
vector space of functions for which one can find a finite set of basis functions.
Weighted sums of these basis functions can then be used to represent any
function in this vector space.
If we take $\R^2$ as an example, there exist numerous possible choices
for bases, for instance the standard basis $\{ {(1,0)}^T, {(0,1)}^T \}$
or $\{ {(1,1)}^T, {(0,1)}^T \}$. Any element in $\R^2$ can be represented
with a weigthed sum of these basis vectors. For function spaces, which are
also vector spaces, one can also consider different choices of basis
functions accordingly.
\par
The choice of basis functions used in this thesis are B-splines, a basis
for the vector space of spline functions that has multiple useful
theoretical and numerical properties. B-splines were first
introduced by Isaac Jacob \textsc{Schoenberg} in 1946 (see \cite{sch46a} and
\cite{sch46b}) and the term is short for \enquote{basis splines}.
It had a big impact in many numerical fields since then.
This thesis is the first to explore the solution theory of the
empirical mode decomposition and
the operator-based methods using B-splines and generally makes
heavy use of them. This is why we introduce B-splines in this section
in such detail, but leave out some of the more technical proofs.
Before introducing B-splines, we naturally first have to define what a
spline function is. To do that, we first introduce the
\begin{definition}[Set of polynomials {\cite[I (1)]{db01}}]\label{def:set_of_polynomials}
	Let $k\in\N$. The \emph{set of polynomials} of order $k$
	is defined as
	\begin{equation*}
		\Pi_k := \left\{
			t \mapsto \sum_{i=0}^{k-1} a_i \cdot t^i
			\mathrel{\Bigg|}
			\left(a_0,\dots,a_{k-1}\right)
			\in
			 \R^{k-1} \times \R_{\neq 0}
		\right\}.
	\end{equation*}
\end{definition}
We distinguish between \enquote{degree}
and \enquote{order}. A linear polynomial with degree $1$
(largest $t$-exponent) has order $2$ (degrees of freedom),
a cubic polynomial with degree $3$ has order $4$. Now that
we have defined polynomials, we can formulate what spline functions
are.
\begin{definition}[Spline function space {\cite[VII (1)]{db01}}]\label{def:spline_function_space}
	Let $k,\ell \in \N$ with $k \le \ell - 1$ and
	$T := {\{ \tau_i \}}_{i = 0}^{\ell - 1}$
	with $\tau_0 < \dots < \tau_{\ell - 1}$.
	The \emph{spline function space}
	$\Sigma_{k,T}$ of order $k$ on $T$ is defined as
	\begin{equation*}
		\Sigma_{k,T} :=
		\left\{
			s \in \mathcal{C}^{k-2}([\tau_0,\tau_{\ell-1}])
			\mathrel{\Big|}
			\forall_{i\in\{ 0,\dots,\ell-2 \}} \colon
			\left. s \right\rvert_{[\tau_i,\tau_{i+1})} \in \Pi_k
		\right\}
	\end{equation*}
\end{definition}
As we can see, a spline function is a smooth function that
is piecewise-defined by polynomials. Analogous to the set
of polynomials the set of linear splines is
$\Sigma_{2,T}$ and the set of cubic splines is
$\Sigma_{4,T}$. The first step towards finding a basis for
$\Sigma_{k,T}$ is to determine the dimension, which we can
say is finite as the grid $T$ is finite.
\begin{proposition}[Spline function space dimension {\cite[IX (44)]{db01}}]
	Let $k,\ell,m,p \in \N$ with $k \le \ell - 1$, $p \in \{m,\dots,\ell-1\}$ and
	$T := {\{ \tau_i \}}_{i = 0}^{\ell - 1}$
	with $\tau_0 = \dots = \tau_m \le \dots \le \tau_p = \dots = \tau_{\ell - 1}$.
	$\Sigma_{k,T}$ is a real vector space with
	\begin{equation*}
		\dim(\Sigma_{k,T}) = k + \ell - 2.
	\end{equation*}
\end{proposition}
\begin{proof}
	That $\Sigma_{k,T}$ is a real vector space follows
	directly from the fact that
	$\mathcal{C}^{k-2}([\tau_0,\tau_{\ell-1}])$
	and $\Pi_k$ are real vector spaces.
	\par
	We find the dimension by constructing an arbitrary
	$s \in \Sigma_{k,T}$. On the first
	piecewise interval $[\tau_0,\tau_1)$ we know that $s$
	is in $\Pi_k$, i.e.\ a polynomial of order $k$ and
	thus with $k$ degrees of freedom. We also have $k$ degrees of
	freedom in the subsequent piecewise interval $[\tau_1,\tau_2)$,
	but require that $s \in \mathcal{C}^{k-2}([\tau_0,\tau_{\ell-1}])$.
	We thus need to demand the $k - 1$ continuity conditions
	\begin{equation*}
		\forall_{i \in \{ 0,\dots,k-2 \}} \colon
		\left. s^{(i)} \right\rvert_{[\tau_0,\tau_1)}(\tau_1) =
		\left. s^{(i)} \right\rvert_{[\tau_1,\tau_2)}(\tau_1),
	\end{equation*}
	leaving $1$ degree of freedom for the interval $[\tau_1,\tau_2)$.
	This holds iteratively for all $\ell - 2$ intervals
	$[\tau_1,\tau_2),\dots,[\tau_{\ell-2},\tau_{\ell-1})$,
	yielding in total $k+\ell-2$ degrees of freedom corresponding
	to the dimension of $\Sigma_{k,T}$.
\end{proof}
Now that we've explored the set of spline functions a bit, we
know that a basis for this vector space needs to have $k+\ell-2$
elements. If we for a moment take a step back and think of an
iterative scheme to construct elements of $\Sigma_{k,T}$
starting with $\Sigma_{1,T}$ (piecewise constant splines), the underlying idea
is to start with piecewise constant
functions for $k=1$, namely indicator functions, and define higher order
splines recursively in such a way that we satisfy piecewise
continuity.
\begin{definition}[Indicator function]\label{def:indicator_function}
	Let $A \subseteq X$. The \emph{indicator function}
	$\indicator_A : X \to \{0,1\}$ is defined as
	\begin{equation*}
		\indicator_A(x) :=
		\begin{cases}
			1 & x \in A \\
			0 & x \notin A.
		\end{cases}
	\end{equation*}
\end{definition}
\begin{definition}[B-spline {\cite[IX (13)]{db01}}]\label{def:b-spline}
	Let $k,\ell,m,p \in \N$ with $k \le \ell - 1$, $p \in \{m,\dots,\ell-1\}$ and
	$T := {\{ \tau_i \}}_{i = 0}^{\ell - 1}$
	with $\tau_0 = \dots = \tau_m \le \dots \le \tau_p = \dots = \tau_{\ell - 1}$.
	The \emph{B-spline} of order $k$ in $\tau_i$
	with $i \in \{ 0,\dots,(\ell - 1) - k \}$
	is defined for $k = 1$ as
	\begin{equation*}
		B_{i,1,T}(t) :=
		\begin{cases}
			\indicator_{[\tau_i,\tau_{i+1})}(t) & i < (\ell - 1) - 1 \\
			\indicator_{[\tau_i,\tau_{i+1}]}(t) & i = (\ell - 1) - 1
		\end{cases}
	\end{equation*}
	and recursively for $k > 1$ as
	\begin{equation*}
		B_{i,k,T}(t) :=
		\frac{t-\tau_i}{\tau_{i+k-1} - \tau_{i}}
		B_{i,k-1,T}(t) +
		\frac{\tau_{i+k}-t}{\tau_{i+k}-\tau_{i+1}}
		B_{i+1,k-1,T}(t).
	\end{equation*}
\end{definition}
This recursive definition, also known as the
\textsc{de Boor}-\textsc{Cox}-\textsc{Mansfield} recursion
formula, not only gives shape to the
concept that has been discussed up until now, but also
provides a convenient way to efficiently evaluate
B-splines recursively, making it especially suitable
for numerical implementations.
\par
As a remark: When the knots $\tau_i$ and
$\tau_{i+1}$ coincide, it holds for the indicator function
$\indicator_{[\tau_i,\tau_{i+1})} \equiv 0$, meaning the respective summand in
the recursive formula of Definition~\ref{def:b-spline} disappears. Without this
knowledge, one might be tempted to assume that we are hitting a case of zero
divided by zero in its coefficient.
\begin{proposition}[B-spline properties]\label{prop:b-spline-properties}
	Let $k,\ell,m,p \in \N$ with $k \le \ell - 1$, $p \in \{m,\dots,\ell-1\}$ and
	$T := {\{ \tau_i \}}_{i = 0}^{\ell - 1}$
	with $\tau_0 = \dots = \tau_m \le \dots \le \tau_p = \dots = \tau_{\ell - 1}$.
	It holds that
	\begin{enumerate}
		\item{%
			$\supp(B_{i,k,T}) = [\tau_i, \tau_{i+k}]$,
		}
		\item{%
			$\forall_{t \in [\tau_0,\tau_{\ell-1}]} \colon B_{i,k,T}(t) \geq 0$,
		}
		\item{%
			$\forall_{i \in \{ 0,\dots,\ell-2 \}} \colon \left. B_{i,k,T} \right\rvert_{[\tau_i,\tau_{i+1}]} \in \Pi_k$,
		}
		\item{%
			$\forall_{i \in \{ 0,\dots,(\ell-1)-k \}} \colon
			B_{i,k,T} \in \mathcal{C}^{k-2}([\tau_0,\tau_{\ell-1}])$.
		}
	\end{enumerate}
\end{proposition}
\begin{proof}
	See \cite[IX (20)]{db01}.
\end{proof}
Another interesting property is that the evaluation of derivatives is also
recursive in nature, similar to the
\textsc{de Boor}-\textsc{Cox}-\textsc{Mansfield} recursion formula.
\begin{proposition}[B-spline derivatives]
	Let $k,\ell,m,p \in \N$ with $k \le \ell - 1$, $p \in \{m,\dots,\ell-1\}$ and
	$T := {\{ \tau_i \}}_{i = 0}^{\ell - 1}$
	with $\tau_0 = \dots = \tau_m \le \dots \le \tau_p = \dots = \tau_{\ell - 1}$.
	It holds that
	\begin{equation*}
		B'_{i,k,T}(t) = (k-1) \cdot
		\left(
			\frac{B_{i,k-1,T}(t)}{\tau_{i+k-1}-\tau_i} -
			\frac{B_{i+1,k-1,T}(t)}{\tau_{i+k}-\tau_{i+1}}
		\right).
	\end{equation*}
\end{proposition}
\begin{proof}
	See \cite[X (8)]{db01}.
\end{proof}
The obvious advantage is that we can not only efficiently evaluate B-splines
themselves for a given grid, we can also do that for their derivatives, making
it possible to work with derivatives in a way not possible with other means
of modelling functions numerically as easily and effectively. This is due to the
recursive nature of B-splines, their compact support, smoothness and positivity, as we'll
also see later in this thesis.
\par
Before we can speak of B-splines as a basis though, we need to solve a
remaining issue. Figure~\ref{fig:b-splines-1} shows all possible B-splines
for varying $k$ and indicates the problem: The number $\ell-k$ of B-splines on the grid decreases for
increasing $k$, even though we want to have $k+\ell-2$ basis functions, a
number that is supposed to increase for increasing $k$.
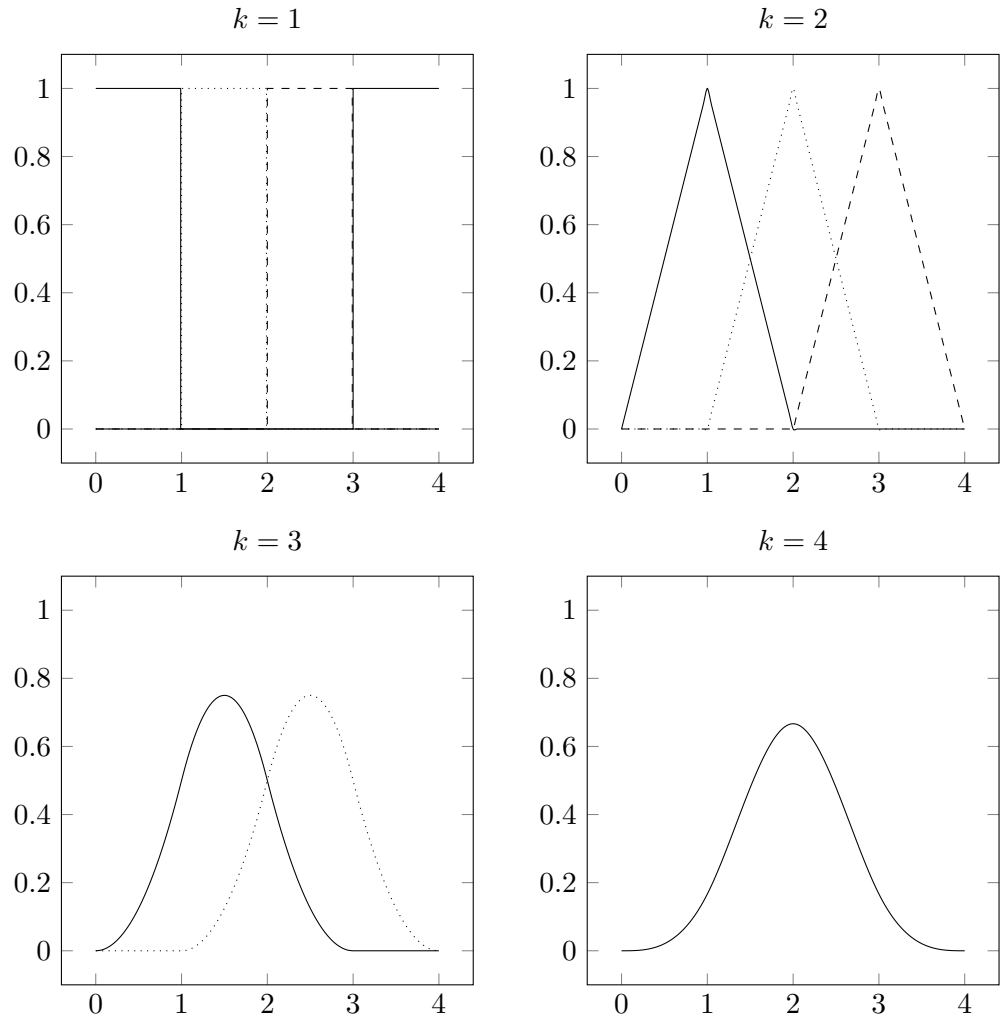
\begin{figure}[htbp]
	\centering
	\begin{tikzpicture}
		\begin{axis}	[
			name=plot1,
			height=7cm,width=7cm,
			title={$k=1$},
			ymin = -0.1,
			ymax = 1.1,
		]
		\addplot[const plot, no marks] coordinates {(0,1) (0.99,1) (0.99,0) (4,0)};
		\addplot[const plot, no marks, dotted] coordinates {(0,0) (1,0) (1,1) (1.99,1) (1.99,0) (4,0)};
		\addplot[const plot, no marks, dashed] coordinates {(0,0) (2,0) (2,1) (2.99,1) (2.99,0) (4,0)};
		\addplot[const plot, no marks] coordinates {(0,0) (3,0) (3,1) (4,1)};
		\end{axis}
		\begin{axis}	[
			name=plot2,
			height=7cm,
			width=7cm,
			title={$k=2$},
			at={($(plot1.east)+(1.5cm,0)$)},
			anchor=west,
			ymin = -0.1,
			ymax = 1.1,
		]
			\addplot [mark=none, smooth] table
				[x=t, y=B_1_2, col sep=comma]
				{examples/spline.data/spline-2.csv};
			\addplot [mark=none, smooth, dotted] table
				[x=t, y=B_2_2, col sep=comma]
				{examples/spline.data/spline-2.csv};
			\addplot [mark=none, smooth, dashed] table
				[x=t, y=B_3_2, col sep=comma]
				{examples/spline.data/spline-2.csv};
		\end{axis}
		\begin{axis}	[
			name=plot3,
			height=7cm,width=7cm,
			title={$k=3$},
			at={($(plot1.south)+(0,-1.5cm)$)},
			anchor=north,
			ymin = -0.1,
			ymax = 1.1,
		]
			\addplot [mark=none, smooth] table
				[x=t, y=B_2_3, col sep=comma]
				{examples/spline.data/spline-3.csv};
			\addplot [mark=none, smooth, dotted] table
				[x=t, y=B_3_3, col sep=comma]
				{examples/spline.data/spline-3.csv};
		\end{axis}
		\begin{axis}	[
			name=plot4,
			height=7cm,
			width=7cm,
			title={$k=4$},
			at={($(plot3.east)+(1.5cm,0)$)},
			anchor=west,
			ymin = -0.1,
			ymax = 1.1,
		]
			\addplot [mark=none, smooth] table
				[x=t, y=B_3_4, col sep=comma]
				{examples/spline.data/spline-4.csv};
		\end{axis}
	\end{tikzpicture}
	\caption{Plots of all $\ell - k$ B-splines $B_{i,k,T}(t)$ for all
	possible $k \le \ell - 1$ with
	the knot vector $T = \{0,1,2,3,4\}$ (hence $\ell = 5$)
	and $i\in\{ 0,\dots,(\ell-1) - k \}$ (from left to right).
	With each increase of $k$ the number of B-spline-functions is reduced by one.
	}
	\label{fig:b-splines-1}
\end{figure}
The solution is to just extend the knot vector in such a way that we conveniently
match the dimension of the spline function space, yielding the
\begin{definition}[Extended knot vector]\label{def:extended_knot_vector}
	Let $k,\ell \in \N$ with $k \le \ell - 1$ and
	$T := {\{ \tau_i \}}_{i = 0}^{\ell - 1}$
	with $\tau_0 < \dots < \tau_{\ell - 1}$.
	The extended knot vector $\Delta_k(T)$ is defined
	with $n := k + \ell - 2$ as
	\begin{equation*}
		\Delta_k(T) \ni: \delta_i :=
		\begin{cases}
			\tau_0 & i \in \{0,\dots,k-1\} \\
			\tau_{i-k+1} & i \in \{k,\dots,k+\ell-3\} = \{k,\dots,n-1\} \\
			\tau_{\ell-1} & i \in \{k+\ell-2,\dots,2\cdot k + \ell - 3\} = \{n,\dots,n+k-1\}.
		\end{cases}
	\end{equation*}
\end{definition}
Intuitively, we repeat the first and last knot $k$ times, and if we take a look
at the resulting plots in Figure~\ref{fig:b-splines-2} we see that the number of
B-splines on the extended grid matches the number of necessary basis functions
for the spline function space. Granted, this argument does not yet prove that these
B-splines based on the extended knot vector form a basis, but it should help
to understand the motivation behind this step before we do that in the following
\begin{figure}[htbp]
	\centering
	\begin{tikzpicture}
		\begin{axis}	[
			name=plot1,
			height=7cm,width=7cm,
			title={$k=1$},
			ymin = -0.1,
			ymax = 1.1,
		]
		\addplot[const plot, no marks] coordinates {(0,1) (0.99,1) (0.99,0) (4,0)};
		\addplot[const plot, no marks, dotted] coordinates {(0,0) (1,0) (1,1) (1.99,1) (1.99,0) (4,0)};
		\addplot[const plot, no marks, dashed] coordinates {(0,0) (2,0) (2,1) (2.99,1) (2.99,0) (4,0)};
		\addplot[const plot, no marks] coordinates {(0,0) (3,0) (3,1) (4,1)};
		\end{axis}
		\begin{axis}	[
			name=plot2,
			height=7cm,
			width=7cm,
			title={$k=2$},
			at={($(plot1.east)+(1.5cm,0)$)},
			anchor=west,
			ymin = -0.1,
			ymax = 1.1,
		]
			\addplot [mark=none, smooth] table
				[x=t, y=B_0_2, col sep=comma]
				{examples/spline.data/spline-2.csv};
			\addplot [mark=none, smooth, dotted] table
				[x=t, y=B_1_2, col sep=comma]
				{examples/spline.data/spline-2.csv};
			\addplot [mark=none, smooth, dashed] table
				[x=t, y=B_2_2, col sep=comma]
				{examples/spline.data/spline-2.csv};
			\addplot [mark=none, smooth] table
				[x=t, y=B_3_2, col sep=comma]
				{examples/spline.data/spline-2.csv};
			\addplot [mark=none, smooth, dotted] table
				[x=t, y=B_4_2, col sep=comma]
				{examples/spline.data/spline-2.csv};
		\end{axis}
		\begin{axis}	[
			name=plot3,
			height=7cm,width=7cm,
			title={$k=3$},
			at={($(plot1.south)+(0,-1.5cm)$)},
			anchor=north,
			ymin = -0.1,
			ymax = 1.1,
		]
			\addplot [mark=none, smooth] table
				[x=t, y=B_0_3, col sep=comma]
				{examples/spline.data/spline-3.csv};
			\addplot [mark=none, smooth, dotted] table
				[x=t, y=B_1_3, col sep=comma]
				{examples/spline.data/spline-3.csv};
			\addplot [mark=none, smooth, dashed] table
				[x=t, y=B_2_3, col sep=comma]
				{examples/spline.data/spline-3.csv};
			\addplot [mark=none, smooth, dotted] table
				[x=t, y=B_3_3, col sep=comma]
				{examples/spline.data/spline-3.csv};
			\addplot [mark=none, smooth, dashed] table
				[x=t, y=B_4_3, col sep=comma]
				{examples/spline.data/spline-3.csv};
			\addplot [mark=none, smooth] table
				[x=t, y=B_5_3, col sep=comma]
				{examples/spline.data/spline-3.csv};
		\end{axis}
		\begin{axis}	[
			name=plot4,
			height=7cm,
			width=7cm,
			title={$k=4$},
			at={($(plot3.east)+(1.5cm,0)$)},
			anchor=west,
			ymin = -0.1,
			ymax = 1.1,
		]
			\addplot [mark=none, smooth] table
				[x=t, y=B_0_4, col sep=comma]
				{examples/spline.data/spline-4.csv};
			\addplot [mark=none, smooth, dotted] table
				[x=t, y=B_1_4, col sep=comma]
				{examples/spline.data/spline-4.csv};
			\addplot [mark=none, smooth, dashed] table
				[x=t, y=B_2_4, col sep=comma]
				{examples/spline.data/spline-4.csv};
			\addplot [mark=none, smooth] table
				[x=t, y=B_3_4, col sep=comma]
				{examples/spline.data/spline-4.csv};
			\addplot [mark=none, smooth, dotted] table
				[x=t, y=B_4_4, col sep=comma]
				{examples/spline.data/spline-4.csv};
			\addplot [mark=none, smooth, dashed] table
				[x=t, y=B_5_4, col sep=comma]
				{examples/spline.data/spline-4.csv};
			\addplot [mark=none, smooth] table
				[x=t, y=B_6_4, col sep=comma]
				{examples/spline.data/spline-4.csv};
		\end{axis}
	\end{tikzpicture}
	\caption{Plots of all $k + \ell - 2$ B-splines $B_{i,k,\Delta_k(T)}(t)$ for all possible $k \le \ell - 1$ with
	the extended knot vector $\Delta_k(T)$ with $T = \{0,1,2,3,4\}$ (hence $\ell = 5$)
	and $i\in\{ 0,\dots,(\ell-1) - k \}$ (from left to right). With
	each increase of $k$ the number of B-spline-functions increases by one, as desired.
	}
	\label{fig:b-splines-2}
\end{figure}
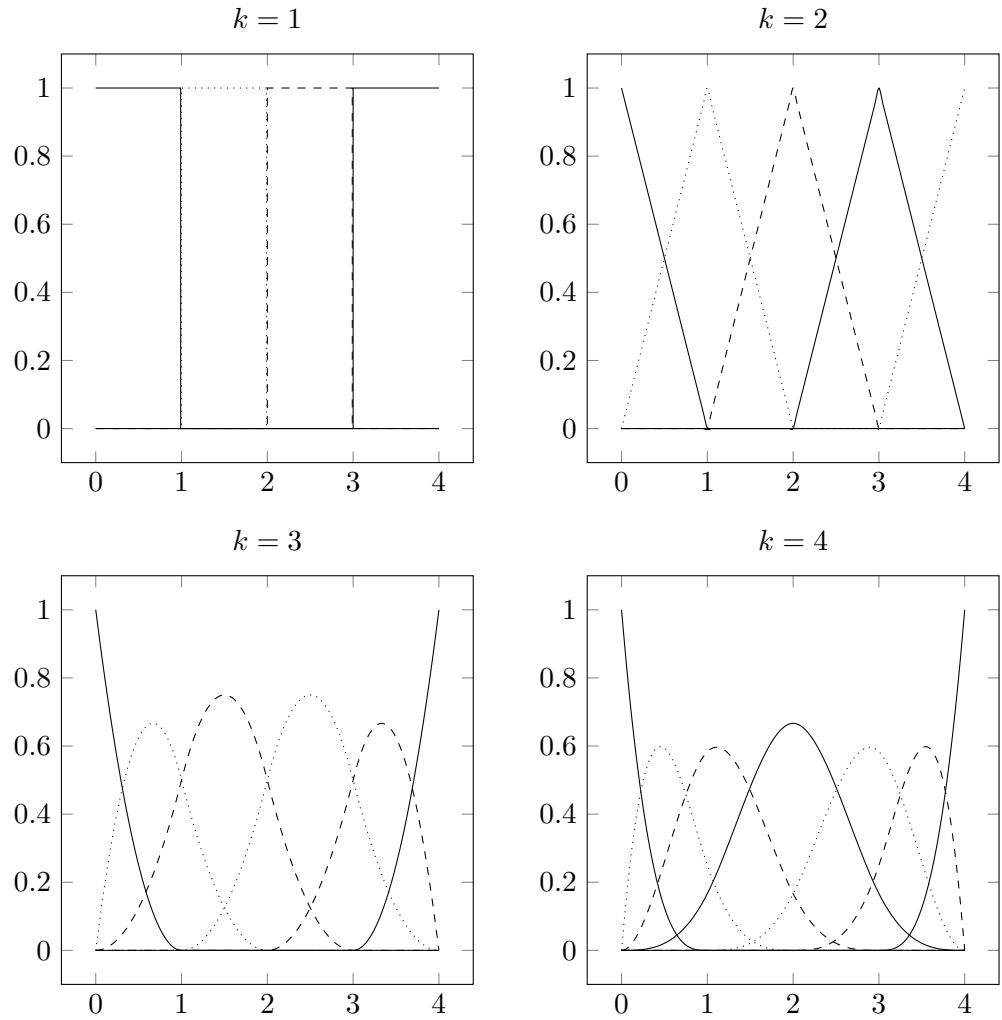
\begin{theorem}[\textsc{Curry}-\textsc{Schoenberg}]
	\label{thm:curry-schoenberg}
	Let $k,\ell \in \N$ with $k \le \ell - 1$ and
	$T := {\{ \tau_i \}}_{i = 0}^{\ell - 1}$
	with $\tau_0 < \dots < \tau_{\ell - 1}$.
	It holds with $n:=k+\ell-2$ that
	\begin{equation*}
		\Sigma_{k,T} = \spn\left(
			\left\{
				B_{0,k,\Delta_k(T)},\dots,B_{n-1,k,\Delta_k(T)}
			\right\}
		\right).
	\end{equation*}
\end{theorem}
\begin{proof}
	See \cite[IX (44)]{db01}.
\end{proof}
With this knowledge we have reached our goal and found
a basis for the spline function space. Given
$\Sigma_{k,T}$ is a real vector space, it makes sense
to define a mapping between it and coefficient
vectors for the B-spline basis. In the following segment
we make use of the results in Chapter~\ref{ch:funcspaceopt} on
the function space order $\preceq$ and
order-preserving isomorphisms.
\begin{definition}[Coefficient spline mapping]\label{def:coefficient_spline_mapping}
	Let $k,\ell \in \N$ with $k \le \ell - 1$,
	$T := {\{ \tau_i \}}_{i = 0}^{\ell - 1}$
	with $\tau_0 < \dots < \tau_{\ell - 1}$,
	$n := k + \ell -2$ and
	$\boldsymbol{s} = (s_0,\dots,s_{n-1}) \in \R^n$.
	The \emph{coefficient spline mapping}
	$\B_{k,T}:(\R^n, \le) \to (\Sigma_{k,T},\preceq)$ is defined as
	\begin{equation*}
		\B_{k,T}(\boldsymbol{s}) := \sum_{i=0}^{n-1} s_i \cdot B_{i,k,\Delta_k(T)}.
	\end{equation*}
\end{definition}
This mapping is both an isomorphism and preserves order, which we prove in
the following
\begin{proposition}
	Let $k,\ell \in \N$ with $k \le \ell - 1$ and
	$T := {\{ \tau_i \}}_{i = 0}^{\ell - 1}$
	with $\tau_0 < \dots < \tau_{\ell - 1}$.
	$\B_{k,T}$ is an order isomorphism of ordered vector-spaces (see
	Definition~\ref{def:oioovs}).
\end{proposition}
\begin{proof}
	This follows directly from Theorem~\ref{thm:curry-schoenberg} and
	$B_{i,k,\Delta_k(T)} \succeq 0$ for $i \in \{0,\dots,n:=k+\ell-2\}$
	mentioned in Proposition~\ref{prop:b-spline-properties}.
\end{proof}
Setting the details aside, what one can take away from this result is 
that manipulations of B-spline functions can equivalently be expressed
in terms of manipulations of their basis coefficients. In the context
of optimization problems considered in Chapter~\ref{ch:emd_analysis},
this enables us to formulate optimization problems in terms of 
B-spline basis coefficients.
\par
From the numerical perspective, we want to find quality conditions with which
we can compare two function space bases. One such aspect is orthogonality, which will
be elaborated as follows.
Consider $\R^2$ again with the standard basis $\{ {(1,0)}^T, {(0,1)}^T \}$.
This is an example for a so-called \enquote{orthogonal basis}, as these
vectors are orthogonal to each other with regard to
the \textsc{Euclid}ean inner product. In turn, this means that in a basis
decomposition, each basis vector only affects one entry of the resulting
vector. In function spaces, which are also vector spaces, bases can also
be orthogonal with regard to an inner product. A more general approach
though is the concept of a basis to be \enquote{locally linearly independent}.
This means that each basis function only affects a small
area of the interval the function operates on (i.e.\ the function has
local support). Thus, in turn, if one seeks
to find fitting coefficients for each basis function to best approximate
a given set of discrete datapoints, a locally linearly independent basis ensures
that each coefficient is only affected by datapoints within that small
area (which corresponds to the support of each basis function). A closely related
concept is that of the well-conditioned basis, where we can relate the
norm of the basis coefficients with the norm of the resulting function.
\begin{proposition}[Well-conditioned basis]\label{prop:well_conditioned_basis}
	Let $k,\ell \in \N$ with $k \le \ell - 1$,
	$T := {\{ \tau_i \}}_{i = 0}^{\ell - 1}$
	with $\tau_0 < \dots < \tau_{\ell - 1}$,
	$n := k + \ell -2$ and
	$\boldsymbol{s} = (s_0,\dots,s_{n-1}) \in \R^n$.
	There exists $c_{k,2} \in (0,1)$ (only depending on $k$) such that
	\begin{equation*}
		c_{k,2} \cdot {\| \boldsymbol{s} \|}_2 \le
		{\left\| \B_{k,\Delta_k(T)}(\boldsymbol{s}) \right\|}_2
		\le {\| \boldsymbol{s} \|}_2.
	\end{equation*}
\end{proposition}
\begin{proof}
	It follows from \cite[XI (8)]{db01} that there exists $c_{k,\infty} \in (0,1)$
	with
	\begin{equation*}
		c_{k,\infty} \cdot {\| \boldsymbol{s} \|}_\infty \le
		{\left\| \B_{k,\Delta_k(T)}(\boldsymbol{s}) \right\|}_\infty
		\le {\| \boldsymbol{s} \|}_\infty.
	\end{equation*}
	Given that ${\left\| \B_{k,\Delta_k(T)}(\boldsymbol{s}) \right\|}_2 \le
	{\left\| \B_{k,\Delta_k(T)}(\boldsymbol{s}) \right\|}_\infty$ and all norms are
	equivalent on $\R^n$ the proposition follows.
\end{proof}
This implies that if there are small disturbances in the B-spline-coefficients
it only leads to small disturbances in the spline functions themselves,
further underlining hat B-splines are a good choice for numerical applications.
Another result of this proposition is that we can find an upper and lower bound
for the supremum norm of a given spline function by the supremum norm of its
coefficient vector.
\par
Another interesting observation is the
\begin{proposition}[partition of unity]\label{prop:partition_of_unity}
	Let $k,\ell \in \N$ with $k \le \ell - 1$ and
	$T := {\{ \tau_i \}}_{i = 0}^{\ell - 1}$
	with $\tau_0 < \dots < \tau_{\ell - 1}$.
	It holds with $n := k + \ell - 2$ that
	\begin{equation*}
		\B_{k,T}((1,\dots,1)) := \sum_{i=0}^{n-1} B_{i,k,\Delta_k(T)} = 1
	\end{equation*}
\end{proposition}
\begin{proof}
	See \cite[IX (36)]{db01}.
\end{proof}
\begin{remark}\label{rem:spline-shorthand}
	The particular form of the extended knot vector according
	to Definition~\ref{def:extended_knot_vector} is considered to be fixed
	in this thesis.
	To simplify notation and bring the focus on the topic at hand instead
	of technicalities we will write $\Sigma_{k}$, $B_{i,k}$ and
	$\B_{k}$ instead of $\Sigma_{k,T}$, $B_{i,k,\Delta_k(T)}$ and
	$\B_{k,\Delta_k(T)}$.
	We just fix the chosen spline knot vector $T$ and assume that in
	the contexts it is used in it has been well-chosen. This assumption
	is not hard to make, given when we prove statements with the general
	variables $k,\ell$ and $n := k + \ell - 2$ we do not lose generality
	even if we ignore $T$'s exact form.
\end{remark}
With the results of this section, we can take a look at other function
space bases and argue why they were not used in this thesis.
One possible choice are \enquote{orthogonal polynomials}. Orthogonality
simplifies data fitting, but it comes at the cost of
numerical behaviour with potentially high polynomial degrees.
The many possible choices of
orthogonal polynomials also require deeper analysis of the matter than
what fits within the scope of this thesis.
\par
Another possible alternative choice are \enquote{radial
basis functions} (RBF), which have been diversely explored in the
context of EMD (see for example \cite{yyj12} and \cite{lww13}).
They are the other extreme compared to orthogonal polynomials
in terms of orthogonality, as each basis function spans across
the entire interval. Additionally, they do not present a locally
linearly independent basis, which can be at the cost of numerical stability.
Given the many choices of radial functions it is also difficult to
evaluate the quality of each choice.
Numerically, due to their non-locality, they yield hard to handle
full rank matrices when used as function bases, which do not scale well
for larger problems. Due to the depth of this matter RBFs will not
be investigated further in this thesis.
\chapter{Empirical Mode Decomposition Model and Analysis}\label{ch:emd_analysis}
As already introduced in Chapter~\ref{ch:introduction} we may consider one
step of the empirical mode decomposition as an optimization problem over the set
of intrinsic mode functions. For a given signal $s$ the cost function
might relate to the amount of residual $r := s - u$ left for a given candidate
function $u$, yielding for instance a problem of the form
\begin{equation}\label{eq:emdop-naive}
	\begin{aligned}
		\inf_{u} \quad &
			{\| s - u \|}_2^2 \\
		\text{s.t.} \quad & u \text{\ IMF}.
	\end{aligned}
\end{equation}
In this chapter, we will only focus on optimization problems of this kind,
namely the extraction of a single IMF (that we formally introduce later) from an input signal using a
cost function of some kind. The EMD method follows by iteratively
using the residual of the previous step as the input signal for the next step.
\par
The objective of this chapter is to clarify what the set of intrinsic
mode functions is. During this process we generalize the concept for arbitrary cost
functions satisfying convex-likeness, which is a weaker form of convexity.
Our goal is to find useful properties for this underlying optimization
problem. This could bring useful results and be a step forward for
the theoretical analysis of the empirical mode decomposition and be
a guide for the development of new heuristic methods.
\section{Intrinsic Mode Functions}\label{sec:intrinsic_mode_functions}
The fundamental building blocks of the empirical
mode decomposition are intrinsic mode functions (IMFs) $u(t)$
of the form
\begin{equation}\label{eq:imf-base_form}
	u(t) := a(t) \cos(\phi(t)),
\end{equation}
where $a(t)$ and $\phi(t)$ satisfy certain conditions we will lay out
later. One can imagine an intrinsic mode function to be a wave of
varying frequency that is enveloped by a varying amplitude, as shown in
Figure~\ref{fig:example-imf}.
\begin{figure}[htbp]
	\centering
	\begin{subfigure}[c]{0.32\textwidth}
		\centering
		\begin{tikzpicture}
			\begin{axis}[
				name=plot1,
				height=4.5cm,width=4.5cm,
			]
				\addplot [domain=0:1, samples=200]{(1 + 10 * x^3) *
				cos(deg((130 - 40 * x^2) * x))};
			\end{axis}
		\end{tikzpicture}
		\subcaption{$u(t)$}
	\end{subfigure}
	\begin{subfigure}[c]{0.32\textwidth}
		\centering
		\begin{tikzpicture}
			\begin{axis}[
				name=plot1,
				height=4.5cm,width=4.5cm,
			]
				\addplot [domain=0:1, samples=200]{(1 + 10 * x^3)};
			\end{axis}
		\end{tikzpicture}
		\subcaption{$a(t)$}
	\end{subfigure}
	\begin{subfigure}[c]{0.32\textwidth}
		\centering
		\begin{tikzpicture}
			\begin{axis}[
				name=plot1,
				height=4.5cm,width=4.5cm,
			]
				\addplot [domain=0:1, samples=200]{(130 - 120 * x^2)};
			\end{axis}
		\end{tikzpicture}
		\subcaption{$\phi'(t)$}
	\end{subfigure}
	\caption{An intrinsic mode function $u(t) := a(t) \cdot \cos(\phi(t))$ with its
	instantaneous amplitude $a(t)$ and frequency $\phi'(t)$.}
	\label{fig:example-imf}
\end{figure}
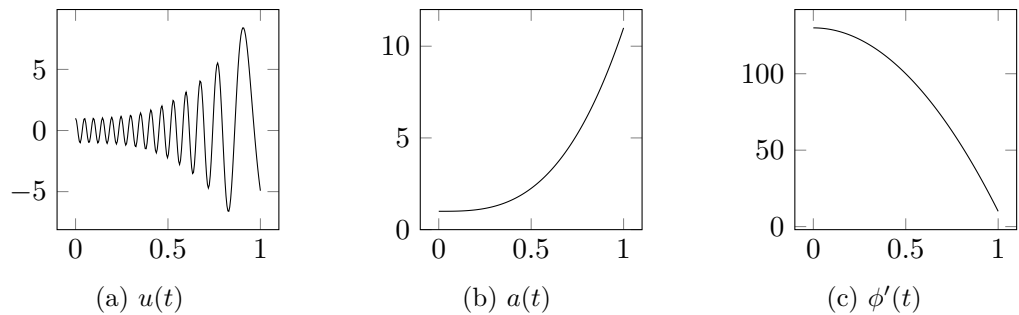
We can see that for given $a(t)$ and $\phi(t)$ the intrinsic mode
function $u(t)$ in Equation~(\ref{eq:imf-base_form}) is fully described.
It follows that the real objects of interest are $a(t)$ and $\phi(t)$,
especially in regard to conditions we want them to satisfy such that the
corresponding IMF has meaningful properties.
\par
The approach we take in this thesis is novel: We represent an IMF as
a function pair $(a(t),\phi(t))$ satisfying a set of IMF conditions instead
of defining an IMF as a function of analytical form $a(t) \cdot \cos(\phi(t))$,
where $a(t)$ and $\phi(t)$ have certain properties. The
crucial advantage of the new approach compared to the
classic one is that the components
$a(t)$ and $\phi(t)$ are \enquote{graspable} rather
than hidden within the IMF itself.
A central objective of this chapter is thus to find a way to extract
$a(t)$ and $\phi$ from an IMF $u(t)$. As we consider optimization
problems over IMFs we need to enforce the IMF conditions in some way,
which requires knowledge of $a(t)$ and $\phi(t)$. This
should not be dependent on such an extraction process until later.
\par
Given the central role of the pair $(a(t),\phi(t))$ for an IMF $u(t)$
we call it the \enquote{soul} of the intrinsic mode function and define
it as follows.
\begin{definition}[Intrinsic mode function soul (IMFS)]
	\label{def:imfs}
	Let $\mu_0, \mu_1, \mu_2 > 0$. The pair 
	$(a,\phi) \in \mathcal{C}^1(\R,\R) \times
	\mathcal{C}^2(\R,\R)$ is an \emph{intrinsic mode
	function soul (IMFS)} with characteristic $(\mu_0,\mu_1,\mu_2)$ if
	and only if
	\begin{align}
		0 &\preceq a \label{eq:imfs:1}\\
		\mu_0 &\preceq \phi' \label{eq:imfs:2}\\
		\left| a' \right| &\preceq \mu_1 \cdot
			\left| \phi' \right| \label{eq:imfs:3}\\
		\left| \phi'' \right| &\preceq \mu_2 \cdot
			\left| \phi' \right| \label{eq:imfs:4}
	\end{align}
	hold. We define the set of IMFSs as ${\mathcal{S}}_{\mu_0,\mu_1,\mu_2}$ and
	call $a$ the \emph{instantaneous amplitude}, $\phi$ the
	\emph{instantaneous phase} and $\phi'$ the \emph{instantaneous
	frequency}.
\end{definition}
The above definition is not arbitrary. To put it into
context with physical reality and other publications, we give the following remarks.
\begin{remark}[Constraint motivations]
	The motivations for Equations (\ref{eq:imfs:1}) and
	(\ref{eq:imfs:2}) are to ensure that both instantaneous
	amplitude and phase have physical meaning, i.e.\ no negative
	amplitude and strictly positive frequency (as the derivative
	of the phase is the frequency). We introduce $\mu_0$ rather
	than demanding $0 \preceq \phi'$ so we are only dealing
	with non-strict inequality constraints (i.e.\ $\preceq$ instead
	of $\prec$).
	\par
	Equations (\ref{eq:imfs:3}) and (\ref{eq:imfs:4}) are there
	to ensure a slowly varying instantaneous amplitude and
	frequency respectively, as we want each intrinsic mode function
	that is extracted to remain within a certain scope.
	The exact nature of this scope depends on the
	type of application and can be shaped with the parameters.
\end{remark}
\begin{remark}[relationship with \cite{dlw11}]
	Definition~\ref{def:imfs} is based on
	\cite[Definition 3.1]{dlw11}, but generalizes it in certain
	aspects by introducing an arbitrary lower bound $\mu_0$ for the frequency and generalizing the single parameter $\epsilon$
	(called \enquote{accuracy}) in \cite{dlw11} into two
	separate parameters $\mu_1$ and $\mu_2$ that are part of the
	characteristic. The latter generalization allows a more fine-grained
	control of the rate of change of the amplitude and frequency
	respectively over time without any trade-offs, which is
	further elaborated in Remark~\ref{rem:imf-characteristic}.
	\par
	One part of the definition, namely that the infimum of $\phi'$ shall be bounded, was left out given there is no practical reason for
	this condition. Given $\phi \in \C^2(\R,\R)$ this would imply
	that $\phi'(t)$ should have finite limits for $t \to \pm \infty$.
	As we can see for instance with the IMF $\cos(t^2)$ with soul $(a,\phi) = (1,t^2)$, and $\phi'(t) = 2t$ in particular, there would
	be no such simple way to represent this simple case with the
	definition given in \cite{dlw11}.
\end{remark}
\begin{remark}[use of modulus]
	Equations (\ref{eq:imfs:3}) and (\ref{eq:imfs:4}) state
	$\left| \phi' \right|$ instead of $\phi'$, even though
	$\phi' \preceq \mu_0 \prec 0$ is guaranteed by Equation
	(\ref{eq:imfs:2}). This is for reasons of consistency
	with the literature (e.g.\ \cite{dlw11}) that chooses the same
	form despite the guarantee.
\end{remark}
For examples and further motivation on the IMF characteristic,
which is more fitting in the chapters on application, see Subsection~\ref{subsec:ethos-imf-characteristic} and Section~\ref{sec:emd_examples}.
\par
The pair $(a,\phi)$ itself may perfectly represent the IMF properties,
but we also need to evaluate the IMF to, for instance,
assess how much residual $r(t) = s(t) - u(t)$ is left with
a given candidate pair $(a,\phi)$. For this purpose, we define
the evaluation as an operator on ${\mathcal{S}}_{\mu_0,\mu_1,\mu_2}$ as follows.
\begin{definition}[Intrinsic mode function operator]\label{def:imf-operator}
	Let $\mu_0, \mu_1, \mu_2 > 0$ and $(a,\phi) \in {\mathcal{S}}_{\mu_0,\mu_1,\mu_2}$.
	The \emph{intrinsic mode function operator} is defined as
	\begin{equation*}
		\I[a,\phi](t) := a(t) \cdot \cos(\phi(t)).
	\end{equation*}
\end{definition}
One has to keep in mind that for a given IMF, there may be multiple
souls that can generate it. This is elaborated in the following
\begin{remark}[IMF soul non-uniqueness]
	Consider the IMF
	\begin{equation*}
		u(t) := (1 + 2 t) \cdot \cos(2\pi t) \cdot \cos(4\pi t)
	\end{equation*}
	on the interval $[0,1]$. This can either be interpreted
	as $u = \I[a,\phi]$ with
	\begin{align*}
		a(t) &:= (1 + 2 t) \cdot \cos(2\pi t),\\
		\phi(t) &:= 4\pi t,
	\end{align*}
	or as $u = \I[\tilde{a},\tilde{\phi}]$ with
	\begin{align*}
		\tilde{a} &:= (1 + 2 t) \cdot \cos(4\pi t),\\
		\tilde{\phi} &:= 2\pi t.
	\end{align*}
	One can possibly exclude such double cases by varying the
	parameters $\mu_0$, $\mu_1$ and $\mu_2$ of the
	IMFS set (see Definition~\ref{def:imfs}), excluding
	possible other candidates by varying the boundaries,
	but this is a heuristical approach and won't be further
	elaborated here.
\end{remark}
Another important aspect is one of information theoretical
nature.
\begin{remark}[Information theory]
	If you consider the information content going
	from $(a,\phi)$ to the IMF $\I(a,\phi)$, the IMF
	operator may present cases where information is
	destroyed. In other words, in such a case it
	becomes impossible to reconstruct $a$ or $\phi$
	from an IMF that was previously generated from them.
	Two examples of such cases can be found in
	Figure~\ref{fig:example-imf-information_loss}.
	They are almost exclusively due to the fact that
	amplitude and phase vary almost equally fast.
	\begin{figure}[htbp]
		\centering
		\begin{subfigure}[c]{0.49\textwidth}
			\centering
			\begin{tikzpicture}
				\begin{axis}[
					name=plot1,
					height=6cm,width=8cm,
				]
					\addplot [domain=0:1, samples=200]{
						(2+cos(deg(2*pi*8*x))) * cos(deg(2*pi*10*x))
					};
					\addplot [domain=0:1, samples=200, densely dashed]{
						(2+cos(deg(2*pi*8*x)))
					};
					\addplot [domain=0:1, samples=200, densely dashed]{
						-(2+cos(deg(2*pi*8*x)))
					};
 				\end{axis}
			\end{tikzpicture}
			\subcaption{$u(t)$}
		\end{subfigure}
		\begin{subfigure}[c]{0.49\textwidth}
			\centering
			\begin{tikzpicture}
				\begin{axis}[
					name=plot1,
					height=6cm,width=8cm,
				]
					\addplot [domain=0:1, samples=200]{
						(3+cos(deg(2*pi*6*x))+cos(deg(2*pi*8*x))) *
						cos(deg(2*pi*8*x))
					};
					\addplot [domain=0:1, samples=200, densely dashed]{
						(3+cos(deg(2*pi*6*x))+cos(deg(2*pi*8*x)))
					};
					\addplot [domain=0:1, samples=200, densely dashed]{
						-(3+cos(deg(2*pi*6*x))+cos(deg(2*pi*8*x)))
					};
				\end{axis}
			\end{tikzpicture}
			\subcaption{$\tilde{u}(t)$}
		\end{subfigure}
		\caption{%
		Two intrinsic mode functions $u$ and $\tilde{u}$ (solid), where
		each is unable to reflect its
		instantaneous amplitude (dashed) because its rate of
		change is almost equal to the instantaneous frequency.
		This example was adapted from \cite[Figure~2]{hyy15}.}
		\label{fig:example-imf-information_loss}
	\end{figure}
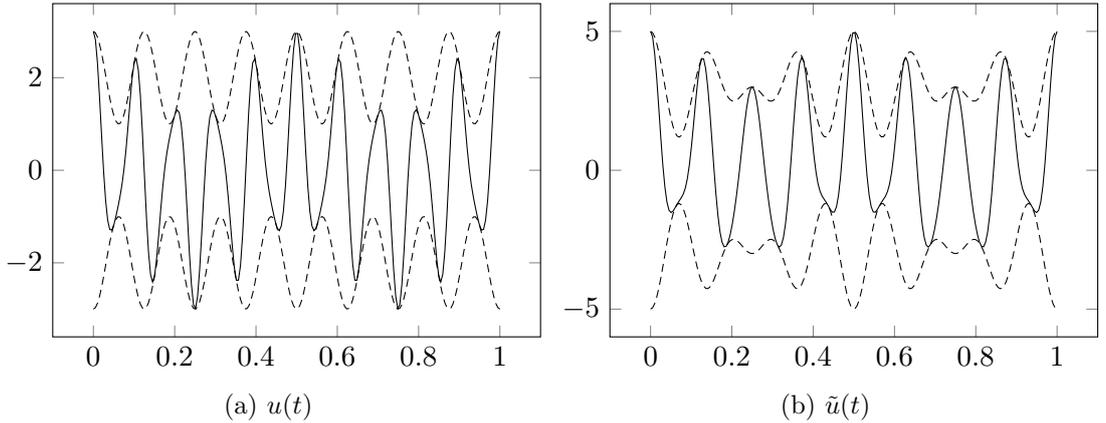
	\par
	One can deduce from this observation that when
	extracting $a$ and $\phi$ from an IMF, it is likely that
	$\mu_1 \ll 1$ and $\mu_2 \ll 1$ hold, i.e.\ 
	that amplitude and phase vary slowly relative to
	each other (and not destroy each other's information content).
	However, as given in the example in Figure~\ref{fig:example-imf-information_loss},
	the ground truth can of course
	still contain more information than what remains after applying
	the intrinsic mode function operator to it.
	For natural inputs, the ground truth is not known.
	Thus, such cases are more of a philosophical aspect
	of this derivation and reflect the analytical nature
	of the set of IMF souls compared to the practical
	nature of the IMF itself.
	\par
	See \cite{hyy15} for a discussion of more of such
	pathological cases and \cite[Definition~3.1]{dlw11}
	for further reading. We note here though that any
	attempt to \enquote{fix} such cases makes little
	sense given one can not create information from
	nothing.
\end{remark}
We will now focus on the parameters $\mu_0$, $\mu_1$ and $\mu_2$
of the set of IMF souls and the role they play when judging
the quality of an extracted IMF.
\begin{remark}[Characteristic]\label{rem:imf-characteristic}
	The IMFS characteristic defined in Definition~\ref{def:imfs} is
	a generalization of the IMF accuracy introduced in
	\cite[Definition~3.1]{dlw11}, which only employs a single
	parameter $\varepsilon > 0$ for both $\mu_1$ and $\mu_2$. This
	is a disadvantageous limitation for approaches aiming to
	only have slowly varying amplitude or frequency while not
	particularily limiting the behaviour of the respective other.
	The parameter $\mu_0$ was added as a lower frequency bound. This
	conveniently eliminates the strict inequality constraint
	$0 \prec \phi'$ from the original definition, which would
	complicate the theoretical analysis unnecessarily with no
	added benefit.
	\par
	Naturally, for a given IMFS we can calculate the characteristic
	$(\mu_0,\mu_1,\mu_2)$ with
	\begin{align*}
		\mu_0 &= \inf_{t \in \R} \phi'(t), \\
		\mu_1 &= \sup_{t \in \R} \left|\frac{a'(t)}{\phi'(t)}\right|, \\
		\mu_2 &= \sup_{t \in \R} \left|\frac{\phi''(t)}{\phi'(t)}\right|.
	\end{align*}
	This makes it possible to assess its relative quality and ascertain conditions
	on its characteristic. This idea is later further explored in the toolbox
	(see Subsection~\ref{subsec:ethos-imf-characteristic}).
\end{remark}
Now that we have formally defined intrinsic mode functions and put
them into the context of physical reality and \cite{dlw11},
we return to our original optimization problem in
Equation~(\ref{eq:emdop-naive}) and express it in terms of our newly
defined set of IMFSs. We obtain the following optimization problem.
\begin{equation}\label{eq:emdop-reformed}
	\begin{aligned}
		\min_{u} \quad &
			{\| s - \I[a,\phi] \|}_2^2 \\
		\text{s.t.} \quad & (a,\phi) \in {\mathcal{S}}_{\mu_0,\mu_1,\mu_2}
	\end{aligned}
\end{equation}
In the ideal case this would be a convex
optimization problem. This means that both the cost function and
the candidate set are convex (according to
Definitions \ref{def:convex_function} and \ref{def:convex_set})
and we have a global minimum.
Given we will later look at cost functions the first step is to see if
our candidate set is convex. We show that in the following
\begin{theorem}\label{thm:imfs-convex}
	Let $\mu_0, \mu_1, \mu_2 \ge 0$. ${\mathcal{S}}_{\mu_0,\mu_1,\mu_2}$ is convex
	according to Definition~\ref{def:convex_set}.
\end{theorem}
\begin{proof}
	Let $(a_1,\phi_1), (a_2,\phi_2) \in {\mathcal{S}}_{\mu_0,\mu_1,\mu_2}$ and
	$q \in [0,1]$. We define an element on the path between
	$(a_1,\phi_1)$ and $(a_2,\phi_2)$ as
	\begin{equation*}
		(a_\star,\phi_\star) := (a_1,\phi_1) + q \cdot
		\left[ (a_2,\phi_2) - (a_1,\phi_1) \right] =
		(1-q) \cdot (a_1,\phi_1) + q \cdot (a_2,\phi_2).
	\end{equation*}
	Component-wise, we obtain
	\begin{equation*}
		a_\star = (1-q) \cdot a_1 + q \cdot a_2
	\end{equation*}
	and
	\begin{equation*}
		\phi_\star = (1-q) \cdot \phi_1 + q \cdot \phi_2.
	\end{equation*}
	We now show that $(a_\star,\phi_\star)$ satisfies the conditions
	from Definition~\ref{def:imfs}:
	\begin{enumerate}
		\item{%
			We note from Equation~(\ref{eq:imfs:1}) that
			$a_1 \succeq 0$ and $a_2 \succeq 0$ and follow\\
			$a_\star = (1-q) \cdot a_1 + q \cdot a_2 \succeq
			(1-q) \cdot 0 + q \cdot 0 = 0$.
		}
		\item{%
			We note from Equation~(\ref{eq:imfs:2}) that
			$\phi'_1 \succeq \mu_0$ and $\phi'_2 \succeq \mu_0$ and follow\\
			$\phi'_\star = (1-q) \cdot \phi'_1 + q \cdot \phi'_2
			\succeq (1-q) \cdot \mu_0 + q \cdot \mu_0 = \mu_0$
		}
		\item{%
			We note from Equations~(\ref{eq:imfs:2}) and (\ref{eq:imfs:3}) that
			$\phi'_1 \succeq \mu_0 \succeq 0$, $\phi'_2 \succeq \mu_0 \succeq 0$,
			$|a'_1| \preceq \mu_1 \cdot |\phi'_1|$ and $|a'_2| \preceq \mu_1 \cdot |\phi'_2|$ and follow\\
			$|a'_\star| = |(1-q) \cdot a'_1 + q \cdot a'_2| \preceq
			(1-q) \cdot |a'_1| + q \cdot |a'_2| \preceq
			\mu_1 \cdot \left( (1-q) \cdot |\phi'_1| + q \cdot |\phi'_2| \right) \preceq
			\mu_1 \cdot | (1-q) \cdot \phi'_1 + q \cdot \phi'_2 | = \mu_1 \cdot |\phi'_\star|$
		}
		\item{%
			We note from Equations~(\ref{eq:imfs:2}) and (\ref{eq:imfs:4}) that
			$\phi'_1 \succeq \mu_0 \succeq 0$, $\phi'_2 \succeq \mu_0 \succeq 0$,
			$|\phi''_1| \preceq \mu_2 \cdot |\phi'_1|$ and $|\phi''_2| \preceq \mu_2 \cdot |\phi'_2|$ and follow\\
			$|\phi''_\star| = |(1-q) \cdot \phi''_1 + q \cdot \phi''_2| \preceq
			(1-q) \cdot |\phi''_1| + q \cdot |\phi''_2| \preceq
			\mu_2 \cdot \left( (1-q) \cdot |\phi'_1| + q \cdot |\phi'_2| \right) =
			\mu_2 \cdot |(1-q) \cdot \phi'_1 + q \cdot \phi'_2 | = \mu_2 \cdot |\phi'_\star| $
		}
	\end{enumerate}
	It follows that $(a_\star,\phi_\star) \in {\mathcal{S}}_{\mu_0,\mu_1,\mu_2}$
	and thus ${\mathcal{S}}_{\mu_0,\mu_1,\mu_2}$ is convex.
\end{proof}
Up to this point we have defined ${\mathcal{S}}_{\mu_0,\mu_1,\mu_2}$ as a set of
function pairs. Let us reconsider the results of Chapters
\ref{ch:b-splines} and \ref{ch:funcspaceopt}: We
introduced a way to relate functions to each other and showed
that the one-to-one relation of spline functions and their
B-spline basis coefficients preserves that order. It thus becomes logical
to use this relation and express intrinsic mode function souls as
a set of pairs of vectors in $\R^n$. Their entries correspond to B-spline
basis coefficients of the spline functions describing instantaneous amplitude
and phase.
\begin{definition}[Intrinsic mode spline function soul (IMSpFS)]
	\label{def:imspfs}
	Let $\mu_0, \mu_1, \mu_2 > 0$ and $k \ge 4$ (for derivability). The pair 
	$(\boldsymbol{a},\boldsymbol{\phi}) \in \R^n \times \R^n$ is an
	\emph{intrinsic mode spline function soul (IMSpFS)} if and only if
	\begin{align}
		0 &\preceq \B_{k}(\boldsymbol{a}) \label{eq:imspfs:1}\\
		\mu_0 &\preceq \B_{k}'(\boldsymbol{\phi}) \label{eq:imspfs:2}\\
		\left| \B_{k}'(\boldsymbol{a}) \right| &\preceq \mu_1 \cdot
			\left| \B_{k}'(\boldsymbol{\phi}) \right| \label{eq:imspfs:3}\\
		\left| \B_{k}''(\boldsymbol{\phi}) \right| &\preceq \mu_2 \cdot
			\left| \B_{k}'(\boldsymbol{\phi}) \right| \label{eq:imspfs:4}
	\end{align}
	hold. We define the set of IMSpFSs as $\boldsymbol{\mathcal{S}}_{\mu_0,\mu_1,\mu_2}$.
\end{definition}
It becomes apparent that by using this reformulation the
handling of IMF souls merely as vectors instead of function
pairs is much simpler. What follows from
Theorem~\ref{thm:imfs-convex} is that, given the
relation between B-splines and their coefficients is
order-preserving, the set of IMSpFS's is also convex.
\section{Cost Functions}\label{sec:cost_functions}
Having obtained the result in Theorem~\ref{thm:imfs-convex},
if we now find a convex cost function that meaningfully judges an intrinsic
mode function soul relative to an input signal we would have solved
the problem of building a convex EMD optimization problem. This is because
we have already shown that the set of IMF souls is convex. Together with
a convex cost function we would then obtain a convex optimization problem.
The search for such a convex cost function will not be within the scope of this
thesis as it might require adding more constraints to the set of IMFSs
or a completely different approach altogether.
Instead, we will take a look at cost functions from \cite{ph08}
and \cite{ph10} that are well-established and have a strong footing
within the classic EMD theory.
\par
We have until now only looked at the nature of intrinsic mode functions
and not how we can actually express which fits our input signal the best.
Each step of the empirical mode decomposition applies to an input signal $s(t)$,
which we want to split up into an intrinsic mode function $u(t)$ and residual
function $r(t)$. To determine the split we want to minimize the \enquote{cost}
a split-up of a signal $s(t)$ into an IMF $u(t)$ and residual $r(t)$ has. This
cost should be relative to the quality of extraction. There are obviously
many possible ways to define such an EMD cost function and we will explore
this topic in the following section.
\par
The final goal is to generalize the optimization problem in
Equation~\ref{eq:emdop-reformed} for an arbitrary cost function later.
\subsection{Canonical}
The simplest idea for an EMD cost function is to look at the residual, as it
corresponds to the classic EMD approach proposed in \cite{hsl+98} and is formally
used in \cite{ph08} and \cite{ph10}. This makes the residual approach the most
common idea for an EMD cost function in the literature. The smaller
the difference between the signal and intrinsic mode function, the less the cost.
That is because we have then extracted as much from the signal as possible.
Additionally, as we will later see in Section~\ref{subsec:nsp-leakage-cf}, it is the
basis for derived EMD cost functions in the context of more advanced separation techniques.
\begin{definition}[Canonical EMD cost function]\label{def:ccf}
	Let $\mu_0, \mu_1, \mu_2 > 0$ and $s \in \mathcal{C}^0(\R,\R)$.
	The \emph{canonical EMD cost function}
	$c_1[s] \colon {\mathcal{S}}_{\mu_0,\mu_1,\mu_2} \to \R$
	is defined as
	\begin{equation*}
		c_1[s](a,\phi) := {\left\| s - \I(a,\phi) \right\|}_2^2 =
		{\left\| s - a \cdot \cos(\phi) \right\|}_2^2.
	\end{equation*}
\end{definition}
Just as with the set of IMFSs we can also formulate the EMD cost function in terms
of B-splines. We do that by expressing it as a function over
$\boldsymbol{\mathcal{S}}_{\mu_0,\mu_1,\mu_2}$ (see Definition~\ref{def:imspfs}) instead of
${\mathcal{S}}_{\mu_0,\mu_1,\mu_2}$. This makes it possible to examine
its convexity as introduced in Chapter~\ref{ch:convexity_theory}.
\begin{definition}[Canonical spline EMD cost function]\label{def:cscf}
	Let $\mu_0, \mu_1, \mu_2 > 0$, $k \ge 4$ (for derivability) and $\boldsymbol{s}\in\R^n$. The
	\emph{canonical spline EMD cost function}
	$\boldsymbol{c}_1[\boldsymbol{s}] \colon \boldsymbol{\mathcal{S}}_{\mu_0,\mu_1,\mu_2} \to \R$
	is defined as
	\begin{equation*}
		\boldsymbol{c}_1[\boldsymbol{s}](\boldsymbol{a},\boldsymbol{\phi}) :=
		c_1[\B_{k}(\boldsymbol{s})](\B_{k}(\boldsymbol{a}),\B_{k}(\boldsymbol{\phi})).
	\end{equation*}
\end{definition}
According to the motivation laid out earlier, we want this cost function to be
convex.
\begin{proposition}\label{prop:cscf-convex}
	The canonical spline EMD cost function $\boldsymbol{c}_1[\boldsymbol{s}]$
	is not convex in $(\boldsymbol{a}, \boldsymbol{\phi})$
	according to Definition~\ref{def:convex_function}.
\end{proposition}
\begin{proof}
	We approach this proof by checking if the requirements of
	Theorem~\ref{thm:gershgorin-hadamard} hold for distinct partial
	derivatives for entries of $\boldsymbol{a}$ and $\boldsymbol{\phi}$.
	We begin with $\boldsymbol{a}$ and first calculate
	the entries of the Hessian matrix $H_{\boldsymbol{c}_1[\boldsymbol{s}]}(\boldsymbol{a})$,
	which means that we consider
	$\boldsymbol{c}_1[\boldsymbol{s}](\boldsymbol{a},\boldsymbol{\phi})$ to only
	vary in $\boldsymbol{a}$. We first note that it holds
	\begin{align*}
		\boldsymbol{c}_1[\boldsymbol{s}](\boldsymbol{a},\boldsymbol{\phi}) & =
			c_1[\B_{k}(\boldsymbol{s})]\!\left(\B_{k}(\boldsymbol{a}),\B_{k}(\boldsymbol{\phi})\right) \\
		& = {\left\| \B_{k}(\boldsymbol{s}) -
			\I\!\left(\B_{k}(\boldsymbol{a}),\B_{k}(\boldsymbol{\phi})\right) \right\|}_2^2 \\
		& = {\left\| \B_{k}(\boldsymbol{s}) - \B_{k}(\boldsymbol{a}) \cdot
			\cos\!\left(\B_{k}(\boldsymbol{\phi})\right)\right\|}_2^2 \\
		& = \int_{-\infty}^{\infty}
			{\left[
				\B_{k}(\boldsymbol{s}) - \left(
					\sum_{i=0}^{n-1} a_i \cdot B_{i,k}(t)
				\right) \cdot
				\cos\!\left(
					\sum_{i=0}^{n-1} \phi_i \cdot B_{i,k}(t)
				\right)
			\right]}^2 \mathrm{d}t
	\end{align*}
	and can deduce for $m,p \in \{ 0,\dots,n-1 \}$
	\begin{align*}
		\frac{\partial \boldsymbol{c}_1[\boldsymbol{s}]}{\partial a_m}(\boldsymbol{a},\boldsymbol{\phi}) &=
			\int_{-\infty}^{\infty}
			(-2) \cdot {\left[
				\B_{k}(\boldsymbol{s}) -
				\left(
					\sum_{i=0}^{n-1} a_i \cdot B_{i,k}(t)
				\right) \cdot
				\cos\!\left(
					\B_{k}(\boldsymbol{\phi})(t)
				\right)
			\right]} \cdot \\
		& \phantom{=\int_{-\infty}^{\infty}\,\,}
			B_{m,k}(t) \cdot
			\cos\!\left(
				\B_{k}(\boldsymbol{\phi})(t)
			\right)
			\mathrm{d}t \\
		&= \int_{-\infty}^{\infty}
			(-2) \cdot
			B_{m,k}(t) \cdot
			\cos\!\left(
				\B_{k}(\boldsymbol{\phi})(t)
			\right) \cdot \\
		& \phantom{=\int_{-\infty}^{\infty}\,}
			{\left[
				\B_{k}(\boldsymbol{s}) -
				\left(
					\sum_{i=0}^{n-1} a_i \cdot B_{i,k}(t)
				\right) \cdot
				\cos\!\left(
					\B_{k}(\boldsymbol{\phi})(t)
				\right)
			\right]}
			\mathrm{d}t,
	\end{align*}
	and consequently as $\forall i \in \{0,\dots,n-1\} \colon B_{i,k} \succeq 0$
	\begin{equation}\label{eq:cf-proof-aderiv}
		\frac{\partial^2 \boldsymbol{c}_1[\boldsymbol{s}]}{\partial a_m \partial a_p}
			(\boldsymbol{a},\boldsymbol{\phi})
			= \int_{-\infty}^{\infty}
			2 \cdot
			B_{m,k}(t) \cdot
			B_{p,k}(t) \cdot
			\cos^2\!\left(
				\B_{k}(\boldsymbol{\phi})(t)
			\right)
			\mathrm{d}t \ge 0.
	\end{equation}
	Our Hessian matrix is of the form
	\begin{equation*}
		H_{\boldsymbol{c}_1[\boldsymbol{s}]}(\boldsymbol{a},\boldsymbol{\phi}) =
		{\left(\frac{\partial^2 \boldsymbol{c}_1[\boldsymbol{s}]}{\partial a_m \partial a_p}
		(\boldsymbol{a},\boldsymbol{\phi})\right)}_{(m,p) \in {\{1,\dots,n\}}^2},
	\end{equation*}
	and we now check the conditions for Theorem~\ref{thm:gershgorin-hadamard}.
	Symmetry follows immediately because the order of partial differentiation does
	not matter for continuously-differentiable functions. What is left to show for convexity is
	that the diagonal entries are strictly positive and the matrix is diagonally dominant.
	We know from Equation~(\ref{eq:cf-proof-aderiv}) that
	\begin{equation*}
		\frac{\partial^2 \boldsymbol{c}_1[\boldsymbol{s}]}{\partial a_m \partial a_m}
			(\boldsymbol{a},\boldsymbol{\phi})
			= \int_{-\infty}^{\infty}
			2 \cdot
			B^2_{m,k}(t) \cdot
			\cos^2\!\left(
				\B_{k}(\boldsymbol{\phi})(t)
			\right)
			\mathrm{d}t > 0,
	\end{equation*}
	which means that the diagonal entries are positive. To show that the matrix
	is diagonally dominant, we first note that Equation~(\ref{eq:cf-proof-aderiv})
	shows that all entries of the Hessian matrix are positive and we thus only
	have to consider the sum of non-diagonal entries without applying the modulus.
	It holds due to Proposition~\ref{prop:partition_of_unity}
	\begin{align*}
		\sum_{\substack{p = 0\\p \neq m}}^{n-1} \frac{\partial^2 \boldsymbol{c}_1[\boldsymbol{s}]}
			{\partial a_m \partial a_p}
			(\boldsymbol{a},\boldsymbol{\phi}) &=
			\sum_{\substack{p = 0\\p \neq m}}^{n-1}
			\int_{-\infty}^{\infty}
			2 \cdot
			B_{m,k}(t) \cdot
			B_{p,k}(t) \cdot
			\cos^2\!\left(
				\B_{k}(\boldsymbol{\phi})(t)
			\right)
			\mathrm{d}t \\
		&=
			\int_{-\infty}^{\infty}
			2 \cdot
			B_{m,k}(t) \cdot
			\left(
				\sum_{\substack{p = 0\\p \neq m}}^{n-1}
				B_{p,k}(t)
			\right) \cdot
			\cos^2\!\left(
				\B_{k}(\boldsymbol{\phi})(t)
			\right)
			\mathrm{d}t \\
		&=
			\int_{-\infty}^{\infty}
			2 \cdot
			B_{m,k}(t) \cdot
			\left(
				1 - B_{m,k}(t)
			\right) \cdot
			\cos^2\!\left(
				\B_{k}(\boldsymbol{\phi})(t)
			\right)
			\mathrm{d}t \\
		&\not<
			\int_{-\infty}^{\infty}
			2 \cdot
			B_{m,k}(t) \cdot
			B_{m,k}(t) \cdot
			\cos^2\!\left(
				\B_{k}(\boldsymbol{\phi})(t)
			\right)
			\mathrm{d}t \\
		&= \sum_{\substack{p = 0\\p \neq m}}^{n-1} \frac{\partial^2 \boldsymbol{c}_1[\boldsymbol{s}]}
			{\partial a_m \partial a_m}(\boldsymbol{a},\boldsymbol{\phi}).
	\end{align*}
	We have shown that the Hessian matrix is symmetric and has strictly positive
	diagonal elements, however, it is not diagonally dominant. With Theorem~\ref{thm:gershgorin-hadamard} alone
	we can thus not conclude that the Hessian matrix is positive definite. The
	Theorem of \textsc{Geršgorin}-\textsc{Hadamard} is by no means exhaustive, but
	one of the most precise methods for this task, which means that the assumption
	that this matrix is not positive definite is well-founded and we can state that the
	canonical cost function is not convex in $\boldsymbol{a}$.
	\par
	We now proceed with $\boldsymbol{\phi}$. Using the cosine sum formula we obtain
	\begin{align*}
		\frac{\partial \boldsymbol{c}_1[\boldsymbol{s}]}{\partial \phi_m}(\boldsymbol{a},\boldsymbol{\phi}) &=
			\int_{-\infty}^{\infty}
			(-2) \cdot {\left[
				\B_{k}(\boldsymbol{s}) - \B_{k}(\boldsymbol{a})(t) \cdot
				\cos\!\left(
					\sum_{i=0}^{n-1} \phi_i \cdot B_{i,k}(t)
				\right)
			\right]} \cdot \\
		& \phantom{=\int_{-\infty}^{\infty}\,\,\,}
			\B_{k}(\boldsymbol{a})(t) \cdot
			B_{m,k}(t) \cdot
			(-1) \cdot
			\sin\!\left(
				\sum_{i=0}^{n-1} \phi_i \cdot B_{i,k}(t)
			\right)
			\mathrm{d}t \\
		&= \int_{-\infty}^{\infty}
			2 \cdot
			\B_{k}(\boldsymbol{a})(t) \cdot
			B_{m,k}(t) \cdot
			\sin\!\left(
				\sum_{i=0}^{n-1} \phi_i \cdot B_{i,k}(t)
			\right) \cdot \\
		& \phantom{=\int_{-\infty}^{\infty}\,}
			{\left[
				\B_{k}(\boldsymbol{s}) - \B_{k}(\boldsymbol{a})(t) \cdot
				\cos\!\left(
					\sum_{i=0}^{n-1} \phi_i \cdot B_{i,k}(t)
				\right)
			\right]}
			\mathrm{d}t,
	\end{align*}
	and consequently
	\begin{align*}
		\frac{\partial^2 \boldsymbol{c}_1[\boldsymbol{s}]}{\partial \phi_m \partial \phi_p}
			(\boldsymbol{a},\boldsymbol{\phi}) &=
			\int_{-\infty}^{\infty}
			2 \cdot
			\B_{k}(\boldsymbol{a})(t) \cdot
			B_{m,k}(t) \cdot
			B_{p,k}(t) \cdot
			\cos\!\left(
				\sum_{i=0}^{n-1} \phi_i \cdot B_{i,k}(t)
			\right)	\cdot \\
		& \phantom{=\int_{-\infty}^{\infty}}
			{\left[
				\B_{k}(\boldsymbol{s}) - \B_{k}(\boldsymbol{a})(t) \cdot
				\cos\!\left(
					\sum_{i=0}^{n-1} \phi_i \cdot B_{i,k}(t)
				\right)
			\right]} +\\
		& \phantom{=\int_{-\infty}^{\infty}}
			2 \cdot
			{\B_{k}(\boldsymbol{a})(t)}^2 \cdot
			B_{m,k}(t) \cdot
			B_{p,k}(t) \cdot
			\sin^2\!\left(
				\sum_{i=0}^{n-1} \phi_i \cdot B_{i,k}(t)
			\right)
			\mathrm{d}t.
	\end{align*}
	As we can see in this expression, especially if we look at diagonal entries
	with $p = m$, they are not strictly positive given the oscillating
	cosine terms and we can not apply Theorem~\ref{thm:gershgorin-hadamard}.
	Granted, only because we can not apply it does not mean that the Hessian
	matrix corresponding to $\boldsymbol{\phi}$ is not positive semidefinite.
	The critical argument that leads to this conclusion though is that the sign
	is arbitrarily controlled by the unrelated parameter $\boldsymbol{a}$
	such that there is always a way to find a counterexample for some
	$\boldsymbol{a}$ such that the Hessian matrix for $\boldsymbol{\phi}$
	is not positive definite. In total, we thus find no general convexity
	property for $\boldsymbol{\phi}$.
\end{proof}
The result of this proposition clearly shows that, at least with this class
of cost functions, the search for a truly convex optimization problem leads to
a dead end. Convexity only makes sense if it applies to the entire function
for all mixed second partial derivatives (even between amplitude and phase).
Only showing it for a subset of the parameters, in our case the amplitude
$\boldsymbol{a}$, is not of much use. However, it shows the approach that must
be taken to analyze future candidates for such cost functions. Considering what
we've seen in the last proof and how close we were to convexity, we can
imply that such candidates will also yield diagonally dominant symmetric
Hessian matrices, and the only real aspect that will matter is the strict
positivity of the diagonal entries.
\par
However, not all is lost only because we have not shown convexity, and we will
go an alternative path in Section~\ref{sec:regularity} using the theory of
convex-likeness to show some useful properties. In the long run though,
the residual-approach might have to be overthought and completely novel
approaches developed, for instance ones making use of information theory
with the goal of maximum information extraction in each step.
\par
Unfortunately, this is not easy and probably even impossible, given we actually
need to evaluate the intrinsic mode function itself to assess the relation
of a candidate IMF in regard to the input signal. One cannot directly
do that with just the soul of the IMF. The reason for the problem
is that the IMF evaluation from its soul
\begin{equation*}
	\I[a,\phi](t) = a(t) \cdot \cos(\phi(t)),
\end{equation*}
or analogously in spline formulation
\begin{equation*}
		\I[\B_{k}(\boldsymbol{a}), \B_{k}(\boldsymbol{\phi})](t) =
		\left( \sum_{i=0}^{n-1} a_i \cdot B_{i,k}(t) \right) \cdot
		\cos\left( \sum_{i=0}^{n-1} \phi_i \cdot B_{i,k}(t) \right)
\end{equation*}
\enquote{moves} $\phi$ from frequency to signal space, which makes any
expression containing it non-convex.
If one manages to find a convex EMD cost function which in some way circumvents
this problem, one has in an instance solved a central part of the previously
discussed problem in regard to the empirical mode decomposition on an analytical
level. A consequence would be a convex
analytical optimization problem and a strong theoretical footing for EMD,
which would have far-reaching effects. From the current standpoint, though,
this feat seems to be impossible to achieve.
\subsection{Leakage Factor}\label{subsec:nsp-leakage-cf}
We have already seen the canonical EMD cost function in Definition~\ref{def:ccf}
in the previous subsection. The motivation behind it is that we want to leave as
little residual as possible and strive for the first IMFs to make up the biggest
part of the signal.
However, serving as a small outlook, what if we do not want to extract as much as
possible in each step and want to control the extraction degree? This has been
discussed in \cite{ph10} and can be achieved heuristically by putting a penalty
on the norm of the extracted IMF and scaling this penalty with a so-called
\enquote{leakage factor}.
\begin{definition}[Leakage factor EMD cost function {\cite[(16)]{ph10}}]\label{def:lfecf}
	Let $\mu_0, \mu_1, \mu_2 > 0$, $\gamma \ge 0$ and $s \in \mathcal{C}^0(\R,\R)$.
	The \emph{leakage factor EMD cost function}
	$c_\ell[s] \colon {\mathcal{S}}_{\mu_0,\mu_1,\mu_2} \to \R$
	is defined as
	\begin{equation*}
		c_\ell[s](a,\phi) := {\left\| s - \I(a,\phi) \right\|}_2^2 +
		\gamma \cdot {\left\| \I(a,\phi) \right\|}_2^2.
	\end{equation*}
\end{definition}
The higher the leakage factor $\gamma$ is chosen, the more we punish the
extraction of \enquote{large} IMFs and let it slip through for one of the
next EMD extraction steps. Analogous to the canonical spline EMD cost function,
we can define a leakage factor spline EMD cost function as follows.
\begin{definition}[Leakage factor spline EMD cost function]\label{def:lfsecf}
	Let $\mu_0, \mu_1, \mu_2 > 0$, $\gamma \ge 0$, $k \ge 4$ (for derivability) and
	$\boldsymbol{s}\in\R^n$. The
	\emph{leakage factor spline EMD cost function}
	$\boldsymbol{c}_\ell[\boldsymbol{s}] \colon \boldsymbol{\mathcal{S}}_{\mu_0,\mu_1,\mu_2} \to \R$
	is defined as
	\begin{equation*}
		\boldsymbol{c}_\ell[\boldsymbol{s}](\boldsymbol{a},\boldsymbol{\phi}) :=
		c_\ell[\B_{k}(\boldsymbol{s})](\B_{k}(\boldsymbol{a}),\B_{k}(\boldsymbol{\phi})).
	\end{equation*}
\end{definition}
Looking at the equation, we can make an interesting
observation that relates the leakage factor cost function
to our canonical cost function.
\begin{remark}\label{rem:lfecf-linear_combination}
	We can directly see that
	\begin{equation}
		c_\ell[s](a,\phi) = c_1[s](a,\phi) + \gamma \cdot c_1[0](a,\phi),
	\end{equation}
	which means that the leakage factor EMD cost function, as $\gamma \ge 0$, is a
	positive linear combination of the canonical EMD cost function.
\end{remark}
This thesis will not further investigate the advantages or disadvantages
of the leakage factor approach itself. However, what we can see is that it
integrates well into the canonical approach and any results we obtain
as follows apply to both the canonical and leakage factor cost functions.
This is especially useful considering the final results in terms of convex-like
functions, as with the above remark we have shown that if the canonical cost function
is convex-like, the leakage-factor cost function is so as well.
\section{General Optimization Problem}\label{sec:general_opt}
Having discussed the nature of intrinsic mode and EMD cost functions, we
can now formulate the general optimization problem that is the core
of each empirical mode decomposition step. As already laid out previously
we are constructing an optimization problem
\begin{align*}
	\inf_{u} \quad &
		{\| s - u \|}_2^2 \\
	\text{s.t.} \quad & u \text{\ IMF}.
\end{align*}
for an input signal $s(t)$ and candidate IMFs $u(t)$. Based on our
IMF construction in Section~\ref{sec:intrinsic_mode_functions} we
have noted that looking at IMF souls $(a(t),\phi(t))$ is much more
useful and the only direct way to theorize the IMF constraints
properly, given we have to explicitly work with $a$ and $\phi$
to steer the extraction process. Consequently, instead of the fixed ${\| s - u \|}_2^2$ in the
sketch in Equation~(\ref{eq:emdop-reformed}) we consider arbitrary EMD cost functions operating
on our set of IMF souls ${\mathcal{S}}_{\mu_0,\mu_1,\mu_2}$, two of which we
presented in Section~\ref{sec:cost_functions}.
\begin{definition}[EMDOP]\label{def:emdop}
	Let $\mu_0, \mu_1, \mu_2 > 0$, $s \in \mathcal{C}^0(\R,\R)$ the
	input signal function and $c[s] \colon {\mathcal{S}}_{\mu_0,\mu_1,\mu_2} \to \R$
	an EMD cost function.
	The \emph{EMD optimization problem (EMDOP) for the input signal $s$} is
	defined as
	\begin{align*}
		\min_{(a,\phi)} \quad &
			c[s](a,\phi) \\
		\text{s.t.} \quad & (a,\phi) \in {\mathcal{S}}_{\mu_0,\mu_1,\mu_2}.
	\end{align*}
\end{definition}
As previously done, we can also express the optimization problem
in terms of B-spline coefficients rather than functions based on the
theoretical groundwork in Chapter~\ref{ch:funcspaceopt}.
\begin{definition}[SpEMDOP]\label{def:spemdop}
	Let $\mu_0, \mu_1, \mu_2 > 0$, $\boldsymbol{s} \in \R^n$ B-spline coefficients
	of the spline input signal function and
	$\boldsymbol{c}[\boldsymbol{s}]\colon \boldsymbol{\mathcal{S}}_{\mu_0,\mu_1,\mu_2} \to \R$
	a spline EMD cost function.
	The \emph{spline EMD optimization problem (SpEMDOP) for input signal $s$} is
	defined as
	\begin{align*}
		\min_{(\boldsymbol{a},\boldsymbol{\phi})} \quad &
			\boldsymbol{c}[\boldsymbol{s}](\boldsymbol{a},\boldsymbol{\phi}) \\
		\text{s.t.} \quad & (\boldsymbol{a},\boldsymbol{\phi}) \in \boldsymbol{\mathcal{S}}_{\mu_0,\mu_1,\mu_2}.
	\end{align*}
\end{definition}
As we have seen in Theorem~\ref{thm:imfs-convex}, the set of intrinsic
mode function souls ${\mathcal{S}}_{\mu_0,\mu_1,\mu_2}$ and its analogue
$\boldsymbol{\mathcal{S}}_{\mu_0,\mu_1,\mu_2} \subset \R^n \times \R^n$ are convex sets.
However, the canonical and leakage factor EMD cost functions are not
convex in $(a,\phi)$, which is a big downside, as we would otherwise have a strong
guarantee that an obtained local minimum is also a global minimum and
each EMD extraction step unique.
The positive aspect of this analysis is that, using the B-spline relation,
we are able to examine this problem at all using this novel formulation.
\par
Given the empirically good results observed with regard to the empirical
mode decomposition in previous publications, it makes one still wonder
why it still works so well
despite the non-convexity of the underlying optimization problem.
Given we now have the tools to theoretically
examine this at the root and because we are not trying to go into the
theory of the search for a convex EMD cost function, we will work with
what is given and instead of convexity focus on the regularity of the
optimization problem.
\section{Regularity}\label{sec:regularity}
We have shown that the SpEMDOP (which is equivalent to the EMDOP)
is not a convex optimization problem,
but we can still examine its regularity. To explain what
regularity is, we take a look at the EMDOP from
Definition~\ref{def:emdop}, which was defined as
\begin{align*}
	\min_{(a,\phi)} \quad &
		c[s](a,\phi) \\
	\text{s.t.} \quad & (a,\phi) \in {\mathcal{S}}_{\mu_0,\mu_1,\mu_2}
\end{align*}
for an EMD cost function $c[s]$.
$s$, $a$ and $\phi$ relate to the input signal $s(t)$ and candidate
IMF soul pair $(a(t),\phi(t))$. When approaching this problem, we vary
$a$ and $\phi$ such that the EMD cost function
is minimized, under the condition that $(a(t),\phi(t))$
are within our set ${\mathcal{S}}_{\mu_0,\mu_1,\mu_2}$ of IMF souls.
However, it is difficult to enforce the
latter condition as this set is too \enquote{large} to check as a whole,
making it necessary to find other ways to \enquote{steer} the candidates
$(a(t),\phi(t))$ in a direction where they in fact are IMF souls.
\par
The approach that can be taken is to modify the cost function
and add a so-called regularization term
$R(a,\phi) \colon {\left( \mathcal{C}^0(\R,\R) \right)}^2 \to \R$.
This term is designed such
that it is exactly $0$ when its arguments satisfy the constraints
and a positive value when they violate them, preferably corresponding
in size to the violation. Given we are aiming to minimize the cost function
of the optimization problem, adding a term to punish violation
of the constraints will, in the best case, enforce them. The advantage
of this regularization approach is that we obtain an unconstrained optimization
problem of the form
\begin{align*}
	\min_{(a,\phi) \in {\left( \mathcal{C}^0(\R,\R) \right)}^2} \quad &
		c[s](a,\phi) + R(a,\phi)
\end{align*}
that is relatively simple to model and implement numerically
using the equivalent B-spline formulation.
A trivial way to define the regularization term is as the
so-called \enquote{characteristic function} of convex analysis
as
\begin{equation*}
	R(a,\phi) =
	\begin{cases}
		0 & (a,\phi) \in {\mathcal{S}}_{\mu_0,\mu_1,\mu_2} \\
		\infty & (a,\phi) \notin {\mathcal{S}}_{\mu_0,\mu_1,\mu_2},
	\end{cases}
\end{equation*}
but for obvious reasons other choices for
$R(a,\phi)$ are much better-suited.
This is because the characteristic function does not distinguish
between candidates close to or far away from the target set
${\mathcal{S}}_{\mu_0,\mu_1,\mu_2}$ and for numerical approaches
we would want to be able to calculate a \enquote{slope} of the
cost function to be able to steer into the optimum in some way.
\par
If we take a look at our constraint for our candidates to be
instrincic mode spline function souls we notice that
(using Definition~\ref{def:imfs} and brackets to group
conditions)
\begin{equation*}
	(a,\phi) \in
	{\mathcal{S}}_{\mu_0,\mu_1,\mu_2} \Leftrightarrow
	\begin{cases}
		0 \preceq a \\
		\mu_0 \preceq \phi' \\
		\left| a' \right| \preceq \mu_1 \cdot
			\left| \phi' \right| \\
		\left| \phi'' \right| \preceq \mu_2 \cdot
			\left| \phi' \right|
	\end{cases}
	\hspace{-0.3cm}\Leftrightarrow
	\begin{cases}
		g_1(a,\phi) :=
			-a \preceq 0\\
		g_2(a,\phi) :=
			\mu_0 - \phi' \preceq 0 \\
		g_3(a,\phi) :=
			\left| a' \right| - \mu_1 \cdot
			\left| \phi' \right| \preceq 0 \\
		g_4(a,\phi) :=
			\left| \phi'' \right| - \mu_2 \cdot
			\left| \phi' \right| \preceq 0.
	\end{cases}
\end{equation*}
So we see that we can formulate four functions $g_1,\dots,g_4$ which
are negative if and only if their parameters are IMF souls.
These functions correspond to four
inequality constraints of the underlying optimization problem.
These can be used in the method of \textsc{Lagrange} multipliers,
that is introduced later, to find a \enquote{perfect} regularization
of the problem based on these functions that are each \enquote{weighted}
and \enquote{added} to the cost function. The optimization problem then
is a two-step process of first finding the optimal parameters
$(a,\phi)$ and then the optimal \enquote{weights} applied to the
constraint functions. It is called
the \enquote{dual problem} as opposed to the constrained
\enquote{primal problem} we started with in Definition~\ref{def:emdop}.
\par
It can be shown that under certain conditions this dual problem yields
the same optimal value as the constrained (primal) optimization problem.
This is known as strong duality and the conditions are called
regularity conditions. One particular sufficient condition for strong
duality is the \enquote{Slater} condition that is presented later,
and we will show that it applies to the spline formulation
SpEMDOP (and EMDOP respectively). This may be a surprising result,
as it is commonly assumed that the \enquote{Slater} condition can only
be shown for convex optimization problems. This is a wrong assumption,
as the requirements for \enquote{Slater} regularity are weaker than
convexity and only require so-called \enquote{convex-like} functions
we will introduce later.
\par
The main result of this section and one of the central results
of this thesis is strong duality for the
SpEMDOP (see Theorem~\ref{th:spemdop-strong_duality})
and EMDOP respectively, as they are equivalent.
The formalism introduced as follows though is inconsequential
for the thesis and can be skipped up to the conclusion in
Section~\ref{sec:emd_conclusion}, which gives a thematic
classification of strong duality of the EMDOP within the operator-based
regularization methods we introduce in Chapter~\ref{ch:oss}.
\subsection{Convex-Like Optimization}\label{subsec:clop}
The theory of convex-like functions and consequently convex-like
optimization problems presented here is based on \cite{jah07} that
formulates constrained optimization problems as cone optimization
problems and constructs the theory of convex-like optimization
problems on top of that. The goal of this subsection is to introduce
the necessary definitions for cone optimization problems and
convex-likeness. To map the results from \cite{jah07} to the SpEMDOP
we reformulate it as a cone optimization problem for which we then show
that it is a convex-like optimization problem. It shall be noted
here that we should remind ourselves of the definitions given
in Chapter~\ref{ch:funcspaceopt}.
\par
First we begin with the introduction of cone optimization, which
is an elegant way to express constrained optimization problems of
higher dimensions and with non-standard orderings. This is necessary
in our case as our constraints do not have a scalar order $\le$
but a function order $\preceq$, for which the classic notation fails.
\begin{definition}[Cone {\cite[Definition~4.1]{jah07}}]\label{def:cone}
	Let $V$ be a vector space and $C \subseteq V$. $C$ is a
	\emph{cone in $V$} if and only if
	\begin{equation*}
		\forall x \in C \colon
		\forall \alpha \in \R_+ \colon
		\alpha \cdot x \in C.
	\end{equation*}
\end{definition}
As we can see, a cone is a set which contains all positive scalar
multiplications of a vector. Consequently we can make the following
\begin{definition}[Convex cone {\cite[Theorem~4.3]{jah07}}]
	Let $V$ be a vector space and $C$ a cone in $V$. $C$ is
	a \emph{convex cone in $V$} if and only if $C$ is a
	convex set.
\end{definition}
We use cone optimization to handle non-standard orders, in our case the function order $\preceq$ that was introduced in Chapter~\ref{ch:funcspaceopt}.
Central to this concept is the concept of a positive cone, which
contains all positive elements of a vector space.
\begin{definition}[Positive cone {\cite[Chapter~V, {\S}1]{sw99}}]
	\label{def:positive_cone}
	Let $(V,\le_V)$ be a preordered vector space.
	The \emph{positive cone of $V$} is defined as
	\begin{equation*}
		V^+ := \left\{ x \in V \mid 0 \le_V x \right\}.
	\end{equation*}
\end{definition}
\begin{proposition}
	Let $(V,\le_V)$ be a preordered vector space.
	$V^+$ is a convex cone.
\end{proposition}
\begin{proof}
	Let $x,y \in V^+$ and $\alpha,\beta \in \R_+$. It holds
	because of the scalar multiplication compatibility of
	the preordered vector space that $\alpha \cdot x \ge_V 0$ and
	$\beta \cdot y \ge_V 0$ and thus it follows with the addition
	compatibility of the preordered vector space that
	$\alpha \cdot x + \beta \cdot y \ge_V 0$.
\end{proof}
Before we can express what convex-likeness means, we first define
a few aspects of notation.
\begin{definition}[\textsc{Minkowski} sum]
	Let $(G,+)$ be a group and $A,B \subseteq G$ be sets. The
	\emph{\textsc{Minkowski} sum} of $A$ and $B$ is defined as
	\begin{equation*}
		A + B :=
		\left\{
			a + b
			\mid
			a \in A \land b \in B
		\right\}.
	\end{equation*}
\end{definition}
We can see that the \textsc{Minkowski} sum is just the
set of all pairwise additions of all elements in both sets.
\begin{definition}[set evaluation]
	Let $A,B$ be sets and $f \colon A \to B$.
	The \emph{set evaluation} of $f$ in $A$ is defined as
	\begin{equation*}
		f(A) :=
		\left\{
			f(a)
			\mid
			a \in A
		\right\}
	\end{equation*}
\end{definition}
The set evaluation of a function is thus just the set of
all evaluations of the function in all elements of the set.
Making use of the \textsc{Minkowski} sum and the set
evaluation, we can now define what a convex-like function is.
\begin{definition}[Convex-like function {\cite[Definition~6.3]{jah07}}]\label{def:clf}
	Let $(V,\le_V),(W,\le_W)$ be real ordered vector spaces,
	$S \subseteq V$ and $f \colon S \to W$.
	$f$ is a \emph{convex-like function in relation to $W^+$}
	if and only if the set
	\begin{equation*}
		M := f(S) + W^+
	\end{equation*}
	is convex.
\end{definition}
As we can see, the idea behind a convex-like function is to say that
if we take the domain of a function $f$ within a vector space $W$
and do a set-addition of all positive elements in $W$ (which is
$W^+$) and find that the resulting set is convex, then the function
$f$ is convex-like. In particular, every convex function is also
convex-like in relation to $\R_+$ (all positive numbers including $0$)
as we know that the epigraph (the set of points lying on or above its graph)
of a convex function is also convex. However, not all convex-like functions
in relation to $\R_+$ are also convex, which we can see in
the following example.
\begin{example}
	Consider the function $f \colon \R \to \R$ with
	\begin{equation*}
		f(t) := \sin(t).
	\end{equation*}
	We know that $f(t)$ is not convex, but it is convex-like
	in relation to $\R_+$, the positive cone of $(\R,\le)$, because
	\begin{equation*}
		M := \sin(\R) + \R_+ = [-1,1] + [0,\infty) = [-1,\infty)
	\end{equation*}
	is a convex set.
\end{example}
The next logical step is to take a look at the canonical spline EMD cost
function and see if it is a convex-like function. This is true as we
can see in the following
\begin{proposition}\label{prop:scecf-clf}
	The canonical spline EMD cost function
	(see Definition~\ref{def:cscf}) is a convex-like function
	in relation to $\R_+$.
\end{proposition}
\begin{proof}
	We defined the canonical spline EMD cost function $\boldsymbol{c}_1
	\colon \boldsymbol{\mathcal{S}}_{\mu_0,\mu_1,\mu_2} \to \R$ with fixed $\boldsymbol{s}\in\R^n$ as
	\begin{align*}
		\boldsymbol{c}_1[\boldsymbol{s}](\boldsymbol{a},\boldsymbol{\phi}) &=
			c_1[\B_{k}(\boldsymbol{s})]\!\left(\B_{k}(\boldsymbol{a}),\B_{k}(\boldsymbol{\phi})\right) \\
		&= {\left\| \B_{k}(\boldsymbol{s}) -
			\I\!\left(\B_{k}(\boldsymbol{a}),\B_{k}(\boldsymbol{\phi})\right) \right\|}_2^2 \\
		&= {\left\| \B_{k}(\boldsymbol{s}) -
			\B_{k}(\boldsymbol{a}) \cdot \cos\left(\B_{k}(\boldsymbol{\phi})\right) \right\|}^2_2
	\end{align*}
	If we, according to Definition~\ref{def:clf}, take
	$(V,\le_V) = (\R^n \times \R^n, \le)$ and
	$(W,\le_W) = (\R, \le)$ (i.e.\ use the canonical orders)
	and note that in this case the domain of our cost function is
	$S = \boldsymbol{\mathcal{S}}_{\mu_0,\mu_1,\mu_2} \subset \R^n \times \R^n$,
	we obtain
	\begin{align*}
		M &= \boldsymbol{c}_1[\boldsymbol{s}](S) + W^+ \\
		&= \boldsymbol{c}_1[\boldsymbol{s}](\boldsymbol{\mathcal{S}}_{\mu_0,\mu_1,\mu_2}) + \R_+ \\
		&= \left\{
				{
					\left\|
						\B_{k}(\boldsymbol{s}) -
						\left(
							\sum_{i=0}^{n-1} a_i \cdot B_{i,k}
						\right) \cdot
						\cos
						\left(
							\sum_{i=0}^{n-1} \phi_i \cdot B_{i,k}
						\right)
					\right\|
				}^2_2
				\mathrel{\Bigg|}
				(\boldsymbol{a},\boldsymbol{\phi}) \in
				\boldsymbol{\mathcal{S}}_{\mu_0,\mu_1,\mu_2}
			\right\} +
			\R_+.
	\end{align*}
	The vector $\boldsymbol{s}$ is fixed, so the matter
	of interest is the right hand side of the subtraction within
	the norm. Fundamentally, we substract all possible IMFs from
	the input signal $\B_{k}(\boldsymbol{s})$ and thus construct
	all residuals and determine their norm.
	Of all norms that we obtain, the minimal norm determines the
	lower bound of the set. In the ideal case, if the residual vanishes
	for a certain IMF, the lower bound is $0$, but it usually is a
	positive constant $q(\boldsymbol{s},\mu_0,\mu_1,\mu_2)$ that
	only depends on $\boldsymbol{s}$ and the predetermined IMF
	characteristic $\mu_0,\mu_1,\mu_2 > 0$.
	The upper bound of this set does not matter, as we add
	$\R_+$ later, and can be set to a constant
	$r(\boldsymbol{s},\mu_0,\mu_1,\mu_2)$ corresponding
	to the norm of the \enquote{worst} residual. It follows that
	\begin{align*}
		M &= [q(\boldsymbol{s},\mu_0,\mu_1,\mu_2),
			r(\boldsymbol{s},\mu_0,\mu_1,\mu_2)) + \R_+\\
		&= [q(\boldsymbol{s},\mu_0,\mu_1,\mu_2),
			r(\boldsymbol{s},\mu_0,\mu_1,\mu_2)) + [0,\infty)\\
		&= [q(\boldsymbol{s},\mu_0,\mu_1,\mu_2),\infty),
	\end{align*}
	which is a convex set.
\end{proof}
Consequently, we can also consider our leakage factor spline EMD
cost function, for which the proof is simpler, based on previous results.
\begin{corollary}
	The leakage factor spline EMD cost function
	(see Definition~\ref{def:lfsecf}) is a convex-like function
	in relation to $\R_+$.
\end{corollary}
\begin{proof}
	This follows directly from Proposition~\ref{prop:scecf-clf} and
	Remark~\ref{rem:lfecf-linear_combination}.
\end{proof}
We have now shown that our two classic cost functions are convex-like
and are now interested in the definition of the convex-like optimization
problem. This is given as follows.
\begin{definition}[Convex-like optimization problem {\cite[(6.2)]{jah07}}]\label{def:clop}
	Let $(V,\le_V),{(W,\le_W)}$ be normed ordered vector spaces,
	$W^+ \ne \emptyset$, $c \colon V \to \R$ a cost function,
	$g \colon V \to W$ and $\emptyset \neq S \subseteq V$.
	The optimization problem
	\begin{align*}	
		\min_{x} \quad &
			c(x) \\
		\text{s.t.} \quad &
			g(x) \in -W^+ \\
		& x \in S
	\end{align*}
	is a \emph{convex-like optimization problem} if and only if
	$K:V\to \R \times W$ defined as
	\begin{equation*}
		K(x) := \left(c(x),g(x)\right)
	\end{equation*}
	is a convex-like function in relation to $\R_+ \times W^+$.
\end{definition}
What we can see is that an optimization problem is a convex-like
optimization problem when the cost function is convex-like
and the constraints can be expressed as a convex-like function
$g(x) \in -W^+$ (which means that the candidate $x$ satisfies
the constraints when $g(x)$ is in the negative cone of $W$,
written as the negation of the positive cone $W^+$). The set
$S$ can just be chosen as $V$, unless it also needs to reflect
some conditions that did not fit into $g$ as it would violate
convex-likeness.
\par
To prove that our SpEMDOP is a convex-like optimization problem
we need the following lemma. It will be later used because the
set ${\mathcal{S}}_{\mu_0,\mu_1,\mu_2}$ cannot be directly expressed using
a convex-like function. We need to consider the superset
${\mathcal{S}}_{0,\mu_1,\mu_2}$ (which is a convex cone) of
${\mathcal{S}}_{\mu_0,\mu_1,\mu_2}$ and move the
remaining conditions into our set $S$.
\begin{lemma}\label{lem:s-convex_cone}
	Let $\mu_1,\mu_2 > 0$. ${\mathcal{S}}_{0,\mu_1,\mu_2}$ is a
	convex cone.
\end{lemma}
\begin{proof}
	We have already shown in in Theorem~\ref{thm:imfs-convex}
	that ${\mathcal{S}}_{0,\mu_1,\mu_2}$ is convex. What is
	left to show is that ${\mathcal{S}}_{0,\mu_1,\mu_2}$ is
	a cone (see Definition~\ref{def:cone}).
	\par
	Let $\alpha > 0, (a,\phi)\in{\mathcal{S}}_{0,\mu_1,\mu_2}$
	and define
	\begin{equation*}
		(a_\star,\phi_\star) :=
		\alpha \cdot (a,\phi) =
		(\alpha \cdot a,\alpha \cdot \phi).
	\end{equation*}
	We now show that $(a_\star,\phi_\star)$ satisfies
	the conditions from Definition~\ref{def:imfs}.
	\begin{enumerate}
		\item{%
			$a_\star = \alpha \cdot a \succeq
			\alpha \cdot 0 = 0$
		}
		\item{%
			$\phi'_\star = \alpha \cdot \phi'
			\succeq \alpha \cdot 0 = 0$
		}
		\item{%
			$|a'_\star| = |\alpha \cdot a'| =
			\alpha \cdot |a'| \preceq \alpha \cdot \mu_1 \cdot |\phi'| =
			\mu_1 \cdot |\alpha \cdot \phi'| = \mu_1 \cdot |\phi'_\star|$
		}
		\item{%
			$|\phi''_\star| = |\alpha \cdot \phi''| =
			\alpha \cdot |\phi''| \preceq \alpha \cdot \mu_2 \cdot |\phi'| =
			\mu_2 \cdot |\alpha \cdot \phi'| = \mu_2 \cdot |\phi'_\star|$
		}
	\end{enumerate}
	It follows that $(a_\star,\phi_\star) \in {\mathcal{S}}_{0,\mu_1,\mu_2}$
	and thus ${\mathcal{S}}_{0,\mu_1,\mu_2}$ is a convex cone.
\end{proof}
\begin{remark}\label{rem:imf-not_convex_cone}
	One important consequence seen in this proof is that
	${\mathcal{S}}_{\mu_0,\mu_1,\mu_2}$ is not a convex cone. This is because
	in general it holds
	\begin{equation*}
			\phi'_\star = \alpha \cdot \phi'
			\succeq \alpha \cdot \mu_0 \not\succeq \mu_0.
	\end{equation*}
	and thus not all scalar multiplications of elements
	in ${\mathcal{S}}_{\mu_0,\mu_1,\mu_2}$ are within ${\mathcal{S}}_{\mu_0,\mu_1,\mu_2}$
\end{remark}
With this lemma shown we can go ahead and formulate the first
central theorem of this section, namely that the SpEMDOP is a
convex-like optimization problem.
\begin{theorem}\label{thm:spemdop-clop}
	Let $\mu_0, \mu_1, \mu_2 > 0$ and $k \ge 4$ (for derivability).
	The SpEMDOP (see Definition~\ref{def:spemdop}) with
	a convex-like spline EMD cost function
	$\boldsymbol{c}[\boldsymbol{s}] \colon \boldsymbol{\mathcal{S}}_{0,\mu_1,\mu_2}
	\to \R$ in relation to $\R_+$ is a convex-like optimization problem of the form
	\begin{align*}
		\min_{(\boldsymbol{a},\boldsymbol{\phi})} \quad &
			\boldsymbol{c}[\boldsymbol{s}](\boldsymbol{a},\boldsymbol{\phi}) \\
		\text{s.t.} \quad & g(\boldsymbol{a},\boldsymbol{\phi}) \in
			-{\left({\mathcal{C}^0(\R,\R)}^4\right)}^+ \\
		& (\boldsymbol{a},\boldsymbol{\phi}) \in S
	\end{align*}
	with $g \colon \R^n \times \R^n \to {\mathcal{C}^0(\R,\R)}^4$  defined as
	\begin{equation*}
		g(\boldsymbol{a},\boldsymbol{\phi}) :=
		\begin{pmatrix}
			-\B_{k}(\boldsymbol{a}) \\
			-\B_{k}'(\boldsymbol{\phi}) \\
			\left| \B_{k}'(\boldsymbol{a}) \right| - \mu_1 \cdot
				\left| \B_{k}'(\boldsymbol{\phi}) \right| \\
			\left| \B_{k}''(\boldsymbol{\phi}) \right| - \mu_2 \cdot
				\left| \B_{k}'(\boldsymbol{\phi}) \right| 
		\end{pmatrix}
	\end{equation*}
	and
	\begin{equation*}
		S := \left\{(\boldsymbol{a},\boldsymbol{\phi}) \in \R^n \times \R^n \mathrel{\Big|}
		\B_{k}'(\boldsymbol{\phi}) \succeq \mu_0 \right\}.
	\end{equation*}
\end{theorem}
\begin{proof}
	According to Definition~\ref{def:clop} we
	can take $(V,\le_V) = (\R^n \times \R^n, \le)$ and
	$(W,\le_W) = ({\mathcal{C}^0(\R,\R)}^4, \preceq)$ (in both cases
	using the canonical orders)
	and note that the cost function $\boldsymbol{c}[\boldsymbol{s}]$
	is already convex-like in relation to $\R_+$ by precondition.
	\par
	What is left to do is to split up the set
	$\boldsymbol{\mathcal{S}}_{\mu_0,\mu_1,\mu_2}$ into a \enquote{cone-component}
	and a residual set $S$. The former is characterized by a
	mapping $g \colon \R^n \times \R^n \to {\mathcal{C}^0(\R,\R)}^4$ such that
	\begin{equation*}
		(\boldsymbol{a},\boldsymbol{\phi}) \in \boldsymbol{\mathcal{S}}_{\mu_0,\mu_1,\mu_2}
		\Leftrightarrow
		\begin{cases}
			g(\boldsymbol{a},\boldsymbol{\phi}) \in -W^+ \\
			(\boldsymbol{a},\boldsymbol{\phi}) \in S.
		\end{cases}
	\end{equation*}
	We know from Remark~\ref{rem:imf-not_convex_cone} that
	$\boldsymbol{\mathcal{S}}_{\mu_0,\mu_1,\mu_2}$ is not a convex cone. Hoewever,
	we know from Lemma~\ref{lem:s-convex_cone} that $\boldsymbol{\mathcal{S}}_{0,\mu_1,\mu_2}$
	is a convex cone, and we want to bring them into relation in some way.
	It holds (by Definition~\ref{def:imspfs}) that (using brackets to group
	conditions)
	\begin{equation*}
		(\boldsymbol{a},\boldsymbol{\phi}) \in \boldsymbol{\mathcal{S}}_{\mu_0,\mu_1,\mu_2}
		\Leftrightarrow
		\begin{cases}
			(\boldsymbol{a},\boldsymbol{\phi}) \in \boldsymbol{\mathcal{S}}_{0,\mu_1,\mu_2} \\
			\B_{k}'(\boldsymbol{\phi}) \succeq \mu_0
		\end{cases}
		\Leftrightarrow
		\begin{cases}
			\begin{cases}
				-\B_{k}(\boldsymbol{a}) \preceq 0\\
				-\B_{k}'(\boldsymbol{\phi}) \preceq 0\\
				\left| \B_{k}'(\boldsymbol{a}) \right| - \mu_1 \cdot
					\left| \B_{k}'(\boldsymbol{\phi}) \right| \preceq 0 \\
				\left| \B_{k}''(\boldsymbol{\phi}) \right| - \mu_2 \cdot
					\left| \B_{k}'(\boldsymbol{\phi}) \right| \preceq 0
			\end{cases}\\
			\B_{k}'(\boldsymbol{\phi}) \succeq \mu_0
		\end{cases}
	\end{equation*}
	and thus, as $\boldsymbol{\mathcal{S}}_{0,\mu_1,\mu_2}$ is a convex cone by
	Lemma~\ref{lem:s-convex_cone} and using the canonical spline isomorphism,
	we can define
	\begin{equation*}
		g(\boldsymbol{a},\boldsymbol{\phi}) :=
		\begin{pmatrix}
			-\B_{k}(\boldsymbol{a}) \\
			-\B_{k}'(\boldsymbol{\phi}) \\
			\left| \B_{k}'(\boldsymbol{a}) \right| - \mu_1 \cdot
				\left| \B_{k}'(\boldsymbol{\phi}) \right| \\
			\left| \B_{k}''(\boldsymbol{\phi}) \right| - \mu_2 \cdot
				\left| \B_{k}'(\boldsymbol{\phi}) \right| 
		\end{pmatrix}
	\end{equation*}
	and
	\begin{equation*}
		(\boldsymbol{a},\boldsymbol{\phi}) \in S
		:\Leftrightarrow
		\B_{k}'(\boldsymbol{\phi}) \succeq \mu_0.
	\end{equation*}
	If a candidate $(\boldsymbol{a},\boldsymbol{\phi})$ satisfies
	$g(\boldsymbol{a},\boldsymbol{\phi}) \preceq 0 \in {\mathcal{C}(\R,\R)}^4$
	and $(\boldsymbol{a},\boldsymbol{\phi}) \in S$ this means that
	$(\boldsymbol{a},\boldsymbol{\phi}) \in \boldsymbol{\mathcal{S}}_{\mu_0,\mu_1,\mu_2}$,
	our constraint set.
	\par
	As $\boldsymbol{\mathcal{S}}_{0,\mu_1,\mu_2}$ is a convex cone it
	follows by construction that $g$ is a convex-like function
	in relation to $W^+$. Consequently, $K \colon \R^n \times \R^n \to
	\R \times {\mathcal{C}^0(\R,\R)}^4$ defined as
	\begin{equation*}
		K(\boldsymbol{a},\boldsymbol{\phi}) :=
		(\boldsymbol{c}[\boldsymbol{s}](\boldsymbol{a},\boldsymbol{\phi}),
		g(\boldsymbol{a},\boldsymbol{\phi}))
	\end{equation*}
	is a convex-like function in relation to $\R_+ \times W^+$.
\end{proof}
Up to this point we have successfully shown that the SpEMDOP is a convex-like
optimization problem. It was not possible to fit the entire set
$\boldsymbol{\mathcal{S}}_{\mu_0,\mu_1,\mu_2}$ into the function $g$, as it
is not a convex cone, and there remained a property to be put into $S$.
However, this remaining property, namely that
$\B_{k}'(\boldsymbol{\phi}) \succeq \mu_0$, is simple enough.
\subsection{\textsc{Slater} Condition and Strong Duality}
Our next point of interest is to examine the regulartiy of the SpEMDOP.
With convex-likeness shown what remains to be seen is if it
also satisfies the \textsc{Slater} condition, which is defined as follows
\begin{definition}[\textsc{Slater} condition {\cite[Lemma~5.9]{jah07}}]\label{def:sc}
	Let $(V,\le_V),(W,\le_W)$ be normed ordered vector spaces,
	$W^+ \ne \emptyset$, $c \colon V \to \R$ a cost function,
	$g \colon V \to W$ and $\emptyset \neq S \subseteq V$.
	The convex-like optimization problem
	\begin{align*}	
		\min_{x} \quad &
			c(x) \\
		\text{s.t.} \quad &
			g(x) \in -W^+ \\
		& x \in S
	\end{align*}
	satisfies the \textsc{Slater} condition if and only if
	\begin{equation*}
		\exists \tilde{x} \in V \colon
		\begin{cases}
			g(\tilde{x}) \in -W^+ \\
			\tilde{x} \in S
		\end{cases}
		\hspace{-0.4cm}\colon
		g(\tilde{x}) \in \interior(-W^+).
	\end{equation*}
\end{definition}
The big advantage of the \textsc{Slater} condition over other
regularity conditions (for strong duality) is that it is
sufficient to find one point that strictly satisfies the
constraints. Even though it will not be further elaborated
here, most other regularity conditions require an examination
on a case-by-case basis for a given candidate. In our case,
finding a single intrinsic mode spline function soul that
is strictly satisfying the constraints is enough to show it
for all cases and candidates. We prove that such a point exists
in the following
\begin{theorem}\label{thm:spemdop-sc}
	Let $k \ge 4$ (for derivability).
	The SpEMDOP (see Definition~\ref{def:spemdop}) with
	the canonical spline EMD cost function
	(see Definition~\ref{def:cscf}) satisfies the
	\textsc{Slater} condition.
\end{theorem}
\begin{proof}
	Let $\mu_0, \mu_1, \mu_2 > 0$.
	We have already shown in Theorem~\ref{thm:spemdop-clop}
	that the SpEMDOP is a convex-like optimization problem
	of the form
	\begin{align*}
		\min_{(\boldsymbol{a},\boldsymbol{\phi})} \quad &
			\boldsymbol{c}_1[\boldsymbol{s}](\boldsymbol{a},\boldsymbol{\phi}) \\
		\text{s.t.} \quad & g(\boldsymbol{a},\boldsymbol{\phi}) \in
			-{\left({\mathcal{C}^0(\R,\R)}^4\right)}^+ \\
		& (\boldsymbol{a},\boldsymbol{\phi}) \in S
	\end{align*}
	with $g \colon \R^n \times \R^n \to {\mathcal{C}^0(\R,\R)}^4$  defined as
	\begin{equation*}
		g(\boldsymbol{a},\boldsymbol{\phi}) :=
		\begin{pmatrix}
			-\B_{k}(\boldsymbol{a}) \\
			-\B_{k}'(\boldsymbol{\phi}) \\
			\left| \B_{k}'(\boldsymbol{a}) \right| - \mu_1 \cdot
				\left| \B_{k}'(\boldsymbol{\phi}) \right| \\
			\left| \B_{k}''(\boldsymbol{\phi}) \right| - \mu_2 \cdot
				\left| \B_{k}'(\boldsymbol{\phi}) \right| 
		\end{pmatrix}
	\end{equation*}
	and
	\begin{equation*}
		S := \left\{(\boldsymbol{a},\boldsymbol{\phi}) \in \R^n \times \R^n \mathrel{\Big|}
		\B_{k}'(\boldsymbol{\phi}) \succeq \mu_0 \right\}.
	\end{equation*}
	Let $(\boldsymbol{a},\boldsymbol{\phi}) \in \R^n \times \R^n$. We can see, considering
	the approach taken in the proof of Theorem~\ref{thm:spemdop-clop}, that
	(using brackets to group conditions)
	\begin{equation*}
		\begin{cases}
			g(\boldsymbol{a},\boldsymbol{\phi}) \in -{\left({\mathcal{C}^0(\R,\R)}^4\right)}^+ \\
			(\boldsymbol{a},\boldsymbol{\phi}) \in S
		\end{cases}
		\Leftrightarrow
		\begin{cases}
			(\boldsymbol{a},\boldsymbol{\phi}) \in {\mathcal{S}}_{0,\mu_1,\mu_2} \\
			(\boldsymbol{a},\boldsymbol{\phi}) \in S.
		\end{cases}
	\end{equation*}
	As (using brackets to group conditions)
	\begin{equation*}
		g(\boldsymbol{a},\boldsymbol{\phi}) \in \interior\left(
		-{\left({\mathcal{C}^0(\R,\R)}^4\right)}^+\right)
		\Leftrightarrow
		(\boldsymbol{a},\boldsymbol{\phi}) \in \interior({\mathcal{S}}_{0,\mu_1,\mu_2})
		\Leftrightarrow
		\begin{cases}
			\B_{k}(\boldsymbol{a}) \succ 0\\
			\B_{k}'(\boldsymbol{\phi}) \succ 0\\
			\left| \B_{k}'(\boldsymbol{a}) \right| \prec \mu_1 \cdot
				\left| \B_{k}'(\boldsymbol{\phi}) \right| \\
			\left| \B_{k}''(\boldsymbol{\phi}) \right| \prec \mu_2 \cdot
				\left| \B_{k}'(\boldsymbol{\phi}) \right|
		\end{cases}
	\end{equation*}
	holds,
	\begin{equation}\label{eq:slater_point-conditions}
		\begin{cases}
			g(\boldsymbol{a},\boldsymbol{\phi}) \in -\interior
				{\left({\mathcal{C}^0(\R,\R)}^4\right)}^+ \\
			(\boldsymbol{a},\boldsymbol{\phi}) \in S
		\end{cases}
		\Leftrightarrow
		\begin{cases}
			\B_{k}(\boldsymbol{a}) \succ 0\\
			\B_{k}'(\boldsymbol{\phi}) \succeq \mu_0\\
			\left| \B_{k}'(\boldsymbol{a}) \right| \prec \mu_1 \cdot
				\left| \B_{k}'(\boldsymbol{\phi}) \right| \\
			\left| \B_{k}''(\boldsymbol{\phi}) \right| \prec \mu_2 \cdot
				\left| \B_{k}'(\boldsymbol{\phi}) \right|
		\end{cases}
	\end{equation}
	 follows with the definition of $S$.
	 Let $c>0$ and $\tilde{a},\tilde{\phi} \colon \R \to \R$ defined as
	 \begin{align*}
	 	\tilde{a}(t) &:= c, \\
	 	\tilde{\phi}(t) &:= \mu_0 \cdot t.
	 \end{align*}
	 We can immediately see that $\tilde{a},\tilde{\phi} \in \Sigma_{k}$
	 and
	 \begin{align*}
		\tilde{a} &\succ 0,\\
		\tilde{\phi}' = \mu_0 &\succeq \mu_0,\\
		\left| \tilde{a}' \right| = 0 &\prec \mu_1 \cdot \mu_0 =
			\mu_1 \cdot \left| \tilde{\phi}' \right|, \\
		\left| \tilde{\phi}'' \right| = 0 &\prec \mu_2 \cdot \mu_0 = \mu_2 \cdot
			\left| \tilde{\phi}' \right|.
	 \end{align*}
	 Thus, using the canonical spline isomorphism, we obtain
	 $(\boldsymbol{\tilde{a}},\boldsymbol{\tilde{\phi}}) := (\B_{k}^{\text{inv}}(\tilde{a}),
	 \B_{k}^{\text{inv}}(\tilde{\phi}))$ satisfying the conditions in
	 Equation~(\ref{eq:slater_point-conditions}), and thus we have shown
	 that the SpEMDOP satisfies the \textsc{Slater} condition.
\end{proof}
The pair $(\tilde{a}(t),\tilde{\phi}(t)) = (c, \mu_0 \cdot t)$ always strictly
satisfies the constraints and is thus the strictly interior point we have been
looking for.
\par
What remains to be seen is what we can deduce from the result that our SpEMDOP
is \textsc{Slater} regular. To do that, we have to introduce the duality theory
on cone optimization problems. This is the part that was left vague in the introduction
of this section and will now be properly defined, especially in regard to the
\textsc{Lagrange} multiplier method.
\begin{definition}[Dual cone {\cite[Definition~D.6]{jah07}}]
	\label{def:dual_cone}
	Let $(V,\le_V)$ be a normed ordered vector space and
	$V^\star := \{ \ell \colon V \to \R \mid \ell \text{ linear} \}$
	its dual space. The \emph{dual cone} of $V$ is defined as
	\begin{equation*}
		V' := \left\{ \ell \in V^\star \mathrel{|}
		\forall_{x \in V^+} \colon \ell(x) \ge 0 \right\}.
	\end{equation*}	
\end{definition}
The dual cone is thus the set of linear functions $\ell$ on $V$
that map positive elements in $V$ to positive numbers in $\R$,
building a bridge from the concept of positiveness in cones to
positive numbers.
In other words, when we take any element in the positive cone
$V^+$ of a vector space $V$ and apply $\ell$ to it, it is mapped to
a positive number. Conversely, any element
in the negative cone $-V^+$ is mapped to a negative number.
\par
Having defined the dual cone, we can now define the \textsc{Lagrange}
functional that has already been introduced at the beginning of
Section~\ref{sec:regularity}.
\begin{definition}[\textsc{Lagrange} functional {\cite[Definition~6.8]{jah07}}]
	\label{def:lf}
	Let $(V,\le_V),(W,\le_W)$ be normed ordered vector spaces,
	$W^+ \ne \emptyset$, $c \colon V \to \R$ a cost function,
	$g \colon V \to W$ and $\emptyset \neq S \subseteq V$.
	The \emph{\textsc{Lagrange} functional}
	$\Lambda \colon S \times W' \to \R$ associated with the optimization problem
	\begin{align*}	
		\min_{x} \quad &
			c(x) \\
		\text{s.t.} \quad &
			g(x) \in -W^+ \\
		& x \in S
	\end{align*}
	is defined as
	\begin{equation*}
		\Lambda(x,\lambda) := c(x) + \lambda(g(x)).
	\end{equation*}
	The function $W' \ni \lambda \colon W \to \R$ is called the \emph{dual variable}.
\end{definition}
As we can see, the dual variable $\lambda$ is taken from the dual cone,
such that the \enquote{sign} of $g(x)$ is preserved. By varying $\lambda$,
we specify how much each subcomponent of $g(x)$ influences the cost function
$\Lambda(X,\lambda)$. With this in mind, we take the idea further
and make the following
\begin{definition}[\textsc{Lagrange} dual functional]\label{def:ldf}
	Let $(V,\le_V),(W,\le_W)$ be normed ordered vector spaces,
	$W^+ \ne \emptyset$, $c \colon V \to \R$ a cost function,
	$g \colon V \to W$ and $\emptyset \neq S \subseteq V$.
	The \emph{\textsc{Lagrange} dual functional}
	$\ul{\Lambda} \colon W' \to \R$ associated with the optimization problem
	\begin{align*}	
		\min_{x} \quad &
			c(x) \\
		\text{s.t.} \quad &
			g(x) \in -W^+ \\
		& x \in S
	\end{align*}
	is defined as
	\begin{equation*}
		\ul{\Lambda}(\lambda) := \inf_{x \in S} \Lambda(x,\lambda).
	\end{equation*}
\end{definition}
In the \textsc{Lagrange} dual functional, we take the
\textsc{Lagrange} functional from earlier and optimize it over
the set of candidates within set $S$. Thus, the only variable
left of this problem is the choice of $\lambda$, so to say the weights
applied to each component of $g(x)$. Consequently, as described at
the beginning of the section, we can define the dual optimization
problem as this optimization over $\lambda$.
\begin{definition}[Dual optimization problem {\cite[(6.4)]{jah07}}]
	\label{def:dop}
	Let $(V,\le_V),(W,\le_W)$ be normed ordered vector spaces,
	$W^+ \ne \emptyset$, $c \colon V \to \R$ a cost function,
	$g \colon V \to W$ and $\emptyset \neq S \subseteq V$.
	The \emph{dual optimization problem} associated with the
	(primal) optimization problem
	\begin{align*}	
		\min_{x} \quad &
			c(x) \\
		\text{s.t.} \quad &
			g(x) \in -W^+ \\
		& x \in S
	\end{align*}
	is defined as
	\begin{align*}
		\max_{\lambda} \quad &
			\ul{\Lambda}(\lambda) \\
		\text{s.t.} \quad &
			\lambda \in W'.
	\end{align*}
\end{definition}
We see that now the set of candidates is the dual cone $W'$ of $W$
and thus we are optimizing over linear functions on $W$. In
the classical optimization theory, the $\lambda$ is a set of scalars
(the \textsc{Lagrange} multipliers), one for each constraint function
that already maps to $\R$. Given we map to positive cones of general
vector spaces, we have to take the little detour and define the
$\lambda$ as a linear function like above. Consistent with the
introduced theory, we can now define the concept of strong duality
as it has already been explained in the beginning.
\begin{definition}[Strong duality {\cite[Theorem~6.7]{jah07}}]
	Let $(V,\le_V),(W,\le_W)$ be normed ordered vector spaces,
	$W^+ \ne \emptyset$, $c \colon V \to \R$ a cost function,
	$g \colon V \to W$ and $\emptyset \neq S \subseteq V$.
	The optimization problem
	\begin{align*}	
		\min_{x} \quad &
			c(x) \\
		\text{s.t.} \quad &
			g(x) \in -W^+ \\
		& x \in S
	\end{align*}
	satisfies \emph{strong duality} if and only if the cost
	functions of the primal and dual optimization problems
	attain the same value in optimality.
\end{definition}
It has to be clear here that this does not mean that
both optimization problems yield the same solution. It just means
that if we solve both optimization problems, the respective
cost functions have the same value. To bring duality and
convex-likeness together, we make the following observations.
\begin{proposition}\label{prop:clop-strong_duality}
	A convex-like optimization problem (see Definition~\ref{def:clop})
	satisfies strong duality if it satisfies the \textsc{Slater}
	condition (see Definition~\ref{def:sc}).
\end{proposition}
\begin{proof}
	See \cite[Theorem~7.12]{jah07}.
\end{proof}
\begin{theorem}\label{th:spemdop-strong_duality}
	The SpEMDOP (see Definition~\ref{def:spemdop}) with
	a convex-like spline EMD cost function
	$\boldsymbol{c}[\boldsymbol{s}] \colon \boldsymbol{\mathcal{S}}_{0,\mu_1,\mu_2}
	\to \R$ in relation to $\R_+$ satisfies strong duality.
\end{theorem}
\begin{proof}
	We have shown in Theorem~\ref{thm:spemdop-clop} that the SpEMDOP
	is a convex-like optimization problem with a convex-like spline
	EMD cost function and in Theorem~\ref{thm:spemdop-sc}
	that it satisfies the \textsc{Slater} condition. It follows directly
	with Proposition~\ref{prop:scecf-clf} that the SpEMDOP satisfies
	strong duality.
\end{proof}
\section{Conclusion}\label{sec:emd_conclusion}
The main result of this chapter is that the newly introduced EMD
optimization problem (EMDOP)
\begin{align*}
	\min_{(a,\phi)} \quad &
		c[s](a,\phi) \\
	\text{s.t.} \quad & (a,\phi) \in {\mathcal{S}}_{\mu_0,\mu_1,\mu_2}.
\end{align*}
with an EMD cost function $c[s]$ (see Definition~\ref{def:ccf} for
the definition of the canonical cost function) over the set
${\mathcal{S}}_{\mu_0,\mu_1,\mu_2}$ of IMF souls (IMFS) (see Definition~\ref{def:imfs})
satisfies strong duality, which has been shown using
its equivalent B-spline formulation. The only condition is that
the cost function is convex-like (see Definition~\ref{def:clf}).
In particular, if we find a convex EMD cost function, given every
convex function is also convex-like, we will also have automatically
shown strong duality as well. In general, the model is thus a
good object to further study the EMD from a theoretical perspective,
as the only variable is the cost function, for which only simple
properties have to be shown to obtain strong results for the entire
EMD optimization problem. This is also the reason why the cost function
has been kept as general as possible in the theoretical derivation.
\par
It is the author's impression that there are only two avenues
to further formalize the empirical mode decomposition, and none of
them is the development of more informal heuristics. The first one
is to find a convex cost function for the entire set of IMF souls,
which would be the optimal scenario. The second one is to add more
constraints to the set of IMF souls, such that an EMD cost function
is convex on this restricted set. This would require an adaption of
the proof in this section and might make some aspects much more
difficult. It remains to be seen which direction will be taken.
\par
In terms of regularization and the role of strong duality, which
at first sight \enquote{only} applies to the method of
\textsc{Lagrange} multipliers, in terms of general regularization
schemes we can make the following remark:
One can imagine the \textsc{Lagrange} multipliers to be the most
perfect regularization term possible. If we look at it intuitively,
it finds a feasible solution and optimally assigns weights to each
constraint such that the cost function is as minimal as possible.
We have shown that strong duality holds and thus that the \textsc{Lagrange}
multiplier method yields the same optimal cost function value, no
matter if one considers the constrained primal problem or the
dual \textsc{Lagrange} problem. Thus it follows that considering
regularization terms (refer to $R(\boldsymbol{a},\boldsymbol{\phi})$
at the beginning of Section~\ref{sec:regularity}) is a valid approach.
If we could not have shown strong duality, even a very good
regularization term, which comes close to the \textsc{Lagrange}
term, would not have the chance to properly \enquote{represent}
the constraints.
\par
Consequently, as the approach outlined in Chapter~\ref{ch:oss} examines
one regularization term approach using operators,
and with the results of this chapter, we can assume that it is
not wrong to approach the empirical mode decomposition like this.
In broader terms, the strong duality shown in this chapter might
even explain why many of the heuristic EMD methods work as well
as they do.
\chapter{Operator-Based Analysis of Intrinsic Mode Functions}\label{ch:oss}
Let us again consider the EMD optimization problem (EMDOP) that
we defined in the previous chapter as
\begin{align*}
	\min_{(a,\phi)} \quad &
		c[s](a,\phi) \\
	\text{s.t.} \quad & (a,\phi) \in
	{\mathcal{S}}_{\mu_0,\mu_1,\mu_2}
\end{align*}
with an EMD cost function $c[s](a,\phi)$ (for instance the canonical cost
function from Definition~\ref{def:ccf}) and the set of IMF souls
${\mathcal{S}}_{\mu_0,\mu_1,\mu_2}$.
At the beginning of Section~\ref{sec:regularity} we looked at the approach
of adding a regularization term $R(a,\phi)$ to
the cost function of the optimization problem. This regularization term
punishes violations of the constraints given by the IMF soul set and,
in the ideal case, \enquote{steers} arbitrary candidate function pairs
$(a,\phi) \in {\left( \mathcal{C}^0(\R,\R) \right)}^2$ into the
desired constraints. The resulting regularized optimization problem
\begin{align*}
	\min_{(a,\phi) \in {\left( \mathcal{C}^0(\R,\R) \right)}^2} \quad &
		c[s](a,\phi) + R(a,\phi)
\end{align*}
is unconstrained and easier to handle than the original constrained
optimization problem. We have examined the regularity of the EMDOP in
Section~\ref{sec:regularity} and found out that if we find a
\enquote{perfect} regularization operator that behaves
equivalently to the regularization of the \textsc{Lagrange}
multiplier method, we can minimize the cost function just
as well as with the constrained optimization problem.
Up to this point though, we have not yet seen a non-trivial definition
of a regularization term for the EMD optimization problem.
\par
The motivation of this chapter is to examine one such approach
for defining a regularization term that is called the
null-space-pursuit (NSP) (see \cite{ph08} and \cite{ph10}) which is
classified as a so-called operator-based signal-separation (OSS) method.
It is based on so-called \enquote{adaptive operators} that have been
introduced with an example in Chapter~\ref{ch:introduction}. The
fundamental idea is as follows: Suppose that we have a function
$h(t)$ that \enquote{contains} a function $\phi(t)$, for example
$h(t)=\cos(\phi(t))$. It is our interest to extract $\phi(t)$ from it.
To approach this problem, we can define an
operator $\mathcal{D}_{\tilde{\phi}}$ with a parameter function
$\tilde{\phi}(t)$ as
\begin{equation*}
	\left(\mathcal{D}_{\tilde{\phi}}h\right)(t) := \frac{\partial^2 h(t)}{\partial t^2} -
		\frac{\tilde{\phi}''(t)}{\tilde{\phi}'(t)} \cdot \frac{\partial h(t)}{\partial t} +
		{(\tilde{\phi}'(t))}^2 \cdot h(t),
\end{equation*}
for which it holds (see Equation~(\ref{eq:diffop-example})) that
\begin{equation*}
	\mathcal{D}_{\phi} h \equiv 0.
\end{equation*}
Adapting it to the EMD optimization problem, our goal
is to find an adaptive operator
$\mathcal{D}_{(\tilde{a},\tilde{\phi})}$ such that for
an intrinsic mode function (IMF) $u(t)$ of the form
$u(t) := a(t) \cdot \cos(\phi(t))$ with instantaneous amplitude
$a(t)$ and phase $\phi(t)$, it holds that
\begin{equation*}
	\mathcal{D}_{(a,\phi)}u \equiv 0.
\end{equation*}
Given any norm is positive definite, this is equivalent to
the norm of the operator vanishing, namely
\begin{equation*}
	{\| \mathcal{D}_{(a,\phi)}u \|}_2^2 = 0.
\end{equation*}
In the ideal case that the operator does not match (i.e.\ vanishes for) other
functions, we can make the following observation: If a function
is \enquote{annihilated} by the operator $\mathcal{D}_{(a,\phi)}$,
we can assume that the function is of the IMF form
$a(t) \cdot \cos(\phi(t))$. We can use that to our advantage
by reminding ourselves how we defined the regularization operator
$R(a,\phi)$. We want it to be exactly zero when the constraints
are satisfied and non-zero otherwise. This corresponds to the
norm of our adaptive operator and it is justified to set the
regularization of our EMDOP to
\begin{equation*}
	R(a,\phi) := {\| \mathcal{D}_{(a,\phi)}u \|}_2^2.
\end{equation*}
If a function is annihilated by the operator we can equivalently
say that the function is in the kernel of this operator. Another
name for the kernel is the \enquote{null-space}, and thus it
becomes clear why this regularization method is called the
null-space-pursuit, as we aim to vary the operator parameters
$(\tilde{a},\tilde{\phi})$ until the operator itself vanishes.
This tells us that the input function is an IMF and what the
underlying instantaneous amplitude and phase look like.
\section{IMF Differential Operator}
Up to this point we have only described the properties we would like
to see from an adaptive operator $\mathcal{D}_{(\tilde{a},\tilde{\phi})}$
for the EMD optimization problem, which we will from now on call
\enquote{IMF differential operator}. We have not yet defined one and
will do that in this section.
\par
The operator we are going to examine is a natural generalization of the
complex-valued differential operator presented in \cite{gphx17} from a
first order to a second order differential operator, which will be
elaborated later.
Before we define it we first define two operators that are used in the
expression of the IMF operator itself and will become more important later.
\begin{definition}[Instantaneous envelope derivation operator]
	\label{def:inst_envelope_deriv_operator}
	The \emph{instantaneous envelope derivation operator} is
	defined as
	\begin{equation*}
		A[a] := \frac{a'}{a}.
	\end{equation*}
\end{definition}
\begin{definition}[Inverse square continuous frequency operator]
	\label{def:inv_sq_cont_freq_operator}
	The \emph{inverse square continuous frequency operator} is
	defined as
	\begin{equation*}
		\Omega[\phi] := \frac{1}{{(\phi')}^2}.
	\end{equation*}
\end{definition}
We can think of the instantaneous envelope derivation
and inverse square continuous frequency operators as derived
expressions of $a$ and $\phi$. Using the notation from 
Definition~\ref{def:differentiation_operator} we define our operator
as follows.
\begin{definition}[IMF differential operator \cite{gphx17}]\label{def:imf-diffop}
	The \emph{IMF differential operator} is defined as
	\begin{align*}
		\mathcal{D}_{(a,\phi)} & := \Omega[\phi] \cdot D^2 +\\
		&\hspace{0.5cm} \left[ -2 \cdot \Omega[\phi] \cdot A[a] +
			\frac{1}{2} \cdot \Omega'[\phi] \right]
			\cdot D^1 +\\
		&\hspace{0.5cm} \left[ \Omega[\phi] \cdot \left( A^2[a] - A'[a] \right) -
		\frac{1}{2} \cdot \Omega'[\phi] \cdot A[a] + 1 \right] \cdot D^0.
	\end{align*}
\end{definition}
As we can see, the IMF differential operator contains derivative operators
of up to order two, which is why we call it a second order differential operator.
\subsection{Properties}
We will now show that it is in fact an operator that annihilates IMFs when
its parameters match the soul of the input IMF. For that, we remind ourselves
of the IMF operator from Definition~\ref{def:imf-operator}.
\begin{proposition}
	Let $(a,\phi) \in {\mathcal{S}}_{\mu_0,\mu_1,\mu_2}$. It holds
	\begin{equation*}
		\mathcal{D}_{(a,\phi)} \I[a,\phi] = 0.
	\end{equation*}
\end{proposition}
\begin{proof}
	We first prepare some results of derivatives for different
	expressions that will occur later. The two operators defined
	earlier behave as follows:
	\begin{align*}
		\Omega'[\phi] &= {\left( \frac{1}{{(\phi')}^2} \right)}' =
			(-2) \cdot \frac{\phi''}{{(\phi')}^3},\\
		A'[a] &= {\left( \frac{a'}{a} \right)}' =
			\frac{a''}{a} - \frac{{(a')}^2}{a^2} = \frac{a''\cdot a - {(a')}^2}{a^2}.
	\end{align*}
	Additionally, we determine the second derivative of an IMF as
	\begin{align*}
		D^2(a \cdot \cos(\phi)) = {\left( a \cdot \cos(\phi) \right)}'' &=
			{\left( a' \cdot \cos(\phi) - a \cdot \phi'
			\cdot \sin(\phi) \right)}' \\
		&= a'' \cdot \cos(\phi) - 2 \cdot a' \cdot \phi' \cdot \sin(\phi) -
			a \cdot \phi'' \cdot \sin(\phi) -
			a \cdot {\left( \phi' \right)}^2 \cos(\phi) \\
		&= \left[ -2 \cdot a' \cdot \phi' - a \cdot \phi'' \right] \cdot
			\sin(\phi) +
			\left[ a'' - a \cdot {\left( \phi' \right)}^2 \right] \cdot
			\cos(\phi).
	\end{align*}
	With these results we can look at the differential operator itself and expand the derivations accordingly by applying the differential operators:
	\begin{align*}
		\mathcal{D}_{(a,\phi)} \I[a,\phi] &=
			\mathcal{D}_{(a,\phi)} \left( a \cdot \cos(\phi) \right)\\
		&= \Omega[\phi] \cdot {\left( a \cdot \cos(\phi) \right)}'' +\\
		&\hspace{0.45cm} \left[ -2 \cdot \Omega[\phi] \cdot A[a] +
			\frac{1}{2} \cdot \Omega'[\phi] \right]
			\cdot {\left( a \cdot \cos(\phi) \right)}' +\\
		&\hspace{0.45cm} \left[ \Omega[\phi] \cdot \left( A^2[a] - A'[a] \right) -
			\frac{1}{2} \cdot \Omega'[\phi] \cdot A[a] + 1 \right]
			\cdot {\left( a \cdot \cos(\phi) \right)}.
	\end{align*}
	To show the proposition we now calculate the derivatives making
	use of the chain rule and simplify:
	\begin{align*}
		\mathcal{D}_{(a,\phi)} \I[a,\phi] &= {\left( \frac{1}{{(\phi')}^2} \right)} \cdot
			\left\{ \left[ -2 \cdot a' \cdot \phi' - a \cdot \phi'' \right] \cdot
			\sin(\phi) +
			\left[ a'' - a \cdot {\left( \phi' \right)}^2 \right] \cdot
			\cos(\phi) \right\} +\\
		&\hspace{0.45cm} \left[ \frac{-2 \cdot a'}{a \cdot {\left( \phi' \right)}^2} +
			\frac{1}{2} \cdot (-2) \cdot \frac{\phi''}{{(\phi')}^3} \right] \cdot 
			{\left( a' \cdot \cos(\phi) - a \cdot \phi' \cdot \sin(\phi) \right)} +\\
		&\hspace{0.45cm} \left[ {\left( \frac{1}{{(\phi')}^2} \right)} \cdot
			\left( \frac{{(a')}^2}{a^2} - \frac{a \cdot a'' - {(a')}^2}{a^2} \right) -
			\frac{1}{2} \cdot (-2) \cdot \frac{a' \cdot \phi''}{a \cdot {(\phi')}^3} + 1\right] \cdot
			{\left( a \cdot \cos(\phi) \right)} \\
		&= \frac{-2 \cdot a' \cdot \phi' - a \cdot \phi''}{{(\phi')}^2} \cdot \sin(\phi) +
			\frac{a'' - a \cdot {\left( \phi' \right)}^2}{{(\phi')}^2} \cdot \cos(\phi) +\\
		&\hspace{0.45cm} \left[ \frac{-2 \cdot a'}{a \cdot {\left( \phi' \right)}^2} -
			\frac{\phi''}{{(\phi')}^3} \right] \cdot 
			{\left( a' \cdot \cos(\phi) - a \cdot \phi' \cdot \sin(\phi) \right)} +\\
		&\hspace{0.45cm} \left[ \frac{2 \cdot {(a')}^2 - a \cdot a''}
			{a^2 \cdot {(\phi')}^2} + \frac{a' \cdot \phi''}{a \cdot {(\phi')}^3} + 1 \right] \cdot
			{\left( a \cdot \cos(\phi) \right)}.
	\end{align*}
	Sine and cosine are separated and we show that their coefficients are zero, implying
	that the entire expression is zero, as follows:
	\begin{align*}
		\mathcal{D}_{(a,\phi)} \I[a,\phi] &=
			\left[ \frac{-2 \cdot a'}{\phi'} - \frac{a \cdot \phi''}{{(\phi')}^2}\right] \cdot
			\sin(\phi) + \left[ \frac{a''}{{(\phi')}^2} - a \right] \cdot \cos(\phi) +\\
		&\hspace{0.45cm} \left[ \frac{2 \cdot a'}{\phi'} + \frac{a \cdot \phi''}
			{{(\phi')}^2} \right] \cdot \sin(\phi) +
			\left[ \frac{-2 \cdot {\left(a'\right)}^2}{a \cdot {\left( \phi' \right)}^2} -
			\frac{a' \cdot \phi''}{{(\phi')}^3} \right] \cdot \cos(\phi) +\\
		&\hspace{0.45cm} \left[ \frac{2 \cdot {(a')}^2}{a \cdot {(\phi')}^2} - \frac{a''}
			{{(\phi')}^2} + \frac{a' \cdot \phi''}{{(\phi')}^3} + a \right] \cdot \cos(\phi) \\
		&= \left[ \frac{-2 \cdot a'}{\phi'} - \frac{a \cdot \phi''}{{(\phi')}^2} +
			\frac{2 \cdot a'}{\phi'} + \frac{a \cdot \phi''}{{(\phi')}^2} \right] \cdot \sin(\phi) +\\
		&\hspace{0.45cm} \left[ \frac{a''}{{(\phi')}^2} - a +
			\frac{-2 \cdot {\left(a'\right)}^2}{a \cdot {\left( \phi' \right)}^2} -
			\frac{a' \cdot \phi''}{{(\phi')}^3} +
			\frac{2 \cdot {(a')}^2}{a \cdot {(\phi')}^2} - \frac{a''}
			{{(\phi')}^2} + \frac{a' \cdot \phi''}{{(\phi')}^3} + a \right] \cdot \cos(\phi)\\
		&= 0. \qedhere
	\end{align*}
\end{proof}
This shows that $\mathcal{D}_{(a,\phi)}$ is in fact an annihilating operator. When considering
the IMF differential operator from Definition~\ref{def:imf-diffop} again,
one can observe that the only way the
$a$ and $\phi$ \enquote{interface} with the operator is through
instantaneous envelope derivation operator $A[a]$
and inverse square continuous frequency operator $\Omega[\phi]$.
Given both are functions just like $a$ and $\phi$, we can
express the operator parametrized by $A$ and $\Omega$ instead of $a$
and $\phi$ and call it the modified IMF operator.
\begin{definition}[Modified IMF operator]\label{def:modified-imf-diffop}
	The \emph{modified IMF operator} is defined as
	\begin{align*}
		\mathcal{\tilde{D}}_{(A,\Omega)} & := \Omega \cdot D^2 +\\
		&\hspace{0.5cm} \left[ -2 \cdot \Omega \cdot A +
			\frac{1}{2} \cdot \Omega' \right]
			\cdot D^1 +\\
		&\hspace{0.5cm} \left[ \Omega \cdot \left( {A}^2 - A' \right) -
		\frac{1}{2} \cdot \Omega' \cdot A + 1 \right] \cdot D^0.
	\end{align*}
\end{definition}
How to determine $a$ from $A$ and $\phi$ from $\Omega$ shall not yet
be of concern here, but for instance in the case of $\Omega$,
the inverse square root of $\Omega$ yields $\phi'$ directly
(compare Definition~\ref{def:inv_sq_cont_freq_operator}).
\par
It is now in our interest to examine the behaviour of the differential
operator under its parameters. As we vary $A$ and $\Omega$, we want to
know that if we found an annihilating pair we really obtained a unique
solution or not. We will approach this question just like the
cost functions in Section~\ref{sec:cost_functions} and consider the
$A$ and $\Omega$ to be spline functions that vary over their
B-spline basis coefficients $\boldsymbol{A} \in \R^n$ and
$\boldsymbol{\Omega} \in \R^n$. To give an example, we consider $A$
to be the spline function (using Definitions \ref{def:b-spline} and
\ref{def:coefficient_spline_mapping})
\begin{equation*}
	\B_{k}(\boldsymbol{A}) = \sum_{i=0}^{n-1} A_i \cdot B_{i,k}
\end{equation*}
that varies over the vector entries $A_i$ of $\boldsymbol{A}$.
If we consider the norm of this function, we can for instance
calculate its partial derivative in $A_m$ for $m \in \{0,\dots,n-1\}$
using the chain rule as
\begin{equation*}
	\frac{\partial {\| \B_{k}(\boldsymbol{A}) \|}_2^2}{\partial A_m} =
	\frac{\partial}{\partial A_m} \left(
		\int_{-\infty}^{\infty} {\left( \B_{k}(\boldsymbol{A}) \right)}^2 \mathrm{d}t
		\right) =
	\int_{-\infty}^{\infty}
		\frac{\partial {\left( \B_{k}(\boldsymbol{A}) \right)}^2}
		{\partial A_m} \mathrm{d}t =
	\int_{-\infty}^{\infty}
		2 \cdot \B_{k}(\boldsymbol{A}) \cdot B_{m,k}(t) \mathrm{d}t.
\end{equation*}
To go even further, we can of course also partially derive again,
this time in $A_p$ for $p \in \{0,\dots,n-1\}$. We obtain
\begin{equation*}
	\frac{\partial^2 {\| \B_{k}(\boldsymbol{A}) \|}_2^2}
		{\partial A_m \partial A_p} =
	\frac{\partial}{\partial A_m} \left(
		\int_{-\infty}^{\infty}
		2 \cdot \B_{k}(\boldsymbol{A}) \cdot B_{m,k}(t) \mathrm{d}t
		\right) =
	\int_{-\infty}^{\infty}
		2 \cdot B_{m,k}(t) \cdot B_{p,k}(t) \mathrm{d}t.
\end{equation*}
The result of this particular observation is that the covariation
of ${\| \B_{k}(\boldsymbol{A}) \|}_2^2$ in $A_m$ and $A_p$ is directly
related to the orthogonality of $B_{m,k}$ and $B_{p,k}$. If the
B-spline basis were truly orthogonal, the final integral will always
be zero. However, the B-spline basis is not orthogonal. Thus, if
$m$ and $p$ are \enquote{close} to each other or even equal, the
integral will be positive. Another interpretation is to consider
the Hessian matrix in partial derivates in $A_i$, which can be
used to prove that a function is convex in multiple variables.
After all, the function ${\| \B_{k}(\boldsymbol{A}) \|}_2^2$ is
a mapping $\R^n \to \R$. Making use of this general technique,
we prove the following
\begin{theorem}\label{thm:imf-operator-convex}
	Let $f \in \mathcal{C}^2(\R,\R)$ and $k \ge 4$ (for derivability).
	${\left\| \tilde{\mathcal{D}}_{(\B_{k}(\boldsymbol{A}),\B_{k}(\boldsymbol{\Omega}))}
	f \right\|}_2^2$ is strictly convex in $\boldsymbol{\Omega}$ but not
	convex in $\boldsymbol{A}$.
\end{theorem}
\begin{proof}
	By definition we obtain that
	\begin{equation*}
		{\left\| \tilde{\mathcal{D}}_{(\B_{k}(\boldsymbol{A}),\B_{k}(\boldsymbol{\Omega}))}
		f \right\|}_2^2 =
		\int_{-\infty}^{\infty}
		{\left(
			\tilde{\mathcal{D}}_{(\B_{k}(\boldsymbol{A}),\B_{k}(\boldsymbol{\Omega}))}f
		\right)}^2
		\mathrm{d}t.
	\end{equation*}
	In particular we first consider the term within the norm itself and expand it
	from its definition (see Definition~\ref{def:modified-imf-diffop})
	\begin{align*}
		\tilde{\mathcal{D}}_{(\B_{k}(\boldsymbol{A}),\B_{k}(\boldsymbol{\Omega}))} f &=
			\B_{k}(\boldsymbol{\Omega}) \cdot f'' +\\
		&\hspace{0.5cm} \bigg[ -2 \cdot \B_{k}(\boldsymbol{\Omega}) \cdot \B_{k}(\boldsymbol{A}) +
			\frac{1}{2} \cdot {\left(\B_{k}(\boldsymbol{\Omega})\right)}' \bigg] \cdot f' +\\
		&\hspace{0.5cm} \bigg[ \B_{k}(\boldsymbol{\Omega}) \cdot \left(
			{\left(\B_{k}(\boldsymbol{A})\right)}^2 - {\left(\B_{k}(\boldsymbol{A})\right)}' \right) -
			\frac{1}{2} \cdot {\left( \B_{k}(\boldsymbol{\Omega}) \right)}' \cdot \B_{k}(\boldsymbol{A}) +
			1 \bigg] \cdot f.
	\end{align*}
	Making use of Definition~\ref{def:coefficient_spline_mapping} to expand
	the spline functions into B-spline expressions it follows
	\begin{align*}
		\tilde{\mathcal{D}}_{(\B_{k}(\boldsymbol{A}),\B_{k}(\boldsymbol{\Omega}))} f 
		&= \left(\sum_{i=0}^{n-1} \Omega_i \cdot B_{i,k}\right) \cdot f'' +\\
		&\hspace{0.5cm} \Bigg[ -2 \cdot 
			\left(\sum_{i=0}^{n-1} \Omega_i \cdot B_{i,k}\right) \cdot
			{\left(\sum_{i=0}^{n-1} A_i \cdot B_{i,k}\right)} +\\
		&\hspace{0.7cm}
			\frac{1}{2} \cdot \left(\sum_{i=0}^{n-1} \Omega_i \cdot
			B'_{i,k}\right) \Bigg] \cdot f' +\\
		&\hspace{0.5cm} \Bigg[ \left(\sum_{i=0}^{n-1}
			\Omega_i \cdot B_{i,k}\right) \cdot \Bigg(
			{\left(\sum_{i=0}^{n-1} A_i \cdot B_{i,k}\right)}^2 -
			{\left(\sum_{i=0}^{n-1} A_i \cdot B'_{i,k}\right)} \Bigg) -\\
		&\hspace{0.7cm}
			\frac{1}{2} \cdot \left(\sum_{i=0}^{n-1} \Omega_i \cdot B'_{i,k}\right)
			\cdot \left(\sum_{i=0}^{n-1} A_i \cdot B_{i,k}\right) +
			1 \Bigg] \cdot f.
	\end{align*}
	We approach this proof checking if the requirements of
	Theorem~\ref{thm:gershgorin-hadamard} hold for distinct partial
	derivatives for entries of $\boldsymbol{A}$ and $\boldsymbol{\Omega}$.
	By applying partial derivatives we obtain that for $m,p \in \{ 0,\dots,n-1 \}$
	it holds with the chain rule for partial derivatives in $\boldsymbol{A}$
	\begin{align}
		\frac{\partial^2 {\left\|\tilde{\mathcal{D}}_{(\B_{k}(\boldsymbol{A}),
			\B_{k}(\boldsymbol{\Omega}))} f\right\|}_2^2}
			{\partial A_m \partial A_p}(\boldsymbol{A},\boldsymbol{\Omega}) &=
			\frac{\partial}{\partial A_p} \left(
				\int_{-\infty}^{\infty}
				2
				\cdot
				\tilde{\mathcal{D}}_{(\B_{k}(\boldsymbol{A}),\B_{k}(\boldsymbol{\Omega}))}f
				\cdot
				\frac{
					\partial
					\tilde{\mathcal{D}}_{(\B_{k}(\boldsymbol{A}),\B_{k}(\boldsymbol{\Omega}))}f
				}{\partial A_m}				
				\mathrm{d}t
			\right) \notag \\
		&= \int_{-\infty}^{\infty}
			2
			\cdot
			\frac{
				\partial
				\tilde{\mathcal{D}}_{(\B_{k}(\boldsymbol{A}),\B_{k}(\boldsymbol{\Omega}))}f
			}{\partial A_m}
			\cdot
			\frac{
				\partial
				\tilde{\mathcal{D}}_{(\B_{k}(\boldsymbol{A}),\B_{k}(\boldsymbol{\Omega}))}f
			}{\partial A_p} + \notag \\
		&\phantom{=\int_{-\infty}^{\infty}\,\,\,}
			2
			\cdot
			\tilde{\mathcal{D}}_{(\B_{k}(\boldsymbol{A}),\B_{k}(\boldsymbol{\Omega}))}f
			\cdot
			\frac{
				\partial^2
				\tilde{\mathcal{D}}_{(\B_{k}(\boldsymbol{A}),\B_{k}(\boldsymbol{\Omega}))}f
			}{\partial A_m \partial A_p}
			\mathrm{d}t \label{eq:proof-diffop-convex-A}
	\end{align}
	and analogously for partial derivatives in $\boldsymbol{\Omega}$
	\begin{align}
		\frac{\partial^2 {\left\|\tilde{\mathcal{D}}_{(\B_{k}(\boldsymbol{A}),
			\B_{k}(\boldsymbol{\Omega}))} f\right\|}_2^2}
			{\partial \Omega_m \partial \Omega_p}(\boldsymbol{A},\boldsymbol{\Omega}) &=
			\int_{-\infty}^{\infty}
			2
			\cdot
			\frac{
				\partial
				\tilde{\mathcal{D}}_{(\B_{k}(\boldsymbol{A}),\B_{k}(\boldsymbol{\Omega}))}f
			}{\partial \Omega_m}
			\cdot
			\frac{
				\partial
				\tilde{\mathcal{D}}_{(\B_{k}(\boldsymbol{A}),\B_{k}(\boldsymbol{\Omega}))}f
			}{\partial \Omega_p} + \notag \\
		&\phantom{=\int_{-\infty}^{\infty}\,\,\,}
			2
			\cdot
			\tilde{\mathcal{D}}_{(\B_{k}(\boldsymbol{A}),\B_{k}(\boldsymbol{\Omega}))}f
			\cdot
			\frac{
				\partial^2
				\tilde{\mathcal{D}}_{(\B_{k}(\boldsymbol{A}),\B_{k}(\boldsymbol{\Omega}))}f
			}{\partial \Omega_m \partial \Omega_p}
			\mathrm{d}t. \label{eq:proof-diffop-convex-Omega}
	\end{align}
	First for partial derivatives in $\boldsymbol{A}$, we consider the bare
	derivatives of the operator (without the norm) that we found within the derivative
	expressions of the operator within the norm. For the first order we find that
	\begin{align*}
		\frac{\partial \tilde{\mathcal{D}}_{(\B_{k}(\boldsymbol{A}),\B_{k}(\boldsymbol{\Omega}))} f }
			{\partial A_m}(\boldsymbol{A},\boldsymbol{\Omega}) &= \left[ -2 \cdot 
			{\B_{k}(\boldsymbol{\Omega})} \cdot B_{m,k} \right] \cdot f' +\\
		&\hspace{0.41cm} \left[ \B_{k}(\boldsymbol{\Omega}) \cdot \left(
			2 \cdot \B_{k}(\boldsymbol{A}) \cdot B_{m,k} - B'_{m,k} \right) -
			\frac{1}{2} \cdot {\left( \B_{k}(\boldsymbol{\Omega}) \right)}'
			\cdot B_{m,k} \right] \cdot f\\
		&= \left[ -2 \cdot 
			{\B_{k}(\boldsymbol{\Omega})} \cdot B_{m,k} \right] \cdot f' +\\
		&\hspace{0.46cm} \Bigg[ \B_{k}(\boldsymbol{\Omega}) \cdot \left(
			2 \cdot {\left(\sum_{i=0}^{n-1} A_i \cdot B_{i,k}\right)} \cdot
			B_{m,k} - B'_{m,k} \right) -\\
		&\hspace{0.7cm} \frac{1}{2} \cdot {\left( \B_{k}(\boldsymbol{\Omega}) \right)}'
			\cdot B_{m,k} \Bigg] \cdot f,
	\end{align*}
	and for the second order we finally obtain
	\begin{equation*}
		\frac{\partial^2 \tilde{\mathcal{D}}_{(\B_{k}(\boldsymbol{A}),\B_{k}(\boldsymbol{\Omega}))} f }
			{\partial A_m \partial A_p}(\boldsymbol{A},\boldsymbol{\Omega}) =
			2 \cdot \B_{k}(\boldsymbol{\Omega}) \cdot B_{m,k} \cdot B_{p,k} \cdot f \not\equiv 0.
	\end{equation*}
	This implies that for $m = p$ the sign of the term in Equation~(\ref{eq:proof-diffop-convex-A})
	is not positive, as the bare operator is not zeroed out in the second summand.
	Thus the necessary condition in Theorem~\ref{thm:gershgorin-hadamard}
	that all diagonal elements of the Hessian matrix must be positive is violated. It follows
	that we can not show convexity for partial derivatives in $\boldsymbol{A}$, as the Hessian
	matrix for derivatives in $\boldsymbol{A}$ can not be shown to be positive semidefinite.
	\par
	For partial derivatives of the bare operator in $\boldsymbol{\Omega}$ we find
	for the first order that
	\begin{align*}
		\frac{\partial \tilde{\mathcal{D}}_{(\B_{k}(\boldsymbol{A}),\B_{k}(\boldsymbol{\Omega}))} f }
			{\partial \Omega_m}(\boldsymbol{A},\boldsymbol{\Omega}) &=
			B_{m,k} \cdot f'' +\\
		&\hspace{0.5cm} \bigg[ -2 \cdot \B_{k}(\boldsymbol{A}) \cdot B_{m,k} +
			\frac{1}{2} \cdot B'_{m,k} \bigg] \cdot f' +\\
		&\hspace{0.5cm} \bigg[ B_{m,k} \cdot \left(
			{\left(\B_{k}(\boldsymbol{A})\right)}^2 - {\left(\B_{k}(\boldsymbol{A})\right)}' \right) -
			\frac{1}{2} \cdot \B_{k}(\boldsymbol{A}) \cdot B'_{m,k} \bigg] \cdot f
	\end{align*}
	which yields for the second order that
	\begin{equation*}
		\frac{\partial^2 \tilde{\mathcal{D}}_{(\B_{k}(\boldsymbol{A}),\B_{k}(\boldsymbol{\Omega}))} f }
			{\partial \Omega_m \partial \Omega_p}(\boldsymbol{A},\boldsymbol{\Omega}) \equiv 0.
	\end{equation*}
	We obtain from this result that the second summand in the integral in
	Equation~(\ref{eq:proof-diffop-convex-Omega}) is zero and it holds that
	\begin{equation*}
		\frac{\partial^2 {\left\|\tilde{\mathcal{D}}_{(\B_{k}(\boldsymbol{A}),
		\B_{k}(\boldsymbol{\Omega}))} f\right\|}_2^2}
		{\partial \Omega_m \partial \Omega_p}(\boldsymbol{A},\boldsymbol{\Omega}) =
		\int_{-\infty}^{\infty}
		2
		\cdot
		\frac{
			\partial
			\tilde{\mathcal{D}}_{(\B_{k}(\boldsymbol{A}),\B_{k}(\boldsymbol{\Omega}))}f
		}{\partial \Omega_m}
		\cdot
		\frac{
			\partial
			\tilde{\mathcal{D}}_{(\B_{k}(\boldsymbol{A}),\B_{k}(\boldsymbol{\Omega}))}f
		}{\partial \Omega_p}
		\mathrm{d}t.
	\end{equation*}
	We can immediately see that the Hessian matrix of
	${\left\|\tilde{\mathcal{D}}_{(\B_{k}(\boldsymbol{A}),
	\B_{k}(\boldsymbol{\Omega}))} f\right\|}_2^2$ for derivatives in $\boldsymbol{\Omega}$
	is symmetric. In particular, for $m = p$, it also holds
	\begin{equation*}
		\frac{\partial^2 {\left\|\tilde{\mathcal{D}}_{(\B_{k}(\boldsymbol{A}),
		\B_{k}(\boldsymbol{\Omega}))} f\right\|}_2^2}
		{\partial \Omega_m^2}(\boldsymbol{A},\boldsymbol{\Omega}) =
		\int_{-\infty}^{\infty}
		2
		\cdot
		{\left(
		\frac{
			\partial
			\tilde{\mathcal{D}}_{(\B_{k}(\boldsymbol{A}),\B_{k}(\boldsymbol{\Omega}))}f
		}{\partial \Omega_m}
		\right)}^2
		\mathrm{d}t > 0,
	\end{equation*}
	which means that the diagonal entries of the Hessian matrix for
	$\boldsymbol{\Omega}$ are strictly positive. It is also diagonally dominant
	with the same argument as in the proof of Proposition~\ref{prop:cscf-convex},
	namely due to the compact support and partial orthogonality of the
	B-spline basis functions. It follows with Theorem~\ref{thm:gershgorin-hadamard}
	that the Hessian matrix for derivatives in $\boldsymbol{\Omega}$ is positive definite
	and by Proposition~\ref{prop:hessian_convexity_condition} that
	${\left\|\tilde{\mathcal{D}}_{(\B_{k}(\boldsymbol{A}),
	\B_{k}(\boldsymbol{\Omega}))} f\right\|}_2^2$ is strictly convex in $\boldsymbol{\Omega}$.
\end{proof}
Given this result we do not have the theoretical guarantee that our
operator gives us a unique and minimal solution, as it is not generally
convex. This is obviously undesirable as our primary motivation is to
find and examine methods that have a stronger theoretical foundation
than the classic heuristic EMD methods.
\subsection{Simplification}\label{subsec:imf-operator-simplification}
The reassuring part of the result in Theorem~\ref{thm:imf-operator-convex}
is that the function is strictly convex if we reduce the variation to
$\boldsymbol{\Omega}$ and keep $\boldsymbol{A}$ constant. Without loss of
generality, if we know that our input IMF has constant amplitude $1$,
namely that it only has the form $\cos(\phi)$, we can apply a simplified differential
operator to it of which we know that it is strictly convex. This assumption may
sound a bit too extravagant, but we will show in Chapter~\ref{ch:hobm}
that it is meaningful and use the results of the following subsection
to extract the instantaneous phase from IMFs with constant amplitude $1$.
\par
If we know that our input IMFs will have the form $\cos(\phi)$, we might
wonder how our IMF differential operator changes under this assumption.
\begin{proposition}
	It holds
	\begin{equation*}
		\mathcal{D}_{(1,\phi)} = \Omega[\phi] \cdot D^2 +
		\frac{1}{2} \cdot \Omega'[\phi] \cdot D^1 +
		D^0.
	\end{equation*}
\end{proposition}
\begin{proof}
	It follows directly from Definition~\ref{def:imf-diffop} and observing that
	\begin{equation*}
		A[1] = \frac{{(1)}'}{1} = 0.\qedhere
	\end{equation*}
\end{proof}
The great simplification of the operator is apparent. We now wonder how
the modified IMF operator behaves under the assumption that $a \equiv 1$.
Making the observation that the instantaneous envelope derivation operator
$A[a]$ (see Definition~\ref{def:inst_envelope_deriv_operator}) vanishes
for $a \equiv 1$, as $A[1] = {(1)}' / 1 = 0$, we can
see that the parametrization for the modified IMF operator is
$(0,\Omega)$ and we can formulate the following
\begin{corollary}\label{cor:modified-imf-diffop-convex}
	Let $f \in \mathcal{C}^2(\R,\R)$ and $k \ge 4$ (for derivability).
	${\left\| \tilde{\mathcal{D}}_{(0,\B_{k}(\boldsymbol{\Omega}))}
	f \right\|}_2^2$ is strictly convex in $\boldsymbol{\Omega}$.
\end{corollary}
\begin{proof}
	This follows directly from Theorem~\ref{thm:imf-operator-convex}.
\end{proof}
Given this convexity property, we have a theoretical guarantee that we reach
a global minimum for a given input IMF and a unique $\boldsymbol{\Omega}$.
As already mentioned earlier, we obtain the desired instantaneous frequency
from $\boldsymbol{\Omega}$ by inversely applying the inverse square
continuous frequency operator in Definition~\ref{def:inv_sq_cont_freq_operator}.
This equates to inverting and taking the square root of $\Omega$, which
is a relatively simple operation.
\subsection{Discretization}
Given the results from this chapter and especially Theorem~\ref{thm:imf-operator-convex}
we will as follows only consider the simple case with constant amplitude $a \equiv 1$
for discretization, as this will also be the only relevant case for the toolbox
presented in Chapter~\ref{ch:hobm} given we can't use the differential operator
to extract the amplitude anyway.
\par
We have shown in Proposition~\ref{cor:modified-imf-diffop-convex} that
${\left\| \tilde{\mathcal{D}}_{(0,\B_{k}(\boldsymbol{\Omega}))}
f \right\|}_2^2$ is strictly convex in $\boldsymbol{\Omega}$, however,
for the discretization we can make two observations to simplify it:
The first is that given $\boldsymbol{\Omega} \in \R^n$ we need a system of
at least $n$ samples and additional boundary conditions to solve the problem.
The second is that given we have a uniform grid we can,
instead of minimizing an integral function, minimize at least $n$ samples
$\left(\tilde{\mathcal{D}}_{(0,\B_{k}(\boldsymbol{\Omega}))}
f\right)(t_i)$ over $\boldsymbol{\Omega}$. We take this detour as
we can see that
$\tilde{\mathcal{D}}_{(0,\B_{k}(\boldsymbol{\Omega}))} f$ is linear
in $\boldsymbol{\Omega}$, yielding a least squares problem of
the form
\begin{equation*}
	\left\| A \cdot \boldsymbol{\Omega} - b \right\|^2,
\end{equation*}
where $A$ is the matrix representing $\tilde{\mathcal{D}}_
{(0,\B_{k}(\boldsymbol{\Omega}))}$ and $b$ is the vector of
samples in $t_i$. This is better than the nonlinear problem
we would obtain by just using the squared integral.
\par
Given the precomputation-concepts of the toolbox we make use of the
precomputed \enquote{extended grid} (see Section~\ref{sec:ethos_toolbox})
and evaluate $\tilde{\mathcal{D}}_{(0,\B_{k}(\boldsymbol{\Omega}))} f$ on the extended
grid, obtaining more than $n$ equations, one for each point on the extended
grid. This way we obtain implicit boundary conditions, saving us from proposing
possibly wrong or ill-chosen ones in the process. See Subsection~\ref{subsec:boundary_effects}
for more reflexions on boundary effects.
\par
The instantaneous frequency is calculated from $\boldsymbol{\Omega}$
by evaluating $\B_k(\boldsymbol{\Omega})$, applying an inversion and
square root and running a B-splines-fit on the resulting data. Given
the nature of the transformation the inverse square
continuous frequency operator in
Definition~\ref{def:inv_sq_cont_freq_operator} specifies it is
most likely impossible to exploit any B-spline property to circumvent this
step and directly work with the B-spline coefficient vector $\boldsymbol{\Omega}$.
\section{Examples}\label{sec:examples-diffop}
Following the previous theoretical perspective, this section gives
a few examples on the numerical behaviour of the simplified
IMF differential operator. For this purpose we restrict ourselves
to IMFs with constant instantaneous amplitude $a \equiv 1$ and known
analytical form and apply our operator to them. The question is
how well we manage to extract the instantaneous phase $\phi$,
which we can assess by comparing the results to the ground truth.
The examples were implemented using the ETHOS-toolbox and can be found
in Listing~\ref{lst:examples-regopc}.
\par
The parameters $(k,q,n)$ given in the figure captions refer to the
spline order $k$, in-fill-count $q$ (see Section~\ref{sec:ethos_toolbox})
and number of B-spline basis functions $n$. See Subsection~\ref{subsec:boundary_effects}
for a discussion on the boundary effects of these examples in the context
of information theory and other literature.
\begin{example}[Constant frequency]\label{ex:diffop-0}
	As an introduction consider the simple IMF
	\begin{equation}\label{eq:example_diffop-0-u}
		u_0(t) := \cos(\phi_0(t)) := \cos(40 \cdot t),
	\end{equation}
	on the interval $[0,1]$ (see Figure~\ref{fig:example_diffop-0-u}).
	The analytical instantaneous frequency $\phi'_0$ is
	$40$, i.e.\ the IMF is of constant frequency.
	\par
	Given the instantaneous amplitude is constantly $1$, we
	can use the toolbox to fit the simple IMF differential operator to $u_0$ to calculate the
	numerical instantaneous frequency $\tilde{\phi}'_0$.
	\par	
	The difference between $\phi'_0$ and $\tilde{\phi}'_0$
	is too small to be visible in a normal plot and thus we examine the semi-log plot of the
	relative error (see Figure~\ref{fig:example_diffop-0-relerr}). We can see that the
	relative error is at most $0.1\%$ briefly at the beginning and stays below
	$10^{-4}$ on the remaining interval.
	\begin{figure}[htbp]
		\centering
		\begin{tikzpicture}
			\begin{axis}[xlabel=$t$, ylabel=$u_0(t)$]
			\addplot [mark=none, smooth] table [x=x, y=y, col sep=comma] {examples/regop.data/0-u.csv};
			\end{axis}
		\end{tikzpicture}
		\caption{Plot of the IMF $u_0$ (see (\ref{eq:example_diffop-0-u})) from
		Example~\ref{ex:diffop-0}.}
		\label{fig:example_diffop-0-u}
	\end{figure}
	\begin{figure}[htbp]
		\centering
		\begin{tikzpicture}
			\begin{semilogyaxis}[xlabel=$t$, ylabel near ticks,
				ylabel=$\left| \frac{\phi'_0(t) -
				\tilde{\phi}'_0(t)}{\phi'_0(t)} \right|$]
			\addplot [mark=none, smooth] table [x=x, y=y, col sep=comma] {examples/regop.data/0-relerr.csv};
			\end{semilogyaxis}
		\end{tikzpicture}
		\caption{Semi-log plot of the relative error between the analytical
		instantaneous frequency $\phi'_0=40$ and the solution
		$\tilde{\phi}'_0$, obtained using the
		IMF differential operator for $(k, q, n) = (4, 4, 180)$, from
		 Example~\ref{ex:diffop-0}.}
		\label{fig:example_diffop-0-relerr}
	\end{figure}
\end{example}
\begin{example}[Harmonic peaks]\label{ex:diffop-1}
	Consider the simple IMF
	\begin{equation}\label{eq:example_diffop-1-u}
		u_1(t) := \cos(\phi_1(t)) := \cos(3 \cdot \sin(3 \cdot \pi \cdot t) +
		16 \cdot \pi \cdot t),
	\end{equation}
	on the interval $[0,1]$ (see Figure~\ref{fig:example_diffop-1-u}).
	It is easy to analytically obtain the instantaneous
	frequency $\phi'_1$ of the signal by calculation, namely
	\begin{equation}\label{eq:example_diffop-1-freq-analytic}
		\phi'_1(t) = 9 \cdot \pi \cdot \cos(3 \cdot \pi \cdot t) +
		16 \cdot \pi,
	\end{equation}
	which you can find pictured in Figure~\ref{fig:example_diffop-1-freq-analytic}.
	\par
	Given the IMF has constant instantaneous amplitude $1$ we can use
	the toolbox to fit the simple IMF differential operator to $u_1$ to calculate the numerical
	instantaneous frequency $\tilde{\phi}'_1$. As can be seen the instantenous frequency
	itself is a wave function too, explaining the irregular shape of the IMF.
	\par	
	The difference between $\phi'_1$ and $\tilde{\phi}'_1$
	is too small to be visible in a normal plot and we thus examine the semi-log plot of the
	relative error (see Figure~\ref{fig:example_diffop-1-relerr}). We can see that the
	relative error is at most $1\%$ briefly at the beginning and between $1$ and $4$ orders
	of magnitude lower on the remaining interval.
	\begin{figure}[htbp]
		\centering
		\begin{tikzpicture}
			\begin{axis}[xlabel=$t$, ylabel=$u_1(t)$]
			\addplot [mark=none, smooth] table [x=x, y=y, col sep=comma] {examples/regop.data/1-u.csv};
			\end{axis}
		\end{tikzpicture}
		\caption{Plot of the IMF $u_1$ (see (\ref{eq:example_diffop-1-u})) from
		Example~\ref{ex:diffop-1}.}
		\label{fig:example_diffop-1-u}
	\end{figure}
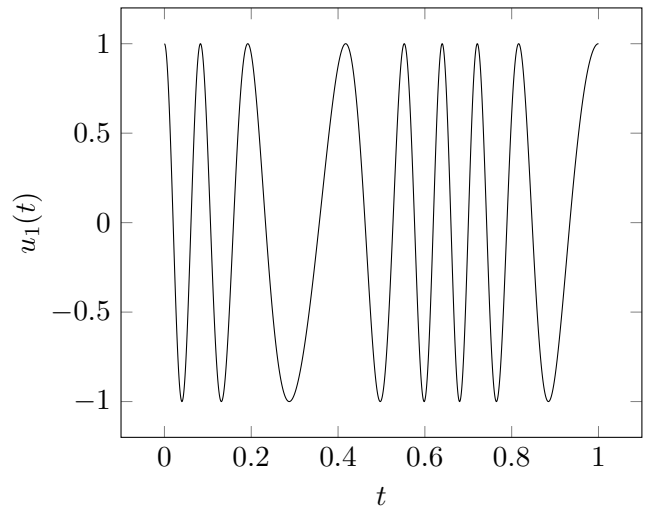	
	\begin{figure}[htbp]
		\centering
		\begin{tikzpicture}
			\begin{axis}[xlabel=$t$, ylabel=$\phi'_1(t)$]
			\addplot [mark=none, smooth] table [x=x, y=y, col sep=comma] {examples/regop.data/1-freq-analytic.csv};
			\end{axis}
		\end{tikzpicture}
		\caption{Plot of the analytical instantaneous frequency
		$\phi'_1$ (see (\ref{eq:example_diffop-1-freq-analytic}))
		from Example~\ref{ex:diffop-1}.}
		\label{fig:example_diffop-1-freq-analytic}
	\end{figure}
	\begin{figure}[htbp]
		\centering
		\begin{tikzpicture}
			\begin{semilogyaxis}[xlabel=$t$, ylabel near ticks,
				ylabel=$\left| \frac{\phi'_1(t) -
				\tilde{\phi}'_1(t)}{\phi'_1(t)} \right|$]
			\addplot [mark=none, smooth] table [x=x, y=y, col sep=comma] {examples/regop.data/1-relerr.csv};
			\end{semilogyaxis}
		\end{tikzpicture}
		\caption{Semi-log plot of the relative error between the analytical
		instantaneous frequency $\phi'_1$
		(see (\ref{eq:example_diffop-1-freq-analytic})) and the solution
		$\tilde{\phi}'_1$, obtained using the
		IMF differential operator for $(k, q, n) = (4, 4, 180)$, from
		 Example~\ref{ex:diffop-1}.}
		\label{fig:example_diffop-1-relerr}
	\end{figure}
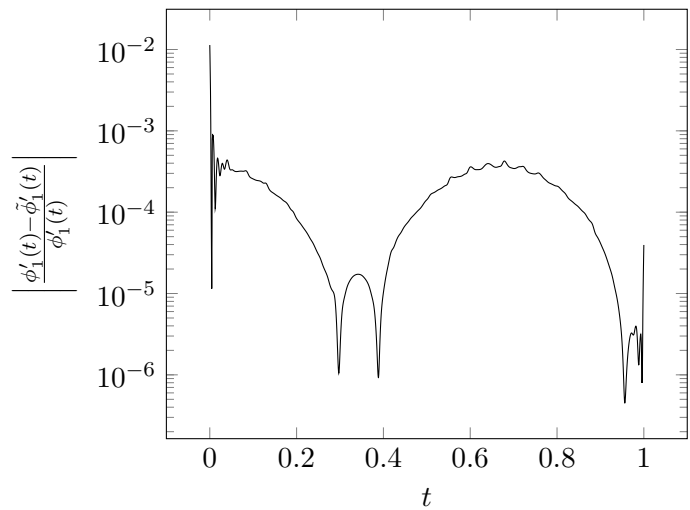
\end{example}
\begin{example}[Sigmoid up-chirp]\label{ex:diffop-2}
	An aspect of interest is an IMF with rapidly increasing frequency in a short
	timeframe. We want to know how well our differential operator handles such a case.
	\par
	A signal whose frequency changes over time is called a \enquote{chirp}, and one
	with increasing frequency over time an \enquote{up-chirp}.
	Even with a rapid increase, as with all natural phenomena, we can reasonably expect
	our frequency to still be smooth. This is best illustrated if we compare the
	frequency with velocity. We can not have sudden changes in velocity of an object
	either, as it would imply infinte acceleration in that moment.
	To take the idea further, we can not have sudden changes in the acceleration either,
	as it would imply infinite jerk (rate of change of acceleration) in that moment,
	et cetera.
	\par
	A good modelling function for this is a sigmoid
	function, more precisely the logistic function, which we will make use of in this
	example.
	Consider the simple IMF
	\begin{equation}\label{eq:example_diffop-2-u}
		u_2(t) := \cos(\phi_2(t)) := \cos\left( 40 \cdot t + \frac{100}{90} \cdot
		\ln\!\left( 1 + \exp\!\left(90 \cdot (t - 0.5)\right)\right)\right)
	\end{equation}
	on the interval $[0,1]$ (see Figure~\ref{fig:example_diffop-2-u}). We calculate
	the instantaneous frequency $\phi'_2$, which happens to be a transformation of the
	logistic function, analytically as
	\begin{equation}\label{eq:example_diffop-2-freq-analytic}
		\phi'_2(t) = 40 + \frac{100}{1 + \exp\!\left( -90 \cdot (t - 0.5) \right)}.
	\end{equation}
	You can find it pictured in Figure~\ref{fig:example_diffop-2-freq-analytic}. It
	represents a sudden frequency increase from $40$ to $140$ in a very short
	timeframe around the middle of the interval $[0,1]$.
	\par
	As in Example~\ref{ex:diffop-1}, given the instantaneous amplitude is constantly $1$, we
	can use the toolbox to fit the simple IMF differential operator to $u_2$ to calculate the
	numerical instantaneous frequency $\tilde{\phi}'_2$.
	\par	
	The difference between $\phi'_2$ and $\tilde{\phi}'_2$
	is too small to be visible in a normal plot and thus we examine the semi-log plot of the
	relative error (see Figure~\ref{fig:example_diffop-2-relerr}). We can see that the
	relative error is at most roughly $0.1\%$ and ranges between around
	$2$ orders of magnitude below that.
	\begin{figure}[htbp]
		\centering
		\begin{tikzpicture}
			\begin{axis}[xlabel=$t$, ylabel=$u_2(t)$]
			\addplot [mark=none, smooth] table [x=x, y=y, col sep=comma] {examples/regop.data/2-u.csv};
			\end{axis}
		\end{tikzpicture}
		\caption{Plot of the IMF $u_2$ (see (\ref{eq:example_diffop-2-u})) from
		Example~\ref{ex:diffop-2}.}
		\label{fig:example_diffop-2-u}
	\end{figure}
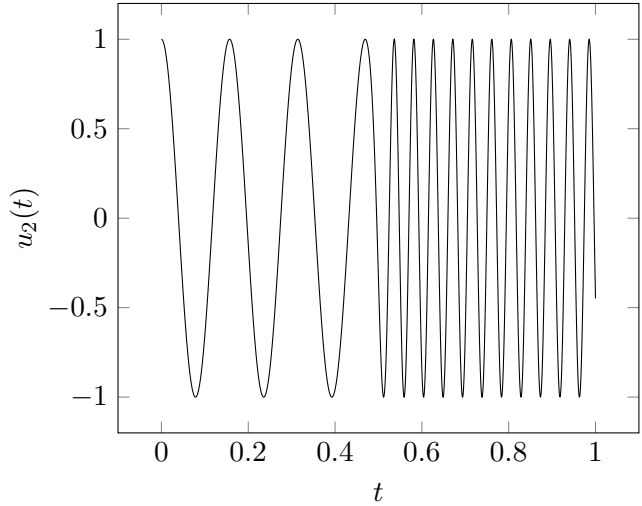
	\begin{figure}[htbp]
		\centering
		\begin{tikzpicture}
			\begin{axis}[xlabel=$t$, ylabel=$\phi'_2(t)$]
			\addplot [mark=none, smooth] table [x=x, y=y, col sep=comma] {examples/regop.data/2-freq-analytic.csv};
			\end{axis}
		\end{tikzpicture}
		\caption{Plot of the analytical instantaneous frequency
		$\phi'_2$ (see (\ref{eq:example_diffop-2-freq-analytic}))
		from Example~\ref{ex:diffop-2}.}
		\label{fig:example_diffop-2-freq-analytic}
	\end{figure}
	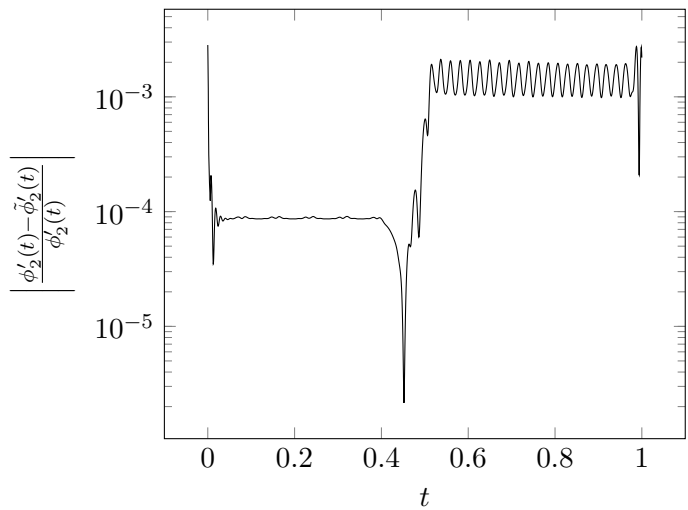
\begin{figure}[htbp]
		\centering
		\begin{tikzpicture}
			\begin{semilogyaxis}[xlabel=$t$, ylabel near ticks,
				ylabel=$\left| \frac{\phi'_2(t) -
				\tilde{\phi}'_2(t)}{\phi'_2(t)} \right|$]
			\addplot [mark=none, smooth] table [x=x, y=y, col sep=comma] {examples/regop.data/2-relerr.csv};
			\end{semilogyaxis}
		\end{tikzpicture}
		\caption{Semi-log plot of the relative error between the analytical
		instantaneous frequency $\phi'_2$
		(see (\ref{eq:example_diffop-2-freq-analytic})) and the solution
		$\tilde{\phi}'_2$, obtained using the
		IMF differential operator for $(k, q, n) = (4, 4, 180)$, from
		 Example~\ref{ex:diffop-2}.}
		\label{fig:example_diffop-2-relerr}
	\end{figure}
\end{example}
\FloatBarrier
\section{Discussion}
In this chapter we have examined the IMF differential operator
(see Definition~\ref{def:imf-diffop}) as a possible means to extract
instantaneous amplitude and frequency from an IMF and to
regularize the EMD optimization problem introduced in 
Definition~\ref{def:emdop}. What we noticed in Theorem~\ref{thm:imf-operator-convex} is that the IMF differential
operator is not convex in the parameters corresponding to amplitude
and frequency, which is why we modified it in
Definition~\ref{def:modified-imf-diffop} to work only on IMFs with
constant amplitude $1$ and only extract the frequency, which we proved
in Corollary~\ref{cor:modified-imf-diffop-convex}. Consequently,
we showed in the examples in Section~\ref{sec:examples-diffop} that
the operator, as expected, successfully extracts the frequency from given
IMFs with constant amplitude $1$.
\par
The limitatin to IMFs with constant amplitude $1$ appears to be a drastic
limitation, but we show in Chapter~\ref{ch:hobm} how to work around it and
extract the amplitude already during the EMD sifting process.
However, when considering the EMD optimization problem from Definition~\ref{def:emdop} the general differential operator from Definition~\ref{def:imf-diffop} is unsuitable as a general regularization
term $R(a,\phi)$ from a theoretical perspective.
Given we have shown that it is not convex, one can even
consider it to be more of a heuristic tool.
\chapter{Hybrid Operator-Based Methods}\label{ch:hobm}
This chapter is the culmination of the results obtained in the previous
chapters. In Chapter~\ref{ch:emd_analysis} we formulated and analyzed
the optimization problem (see Equation~(\ref{eq:emdop-reformed}))
\begin{equation*}
	\begin{aligned}
		\min_{u} \quad &
			{\| s - \I[a,\phi] \|}_2^2 \\
		\text{s.t.} \quad & (a,\phi) \in {\mathcal{S}}_{\mu_0,\mu_1,\mu_2}
	\end{aligned}
\end{equation*}
In particular we have shown that this optimization problem can also be
expressed as a regularized optimization problem
(see Section~\ref{sec:regularity})
\begin{align*}
	\min_{(a,\phi) \in {\left( \mathcal{C}^0(\R,\R) \right)}^2} \quad &
		c[s](a,\phi) + R(a,\phi).
\end{align*}
$R(a,\phi)$ is a regularization term that punishes solution candidates
of the optimization problem that are not an IMF soul
(see Definition~\ref{def:imfs}).
\par
In Chapter~\ref{ch:oss} we introduced one possible way to define
this regularization term. We made use of a differential operator
$\mathcal{D}_{(\tilde{a},\tilde{\phi})}$ with parameters
$\tilde{a}$ and $\tilde{\phi}$ that annihilates IMF functions
$u(t) := a(t) \cdot \cos(\phi(t))$ when $\tilde{a} = a$ and
$\tilde{\phi} = \phi$. Unfortunately, this operator does not
yield unique results for a given input IMF
(see Theorem~\ref{thm:imf-operator-convex}) and thus is
highly reliant on heuristics to work.
\par
However, when reduced to input IMFs with constant amplitude $1$,
the differential operator is convex and the resulting frequency for
an IMF is unique (see Corollary~\ref{cor:modified-imf-diffop-convex}). This result does not look very useful,
but can be leveraged when combined with the classic EMD method
proposed by \cite{hsl+98}.
The result of this combination is a hybrid of classic and
modern methods and will be introduced later in this chapter. Before
considering this approach, we first introduce the classic EMD method.
\section{Classic EMD method}\label{sec:classic_emd}
The classic EMD method was first proposed in \cite{hsl+98}
and will be described as follows. The EMD is a multistep
method, but we will without loss of generality
only consider a single extraction step.
In this step we separate a given multicomponent signal $s(t)$
into an IMF $u(t)$ and a residual $r(t)$.
This is without loss of generality, as subsequent extraction
steps are realized by considering the
residual of the previous step as the input signal for the current
step. Continuing this process, we sooner or later obtain a residual
that does not contain any more IMFs. The stopping criterion
might for instance be when the residual has no or at
most one local extremum, but this is not within the scope of this thesis.
\par
After the separation of $s(t)$ into $u(t)$ and $r(t)$,
one can determine the instantaneous amplitude $a(t)$
and phase $\phi(t)$ of $u(t) := a(t) \cdot \cos(\phi(t))$ by
complexification of $u(t)$ using the \textsc{Hilbert}
transform, which will not be further elaborated here.
We note here though that this \textsc{Hilbert} transform
provides some numerical challenges. In particular,
it requires heuristics to work properly in the numerical
context, which is why alternatives to this approach are
desired and presented in this thesis.
\par
The process of
separation is called \enquote{sifting} in the original
paper \cite{hsl+98} and commonly referred to as the
empirical mode decomposition (EMD). The extraction of $a(t)$
and $\phi(t)$ from $u(t)$ is called the
\enquote{\textsc{Hilbert} spectral-analysis} (HSA).
The complete process of EMD and HSA is referred
to as the \enquote{\textsc{Hilbert}-\textsc{Huang}-transform}
(HHT). Of note here is though that because the
HHT describes a very specific approach using the
\textsc{Hilbert} transform, one finds that the term
\enquote{EMD} is often used to also include the
spectral analysis part that makes use of some
other method.
\par
The sifting method of separating the signal $s(t)$ into
an IMF $u(t)$ and residual $r(t)$ can be separated into
three steps, illustrated in Figure~\ref{fig:sifting-2}
and given in Algorithm~\ref{alg:sifting}.
\begin{figure}[htbp]
	\centering
	\begin{subfigure}[c]{0.32\textwidth}
		\centering
		\begin{tikzpicture}
			\begin{axis}[
				name=plot1,
				height=4.5cm,width=4.5cm,
				ymin=-5,
			]
				\addplot [domain=0:1, samples=200]
					{2 + 3 * (1 + x)^3 +
					0.8 * (2 + cos(deg(4 * pi * x))) * sin(deg(30 * pi * x))};
				\addplot [domain=0:1, samples=200, densely dashed]
					{2 + 3 * (1 + x)^3 +
					0.8 * (2 + cos(deg(4 * pi * x))) * 1};
				\addplot [domain=0:1, samples=200, densely dashed]
					{2 + 3 * (1 + x)^3 +
					0.8 * (2 + cos(deg(4 * pi * x))) * (-1)};
			\end{axis}
		\end{tikzpicture}
		\subcaption{envelope estimation}
	\end{subfigure}
	\begin{subfigure}[c]{0.32\textwidth}
		\centering
		\begin{tikzpicture}
			\begin{axis}[
				name=plot1,
				height=4.5cm,width=4.5cm,
				ymin=-5,
			]
				\addplot [domain=0:1, samples=200, densely dashed]
					{2 + 3 * (1 + x)^3};
				\addplot [domain=0:1, samples=200]
					{2 + 3 * (1 + x)^3 +
					0.8 * (2 + cos(deg(4 * pi * x))) * sin(deg(30 * pi * x))};
			\end{axis}
		\end{tikzpicture}
		\subcaption{envelope averaging}
	\end{subfigure}
	\begin{subfigure}[c]{0.32\textwidth}
		\centering
		\begin{tikzpicture}
			\begin{axis}[
				name=plot1,
				height=4.5cm,width=4.5cm,
				ymin=-5,
			]
				\addplot [domain=0:1, samples=200]
					{2 + 3 * (1 + x)^3};
				\addplot [domain=0:1, samples=200]
					{0.8 * (2 + cos(deg(4 * pi * x))) * sin(deg(30 * pi * x))};
				\addplot [domain=0:1, samples=200, densely dashed]
					{0.8 * (2 + cos(deg(4 * pi * x))) * 1};
			\end{axis}
		\end{tikzpicture}
		\subcaption{separation}
	\end{subfigure}
	\caption{%
		Visualization of the EMD sifting process of a $1$-component
		signal.
	}
	\label{fig:sifting-2}
\end{figure}
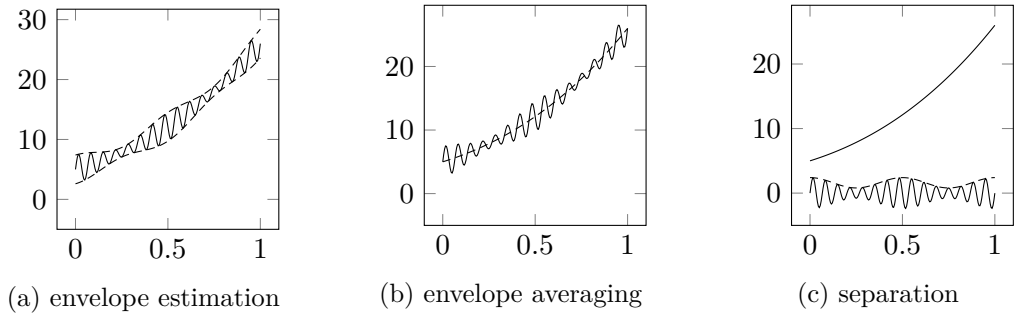
\begin{algorithm}[htbp]
	\begin{algorithm2e}[H]
		\SetKwInOut{Input}{input}
		\SetKwInOut{Output}{output}
		\Input{multicomponent signal $s \in \mathcal{C}^2(\R,\R)$
		}
		\Output{intrinsic mode function $u \in \mathcal{C}^2(\R,\R)$ \\
			instantaneous amplitude $a \in \mathcal{C}^2(\R,\R)$ \\
			residual $r \in \mathcal{C}^2(\R,\R)$		
		}
		\BlankLine
		$\ul{a} \longleftarrow \text{LowerEnvelope(s)}$\;
		$\ol{a} \longleftarrow \text{UpperEnvelope}(s)$\;
		\BlankLine
		$r \longleftarrow \frac{1}{2} \cdot (\ul{a} + \ol{a})$\;
		$u \longleftarrow s - r$\;
		$a \longleftarrow \ol{a} - r$\;
	\end{algorithm2e}
	\caption{EMD sifting algorithm.}
	\label{alg:sifting}
\end{algorithm}
The first step is to estimate the lower and upper envelopes $\ul{a}(t)$ 
and $\ul{a}(t)$ of the input signal. What an envelope is exactly will
be defined later. The second is to take the average of $\ul{a}(t)$ 
and $\ul{a}(t)$, yielding the residual $r(t)$, and the third is to
separate the signal into residual and IMF $u(t)$ by subtracting $r(t)$
from $s(t)$. The instantaneous amplitude $a(t)$ of $u(t) :=
a(t) \cdot \cos(\phi(t))$ follows naturally by subtracting $r$ from
the upper envelope $\ol{a}$.
\par
From this observation we can conclude two things: The first is that
the envelope estimation is central to the EMD method. The second
is that given we obtain the IMF $u(t)$ and its instantaneous amplitude
$a(t)$ naturally from the sifting process, we can make use of our
differential operator to extract the instantaneous phase $\phi(t)$.
This is because the IMF $\tilde{u}(t) := u(t) / a(t)$ has amplitude
$1$ and makes it possible to use the differential operator introduced
in Chapter~\ref{ch:oss} in a theoretically meaningful way.
What is left to do is to analyze the envelope estimation method itself,
which we will do as follows.
\section{Envelope Estimation}\label{sec:envelope_estimation}
An envelope is not uniquely classified, but defined as a
function that encloses a function either from above (\enquote{upper
envelope}) or below (\enquote{lower envelope}).
\begin{definition}[Lower/upper envelope]\label{def:envelope}
	Let $v,f \in \mathcal{C}^0(\R,\R)$. $v$ is a \emph{lower envelope}
	of $f$ if and only if
	\begin{equation*}
		v \preceq f.
	\end{equation*}
	$v$ is an \emph{upper envelope} of $f$ if and only if
	\begin{equation*}
		v \succeq f.
	\end{equation*}
\end{definition}
As an example, an upper envelope for $\cos(t)$ is the constant function
$1$ and a lower envelope is the constant function $-1$, but we can also
choose $2$ and $-2$ or $\cos(t)$ for both as lower and upper amplitudes
(see Figure~\ref{fig:sin-upper_envelopes}).
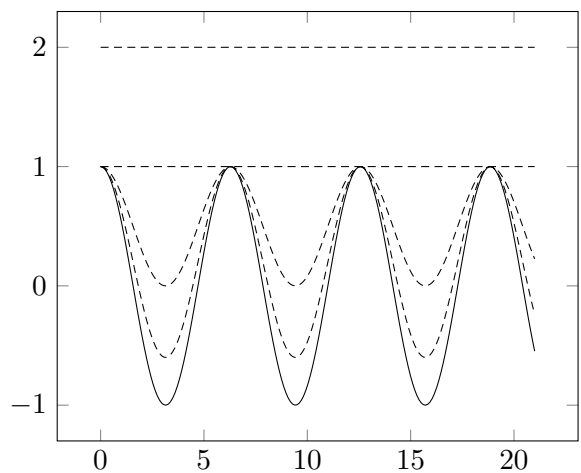
\begin{figure}[htbp]
	\centering
	\begin{tikzpicture}
		\begin{axis}[]
		\addplot [domain=0:21, samples=200]{cos(deg(x))};
		\addplot [densely dashed, domain=0:21, samples=200]{0.2 + 0.8 * cos(deg(x)))};
		\addplot [densely dashed, domain=0:21, samples=200]{0.5 + 0.5 * cos(deg(x)))};
		\addplot [densely dashed, domain=0:21, samples=20]{1};
		\addplot [densely dashed, domain=0:21, samples=20]{2};
		\end{axis}
	\end{tikzpicture}
	\caption{Examples (dashed) for upper envelopes of $\cos(t)$ (solid).}
	\label{fig:sin-upper_envelopes}
\end{figure}
We can thus note that by far there is no unique choice for a lower
and upper envelope of a function and we will have to specify more
requirements the envelopes have to fulfill.
In the context of the empirical mode decomposition, determining the
lower and upper envelopes of an input signal is the central step to
obtain the residual and IMF, as explained in Section~\ref{sec:classic_emd}.
\par
As follows, we will, without loss of generality, only consider the
upper envelope estimation. The procedure for the lower envelope follows
respectively, given the following
\begin{proposition}\label{prop:lower_envelope-upper_envelope}
	Let $f \in \mathcal{C}^0(\R,\R)$ and $v \in \mathcal{C}^0(\R,\R)$
	be a lower envelope of $s$.
	It holds that $-v$ is an upper envelope of $-f$.
\end{proposition}
\begin{proof}
	It holds by Definition~\ref{def:envelope} that $v \preceq f$ and
	\begin{equation*}
		v \preceq f
		\Leftrightarrow -v \succeq -f. \qedhere
	\end{equation*}
\end{proof}
Thus, to determine the lower envelope we simply determine the negated
upper envelope of the negated input function.
\subsection{Classic Envelope Estimation}
Knowing the requirements for an upper envelope in the context of the
empirical mode decomposition listed previously, we now take a look at the
classic envelope estimation proposed in \cite[Section~5]{hsl+98}.
When we reconsider the previous example $\cos(t)$ (which has upper envelope $1$)
we see that the function assumes the value
$1$ in its local maxima. Consequently, we can propose that an IMF assumes
the value of its upper envelope in its local maxima and we obtain the upper
envelope by interpolating them. The corresponding algorithm in pseudocode
can be found in Algorithm~\ref{alg:envelope_estimation-classic}.
\begin{algorithm}[htbp]
	\begin{algorithm2e}[H]
		\SetKwInOut{Input}{input}
		\SetKwInOut{Output}{output}
		\Input{multicomponent signal $s \in \mathcal{C}^2(\R,\R)$}
		\Output{upper envelope $m \in \mathcal{C}^2(\R,\R)$}
		\BlankLine
		$P \longleftarrow \{ (t, s(t)) \in \R\times\R \mid s'(t) = 0 \land s''(t) < 0 \}$\;
		$m \longleftarrow \text{Interpolate}(P)$\;
	\end{algorithm2e}
	\caption{Classic upper envelope estimation algorithm.}
	\label{alg:envelope_estimation-classic}
\end{algorithm}
\par
The Interpolate-method in the Algorithm is left out by choice and
means the fitting of a B-spline-curve to each point in the set $p$.
\par
The problem is that with varying amplitude the estimated envelope tends to
dip below the signal, thus violating the definition of an envelope not to
cut the signal at any moment. This is illustrated
in Figure~\ref{fig:classic_sifting_problem}.
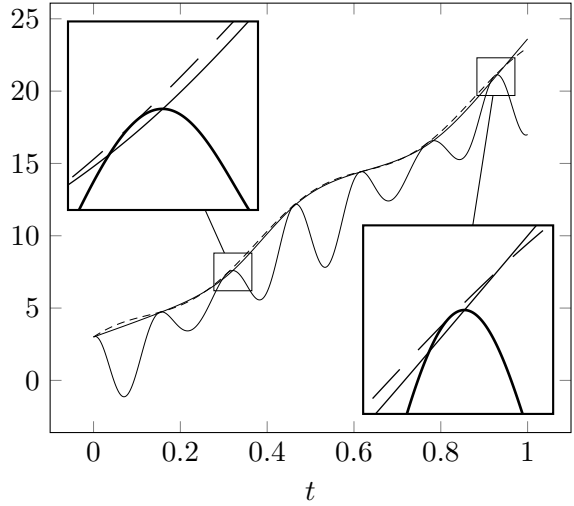
\begin{figure}[htbp]
	\centering
	\begin{tikzpicture}[spy using outlines={rectangle,lens={scale=5},
	size=2.5cm, connect spies}]
		\begin{axis}[xlabel=$t$]
		\addplot [mark=none, smooth, very thin] table [x=x, y=y, col sep=comma] {examples/envelope.data/x-s.csv};
		\addplot [mark=none, smooth, ultra thin] table [x=x, y=y, col sep=comma] {examples/envelope.data/x-a-classic.csv};
		\addplot [mark=none, smooth, ultra thin, densely dashed] table [x=x, y=y, col sep=comma] {examples/envelope.data/x-a-analytic.csv};
		\coordinate (spypoint_0) at (axis cs:0.321,7.5);
		\coordinate (glass_0) at (axis cs:0.16,18.3);
		\coordinate (spypoint_1) at (axis cs:0.927,21);
		\coordinate (glass_1) at (axis cs:0.84,4.2);	
		\end{axis}
		\spy on (spypoint_0) in node[fill=white] at (glass_0);
		\spy on (spypoint_1) in node[fill=white] at (glass_1);
	\end{tikzpicture}
	\caption{Multicomponent signal $s_0$ (thick) from Example~\ref{ex:envelope-0}
	with analytical upper envelope (dashed) and estimated upper envelope using
	the classic sifting algorithm (thin). Two sections where the latter cuts
	the signal are enlarged.}
	\label{fig:classic_sifting_problem}
\end{figure}
As we can see, the classic method of interpolating the local maxima
reaches its limits very quickly and is in general not a very good
envelope estimation method, given it violates the definition.
\subsection{Iterative Slope Envelope Estimation}
There have been multiple approaches to the problem
with the intersection of envelope and signal that we described earlier.
\cite[Subsection~2.3]{hk13} introduced an optimization scheme to obtain
the envelope, strictly enforcing the nature of the envelope definition,
but at the cost of the smoothness of the resulting amplitude estimation.
\cite{hph12} approached the problem by analytically moving the interpolation
points from the local maxima to more fitting spots, with the disadvantage
that these approaches only work where it is at least possible to estimate
the current frequency. Additionally, it only allows to work with IMFs and
not a multicomponent signal, which we are relying on in the sifting process,
because it is not possible to easily find analytical results taking the
entire multicomponent signal into account.
\par
If we take a step back and think how a human would draw an uppper envelope of
a signal $s(t)$ by hand, we see that the result $m_0(t)$ of the classic
sifting can be considered as a first step toward a better envelope
estimation which just needs some refining.
We do that by taking $m_0(t)$ and finding every point on the signal $s(t)$
where $m_0'(t) = s'(t)$ (matching slope) and $s''(t) < 0$ (negative curvature)
hold. We obtain the upper envelope $m_1(t)$ by interpolating these points.
Repeating this process yields a curve with a better fitting, as it
becomes by definition a tangential curve.
\par
The algorithm describing this process can be found as pseudocode in
Algorithm~\ref{alg:envelope_estimation-iterative_slope}. We begin
with a multicomponent signal $s(t)$ and a tolerance.
Our estimated upper envelope $m$ is first initialized to
the zero-function before entering the main loop, in which $m$
is copied to $\tilde{m}$ and $m$ set to the
next envelope estimate iterate. If the difference between
the previous and current envelope estimate iterate
is strictly smaller than our tolerance $\varepsilon$
in the supremum norm, we are done.
\par
We make use of the supremum norm given it is easy
to calculate a close upper bound of it within the
well-conditioned B-spline basis (see
\ref{prop:well_conditioned_basis}), which amounts
to just the supremum norm of the respective vector
of  B-spline basis coefficients.
\begin{algorithm}[htbp]
	\begin{algorithm2e}[H]
		\SetKwInOut{Input}{input}
		\SetKwInOut{Output}{output}
		\Input{multicomponent signal $s \in \mathcal{C}^2(\R,\R)$\\
			tolerance $\varepsilon > 0$
		}
		\Output{upper envelope $m \in \mathcal{C}^2(\R,\R)$}
		\BlankLine
		$m \longleftarrow 0 \in \mathcal{C}^0(\R,\R)$\;
		\Repeat{${\| m - \tilde{m} \|}_\infty < \varepsilon$}{
			$\tilde{m} \longleftarrow m$\;
			$p \longleftarrow \{ t \in \R \mid s'(t) = \tilde{m}'(t) \land s''(t) < 0 \}$\;
			$m \longleftarrow \text{Interpolate}(p)$\;
		}
	\end{algorithm2e}
	\caption{Iterative slope upper envelope estimation algorithm.}
	\label{alg:envelope_estimation-iterative_slope}
\end{algorithm}
\begin{remark}[Generalization of the classic envelope estimation method]\label{rem:iterative_slope-generalization}
	Let us compare Algorithms \ref{alg:envelope_estimation-classic}
	and \ref{alg:envelope_estimation-iterative_slope}.
	We remind ourselves that to determine the upper envelope,
	the classic method interpolates the local maxima.
	The slope in the maxima is $0$ and the curvature is
	negative. It is easy to see that the first iteration of the
	iterative slope envelope algorithm is simply the classic
	envelope estimation, because the slope of the $0$-function
	is also zero. Thus, all slope matches in the signal are those
	where the slope is zero.
	\par
	Forcing the algorithm to finish after the first iteration by
	setting $\varepsilon = \infty$ we obtain the classic
	method. We can thus say that the proposed upper envelope
	estimation algorithm is a generalization of the classic
	algorithm.
\end{remark}
Obtaining the lower envelope of a given input signal $s$ is analogous
to Proposition~\ref{prop:lower_envelope-upper_envelope} by
determining the negative upper envelope of the negated input
signal $-s$. Given these negations are linear time operations
there is no effect on the run-time of the algorithm regardless
of whether we estimate the upper or lower envelope.
\par
An advantage of this algorithm over the method presented in \cite{hph12}
is that we do not need to estimate the instantaneous frequency and do
not require the input signal to have any special form. Given our new
method is a generalization of the classic envelope estimation,
it fits more naturally into the existing methods. Moreover, we
solve the intersection problem as described in
Figure~\ref{fig:classic_sifting_problem} and obtain meaningful envelopes
that satisfy the definition.
\subsection{Examples}\label{subsec:envelope_estimation_examples}
The following examples were implemented using the ETHOS-toolbox developed
in the course of this thesis and can be found in Listing~\ref{lst:examples-envelopec}.
The parameters $(k,q,n,\varepsilon)$ given in the figure captions refer to the
spline order $k$, in-fill-count $q$ (see Section~\ref{sec:ethos_toolbox}),
number of B-spline basis functions $n$ (see Definition~\ref{def:b-spline}) and envelope extraction
tolerance $\varepsilon$ (see Algorithm~\ref{alg:envelope_estimation-iterative_slope}).
See Subsection~\ref{subsec:boundary_effects} for a discussion on the boundary
effects of these examples in the context of information theory and other literature.
\begin{example}[Ladder]\label{ex:envelope-0}
	Consider the composite signal
	\begin{equation}\label{eq:example_envelope-0-s}
		s_0(t) := 40 \cdot t + (20 + 10 \cdot \cos(5 \cdot \pi \cdot t))
		\cdot \cos(25 \cdot \pi \cdot t)
	\end{equation}
	on the interval [0,1]. Beginning with the highest frequency
	component, the first analytical envelope to be extracted by the sifting process
	is
	\begin{equation}\label{eq:example_envelope-0-m}
		m_0(t) := 40 \cdot t + (20 + 10 \cdot \cos(5 \cdot \pi \cdot t)).
	\end{equation}
	In Figure~\ref{fig:example_envelope-0-s-m} you can see the result of
	the proposed iterative slope sifting process compared with the
	analytical envelope $m_0$.
	\par
	Due to the little differences in most
	parts we examine the semi-log plot of the relative error
	(see Figure~\ref{fig:example_envelope-0-relerr}) for both the classic
	and iterative slope sifting processes. We can see that the
	relative error for the proposed iterative slope method is up to
	an order of magnitude less in
	some parts while staying equally good e.g.\ in the boundary regions,
	which is more due to an information theoretical reason and not a
	quality criterion of the sifting algorithm.
	\par
	What is more important is that the iterative slope envelope is a true
	envelope in that it does not cut the signal in any location like
	the envelope obtained with the classic sifting algorithm. This is
	due to the fact that the iterative slope envelope is a tangent
	by construction.
	\begin{figure}[htbp]
		\centering
		\begin{tikzpicture}
			\begin{axis}[xlabel=$t$]
			\addplot [mark=none, smooth] table [x=x, y=y, col sep=comma] {examples/envelope.data/0-s.csv};
			\addplot [mark=none, smooth, thin] table [x=x, y=y, col sep=comma] {examples/envelope.data/0-a.csv};
			\addplot [mark=none, smooth, dotted] table [x=x, y=y, col sep=comma] {examples/envelope.data/0-a-analytic.csv};
			\end{axis}
		\end{tikzpicture}
		\caption{Plot of the multicomponent signal $s_0$ (see (\ref{eq:example_envelope-0-s})), its analytical highest frequency
		component envelope $m_0$ (dotted) (see (\ref{eq:example_envelope-0-m}))
		and calculated iterative slope envelope (thin) for
		$(k, q, n, \varepsilon) = (4, 4, 180, 0.01)$ from
		Example~\ref{ex:envelope-0}.}
		\label{fig:example_envelope-0-s-m}
	\end{figure}
	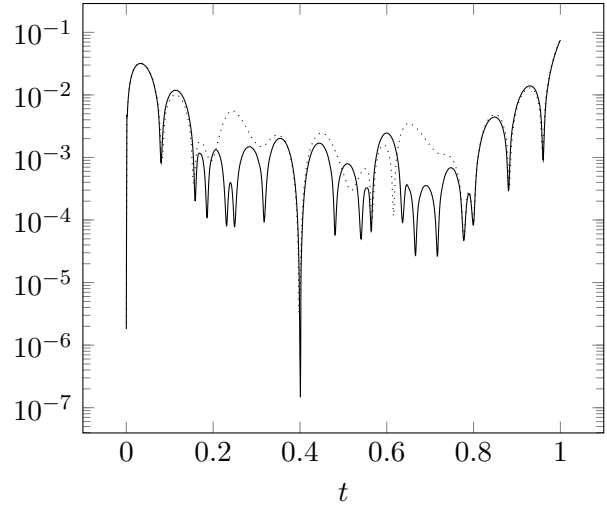
\begin{figure}[htbp]
		\centering
		\begin{tikzpicture}
			\begin{semilogyaxis}[xlabel=$t$]
			\addplot [mark=none, smooth] table [x=x, y=y, col sep=comma] {examples/envelope.data/0-relerr.csv};
			\addplot [mark=none, smooth, dotted] table [x=x, y=y, col sep=comma] {examples/envelope.data/0-relerr-classic.csv};
			\end{semilogyaxis}
		\end{tikzpicture}
		\caption{Semi-log plot of the relative errors between the analytical
		highest frequency component envelope $m_0$
		(see (\ref{eq:example_envelope-0-m})) and both the envelopes
		obtained using the classic (dotted) and iterative slope
		sifting algorithms for $(k,q,n,\varepsilon) = (4,4,180,0.01)$
		from Example~\ref{ex:envelope-0}.}
		\label{fig:example_envelope-0-relerr}
	\end{figure}
\end{example}
\begin{example}[{\cite[Figure~2]{hph12}}]\label{ex:envelope-1}
	Consider the IMF
	\begin{equation}\label{eq:example_envelope-1-s}
		s_1(t) := \frac{1}{16} \cdot (t^2 + 2) \cdot
		\cos(\pi \cdot \sin(8 \cdot t) + \pi)
	\end{equation}
	on the interval [-4,4]. The analytical envelope to be
	extracted by the sifting process is
	\begin{equation}\label{eq:example_envelope-1-m}
		m_1(t) := t^2 + 2.
	\end{equation}
	In Figure~\ref{fig:example_envelope-1-s-m} you can see the result of
	the proposed iterative slope sifting process compared with the
	analytical envelope $m_1$.
	\par
	Due to the little differences in most
	parts we examine the semi-log plot of the relative error
	(see Figure~\ref{fig:example_envelope-1-relerr}) for both the classic
	and iterative slope sifting processes. We can see that the
	relative error for the proposed iterative slope method is equal
	to that of the classic sifting method and even up to an order of magnitude
	lower in the increasing branch of $s_1$.
	\par	
	The reason the error is
	not symmetric like that of the classic sifting method is because
	even though the maxima are symmetrically distributed, $s_1$ itself
	is not symmetric. The classic sifting method only considers the maxima
	though and thus is oblivious to the shape of $s_1$ itself, unlike
	the iterative slope method.
	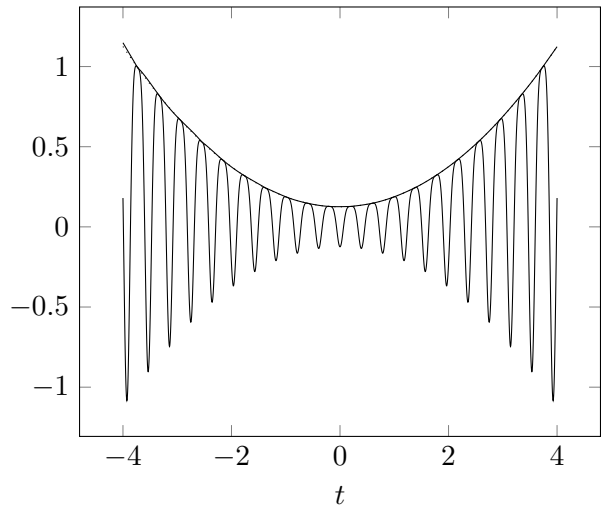
\begin{figure}[htbp]
		\centering
		\begin{tikzpicture}
			\begin{axis}[xlabel=$t$]
			\addplot [mark=none, smooth] table [x=x, y=y, col sep=comma] {examples/envelope.data/1-s.csv};
			\addplot [mark=none, smooth, thin] table [x=x, y=y, col sep=comma] {examples/envelope.data/1-a.csv};
			\addplot [mark=none, smooth, dotted] table [x=x, y=y, col sep=comma] {examples/envelope.data/1-a-analytic.csv};
			\end{axis}
		\end{tikzpicture}
		\caption{Plot of the multicomponent signal $s_1$ (see (\ref{eq:example_envelope-1-s})), its analytical highest frequency
		component envelope $m_1$ (dotted) (see (\ref{eq:example_envelope-1-m}))
		and calculated iterative slope envelope (thin) for
		$(k, q, n, \varepsilon) = (4, 4, 180, 0.1)$ from
		Example~\ref{ex:envelope-1}.}
		\label{fig:example_envelope-1-s-m}
	\end{figure}
	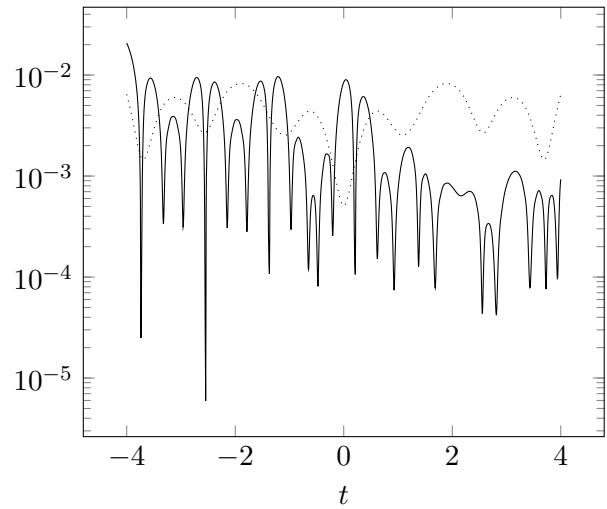
\begin{figure}[htbp]
		\centering
		\begin{tikzpicture}
			\begin{semilogyaxis}[xlabel=$t$]
			\addplot [mark=none, smooth] table [x=x, y=y, col sep=comma] {examples/envelope.data/1-relerr.csv};
			\addplot [mark=none, smooth, dotted] table [x=x, y=y, col sep=comma] {examples/envelope.data/1-relerr-classic.csv};
			\end{semilogyaxis}
		\end{tikzpicture}
		\caption{Semi-log plot of the relative errors between the analytical
		highest frequency component envelope $m_1$
		(see (\ref{eq:example_envelope-1-m})) and both the envelopes
		obtained using the classic (dotted) and iterative slope
		sifting algorithms for $(k,q,n,\varepsilon) = (4,4,180,0.1)$
		from Example~\ref{ex:envelope-1}.}
		\label{fig:example_envelope-1-relerr}
	\end{figure}
\end{example}
What we can clearly see is that the newly presented iterative slope sifting
algorithm provides a better envelope estimation than the classic sifting algorithm. Of note is especially the intuition behind it and the fact that
it is a generalization of the classic method. For this reason, we will
make use of it in our hybrid EMD algorithm presented in the next section.
\FloatBarrier
\section{Hybrid EMD Algorithm}\label{sec:emd_algo}
This section presents a new EMD algorithm making use of
the new iterative slope sifting algorithm and the differential-
operator-based method presented earlier. Given the former is
considered a more classic approach compared to the operator-based
signal-separation and the latter operator-based method is a modern
concept it is fitting to call this algorithm a \enquote{hybrid}
algorithm. The entire procedure for a single step of the decomposition
is given in Algorithm~\ref{alg:emd}, but we will construct the method
step by step in the following section for a given multicomponent signal $s$.
The complete decomposition is obtained by successive runs of the
algorithm with the residual subtracted from the signal as the input
for the next step.
\begin{algorithm}[htbp]
	\begin{algorithm2e}[H]
		\SetKwInOut{Input}{input}
		\SetKwInOut{Output}{output}
		\Input{multicomponent signal $s \in \mathcal{C}^2(\R,\R)$\\
			tolerance $\varepsilon > 0$
		}
		\Output{intrinsic mode function $u \in \mathcal{C}^2(\R,\R)$\\
			instantaneous amplitude $a \in \mathcal{C}^2(\R,\R)$\\
			instantaneous frequency $\phi' \in \mathcal{C}^2(\R,\R)$\\
			residual $r \in \mathcal{C}^2(\R,\R)$
		}
		\BlankLine
		$\ol{a} \leftarrow \text{UpperEnvelope}(s,\varepsilon)$\;
		$\ul{a} \leftarrow -\text{UpperEnvelope}(-s,\varepsilon)$\;
		$r \leftarrow \frac{1}{2} \cdot (\ol{a} + \ul{a})$\;
		$u \leftarrow s - r$\;
		$a \leftarrow \ol{a} - r$\;
		$\Omega^\star \leftarrow \arg\min_{\Omega \in \mathcal{C}^2(\R,\R)}
			\left(
				{\left\|\tilde{\mathcal{D}}_{(0,\Omega)}(u / a)\right\|}_2^2
			\right)$\;
		$\phi' \leftarrow \frac{1}{\sqrt{\Omega^\star}}$\;
	\end{algorithm2e}
	\caption{Hybrid operator-based empirical mode decomposition algorithm.}
	\label{alg:emd}
\end{algorithm}
\par
The first step is to determine the upper and lower envelopes $\ol{a}$
and $\ul{a}$ of $s$.
We make use of the fact that the lower envelope is just the negation
of the upper envelope of the negated signal $s$. This is why we
previously only considered the upper envelope estimation, as the lower
envelope estimation follows as a corollary.
\par
The idea behind the following steps to obtain $r$, $u$ and $a$ were first
introduced in \cite{hsl+98}. Once we've determined $\ol{a}$ and $\ul{a}$
we can calculate the residual $r$ as their mean.
The intrinsic mode function $u$ is obtained by subtracting $r$ from $s$
and the instantaneous amplitude $a$ is calculated by subtracting $r$ from
$\ol{a}$.
\par
Next we use the simple case (see Subsection~\ref{subsec:imf-operator-simplification})
of the modified IMF differential operator
(see Definition~\ref{def:modified-imf-diffop})
to obtain our instantaneous frequency. We first solve the NSP optimization
problem for the inverse square continuous frequency operator $\Omega[\phi]$
(see Definition~\ref{def:inv_sq_cont_freq_operator}) and then calculate
the instantaneous frequency directly by applying the inverse square
root.
\section{ETHOS Toolbox}\label{sec:ethos_toolbox}
The central numerical piece of this thesis is the ETHOS toolbox.
It stands
for \enquote{\ul{E}MD \ul{T}oolbox using \ul{H}ybrid
\ul{O}perator-Based Methods and B-\ul{s}plines} and has been
developed in the course of this thesis to provide an implementation
for the new concepts presented in this work, making it possible
to do an empirical mode decomposition on a discrete input signal.
The implementation language is C99 (see \cite{iso99}),
making use of the GNU Scientific Library (see \cite{gdt+18}) for
the numerical backend (including B-splines).
All examples in this and previous chapters have been realized in
this toolbox (see Section~\ref{sec:code-examples}) and the
entire source code is listed under Section~\ref{sec:code-ethos_toolbox}.
\par
The main header exposing the toolbox function is
\texttt{ethos.h} (see Listing~\ref{lst:ethos-ethosh}) with
the main datatype \texttt{struct ethos}. Nearly all
functions take a \texttt{struct ethos} as input and it is
the main storage for system parameters and precomputed
data.
\par
As follows, we will take a look at the most important
functions with regard to the decomposition process. For
all exposed functions refer to \texttt{ethos.h}
(see Listing~\ref{lst:ethos-ethosh}).
\subsection{Initialization and Precomputation}
The main initialization and first step of any program
using the ETHOS toolbox is done by \texttt{ethos\_init()}
(see Listing~\ref{lst:ethos-ethosc}). It takes a
pointer to an ETHOS-struct \texttt{e}, vector \texttt{T}
of length \texttt{N} with spline
order \texttt{k}, in-fill-count \texttt{q},
density \texttt{d} and grid-type \texttt{g} as input and
fills the given ETHOS-struct with the necessary parameters
and precomputes data for later use. The parameters are
explained as follows.
\par
The vector \texttt{T} contains all time-steps of the
discrete input signal, or comparable like a superset of
multiple possible interpolation areas, and it is our
interest to only take a certain subset of these steps for our spline
knot-vector. This is controlled by the parameters
\texttt{d}, controlling the density
of the spline knot vector relative to the input vector
and residing in the interval $(0,1)$, and \texttt{g},
controlling the way the selection is made
(uniformly or adaptively).
\par
When the spline knot-vector is obtained, it is uniformly
in-filled with \texttt{q} points between each spline knot.
The motivation for this process is to be able to pre-evaluate
the splines and their derivatives during the initialization step
on this \enquote{extended grid} (i.e.\ the in-filled knot-vector).
The reason why the knot-vector is not just made denser is
because we want to limit the number of basis functions and
for the sake of plotting or general evaluation do not need so
many basis functions, because B-splines as is provide a great
amount of smoothness.
\par
The evaluation happens only for the non-zero parts of each
basis function, resulting in linear memory complexity for
this precomputation step.
\par
On this extended grid the B-splines and their first and
second derivatives are evaluated and the results stored
in \texttt{dB} within the ETHOS struct, just like the
spline knot-vector in \texttt{grid} and in-filled form
in \texttt{extgrid}.
\par
The size of the B-spline basis \texttt{n} is of great
importance, as each function is stored internally as a
vector of this length, corresponding as coefficients
of the B-spline basis.
\subsection{Data Filtering}
After initialization, the next step is to filter the
discrete input data, with the goal of obtaining the
B-spline coefficients for this given function on the
initialized grid. The function for this purpose is
\texttt{ethos\_fit()} (see Listing~\ref{lst:ethos-ethosc})
and it takes an arbitrary discrete
input signal \texttt{S} with the time-steps \texttt{T},
which do not have to agree completely with the \texttt{T}
used in the initialization, but should agree on the
start- and endpoints.
\par
During the fitting process, the B-splines of the
initialized basis are evaluated on all points in
\texttt{T} and a weighted least-squares system solved with
a set of low-weight smoothness-terms of second order
besides the interpolation terms for each given datapoint.
It returns the B-spline coefficients in \texttt{s} best fitting the
given discrete input data as a vector, which is the
standard way of handling functions within the toolbox.
Even though not explicitly expressed, all of the
coefficient vectors have length \texttt{n} found
in the ETHOS struct.
\subsection{Boundary Effects and Extension}\label{subsec:boundary_effects}
The boundary effects we have seen in the previous
examples opens up some questions that will be
addressed here. As you can, for example, see
in Figure~\ref{fig:example_emd-1-r-2}
the error goes up as it reaches the boundary.
A common countermeasure often (silently) employed
in the literature is to extend the signal beyond
the boundary, either by mirroring or other methods
(see \cite{wr10} for further reading).
This way, the \enquote{shock} the algorithm is exposed to
is moved into the mirrored section or dampened, not
as heavily affecting the interior part one actually cares about.
\par
When looking at this matter in an information
theoretical way this technique of extension is
rather dishonest about the performance of such
an algorithm and generates information where there
is none. It might be forgivable for applications
that care about a good represenation, but the real
challenge is to design robust algorithms and make them
comparable among each other without silent tricks like
this one. Moreover, there is not a canonical way to
extend beyond the boundary and it presents itself
more as its own field of research. This is the reason
why the author chose not to use boundary extension
methods for his examples and keep them
honest with regard to the boundary effects.
\par
Despite these ethical concerns, the ETHOS toolbox includes
\texttt{ethos\_extend\_boundary()} (see Listing~\ref{lst:ethos-ethosc})
which takes an arbitrary discrete input signal \texttt{S} with the
time-steps \texttt{T} and length \texttt{N} and calculates an extended signal
\texttt{Se} with the time-steps \texttt{Te} and length \texttt{Ne}
by mirroring the signal into the extended area.
This extension is parametrized by \texttt{ratio} between $0$ and $1$,
extending the signal by this fraction both on the left and right side.
\subsection{Decomposition}\label{subsec:decomposition}
The function to do the signal decomposition itself is
\texttt{ethos\_emd()} (see Listing~\ref{lst:ethos-ethosc}).
It takes a pointer to an ETHOS-struct \texttt{e},
B-spline coefficient vectors \texttt{u} of the output IMF,
\texttt{a} of the amplitude of the output IMF,
\texttt{freq} of the frequency of the output IMF
and \texttt{s} of the input signal
of length \texttt{n} and tolerance \texttt{eps}.
\par
The procedure aligns with Algorithm~\ref{alg:emd},
filling the input signal vector \texttt{s} with
the residual and \texttt{u}, \texttt{a}
and \texttt{freq} with the extracted IMF and its
amplitude and frequency respectively. The tolerance
\texttt{eps} is the tolerance for the iterative
slope algorithm (see Algorithm~\ref{alg:envelope_estimation-iterative_slope}).
\par
Subsequent invocations of \texttt{ethos\_emd()}
yield the complete decomposition.
\subsection{Plotting}
The ETHOS toolbox provides two ways of plotting data, either
as a CSV-output or output meant as input for the \texttt{graph(1)}
command of the GNU plotting utilities. The input to those
plotting functions can either be a set of discrete points or
a spline function, represented with a coefficient vector.
The former is implemented as \texttt{ethos\_plot\_points()},
the latter as \texttt{ethos\_plot\_spline()} (see
Listing~\ref{lst:ethos-ethosc}).
\subsection{Envelope Estimation}
The procedure to estimate the envelope is already used
in \texttt{ethos\_emd()} described in Subsection~\ref{subsec:decomposition}
and not directly part of the decomposition path, however, it might
be of interest to test the envelope estimation itself separately.
\par
This estimation is achieved with \texttt{ethos\_upper\_envelope()}
(see Listing~\ref{lst:ethos-ethosc}). It takes a pointer to an
ETHOS-struct \texttt{e}, B-spline coefficient vectors
\texttt{m} of the upper envelope and \texttt{s} of the input signal and
tolerance \texttt{eps}.
The procedure aligns with Algorithm~\ref{alg:envelope_estimation-iterative_slope}
and stores the upper envelope estimation of \texttt{s} in \texttt{m}
with the tolerance \texttt{eps}.
\par
It can also be used to estimate the envelope using the classic
method by setting \texttt{eps} to \texttt{INFINITY}
(see Remark~\ref{rem:iterative_slope-generalization}), defined
in \cite[\texttt{math.h}]{cie97} to represent infinity. This allows an easy comparison
of both methods, as done in Subsection~\ref{subsec:envelope_estimation_examples}.
\subsection{IMF Characteristic}\label{subsec:ethos-imf-characteristic}
When we defined intrinsic mode functions in Definition~\ref{def:imfs}
we parametrized the model with three parameters $(\mu_0,\mu_1,\mu_2)$ with
$\mu_0,\mu_1,\mu_2 > 0$, called the characteristic. We also established
the connection with the IMF accuracy presented in \cite[Definition 3.1]{dlw11} 
and the role and calculation of each parameter in Remark~\ref{rem:imf-characteristic}. The former especially underlines the relevance of these parameters and makes it interesting to further explore them instead of just
treating them as a theoretical tool.
\par
For the purpose of determining the characteristic for a given IMF, the ETHOS toolbox offers the \texttt{ethos\_characteristic()}
(see Listing~\ref{lst:ethos-ethosc}) function. It takes a pointer to an
ETHOS-struct \texttt{e}, vectors \texttt{a} and \texttt{freq}
of length \texttt{n} and a vector \texttt{mu} of length $3$,
filling \texttt{mu} with the three characteristic values.
The vectors \texttt{a} and \texttt{freq} are the B-spline
coefficients for the amplitude and frequency functions relative
to the current spline environment given with \texttt{e}.
\par
Even though it is not possible to control the characteristic of each
extracted IMF during the decomposition it is nevertheless possible
to ascertain the quality of the extraction afterwards using this tool.
\section{Examples}\label{sec:emd_examples}
These examples were implemented using the ETHOS-toolbox and can be found
in Listing~\ref{lst:examples-emdc}.
The parameters $(k,q,n,\varepsilon)$ given in the figure captions refer to the
spline order $k$, in-fill-count $q$ (see Section~\ref{sec:ethos_toolbox}),
number of B-spline basis functions $n$ and envelope extraction
tolerance $\varepsilon$ (see Algorithm~\ref{alg:emd}).
See Subsection~\ref{subsec:boundary_effects} for a discussion on the boundary
effects of these examples in the context of information theory and other literature.
\begin{example}\label{ex:emd-0}
	This example was inspired by \cite[Example~1]{hs11}.
	Consider the multicomponent signal
	\begin{equation}\label{eq:example_emd-0-s}
		s_0(t) := u_{0,0}(t) + u_{0,1}(t) + 20 \cdot (t + 1)
	\end{equation}
	with the first IMF component (characteristic
	$\sim(6.60\cdot 10^1, 1.52\cdot 10^{-2}, 1.43 \cdot 10^0)$)
	\begin{equation*}
		u_{0,0}(t) := a_{0,0} \cdot \cos(\phi_{0,0}(t))
		:= (t + 1) \cdot \cos((15 \cdot t + 21) \cdot \pi \cdot t)
	\end{equation*}
	and the second IMF component (characteristic
	$\sim(1.57\cdot 10^1,1.91 \cdot 10^{-1}, 1.88 \cdot 10^{-12})$)
	\begin{equation*}
		u_{0,1}(t) := a_{0,1} \cdot \cos(\phi_{0,1}(t))
		:= (3 \cdot t + 1) \cdot \cos(5 \cdot \pi \cdot t).
	\end{equation*}
	We calculate the instantaneous frequencies $\phi'_{0,0}$ and
	$\phi'_{0,1}$ of both IMF components $u_{0,0}$ and $u_{0,1}$ as
	\begin{equation*}
		\phi'_{0,0} = 15 \cdot \pi \cdot t + (15 \cdot t + 21) \cdot \pi
	\end{equation*}
	and
	\begin{equation*}
		\phi'_{0,1} = 5 \cdot \pi.
	\end{equation*}
	Our objective is to run a full EMD on this input signal $s_0$, which
	means that we, in each step, identify an IMF that we will further
	analyze to obtain its instantaneous amplitude and frequency.
	Due to the nature of our HOST-EMD
	algorithm we first extract high-frequency components only to continue
	to extract successively lower frequency components in subsequent steps,
	corresponding to the target in this example to extract $u_{0,0}$ first
	and then $u_{0,1}$.
	\par
	In the first step we find the IMF $\tilde{u}_{0,0}$ with instantaneous
	amplitude $\tilde{a}_{0,0}$ and frequency $\tilde{\phi}'_{0,0}$
	(see Figure~\ref{fig:example_emd-0-step_0}). Splitting $u_{0,0}$
	analytically from the input signal $s_0$ we obtain our first residual
	\begin{equation}\label{eq:example_emd-0-r-1}
		r_{0,1}(t) := s_0 - u_{0,0}
	\end{equation}
	(see Figure~\ref{fig:example_emd-0-r-1}), whose calculated form
	$\tilde{r}_{0,1}$ we will use as the input signal for our second step.
	\par
	Analogously, we find the IMF $\tilde{u}_{0,1}$ with instantaneous
	amplitude $\tilde{a}_{0,1}$ and frequency $\tilde{\phi}'_{0,1}$
	in the second
	step (see Figure~\ref{fig:example_emd-0-step_1}) and further
	splitting $u_{0,1}$ from $r_{0,1}$ analytically yields the second
	residual
	\begin{equation}\label{eq:example_emd-0-r-2}
		r_{0,2}(t) := r_{0,1} - u_{0,1} = s_0 - u_{0,0} - u_{0,1}
	\end{equation}
	(see Figure~\ref{fig:example_emd-0-r-2}) with its calculated form
	$\tilde{r}_{0,2}$, which we identify
	as the last residual given it obviously contains no further
	IMF components.
	\par
	Discussing errors is more difficult than in the other examples
	presented in this thesis as we have two algorithms working in concert,
	namely the proposed iterative slope sifting and differential operator
	extraction algorithms. Nevertheless, they turn out to be working
	independently and what we can note is that the only significant errors
	are visible at the boundaries of the signal, which is to be expected
	as it is blind for the analytical nature of the input signal.
	In the \enquote{interior} of the signal, the relative errors range
	between $10^{-2}$ and $10^{-6}$.
	\begin{figure}[htbp]
		\centering
		\begin{tikzpicture}
			\begin{axis}
				\addplot [mark=none, smooth] table
					[x=x, y=y, col sep=comma]
					{examples/emd.data/0-s.csv};
			\end{axis}
		\end{tikzpicture}
		\caption{Plot of the multicomponent signal
		$s_{0}$ (see (\ref{eq:example_emd-0-s})) from
		Example~\ref{ex:emd-0}.}
		\label{fig:example_emd-0-s}
	\end{figure}
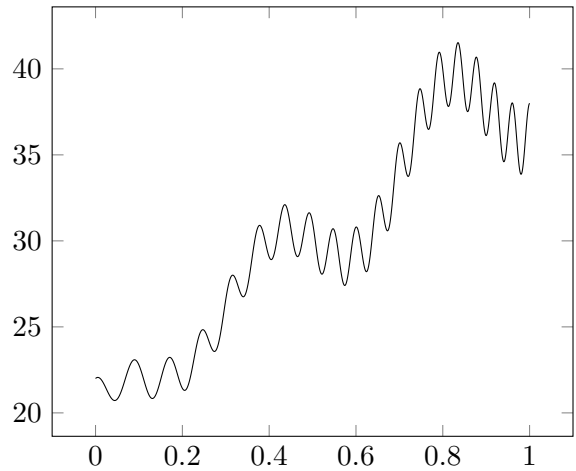
	\begin{figure}[htbp]
		\centering
		\begin{tikzpicture}
			\begin{axis}[
				name=plot1,
				height=4.5cm,width=4.5cm,
				title=$u_{0,0}(t)$,
			]
				\addplot [mark=none, smooth] table
					[x=x, y=y, col sep=comma]
					{examples/emd.data/0-u-0.csv};
				\addplot [mark=none, smooth, dotted] table
					[x=x, y=y, col sep=comma]
					{examples/emd.data/0-u-0-analytic.csv};
			\end{axis}
			\begin{axis}[
				name=plot2,
				at={($(plot1.east)+(1.2cm,0)$)},
				anchor=west,height=4.5cm,width=4.5cm,
				title=$a_{0,0}(t)$,
			]
				\addplot [mark=none, smooth] table
					[x=x, y=y, col sep=comma]
					{examples/emd.data/0-a-0.csv};
				\addplot [mark=none, smooth, dotted] table
					[x=x, y=y, col sep=comma]
					{examples/emd.data/0-a-0-analytic.csv};
			\end{axis}
			\begin{axis}[
				name=plot3,
				at={($(plot2.east)+(1.2cm,0)$)},
				anchor=west,height=4.5cm,width=4.5cm,
				title=$\phi'_{0,0}(t)$
			]
				\addplot [mark=none, smooth] table
					[x=x, y=y, col sep=comma]
					{examples/emd.data/0-freq-0.csv};
				\addplot [mark=none, smooth, dotted] table
					[x=x, y=y, col sep=comma]
					{examples/emd.data/0-freq-0-analytic.csv};
			\end{axis}
			\begin{semilogyaxis}[
				name=plot4,
				at={($(plot1.south)-(0,0.8cm)$)},
				anchor=north,height=4.5cm,width=4.5cm,
				xlabel=$t$
			]
				\addplot [mark=none, smooth] table
					[x=x, y=y, col sep=comma]
					{examples/emd.data/0-u-0-abserr.csv};
			\end{semilogyaxis}
			\begin{semilogyaxis}[
				name=plot5,
				at={($(plot4.east)+(1.2cm,0)$)},
				anchor=west,height=4.5cm,width=4.5cm,
				xlabel=$t$
			]
				\addplot [mark=none, smooth] table
					[x=x, y=y, col sep=comma]
					{examples/emd.data/0-a-0-relerr.csv};
			\end{semilogyaxis}
			\begin{semilogyaxis}[
				name=plot6,
				at={($(plot5.east)+(1.2cm,0)$)},
				anchor=west,height=4.5cm,width=4.5cm,
				xlabel=$t$
			]
				\addplot [mark=none, smooth] table
					[x=x, y=y, col sep=comma]
					{examples/emd.data/0-freq-0-relerr.csv};
			\end{semilogyaxis}
		\end{tikzpicture}
		\caption{Plots of the first extracted
		IMF-component $\tilde{u}_{0,0}$ of $s_0$
		(see (\ref{eq:example_emd-0-s}))
		and its respective instantaneous amplitude $\tilde{a}_{0,0}$ and
		frequency $\tilde{\phi}'_{0,0}$ in the first row (with
		analytical solutions $u_{0,0}, a_{0,0}$ and
		$\phi'_{0,0}$ (dotted)) with
		semi-log plots of the absolute (for $u_{0,0}(t)$) and
		relative errors (for $a_{0,0}(t)$ and $\phi'_{0,0}(t)$) compared to the
		analytical solutions in the second row for
		$(k, q, n, \varepsilon) = (4, 4, 180, 0.01)$
		from Example~\ref{ex:emd-0}.}
		\label{fig:example_emd-0-step_0}
	\end{figure}
	\begin{figure}[htbp]
		\centering
		\begin{tikzpicture}
			\begin{axis}	[
				name=plot1,
				height=7cm,width=7cm,
			]
				\addplot [mark=none, smooth] table
					[x=x, y=y, col sep=comma]
					{examples/emd.data/0-r-1.csv};
				\addplot [mark=none, smooth, dotted] table
					[x=x, y=y, col sep=comma]
					{examples/emd.data/0-r-1-analytic.csv};
			\end{axis}
			\begin{semilogyaxis}	[
				name=plot2,
				height=7cm,width=7cm,
				at={($(plot1.east)+(1.5cm,0)$)},
				anchor=west,
			]
				\addplot [mark=none, smooth] table
					[x=x, y=y, col sep=comma]
					{examples/emd.data/0-r-1-relerr.csv};
			\end{semilogyaxis}
		\end{tikzpicture}
		\caption{Plots of the calculated first residual $\tilde{r}_{0,1}$
		(see (\ref{eq:example_emd-0-r-1})) and its analytical solution
		$r_{0,1}$ (dotted) on
		the left and a semi-log plot of the relative error between both
		on the right for $(k, q, n, \varepsilon) = (4, 4, 180, 0.01)$
		from Example~\ref{ex:emd-0}.}
		\label{fig:example_emd-0-r-1}
	\end{figure}
	\begin{figure}[htbp]
		\centering
		\begin{tikzpicture}
			\begin{axis}[
				name=plot1,
				height=4.5cm,width=4.5cm,
				title=$u_{0,1}(t)$,
			]
				\addplot [mark=none, smooth] table
					[x=x, y=y, col sep=comma]
					{examples/emd.data/0-u-1.csv};
				\addplot [mark=none, smooth, dotted] table
					[x=x, y=y, col sep=comma]
					{examples/emd.data/0-u-1-analytic.csv};
			\end{axis}
			\begin{axis}[
				name=plot2,
				at={($(plot1.east)+(1.2cm,0)$)},
				anchor=west,height=4.5cm,width=4.5cm,
				title=$a_{0,1}(t)$,
			]
				\addplot [mark=none, smooth] table
					[x=x, y=y, col sep=comma]
					{examples/emd.data/0-a-1.csv};
				\addplot [mark=none, smooth, dotted] table
					[x=x, y=y, col sep=comma]
					{examples/emd.data/0-a-1-analytic.csv};
			\end{axis}
			\begin{axis}[
				name=plot3,
				at={($(plot2.east)+(1.2cm,0)$)},
				anchor=west,height=4.5cm,width=4.5cm,
				title=$\phi'_{0,1}(t)$
			]
				\addplot [mark=none, smooth] table
					[x=x, y=y, col sep=comma]
					{examples/emd.data/0-freq-1.csv};
				\addplot [mark=none, smooth, dotted] table
					[x=x, y=y, col sep=comma]
					{examples/emd.data/0-freq-1-analytic.csv};
			\end{axis}
			\begin{semilogyaxis}[
				name=plot4,
				at={($(plot1.south)-(0,0.8cm)$)},
				anchor=north,height=4.5cm,width=4.5cm,
				xlabel=$t$
			]
				\addplot [mark=none, smooth] table
					[x=x, y=y, col sep=comma]
					{examples/emd.data/0-u-1-abserr.csv};
			\end{semilogyaxis}
			\begin{semilogyaxis}[
				name=plot5,
				at={($(plot4.east)+(1.2cm,0)$)},
				anchor=west,height=4.5cm,width=4.5cm,
				xlabel=$t$
			]
				\addplot [mark=none, smooth] table
					[x=x, y=y, col sep=comma]
					{examples/emd.data/0-a-1-relerr.csv};
			\end{semilogyaxis}
			\begin{semilogyaxis}[
				name=plot6,
				at={($(plot5.east)+(1.2cm,0)$)},
				anchor=west,height=4.5cm,width=4.5cm,
				xlabel=$t$
			]
				\addplot [mark=none, smooth] table
					[x=x, y=y, col sep=comma]
					{examples/emd.data/0-freq-1-relerr.csv};
			\end{semilogyaxis}
		\end{tikzpicture}
		\caption{Plots of the second extracted
		IMF-component $\tilde{u}_{0,1}$ of $s_0$
		(see (\ref{eq:example_emd-0-s}))
		and its respective instantaneous amplitude $\tilde{a}_{0,1}$ and
		frequency $\tilde{\phi}'_{0,1}$ in the first row (with
		analytical solutions $u_{0,1}, a_{0,1}$ and
		$\phi'_{0,1}$ (dotted)) with
		semi-log plots of the absolute (for $u_{0,1}(t)$) and
		relative errors (for $a_{0,1}(t)$ and $\phi'_{0,1}(t)$) compared to the
		analytical solutions in the second row for
		$(k, q, n, \varepsilon) = (4, 4, 180, 0.01)$
		from Example~\ref{ex:emd-0}.}
		\label{fig:example_emd-0-step_1}
	\end{figure}
	\begin{figure}[htbp]
		\centering
		\begin{tikzpicture}
			\begin{axis}	[
				name=plot1,
				height=7cm,width=7cm,
			]
				\addplot [mark=none, smooth] table
					[x=x, y=y, col sep=comma]
					{examples/emd.data/0-r-2.csv};
				\addplot [mark=none, smooth, dotted] table
					[x=x, y=y, col sep=comma]
					{examples/emd.data/0-r-2-analytic.csv};
			\end{axis}
			\begin{semilogyaxis}	[
				name=plot2,
				height=7cm,width=7cm,
				at={($(plot1.east)+(1.5cm,0)$)},
				anchor=west,
			]
				\addplot [mark=none, smooth] table
					[x=x, y=y, col sep=comma]
					{examples/emd.data/0-r-2-relerr.csv};
			\end{semilogyaxis}
		\end{tikzpicture}
		\caption{Plots of the calculated second residual $\tilde{r}_{0,2}$
		(see (\ref{eq:example_emd-0-r-2})) and its analytical solution
		 $r_{0,2}$ (dotted) on
		the left and a semi-log plot of the relative error between both
		on the right for $(k, q, n, \varepsilon) = (4, 4, 180, 0.01)$
		from Example~\ref{ex:emd-0}.}
		\label{fig:example_emd-0-r-2}
	\end{figure}
\end{example}
\begin{example}\label{ex:emd-1}
	This example was inspired by \cite[Example~2]{hs11}.
	Consider the multicomponent signal
	\begin{equation}\label{eq:example_emd-1-s}
		s_1(t) := u_{1,0}(t) + u_{1,1}(t) + 25 \cdot t^3
	\end{equation}
	with the first IMF component (characteristic
	$\sim(2.51\cdot 10^2, 1.19 \cdot 10^{-2}, 1.88 \cdot 10^{-12})$)
	\begin{equation*}
		u_{1,0}(t) := a_{1,0} \cdot \cos(\phi_{1,0}(t))
		:= (t + 1) \cdot \cos((15 \cdot t + 21) \cdot \pi \cdot t)
	\end{equation*}
	and the second IMF component (characteristic
	$\sim(4.71 \cdot 10^1, 4.24 \cdot 10^{-1}, 2.86 \cdot 10^{-12})$)
	\begin{equation*}
		u_{1,1}(t) := a_{1,1} \cdot \cos(\phi_{1,1}(t))
		:= (3 \cdot t + 1) \cdot \cos(5 \cdot \pi \cdot t).
	\end{equation*}
	We calculate the instantaneous frequencies $\phi'_{1,0}$ and
	$\phi'_{1,1}$ of both IMF components $u_{1,0}$ and $u_{1,1}$ as
	\begin{equation*}
		\phi'_{1,0} = 15 \cdot \pi \cdot t + (15 \cdot t + 21) \cdot \pi
	\end{equation*}
	and
	\begin{equation*}
		\phi'_{1,1} = 5 \cdot \pi.
	\end{equation*}
	Our objective is to run a full EMD on this input signal $s_0$, which
	means that we, in each step, identify an IMF that we will further
	analyze to obtain its instantaneous amplitude and frequency.
	Due to the nature of our HOST-EMD
	algorithm we first extract high-frequency components only to continue
	to extract successively lower frequency components in subsequent steps,
	corresponding to the target in this example to extract $u_{1,0}$ first
	and then $u_{1,1}$.
	\par
	In the first step we find the IMF $\tilde{u}_{1,0}$ with instantaneous
	amplitude $\tilde{a}_{1,0}$ and frequency $\tilde{\phi}'_{1,0}$
	(see Figure~\ref{fig:example_emd-1-step_0}). Splitting $u_{1,0}$
	analytically from the input signal $s_1$ we obtain our first residual
	\begin{equation}\label{eq:example_emd-1-r-1}
		r_{1,1}(t) := s_1 - u_{1,0}
	\end{equation}
	(see Figure~\ref{fig:example_emd-1-r-1}), whose calculated form
	$\tilde{r}_{1,1}$ we will use as the input signal for our second step.
	\par
	Analogously, we find the IMF $\tilde{u}_{1,1}$ with instantaneous
	amplitude $\tilde{a}_{1,1}$ and frequency $\tilde{\phi}'_{1,1}$
	in the second
	step (see Figure~\ref{fig:example_emd-1-step_1}) and further
	splitting $u_{1,1}$ from $r_{1,1}$ analytically yields the second
	residual
	\begin{equation}\label{eq:example_emd-1-r-2}
		r_{1,2}(t) := r_{1,1} - u_{1,1} = s_1 - u_{1,0} - u_{1,1}
	\end{equation}
	(see Figure~\ref{fig:example_emd-1-r-2}) with its calculated form
	$\tilde{r}_{1,2}$, which we identify
	as the last residual given it obviously contains no further
	IMF components.
	\par
	Discussing errors is more difficult than in the other examples
	presented in this thesis as we have two algorithms working in concert,
	namely the proposed iterative slope sifting and differential operator
	extraction algorithms. Nevertheless, they turn out to be working
	independently and what we can note is that the only significant errors
	are visible at the boundaries of the signal, which is to be expected
	as it is blind for the analytical nature of the input signal.
	In the \enquote{interior} of the signal, the relative errors range
	between $10^{-2}$ and $10^{-6}$.
	\begin{figure}[htbp]
		\centering
		\begin{tikzpicture}
			\begin{axis}
				\addplot [mark=none, smooth] table
					[x=x, y=y, col sep=comma]
					{examples/emd.data/1-s.csv};
			\end{axis}
		\end{tikzpicture}
		\caption{Plot of the multicomponent signal
		$s_{1}$ (see (\ref{eq:example_emd-1-s})) from
		Example~\ref{ex:emd-1}.}
		\label{fig:example_emd-1-s}
	\end{figure}
	\begin{figure}[htbp]
		\centering
		\begin{tikzpicture}
			\begin{axis}[
				name=plot1,
				height=4.5cm,width=4.5cm,
				title=$u_{1,0}(t)$,
			]
				\addplot [mark=none, smooth] table
					[x=x, y=y, col sep=comma]
					{examples/emd.data/1-u-0.csv};
				\addplot [mark=none, smooth, dotted] table
					[x=x, y=y, col sep=comma]
					{examples/emd.data/1-u-0-analytic.csv};
			\end{axis}
			\begin{axis}[
				name=plot2,
				at={($(plot1.east)+(1.2cm,0)$)},
				anchor=west,height=4.5cm,width=4.5cm,
				title=$a_{1,0}(t)$,
			]
				\addplot [mark=none, smooth] table
					[x=x, y=y, col sep=comma]
					{examples/emd.data/1-a-0.csv};
				\addplot [mark=none, smooth, dotted] table
					[x=x, y=y, col sep=comma]
					{examples/emd.data/1-a-0-analytic.csv};
			\end{axis}
			\begin{axis}[
				name=plot3,
				at={($(plot2.east)+(1.2cm,0)$)},
				anchor=west,height=4.5cm,width=4.5cm,
				title=$\phi'_{1,0}(t)$
			]
				\addplot [mark=none, smooth] table
					[x=x, y=y, col sep=comma]
					{examples/emd.data/1-freq-0.csv};
				\addplot [mark=none, smooth, dotted] table
					[x=x, y=y, col sep=comma]
					{examples/emd.data/1-freq-0-analytic.csv};
			\end{axis}
			\begin{semilogyaxis}[
				name=plot4,
				at={($(plot1.south)-(0,0.8cm)$)},
				anchor=north,height=4.5cm,width=4.5cm,
				xlabel=$t$
			]
				\addplot [mark=none, smooth] table
					[x=x, y=y, col sep=comma]
					{examples/emd.data/1-u-0-abserr.csv};
			\end{semilogyaxis}
			\begin{semilogyaxis}[
				name=plot5,
				at={($(plot4.east)+(1.2cm,0)$)},
				anchor=west,height=4.5cm,width=4.5cm,
				xlabel=$t$
			]
				\addplot [mark=none, smooth] table
					[x=x, y=y, col sep=comma]
					{examples/emd.data/1-a-0-relerr.csv};
			\end{semilogyaxis}
			\begin{semilogyaxis}[
				name=plot6,
				at={($(plot5.east)+(1.2cm,0)$)},
				anchor=west,height=4.5cm,width=4.5cm,
				xlabel=$t$
			]
				\addplot [mark=none, smooth] table
					[x=x, y=y, col sep=comma]
					{examples/emd.data/1-freq-0-relerr.csv};
			\end{semilogyaxis}
		\end{tikzpicture}
		\caption{Plots of the first extracted
		IMF-component $\tilde{u}_{1,0}$ of $s_1$
		(see (\ref{eq:example_emd-1-s}))
		and its respective instantaneous amplitude $\tilde{a}_{1,0}$ and
		frequency $\tilde{\phi}'_{1,0}$ in the first row (with
		analytical solutions $u_{1,0}, a_{1,0}$ and
		$\phi'_{1,0}$ (dotted)) with
		semi-log plots of the absolute (for $u_{1,0}(t)$) and
		relative errors (for $a_{1,0}(t)$ and $\phi'_{1,0}(t)$) compared to the
		analytical solutions in the second row for
		$(k, q, n, \varepsilon) = (4, 4, 180, 0.01)$
		from Example~\ref{ex:emd-1}.}
		\label{fig:example_emd-1-step_0}
	\end{figure}
	\begin{figure}[htbp]
		\centering
		\begin{tikzpicture}
			\begin{axis}	[
				name=plot1,
				height=7cm,width=7cm,
			]
				\addplot [mark=none, smooth] table
					[x=x, y=y, col sep=comma]
					{examples/emd.data/1-r-1.csv};
				\addplot [mark=none, smooth, dotted] table
					[x=x, y=y, col sep=comma]
					{examples/emd.data/1-r-1-analytic.csv};
			\end{axis}
			\begin{semilogyaxis}	[
				name=plot2,
				height=7cm,width=7cm,
				at={($(plot1.east)+(1.5cm,0)$)},
				anchor=west,
			]
				\addplot [mark=none, smooth] table
					[x=x, y=y, col sep=comma]
					{examples/emd.data/1-r-1-relerr.csv};
			\end{semilogyaxis}
		\end{tikzpicture}
		\caption{Plots of the calculated first residual $\tilde{r}_{1,1}$
		(see (\ref{eq:example_emd-1-r-1})) and its analytical solution
		$r_{1,1}$ (dotted) on
		the left and a semi-log plot of the relative error between both
		on the right for $(k, q, n, \varepsilon) = (4, 4, 180, 0.01)$ 
		from Example~\ref{ex:emd-1}.}
		\label{fig:example_emd-1-r-1}
	\end{figure}
	\begin{figure}[htbp]
		\centering
		\begin{tikzpicture}
			\begin{axis}[
				name=plot1,
				height=4.5cm,width=4.5cm,
				title=$u_{1,1}(t)$,
			]
				\addplot [mark=none, smooth] table
					[x=x, y=y, col sep=comma]
					{examples/emd.data/1-u-1.csv};
				\addplot [mark=none, smooth, dotted] table
					[x=x, y=y, col sep=comma]
					{examples/emd.data/1-u-1-analytic.csv};
			\end{axis}
			\begin{axis}[
				name=plot2,
				at={($(plot1.east)+(1.2cm,0)$)},
				anchor=west,height=4.5cm,width=4.5cm,
				title=$a_{1,1}(t)$,
			]
				\addplot [mark=none, smooth] table
					[x=x, y=y, col sep=comma]
					{examples/emd.data/1-a-1.csv};
				\addplot [mark=none, smooth, dotted] table
					[x=x, y=y, col sep=comma]
					{examples/emd.data/1-a-1-analytic.csv};
			\end{axis}
			\begin{axis}[
				name=plot3,
				at={($(plot2.east)+(1.2cm,0)$)},
				anchor=west,height=4.5cm,width=4.5cm,
				title=$\phi'_{1,1}(t)$
			]
				\addplot [mark=none, smooth] table
					[x=x, y=y, col sep=comma]
					{examples/emd.data/1-freq-1.csv};
				\addplot [mark=none, smooth, dotted] table
					[x=x, y=y, col sep=comma]
					{examples/emd.data/1-freq-1-analytic.csv};
			\end{axis}
			\begin{semilogyaxis}[
				name=plot4,
				at={($(plot1.south)-(0,0.8cm)$)},
				anchor=north,height=4.5cm,width=4.5cm,
				xlabel=$t$
			]
				\addplot [mark=none, smooth] table
					[x=x, y=y, col sep=comma]
					{examples/emd.data/1-u-1-abserr.csv};
			\end{semilogyaxis}
			\begin{semilogyaxis}[
				name=plot5,
				at={($(plot4.east)+(1.2cm,0)$)},
				anchor=west,height=4.5cm,width=4.5cm,
				xlabel=$t$
			]
				\addplot [mark=none, smooth] table
					[x=x, y=y, col sep=comma]
					{examples/emd.data/1-a-1-relerr.csv};
			\end{semilogyaxis}
			\begin{semilogyaxis}[
				name=plot6,
				at={($(plot5.east)+(1.2cm,0)$)},
				anchor=west,height=4.5cm,width=4.5cm,
				xlabel=$t$
			]
				\addplot [mark=none, smooth] table
					[x=x, y=y, col sep=comma]
					{examples/emd.data/1-freq-1-relerr.csv};
			\end{semilogyaxis}
		\end{tikzpicture}
		\caption{Plots of the second extracted
		IMF-component $\tilde{u}_{1,1}$ of $s_1$
		(see (\ref{eq:example_emd-1-s}))
		and its respective instantaneous amplitude $\tilde{a}_{1,1}$ and
		frequency $\tilde{\phi}'_{1,1}$ in the first row (with
		analytical solutions $u_{1,1}, a_{1,1}$ and
		$\phi'_{1,1}$ (dotted)) with
		semi-log plots of the absolute (for $u_{1,1}(t)$) and
		relative errors (for $a_{1,1}(t)$ and $\phi'_{1,1}(t)$) compared to the
		analytical solutions in the second row for
		$(k, q, n, \varepsilon) = (4, 4, 180, 0.01)$
		from Example~\ref{ex:emd-1}.}
		\label{fig:example_emd-1-step_1}
	\end{figure}
	\begin{figure}[htbp]
		\centering
		\begin{tikzpicture}
			\begin{axis}	[
				name=plot1,
				height=7cm,width=7cm,
			]
				\addplot [mark=none, smooth] table
					[x=x, y=y, col sep=comma]
					{examples/emd.data/1-r-2.csv};
				\addplot [mark=none, smooth, dotted] table
					[x=x, y=y, col sep=comma]
					{examples/emd.data/1-r-2-analytic.csv};
			\end{axis}
			\begin{semilogyaxis}	[
				name=plot2,
				height=7cm,width=7cm,
				at={($(plot1.east)+(1.5cm,0)$)},
				anchor=west,
			]
				\addplot [mark=none, smooth] table
					[x=x, y=y, col sep=comma]
					{examples/emd.data/1-r-2-relerr.csv};
			\end{semilogyaxis}
		\end{tikzpicture}
		\caption{Plots of the calculated second residual $\tilde{r}_{1,2}$
		(see (\ref{eq:example_emd-1-r-2})) and its analytical solution
		 $r_{1,2}$ (dotted) on
		the left and a semi-log plot of the relative error between both
		on the right for $(k, q, n, \varepsilon) = (4, 4, 180, 0.01)$
		from Example~\ref{ex:emd-1}.}
		\label{fig:example_emd-1-r-2}
	\end{figure}
\end{example}
\FloatBarrier
\chapter{Summary and Outlook}
In the course of this thesis we started off with the construction
of an analytical model of the empirical mode decomposition and showed
that it satisfies strong duality, namely due to being \textsc{Slater}
regular, using B-spline properties and the theory of convex-like
optimization. This strong duality yielded a theoretical justification for
reformulating the constrained optimization problem into an unconstrained
optimization problem with a regularization term that enforces the
constraints.
\par
In the context of EMD, we examined one possible modern approach to such a
regularization term: The operator-based signal-separation (OSS)
null-space-pursuit (NSP) method that makes use of adaptive differential
operators to determine instantaneous amplitude and phase from a given IMF.
We observed that the operator can only yield unique results when the
IMF that is to be regularized has constant amplitude $1$, which is
a strong limitation.
\par
We considered the classic EMD algorithm and noted the following:
Sifting, the process in which the signal is separated into an
IMF and a residual, is highly dependent on a good method for estimating
the upper envelope of a multicomponent signal. We identified the
weakness in the classic method that the envelope may intersect with
the signal itself, violating the definition of an envelope.
We presented a new approach called iterative slope envelope estimation
that solves this problem. Additionally, it is a generalization of the
classic envelope estimation method.
We also made the following observation: During sifting we also
obtain the amplitude of the IMF. This in turn meant that we can just
divide the IMF by its amplitude and obtain an IMF with constant
instantaneous amplitude $1$. This meant that we could use the
differential operator from the NSP method mentioned earlier and
be sure that it behaved properly.
\par
Using both the new envelope estimation method and the differential
operator, we have obtained an approach that is a mix of classic
and modern methods. Consequently, we defined a hybrid EMD method
and examined it using multiple examples.
These examples were implemented in a toolbox called ETHOS
using the GNU Scientific Library, which was explained and documented
subsequently. One newly discovered approach in this process was to
evaluate the quality of each extracted IMF by calculating its
characteristic.
\par
It is clear that only by building a strict theoretical foundation that
went further than previous works it was possible to obtain the results
laid out in this thesis. This foundation included convex constraints
instead of regularization terms, using B-splines for the theoretical and
practical modelling of functions and the theory of convex-like
optimization.
The author suspects that given the underlying EMD optimization-problem is not
convex, previous attempts to show regularity were not followed through.
This is because it is widely assumed that convexity is a requirement for showing
\textsc{Slater} regularity. This is wrong,
as convex-likeness is sufficient, but generally not well-known.
The attractiveness of \textsc{Slater} regularity is due to the fact
that it applies to the entire optimization problem instead of just specific
points, and showing it for the EMD optimization problem yields the
regularity for any scenario.
\par
A possible outlook for further works would be to expand the iterative slope
sifting algorithm to higher dimensions to explore their usefulness in
multidimensional EMD. In the context of the ETHOS toolbox, a possible
field of research would be to use more specific tools for
sparse optimization, possibly based on the modern GHOST
(General, Hybrid and Optimized Sparse Toolkit) sparse library (see \cite{ktrz+17}),
and refine the B-spline data fitting process, especially concerning
preventing overfitting and underfitting.
\par
In the end, what is easy to see is that signal analysis as a whole and
EMD in particular are complex topics with many open questions, not only
in the theoretical sense but also in practical terms. Examples include
problems like mode mixing and noise distortion of an input signal.
What remains to be seen is how these problems can be solved most
effectively: In a preprocessing step before applying the
EMD method or as a part of a newly devised EMD method. Or maybe they are
simply unsolvable in terms of information theory. No matter the outcome,
given the far-reaching applications it has in many different fields,
every little problem solved in signal analysis may have far-reaching
consequences.
\appendix
\chapter{Function Space Order and Operators}\label{ch:funcspaceopt}
This thesis makes use of operations on functions rather than scalars,
for instance as candidates in optimization problems. Because of that we
want to understand how function spaces work and how we can map them to
other vector spaces we can handle more easily. Even though the concepts
laid out as follows seem to be very intuitive and might not even need
further explanation beyond notation, the formal aspect of this topic
shall not be missed for completeness sake but also not unnecessarily
complicate the main matter, which is why this chapter is in the appendix.
\par
The functions we are dealing with are relatively smooth functions
$\R \to \R$.
Formally speaking, we have a $p \geq 0$ such that our function can be
differentiated $p$ times and the $p$th derivative is continuous.
This set is commonly denoted as $\mathcal{C}^p(\R,\R)$, the set of
$p$-continuously differentiable functions. It is easy to see
that $\mathcal{C}^0(\R,\R) \supseteq \mathcal{C}^1(\R,\R) \supseteq
\mathcal{C}^2(\R,\R) \supseteq \dots \supseteq \mathcal{C}^\infty(\R,\R)$,
so if we define something for $\mathcal{C}^0(\R,\R)$ it automatically holds
for all functions from $\mathcal{C}^p(\R,\R)$ with $p \geq 0$. This is why,
as follows, we will only use the set of continuous functions
$\mathcal{C}^0(\R,\R)$ in our definitions to keep everything relatively
general.
\par
One aspect of interest is to be able to compare two functions in some
way. In other words, we want to find an order relation on
$\mathcal{C}^0(\R,\R)$. One way to do that is as follows: We say that
a function succeeds another function if and only if the former is
pointwise greater than or equal to the latter. We define
precedence respectively and note that, obviously, there are functions
which can not be compared this way. This concept
is illustrated in Figure~\ref{fig:order} and is the function space
order used in the course of this thesis.
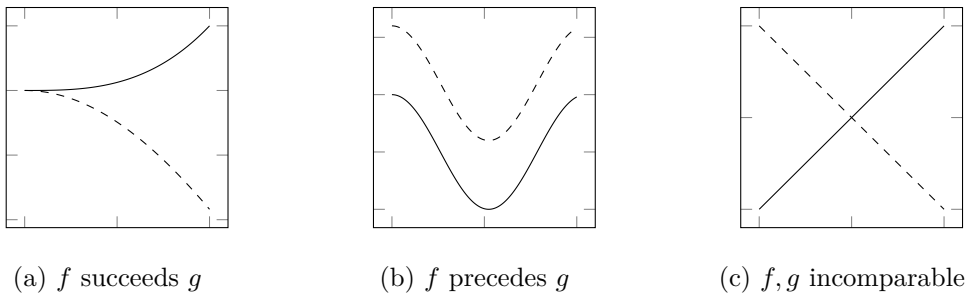
\begin{figure}[htbp]
	\begin{subfigure}[c]{0.32\textwidth}
		\centering
		\begin{tikzpicture}
			\begin{axis}[
				name=plot1,
				height=4.5cm,width=4.5cm,
				xticklabels={,,},
				yticklabels={,,},
			]
				\addplot [domain=0:1, samples=200]{1 + 0.2 * x^3};
				\addplot [dashed, domain=0:1, samples=200]
					{0.2 + 0.8 * cos(deg(x)))};
			\end{axis}
		\end{tikzpicture}
		\subcaption{$f$ succeeds $g$}
	\end{subfigure}
	\begin{subfigure}[c]{0.32\textwidth}
		\centering
		\begin{tikzpicture}
			\begin{axis}[
				name=plot1,
				height=4.5cm,width=4.5cm,
				xticklabels={,,},
				yticklabels={,,},
			]
				\addplot [domain=0:1, samples=200]{cos(6 * deg(x))};
				\addplot [dashed, domain=0:1, samples=200]
					{1.2 + cos(6 * deg(x))};
			\end{axis}
		\end{tikzpicture}
		\subcaption{$f$ precedes $g$}
	\end{subfigure}
	\begin{subfigure}[c]{0.32\textwidth}
		\centering
		\begin{tikzpicture}
			\begin{axis}[
				name=plot1,
				height=4.5cm,width=4.5cm,
				xticklabels={,,},
				yticklabels={,,},
			]
				\addplot [domain=0:1, samples=200]{x)};
				\addplot [dashed, domain=0:1, samples=200]
					{1.0-x};
			\end{axis}
		\end{tikzpicture}
		\subcaption{$f,g$ incomparable}
	\end{subfigure}
	\caption{Illustration of an intuitive concept for a
		partial function space order
		(see Proposition~\ref{prop:partial_order_on_C0})
		exemplified with two functions $f$ (solid)
		and $g$ (dashed).}
	\label{fig:order}
\end{figure}
\par
Before we are able to further formalize this idea, we first need
to introduce some preliminary definitions of what an order is and
how to define ordered vector spaces.
\begin{definition}[Preorder {\cite[Section~2.4]{cie97}}]
	Let $S$ be a set and $\le$ a binary relation on $S$. $\le$
	is a \emph{preorder on $S$} if and only if
	\begin{enumerate}
		\item{%
			\emph{Reflexivity}:
			\begin{equation*}
				\forall x \in S \colon
				x \le x,
			\end{equation*}
		}
		\item{%
			\emph{Transitivity}:
			\begin{equation*}
				\forall x,y,z \in S \colon
				a \le b \land b \le c \Rightarrow
				a \le c.
			\end{equation*}
		}
	\end{enumerate}
\end{definition}
A preorder is the weakest order we can find for a vector space.
If we can also show antisymmetry of a preorder, we obtain
a partial order, given in the following
\begin{definition}[Partial order {\cite[Section~2.4]{cie97}}]
	Let $S$ be a set and $\le$ a preorder on $S$. $\le$
	is a \emph{partial order on $S$} if and only if
	$\le$ is antisymmetric, i.e.
	\begin{equation*}
		\forall a,b \in S \colon
		a \le b \land b \le a \Rightarrow a = b.
	\end{equation*}
\end{definition}
One step higher would be a total order, which adds the
connex property meaning that all elements of the set are
comparable. As we've seen in Figure~\ref{fig:order} this
is not possible for vector spaces given the presence of
indeterminate relations.
\par
If we find a binary relation on a vector space that is
either a preorder or partial order we would, as the next
step, hope that the orders are compatible with vector
space operations. In other words, this means that addition
and scalar multiplication preserve relations intuitively.
This is reflected in the following
\begin{definition}[(Pre)ordered vector space {\cite[Chapter~II, {\S}2.5]{bou03}}]
	Let $V$ be a real vector space and $\le$ a binary relation on $V$.
	$V$ is a \emph{(pre)ordered vector space} if and only if
	\begin{enumerate}
		\item{%
			$\le$ is a preorder or partial order respectively,
		}
		\item{%
			\emph{Addition compatibility}:
			\begin{equation*}
				\forall x,y,z \in V \colon
				x \le y \Rightarrow x + z \le y + z,
			\end{equation*}
		}
		\item{%
			\emph{Scalar multiplication compatibility}:
			\begin{equation*}
				\forall x,y \in V \colon
				\forall \lambda \ge 0 \colon
				x \le y \Rightarrow \lambda \cdot x \le
				\lambda \cdot y.
			\end{equation*}
		}
	\end{enumerate}
\end{definition}
With the definitions in place we can now formalize what has
been discussed at the beginning of the section and illustrated in
Figure~\ref{fig:order}.
\begin{proposition}\label{prop:partial_order_on_C0}
	Let $a,b \in \mathcal{C}^0(\R,\R)$.
	$(\mathcal{C}^0(\R,\R), \preceq)$ is an ordered vector space
	with
	\begin{equation*}
		a \preceq b
		:\Leftrightarrow
		\forall t \in \R \colon a(t) \leq b(t).
	\end{equation*}
\end{proposition}
\begin{proof}
	Let $a,b,c \in \mathcal{C}^0(\R,\R)$. We first show that $\preceq$
	is a partial order.
	\begin{enumerate}
		\item{%
			\emph{Reflexivity}:
			\begin{equation*}
				\forall t \in \R \colon a(t) \leq a(t)
				\Leftrightarrow
				a \preceq a,
			\end{equation*}
		}
		\item{%
			\emph{Transitivity}:
			\begin{align*}
				a \preceq b \land b \preceq c & \Leftrightarrow
					\forall t \in \R \colon (a(t) \leq b(t) \land b(t) \leq c(t)) \\
				& \Rightarrow
					\forall t \in \R \colon a(t) \leq c(t) \\
				& \Leftrightarrow
					a \preceq c,
			\end{align*}
		}
		\item{%
			\emph{Antisymmetry}:
			\begin{align*}
				a \preceq b \land b \preceq a & \Leftrightarrow
					\forall t \in \R \colon (a(t) \leq b(t) \land b(t) \leq a(t)) \\
				& \Leftrightarrow
					\forall t \in \R \colon a(t) = b(t) \\
				& \Leftrightarrow
					a = b.
			\end{align*}
		}
	\end{enumerate}
	Let $\lambda \ge 0$. We now show that $(\mathcal{C}^0(\R,\R), \preceq)$
	satisfies the two axioms
	of an ordered vector space such that the order is compatible
	with the vector space operations.
	\begin{enumerate}
		\item{%
			\emph{Addition compatibility}:
			\begin{align*}
				a \preceq b & \Leftrightarrow
					\forall t \in \R \colon a(t) \leq b(t) \\
				& \Leftrightarrow
					\forall t \in \R \colon a(t) + c(t) \leq b(t) + c(t) \\
				& \Leftrightarrow
					\forall t \in \R \colon (a + c)(t) \leq (b + c)(t) \\
				& \Leftrightarrow
					a + c \preceq b + c,
			\end{align*}
		}
		\item{%
			\emph{Scalar multiplication compatibility}:
			\begin{align*}
				a \preceq b & \Leftrightarrow
					\forall t \in \R \colon a(t) \leq b(t) \\
				& \Leftrightarrow
					\forall t \in \R \colon \lambda \cdot a(t)
					\leq \lambda \cdot b(t) \\
				& \Leftrightarrow
					\forall t \in \R \colon (\lambda \cdot a)(t)
					\leq (\lambda \cdot b)(t) \\
				& \Leftrightarrow
					\lambda \cdot a \preceq
					\lambda \cdot b. \qedhere
			\end{align*}
		}
	\end{enumerate}
\end{proof}
We have now shown that $\mathcal{C}^0(\R,\R)$ is an ordered vector
space with the order relation $\preceq$, but this is just one point of
interest. It is often useful to transform objects that are hard to work
with to a space where that is easy, perform operations on them and
then transform them back. One common example are polar coordinate transforms
which dramatically simplify many complicated integrals.
A requirement for such a transform in the general sense is that it is an
isomorphism and preserves the nature and/or relations of the objects.
In the context of optimization problems this means that a transform shall
preserve the relative order of objects. If a function precedes another its
transform shall do the same relative to the transform of the other function.
\par
The motivation for this approach in this thesis is found within the strong
link between spline functions and their B-spline base coefficients, as introduced in
Chapter~\ref{ch:b-splines}. The base coefficients of a spline function is
a vector in $\R^n$ and fully describes it. Due to the nature of the
B-splines the intuitive partial order $\leq$ on $\R^n$ with regard to the
base coefficients is equivalent to the partial order $\preceq$ on
the set of spline functions $\Sigma_k$. To put it differently, if
a spline function succeeds another, so do their base coefficients.
The advantage this brings is apparent: Vectors in $\R^n$ are much easier
to handle both theoretically and numerically, and any optimization problem
can be trivially transformed between both representations.
\par
Now that we have understood the motivation behind this approach, we
can formalize such a transform between ordered vector spaces in the
following
\begin{definition}[Order isomorphism of ordered vector-spaces
	{\cite[Section~4.1]{cie97}}]
	\label{def:oioovs}
	Let $(V,\le_V), (W,\le_W)$ be ordered vector spaces and
	$f\colon V \to W$. $f$ is an \emph{order isomorphism of
	ordered vector-spaces} if and only if
	\begin{enumerate}
		\item{%
			$f$ is an isomorphism,
		}
		\item{%
			$\forall x,y \in V \colon x \le_V y \Leftrightarrow
			f(x) \le_W f(y)$.
		}
	\end{enumerate}
\end{definition}
As we can see, preserving the order structure opens up new
possibilities with regard to optimization problems. In our case,
even though we formulate it over a function space, we can
transform our candidates into $\R^n$ and examine the equivalent
optimization problem over $\R^n$, which is much more accessible
than the function space in many aspects.
\par
As another remark we define the following operators on functions
that serve notational purposes but can be considered to be more
or less intuitive.
\begin{definition}[Modulus operator]\label{def:modulus_operator}
	Let $a \in \mathcal{C}^0(\R,\R)$.
	The \emph{modulus operator} $|\cdot| : \mathcal{C}^0(\R,\R) \to
	\mathcal{C}^0(\R,\R)$ is defined as
	\begin{equation*}
		|a|(t) := |a(t)|.
	\end{equation*}
\end{definition}
\begin{definition}[Power operator]\label{def:power_operator}
	Let $p \geq 0$ and $a \in \mathcal{C}^0(\R,\R)$.
	The \emph{power operator} $\cdot^p : \mathcal{C}^0(\R,\R) \to
	\mathcal{C}^0(\R,\R)$ is defined as
	\begin{equation*}
		a^p(t) := {\left(a(t)\right)}^p.
	\end{equation*}
\end{definition}
\begin{definition}[Differentiation operator]\label{def:differentiation_operator}
	Let $p \in \N_0$ and $a \in \mathcal{C}^0(\R,\R)$. The \emph{differentiation operator}
	$D^p: \mathcal{C}^0(\R,\R) \to \mathcal{C}^0(\R,\R)$ is defined as
	\begin{equation*}
		\left(D^p a\right)(t) :=
		\begin{cases}
			a(t) & p = 0 \\
			\frac{\mathrm{d}^p a}{\mathrm{d} t^p}(t) & p > 0
		\end{cases}
	\end{equation*}
\end{definition}
What shall be apparent from the results of this chapter is that
this approach saves us from a lot of cumbersome notation and complexity
down the road. The alternative of always thinking of a function as a
set of samples may also solve the discretization problem, but requires
much more care in terms of parametrization and is much harder to access
theoretically. The latter is especially apparent when it comes to derivatives,
which are a crucial part of the theory of this thesis.
\chapter{Convexity Theory}\label{ch:convexity_theory}
The theory of convex sets and functions plays a central role in this
thesis and we will introduce it in this chapter. Even though most
of the given definitions may be known to the reader, the resulting
Theorem~\ref{thm:gershgorin-hadamard} is used in all proofs
in the thesis for showing function convexity. Before we dive
in further, we introduce both concepts of convex sets and convex functions.
\begin{definition}[Convex set]\label{def:convex_set}
	Let $V$ be a real or complex vector space and $C \subseteq V$.
	$C$ is \emph{convex} if and only if
	\begin{equation*}
		\forall a,b \in C \colon \forall \lambda \in [0,1] \colon
		\lambda \cdot a + (1 - \lambda) \cdot b \in C.
	\end{equation*}
\end{definition}
\begin{definition}[(Strictly) convex function]\label{def:convex_function}
	Let $n \in \N$, $C \subseteq \R^n$ convex and $f \colon C \to \R$.
	$f$ is a \emph{strictly convex function} or \emph{convex function}
	if and only if
	\begin{equation*}
		\forall a,b \in C \colon \forall \lambda \in [0,1] \colon
		f(\lambda \cdot a + (1 - \lambda) \cdot b) <
		\lambda \cdot f(a) + (1 - \lambda) \cdot f(b)
	\end{equation*}
	or
	\begin{equation*}
		\forall a,b \in C \colon \forall \lambda \in [0,1] \colon
		f(\lambda \cdot a + (1 - \lambda) \cdot b) \le
		\lambda \cdot f(a) + (1 - \lambda) \cdot f(b)
	\end{equation*}
	respectively.
\end{definition}
For the purpose of determining convexity later the multidimensional
pendant of a second derivative is introduced as the Hessian matrix,
the matrix that contains all possible combinations of
mixed second partial derivatives.
\begin{definition}[Hessian matrix {\cite[Section~3.9]{np06}}]\label{def:hessian}
	Let $n \in \N$ and $f \in \mathcal{C}^2(\R^n,\R)$. The \emph{Hessian matrix}
	$H_f(x) \in \R^{n \times n}$ of $f$ in $x \in \R^n$ is defined as
	\begin{equation*}
		H_f(x) := {\left( \frac{\partial^2 f(x)}{\partial x_i \partial x_j} \right)}_
		{i,j \in \{1,\dots,n\}}.
	\end{equation*}
\end{definition}
Another important concept is the positive definiteness that
can more or less be imagined to be a generalization of
positivity of scalars into the realm of matrices.
\begin{definition}[Positive (semi)definite]
	Let $n \in \N$ and $A \in \R^{n \times n}$. $A$ is \emph{positive definite} or
	\emph{positive semidefinite} if and only if
	\begin{equation*}
		\forall v \in \R^n \colon v^T \cdot A \cdot v > 0
	\end{equation*} or
	\begin{equation*}
		\forall v \in \R^n \colon v^T \cdot A \cdot v \ge 0
	\end{equation*}
	respectively.
\end{definition}
One thing to note here though is that a matrix can in fact have
negative entries and be positive definite, but also have only
positive entries and still not be positive definite.
Using these definitions, we can give the following condition for a
sufficiently smooth function to be convex.
\begin{proposition}[Hessian (strict) convexity condition]
	\label{prop:hessian_convexity_condition}
	Let $n \in \N$ and $f \in \mathcal{C}^2(\R^n,\R)$. $f$ is
	strictly convex or convex if
	and only if for all $x \in \R^n$ the Hessian matrix $H_f(x)$ is
	positive definite or positive semidefinite respectively.
\end{proposition}
\begin{proof}
	See \cite[Corollary~3.9.5]{np06}.
\end{proof}
The result is that object of interest is the Hessian matrix of
a given function we want to examine and its definiteness. The
direct approach to check definiteness is complicated, especially
for large systems like the ones we are dealing with in this thesis.
One useful result to simplify this process is the
\textsc{Gershgorin}-\textsc{Hadamard}-theorem that gives conditions
which are easy to check, yielding a positive definite matrix.
To formulate the theorem we must first consider so-called strictly
diagonally dominant matrices where the absolute value of a diagonal
entry is strictly larger than the sum of the absolute values of all
other entries in that row.
\begin{definition}[Strictly diagonally dominant {\cite[Definition~6.1.9]{hj12}}]
	Let $n \in \N$ and $A \in \R^{n \times n}$. $A$ is
	\emph{strictly diagonally dominant} if	and only if
	\begin{equation*}
		\forall i \in \{1,\dots,n\} \colon |a_{ii}| >
		\sum_{j \in \{1,\dots,n\} \setminus \{i\}} |a_{ij}|.
	\end{equation*}
\end{definition}
\begin{theorem}[\textsc{Geršgorin}-\textsc{Hadamard} {\cite[Theorem~6.1.10]{hj12}}]
	\label{thm:gershgorin-hadamard}
	Let $n \in \N$ and $M \in \R^{n \times n}$ symmetric, diagonally
	dominant and $\forall i \in \{1,\dots,n\} \colon a_{ii} > 0$.
	Then $M$ is positive definite, non-singular and every eigenvalue of $M$
	is positive.
\end{theorem}
\begin{proof}
	See \cite[Theorem~6.1.10]{hj12}.
\end{proof}
Using this theorem it shall be possible to approach even large systems
in regard to positive definiteness. When we show positive definitness
of the Hessian matrix, convexity follows and we have obtained our desired
result.
\chapter{Notation Directory}
\section{Chapter~\ref*{ch:introduction}: \nameref*{ch:introduction}}
\begin{longtable}{p{2cm} p{10cm}}
	$\Phi(t)$ & 1-periodic function;
		see Equation~(\ref{eq:fourier-series}) \\
	$c_j$ & \textsc{Fourier} coefficients;
		see Equation~(\ref{eq:fourier-coefficients}) \\
	$(\mathcal{F}s)(f)$ & \textsc{Fourier} transform;
		see Equation~(\ref{eq:fourier-transform}) \\
	$\left\langle \cdot,\cdot \right\rangle$ & standard inner product;
		see Equation~(\ref{eq:standard_inner_product}) \\
	$\psi_{j,k}$ & wavelet basis function;
		see Equation~(\ref{eq:wavelet-basis}) \\
	$\psi$ & mother wavelet;
		see Equation~(\ref{eq:wavelet-basis}) \\
	$c_{j,k}$ & wavelet coefficients;
		see Equation~(\ref{eq:wavelet-coefficients}) \\
	${\| \cdot \|}_2$ & $2$-norm based on the standard inner product
\end{longtable}
\section{Chapter~\ref*{ch:b-splines}: \nameref*{ch:b-splines}}
\begin{longtable}{p{2cm} p{10cm}}
	$\Pi_p$ & set of polynomials of order $p$;
		see Definition~\ref{def:set_of_polynomials} \\
	$\Sigma_{k,T}$ & spline function space;
		see Definition~\ref{def:spline_function_space} \\
	$\Sigma_{k}$ & shorthand notation for $\Sigma_{k,T}$;
		see Remark~\ref{rem:spline-shorthand} \\
	$k$ & spline function order;
		see Definition~\ref{def:spline_function_space} \\
	$\ell$ & size of spline knot vector;
		see Definition~\ref{def:spline_function_space} \\
	$T$ & spline knot vector;
		see Definition~\ref{def:spline_function_space} \\
	$\indicator(A)$ & indicator function on set $A$;
		see Definition~\ref{def:indicator_function} \\
	$B_{i,k,T}$ & B-spline function;
		see Definition~\ref{def:b-spline} \\
	$B_{i,k}$ & shorthand notation for $B_{i,k,\Delta_k(T)}$;
		see Remark~\ref{rem:spline-shorthand} \\
	$\Delta_k(T)$ & extended knot vector;
		see Definition~\ref{def:extended_knot_vector} \\
	$n$ & number of B-spline functions ($k + \ell - 2$);
		see Definition~\ref{def:extended_knot_vector} \\
	$\B_{k,T}$ & coefficient spline mapping;
		see Definition~\ref{def:coefficient_spline_mapping} \\
	$\B_{k}$ & shorthand notation for $\B_{k,T}$;
	    see Remark~\ref{rem:spline-shorthand} \\
	$c_{k,\infty}$ & B-spline condition constant;
	    see Proposition~\ref{prop:well_conditioned_basis}
\end{longtable}
\section{Chapter~\ref*{ch:emd_analysis}: \nameref*{ch:emd_analysis}}
\begin{longtable}{p{2cm} p{10cm}}
	$a$ & instantaneous amplitude;
		see Defintion~\ref{def:imfs} \\
	$\phi$ & instantaneous phase;
		see Definition~\ref{def:imfs} \\
	$\phi'$ & instantaneous frequency;
		see Definition~\ref{def:imfs} \\
	${\mathcal{S}}_{\mu_0,\mu_1,\mu_2}$ & set of intrinsic mode function souls (IMFS)
		with characteristic $(\mu_0,\mu_1,\mu_2)$;
		see Definition~\ref{def:imfs} \\
	$\boldsymbol{\mathcal{S}}_{\mu_0,\mu_1,\mu_2}$ & set of intrinsic mode spline
		function souls (IMSpFS);
		see Definition~\ref{def:imspfs} \\
	$\I[a,\phi]$ & intrinsic mode function operator;
		see Definition~\ref{def:imf-operator} \\
	$c_1[s](a,\phi)$ & canonical EMD cost function;
		see Definition~\ref{def:ccf} \\
	$\boldsymbol{c}_1[\boldsymbol{s}](\boldsymbol{a},\boldsymbol{\phi})$ & canonical spline EMD cost function;
		see Definition~\ref{def:cscf} \\
	$c_\ell[s](a,\phi)$ & leakage factor EMD cost function;
		see Definition~\ref{def:lfecf} \\
	$\boldsymbol{c}_\ell[\boldsymbol{s}](\boldsymbol{a},\boldsymbol{\phi})$ & leakage factor spline EMD cost function;
		see Definition~\ref{def:lfsecf} \\
	$V^+$ & positive cone of cone $V$;
		see Definition~\ref{def:positive_cone} \\
	$V'$ & dual cone of cone $V$;
		see Definition~\ref{def:dual_cone} \\
	$\Lambda(x,\lambda)$ & \textsc{Lagrange} function;
		see Definition~\ref{def:lf} \\
	$\ul{\Lambda}(\lambda)$ & \textsc{Lagrange} dual function;
		see Definition~\ref{def:ldf}
\end{longtable}
\section{Chapter~\ref*{ch:oss}: \nameref*{ch:oss}}
\begin{longtable}{p{2cm} p{10cm}}
	$A[a]$ & instantaneous envelope derivation operator;
		see Definition~\ref{def:inst_envelope_deriv_operator} \\
	$\Omega[\phi]$ & inverse square continuous frequency operator;
		see Definition~\ref{def:inv_sq_cont_freq_operator} \\
	$\mathcal{D}_{(a,\phi)}$ & IMF differential operator;
		see Definition~\ref{def:imf-diffop} \\
	$\tilde{\mathcal{D}}_{(A,\Omega)}$ & modified IMF differential operator;
		see Definition~\ref{def:modified-imf-diffop}
\end{longtable}
\section{Chapter~\ref*{ch:funcspaceopt}: \nameref*{ch:funcspaceopt}}
\begin{longtable}{p{2cm} p{10cm}}
	$\N$ & set of natural numbers $\ge 1$ \\
	$\N_0$ & set of natural numbers $\ge 0$ \\
	$\mathcal{C}^p(A,B)$ & set of $p$-continuously differentiable
		functions $A \to B$ \\
	$\preceq$ & partial order on $\mathcal{C}^0(\R,\R)$;
		see Proposition~\ref{prop:partial_order_on_C0} \\
	$| \cdot |$ & modulus operator;
		see Definition~\ref{def:modulus_operator} \\
	${(\cdot)}^p$ & power operator;
		see Definition~\ref{def:power_operator} \\
	$D^p$ & differentiation operator;
		see Definition~\ref{def:differentiation_operator}
\end{longtable}
\section{Chapter~\ref*{ch:convexity_theory}: \nameref*{ch:convexity_theory}}
\begin{longtable}{p{2cm} p{10cm}}
	$H_f(x)$ & Hessian matrix of $f$ in $x$; see Definition~\ref{def:hessian}
\end{longtable}
\chapter{Code Listings}\label{ch:code}
\lstset{basicstyle=\ttfamily\normalsize,inputencoding=utf8/latin1,
	morekeywords={},float,numberstyle=\footnotesize\noncopynumber,
	breakatwhitespace=false,breaklines=false,escapeinside={\%*}{*)},
	columns=fullflexible,keepspaces=true,numbers=left,
	frame=single,showstringspaces=true,keywordstyle=\ttfamily}
\section{ETHOS Toolbox}\label{sec:code-ethos_toolbox}
\subsection{ethos.h}
\begin{lstlisting}[language=C,label=lst:ethos-ethosh]
/* See LICENSE file for copyright and license details. */
#ifndef ETHOS_H
#define ETHOS_H

#include <sys/types.h>

enum ethos_plot_format {
	ETHOS_PLOT_CSV,
	ETHOS_PLOT_GRAPH,
};

struct ethos {
	void *bw;
	size_t n;
	double a;
	double b;
	int k;
	int q;
	double *grid;
	size_t ngrid;
	double *extgrid;
	size_t nextgrid;
	double ***dB;
	size_t *istart;
	size_t *iend;
};

enum ethos_grid {
	ETHOS_GRID_UNIFORM,
};

int ethos_init(struct ethos *e, const double *T, size_t N, int k, int q,
               double dens, enum ethos_grid g);
void ethos_free(struct ethos *e);

int ethos_extend_boundary(double ratio, const double *T, const double *S,
                          size_t N, double **Te, double **Se,
                          size_t *Ne);
int ethos_fit(const struct ethos *e, double *s, const double *T,
              const double *S, size_t N);

int ethos_plot_points(double *T, double *S, size_t N, const char *fname,
                      enum ethos_plot_format p);
int ethos_plot_spline(const struct ethos *e, const double *c,
                      const char *fname, enum ethos_plot_format p);

int ethos_upper_envelope(const struct ethos *e, double *m,
                         const double *s, double eps);
int ethos_frequency_from_simple_imf(const struct ethos *e, double *u);
int ethos_emd(const struct ethos *e, double *u, double *a, double *freq,
              double *s, double eps);

int ethos_characteristic(const struct ethos *e, double mu[static 3],
                         const double *a, const double *freq);

double ethos_norm_sup_bound_spline(const struct ethos *e,
                                   const double *c);
int ethos_relerror(const struct ethos *e, double *relerror, double *c,
                   double *c_analytic);

#endif /* ETHOS_H */
\end{lstlisting}
\subsection{ethos.c}
\begin{lstlisting}[language=C,label=lst:ethos-ethosc]
/* See LICENSE file for copyright and license details. */
#include <errno.h>
#include <gsl/gsl_bspline.h>
#include <gsl/gsl_linalg.h>
#include <gsl/gsl_matrix.h>
#include <gsl/gsl_multifit.h>
#include <gsl/gsl_statistics.h>
#include <gsl/gsl_vector.h>
#include <limits.h>
#include <math.h>
#include <stdarg.h>
#include <stddef.h>
#include <stdint.h>
#include <stdio.h>
#include <stdlib.h>
#include <string.h>
#include <sys/types.h>

#include "ethos.h"

#undef MIN
#define MIN(x,y)  ((x) < (y) ? (x) : (y))
#undef MAX
#define MAX(x,y)  ((x) > (y) ? (x) : (y))
#undef LEN
#define LEN(x) (sizeof (x) / sizeof *(x))

static void
warn(const char *fmt, ...)
{
	va_list ap;

	fprintf(stderr, "libethos: ");

	va_start(ap, fmt);
	vfprintf(stderr, fmt, ap);
	va_end(ap);

	if (fmt[0] && fmt[strlen(fmt) - 1] == ':') {
		fputc(' ', stderr);
		perror(NULL);
	} else {
		fputc('\n', stderr);
	}
}

/*
 * This is sqrt(SIZE_MAX+1), as s1*s2 <= SIZE_MAX
 * if both s1 < MUL_NO_OVERFLOW and s2 < MUL_NO_OVERFLOW
 */
#define MUL_NO_OVERFLOW	((size_t)1 << (sizeof(size_t) * 4))

static void *
_reallocarray(void *optr, size_t nmemb, size_t size)
{
	if ((nmemb >= MUL_NO_OVERFLOW || size >= MUL_NO_OVERFLOW) &&
	    nmemb > 0 && SIZE_MAX / nmemb < size) {
		errno = ENOMEM;
		return NULL;
	}
	return realloc(optr, size * nmemb);
}

static void *
ereallocarray(void *optr, size_t nmemb, size_t size)
{
	void *p;

	if (!(p = _reallocarray(optr, nmemb, size))) {
		warn("reallocarray: Out of memory");
		exit(1);
	}

	return p;
}

static void
util_array_vector_wrap(gsl_vector *v, const double *arr, size_t len)
{
	v->size   = len;
	v->stride = 1;
	v->data   = (double *)arr;
	v->block  = NULL;
	v->owner  = 0;
}

int
ethos_init(struct ethos *e, const double *T, size_t N, int k, int q,
          double dens, enum ethos_grid g)
{
	gsl_matrix *dB;
	gsl_vector *gridv;
	size_t i, j, m, nderiv;
	double t;

	/* input validation */
	if (dens <= 0 || dens > 1) {
		warn("ethos_init: Grid-density dens must be in (0,1]");
		return 1;
	}
	if (k < 1) {
		warn("ethos_init: Parameter k must be larger than 0");
		return 1;
	}
	if (q < 0) {
		warn("ethos_init: The parameter q must be positive");
		return 1;
	}
	if (N == 0) {
		warn("ethos_init: Empty time vector");
		return 1;
	}

	/* initialization */
	e->ngrid = (size_t)(dens * N);
	if (e->ngrid + k <= 2) {
		warn("ethos_init: Density too low for given parameters");
		return 1;
	}
	e->bw = gsl_bspline_alloc(k, e->ngrid);
	e->n = e->ngrid + k - 2;
	e->k = k;
	e->q = q;
	e->a = T[0];
	e->b = T[N - 1];

	/* breakpoint vector */
	e->grid = ereallocarray(NULL, e->ngrid, sizeof(*e->grid));
	switch (g) {
	case ETHOS_GRID_UNIFORM:
		for (i = 0; i < e->ngrid; i++) {
			e->grid[i] = e->a + (double)(e->b - e->a) /
			             (e->ngrid - 1) * i;
		}
		break;
	default:
		warn("ethos_init: Invalid grid style");
		free(e->grid);
		gsl_bspline_free(e->bw);
		return 1;
	}

	gridv = gsl_vector_alloc(e->ngrid);
	for (i = 0; i < e->ngrid; i++) {
		gsl_vector_set(gridv, i, e->grid[i]);
	}
	gsl_bspline_knots(gridv, e->bw);
	gsl_vector_free(gridv);

	/* build extended grid */
	e->nextgrid = e->ngrid + q * (e->ngrid - 1);
	e->extgrid = ereallocarray(NULL, e->nextgrid,
	                           sizeof(*e->extgrid));
	for (i = 0; i < (e->ngrid - 1); i++) {
		for (j = 0; j <= (size_t)q; j++) {
			e->extgrid[i * (q + 1) + j] = e->grid[i] + j *
				(e->grid[i + 1] - e->grid[i]) / (q + 1);
		}
	}
	e->extgrid[e->ngrid + q * (e->ngrid - 1) - 1] =
		e->grid[e->ngrid - 1];

	/* evaluate derivatives 0, 1, 2 on extended grid */
	nderiv = 2;
	dB = gsl_matrix_alloc(e->k, nderiv + 1);

	e->dB = ereallocarray(NULL, e->nextgrid, sizeof(*(e->dB)));
	for (m = 0; m < e->nextgrid; m++) {
		e->dB[m] = ereallocarray(NULL, e->k, sizeof(**(e->dB)));
		for (i = 0; i <= (size_t)e->k; i++) {
			e->dB[m][i] = ereallocarray(NULL, nderiv + 1,
			                            sizeof(***(e->dB)));
		}
	}
	e->istart = ereallocarray(NULL, e->nextgrid,
	                          sizeof(*(e->istart)));
	e->iend = ereallocarray(NULL, e->nextgrid, sizeof(*(e->iend)));

	for (m = 0; m < e->nextgrid; m++) {
		t = e->extgrid[m];

		/* B_i(e->extgrid[m]) for all i where B_i is non-zero */
		gsl_bspline_deriv_eval_nonzero(t, nderiv, dB,
		                               &(e->istart[m]),
		                               &(e->iend[m]), e->bw);

		/* store in ethos-struct */
		for (i = e->istart[m]; i <= e->iend[m]; i++) {
			for (j = 0; j <= nderiv; j++) {
				e->dB[m][i - e->istart[m]][j] =
					gsl_matrix_get(dB,
					               i - e->istart[m],
					               j);
			}
		}
	}
	gsl_matrix_free(dB);

	return 0;
}

void
ethos_free(struct ethos *e)
{
	size_t m, i;

	gsl_bspline_free(e->bw);
	e->bw = NULL;

	free(e->grid);
	e->grid = NULL;

	free(e->extgrid);
	e->extgrid = NULL;

	for (m = 0; m < e->nextgrid; m++) {
		for (i = 0; i <= (size_t)e->k; i++) {
			free(e->dB[m][i]);
			e->dB[m][i] = NULL;
		}
		free(e->dB[m]);
		e->dB[m] = NULL;
	}
	free(e->dB);
	e->dB = NULL;

	free(e->istart);
	e->istart = NULL;
	free(e->iend);
	e->iend = NULL;
}

int
ethos_extend_boundary(double ratio, const double *T, const double *S,
                      size_t N, double **Te, double **Se, size_t *Ne)
{
	double a, b, delta;
	size_t nmirr_l, nmirr_r, i, ind;

	if (ratio < 0 || ratio > 1) {
		warn("ethos_mirror_boundary: Ratio must be in [0,1]");
		return 1;
	}

	/* get extent of the array and range of mirror */
	a = T[0];
	b = T[N - 1];
	delta = (b - a) * ratio;
	for (nmirr_l = 0; nmirr_l < N; nmirr_l++) {
		if (T[nmirr_l] >= a + delta) {
			break;
		}
	}
	for (nmirr_r = 0; nmirr_r < N; nmirr_r++) {
		if (T[N - nmirr_r - 1] <= b - delta) {
			break;
		}
	}

	/* allocate new vectors */
	*Ne = N + (nmirr_l - 1) + (nmirr_r - 1);
	*Te = ereallocarray(NULL, *Ne, sizeof(**Te));
	*Se = ereallocarray(NULL, *Ne, sizeof(**Se));

	/* fill mirror vectors */
	for (i = 0; i < nmirr_l - 1; i++) {
		ind = nmirr_l - 1 - i;
		(*Te)[i] = a - (T[ind] - a);
		(*Se)[i] = S[ind];
	}
	for (i = 0; i < N; i++) {
		ind = i + (nmirr_l - 1);
		(*Te)[ind] = T[i];
		(*Se)[ind] = S[i];
	}
	for (i = 0; i < nmirr_r - 1; i++) {
		ind = i + (nmirr_l - 1) + N;
		(*Te)[ind] = b + (b - T[N - 2 - i]);
		(*Se)[ind] = S[N - 2 - i];
	}

	return 0;
}

int
ethos_fit(const struct ethos *e, double *s, const double *T,
         const double *S, size_t N)
{
	gsl_matrix *F, *cov;
	gsl_multifit_linear_workspace *mw;
	gsl_vector s_v, *y, *w, *B;
	size_t neq, i, start, end, j, m;
	double chisq, tss, Rsq;
	int ud;

	/* check S if it is contained in spline knot vector */
	if (T[0] < e->a || (N > 0 && T[N - 1] > e->b)) {
		warn("ethos_fit: input data exceeds spline bounds");
		return 1;
	}

	/* coefficient vector */
	util_array_vector_wrap(&s_v, s, e->n);

	/* is the system underdetermined? */
	ud = (e->n > N);
	neq = N + (ud ? e->nextgrid : 0);

	/* fit matrix */
	F = gsl_matrix_alloc(neq, e->n);

	B = gsl_vector_alloc(e->k);
	for (i = 0; i < N; i++) {
		/* compute B_j(T[i]) for all j where B_j is non-zero */
		gsl_bspline_eval_nonzero(T[i], B, &start, &end, e->bw);

		/* fill row i */
		for (j = start; j <= end; j++) {
			gsl_matrix_set(F, i, j,
			               gsl_vector_get(B, j - start));
		}
	}
	gsl_vector_free(B);

	if (ud) {
		/* fill smoothness terms */
		for (m = 0; m < e->nextgrid; m++) {
			for (i = e->istart[m]; i <= e->iend[m]; i++) {
				gsl_matrix_set(F, N + m, i,
				               e->dB[m][i -
				               e->istart[m]][2]);
			}
		}
	}

	/* run the fit */
	cov = gsl_matrix_alloc(e->n, e->n);

	/* weight vector and right hand side */
	y = gsl_vector_alloc(neq);
	w = gsl_vector_alloc(neq);
	for (i = 0; i < neq; i++) {
		gsl_vector_set(y, i, (i < N) ? S[i] : 0.0);
		gsl_vector_set(w, i, (i < N) ? 1.0 : 1E-12);
	}

	mw = gsl_multifit_linear_alloc(neq, e->n);
	gsl_multifit_wlinear(F, w, y, &s_v, cov, &chisq, mw);
	/*gsl_multifit_linear_free(mw); GSL-BUG */

	/* statistics (total sum of squares) (dof = neq - e->n) */
	tss = gsl_stats_wtss(y->data, 1, s, 1, e->n);
	Rsq = 1.0 - chisq / tss;
	(void)Rsq;

	/* cleanup */
	gsl_vector_free(y);
	gsl_vector_free(w);
	gsl_matrix_free(cov);
	/* gsl_matrix_free(F); GSL-BUG */

	return 0;
}

int
ethos_plot_points(double *T, double *S, size_t N, const char *fname,
                 enum ethos_plot_format p)
{
	FILE *fp;
	size_t i;

	if (fname) {
		if (!(fp = fopen(fname, "w"))) {
			warn("fopen '%s': %s\n", fname, strerror(errno));
			return 1;
		}
	} else {
		fp = stdout;
	}

	switch (p) {
	case ETHOS_PLOT_GRAPH:
		fprintf(fp, "#m=0,S=3\n");
		for (i = 0; i < N; i++) {
			fprintf(fp, "%g %g\n", T[i], S[i]);
		}
		fprintf(fp, "\n\n");
		break;
	case ETHOS_PLOT_CSV:
		fprintf(fp, "x,y\n");
		for (i = 0; i < N; i++) {
			fprintf(fp, "%g,%g\n", T[i], S[i]);
		}
		break;
	default:
		warn("ethos_plot_points: Invalid plot format");
		return 1;
	}

	if (fname && fclose(fp)) {
		warn("fclose: %s\n", strerror(errno));
		return 1;
	}

	return 0;
}

int
ethos_plot_spline(const struct ethos *e, const double *c,
                 const char *fname, enum ethos_plot_format p)
{
	FILE *fp;
	size_t m, i;
	double s;

	if (fname) {
		if (!(fp = fopen(fname, "w"))) {
			warn("fopen '%s': %s\n", fname, strerror(errno));
			return 1;
		}
	} else {
		fp = stdout;
	}

	switch (p) {
	case ETHOS_PLOT_GRAPH:
		fprintf(fp, "#m=1,S=0\n");
		for (m = 0; m < e->nextgrid; m++) {
			s = 0;
			for (i = e->istart[m]; i <= e->iend[m]; i++) {
				s += c[i] * e->dB[m][i -
				     e->istart[m]][0];
			}
			fprintf(fp, "%g %g\n", e->extgrid[m], s);
		}
		fprintf(fp, "\n\n");
		break;
	case ETHOS_PLOT_CSV:
		fprintf(fp, "x,y\n");
		for (m = 0; m < e->nextgrid; m++) {
			s = 0;
			for (i = e->istart[m]; i <= e->iend[m]; i++) {
				s += c[i] * e->dB[m][i -
				     e->istart[m]][0];
			}
			fprintf(fp, "%g,%g\n", e->extgrid[m], s);
		}
		break;
	default:
		warn("ethos_plot_spline: Invalid plot format");
		return 1;
	}

	if (fname && fclose(fp)) {
		warn("fclose '%s': %s\n", fname, strerror(errno));
		return 1;
	}

	return 0;
}

static int
ethos_upper_envelope_step(const struct ethos *e, double *r,
                          const double *s)
{
	struct {
		double *t;
		double *v;
		size_t len;
	} amp;
	size_t i, m;
	double s0, s0prev, s1, s1prev, s2, r1, r1prev;
	int r1_gt_s1_prev;

	amp.len = 0;
	amp.t = ereallocarray(NULL, amp.len, sizeof(*amp.t));
	amp.v = ereallocarray(NULL, amp.len, sizeof(*amp.v));

	/*
	 * identify spots where derivatives match up and curvature is
	 * negative
	 */
	for (m = 0; m < e->nextgrid; s0prev = s0, s1prev = s1,
             r1prev = r1, r1_gt_s1_prev = (r1 > s1), m++) {
		s0 = 0.0;
		s1 = 0.0;
		s2 = 0.0;
		r1 = 0.0;

		for (i = e->istart[m]; i <= e->iend[m]; i++) {
			s0 += s[i] * e->dB[m][i - e->istart[m]][0];
			s1 += s[i] * e->dB[m][i - e->istart[m]][1];
			s2 += s[i] * e->dB[m][i - e->istart[m]][2];
			r1 += r[i] * e->dB[m][i - e->istart[m]][1];
		}

		if (m > 0) {
			/*
			 * check if r1_gt_s1_prev != (r1 > s1)
			 * -> there was a match r1 = s1
			 */
			if (r1_gt_s1_prev != (r1 > s1) && s2 <= 0) {
				goto addpoint;
			}
		}
		continue;
addpoint:
		amp.len++;
		amp.t = ereallocarray(amp.t, amp.len, sizeof(*amp.t));
		amp.v = ereallocarray(amp.v, amp.len, sizeof(*amp.v));

		if (fabs(s1 - r1) > fabs(s1prev - r1prev)) {
			/* the previous gridpoint is closer to the max */
			amp.t[amp.len - 1] = e->extgrid[m - 1];
			amp.v[amp.len - 1] = s0prev;
		} else {
			amp.t[amp.len - 1] = e->extgrid[m];
			amp.v[amp.len - 1] = s0;
		}
	}

	if (ethos_fit(e, r, amp.t, amp.v, amp.len)) {
		warn("ethos_fit: Failed");
		free(amp.t);
		free(amp.v);
		return 1;
	}

	free(amp.t);
	free(amp.v);

	return 0;
}

int
ethos_upper_envelope(const struct ethos *e, double *m, const double *s,
                     double eps)
{
	size_t i, steps;
	double *mt, *diff;

	if (eps <= 0) {
		warn("ethos_upper_envelope: Tolerance must be strictly "
		     "positive");
		return 1;
	}

	mt = ereallocarray(NULL, e->n, sizeof(*mt));
	diff = ereallocarray(NULL, e->n, sizeof(*diff));

	/* set m to zero-function */
	for (i = 0; i < e->n; i++) {
		m[i] = 0.0;
	}

	steps = 0;
	do {
		/* set mt to m */
		for (i = 0; i < e->n; i++) {
			mt[i] = m[i];
		}

		if (ethos_upper_envelope_step(e, m, s)) {
			return 1;
		}

		/* calculate difference */
		for (i = 0; i < e->n; i++) {
			diff[i] = m[i] - mt[i];
		}

		/* premature termination after 10 steps */
		if (++steps >= 10) {
			warn("ethos_upper_envelope: could not satisfy "
			     "bound");
			break;
		}
	} while (ethos_norm_sup_bound_spline(e, diff) > eps);

	free(mt);
	free(diff);

	return 0;
}

static int
ethos_divide_spline(const struct ethos *e, double *a, const double *b)
{
	size_t m, i;
	double *eval, a_ev, b_ev;

	eval = ereallocarray(NULL, e->nextgrid, sizeof(*eval));

	for (m = 0; m < e->nextgrid; m++) {
		a_ev = 0.0;
		b_ev = 0.0;

		for (i = e->istart[m]; i <= e->iend[m]; i++) {
			a_ev += a[i] * e->dB[m][i - e->istart[m]][0];
			b_ev += b[i] * e->dB[m][i - e->istart[m]][0];
		}

		eval[m] = a_ev / b_ev;
	}

	return (ethos_fit(e, a, e->extgrid, eval, e->nextgrid) ? 1 : 0);
}

int
ethos_frequency_from_simple_imf(const struct ethos *e, double *u)
{
	gsl_matrix *F, *cov;
	gsl_multifit_linear_workspace *mw;
	gsl_vector *freq, *y;
	size_t i, m;
	double tmp1, chisq, tss, Rsq, *freqeval, om;

	/* coefficient vector */
	freq = gsl_vector_alloc(e->n);

	/* fit matrix */
	F = gsl_matrix_alloc(e->nextgrid, e->n);

	for (m = 0; m < e->nextgrid; m++) {
		for (i = e->istart[m]; i <= e->iend[m]; i++) {
			/*
			 * Omega  s''
			 */

			/* preprocess sum_{i=0}^{n-1} s_j B_j'' */
			tmp1 = 0.0;
			for (i = e->istart[m]; i <= e->iend[m]; i++) {
				tmp1 += u[i] * e->dB[m][i -
				        e->istart[m]][2];
			}

			/* initialize entry */
			for (i = e->istart[m]; i <= e->iend[m]; i++) {
				gsl_matrix_set(F, m, i, tmp1 *
				               e->dB[m][i -
				               e->istart[m]][0]);
			}

			/*
			 * 1/2 Omega' s'
			 */

			/* preprocess sum_{i=0}^{n-1} s_j B_j' */
			tmp1 = 0.0;
			for (i = e->istart[m]; i <= e->iend[m]; i++) {
				tmp1 += u[i] * e->dB[m][i -
				        e->istart[m]][1];
			}

			for (i = e->istart[m]; i <= e->iend[m]; i++) {
				gsl_matrix_set(F, m, i,
				               gsl_matrix_get(F, m, i) +
				               0.5 * tmp1 *
				               e->dB[m][i -
				               e->istart[m]][1]);
			}
		}
	}

	/* run the fit */
	cov = gsl_matrix_alloc(e->n, e->n);

	/* right hand side */
	y = gsl_vector_alloc(e->nextgrid);
	for (m = 0; m < e->nextgrid; m++) {
		tmp1 = 0.0;
		for (i = e->istart[m]; i <= e->iend[m]; i++) {
			tmp1 += -u[i] * e->dB[m][i - e->istart[m]][0];
		}
		gsl_vector_set(y, m, tmp1);
	}

	mw = gsl_multifit_linear_alloc(e->nextgrid, e->n);

	gsl_multifit_linear(F, y, freq, cov, &chisq, mw);
	/*gsl_multifit_linear_free(mw); GSL-BUG */

	/* statistics (total sum of squares) (dof = e->nxg + 2 - e->n) */
	tss = gsl_stats_wtss(y->data, 1, freq->data, 1, e->n);
	Rsq = 1.0 - chisq / tss;
	(void)Rsq;

	/* fill output */
	for (i = 0; i < e->n; i++) {
		u[i] = gsl_vector_get(freq, i);
	}

	/* convert omega to frequency */
	freqeval = ereallocarray(NULL, e->nextgrid, sizeof(*freqeval));

	for (m = 0; m < e->nextgrid; m++) {
		om = 0.0;

		for (i = e->istart[m]; i <= e->iend[m]; i++) {
			om += u[i] * e->dB[m][i - e->istart[m]][0];
		}

		freqeval[m] = 1 / sqrt(fabs(om));
	}

	if (ethos_fit(e, u, e->extgrid, freqeval, e->nextgrid)) {
		return 1;
	}
	free(freqeval);

	/* cleanup */
	gsl_vector_free(freq);
	gsl_vector_free(y);
	gsl_matrix_free(cov);
	/* gsl_matrix_free(F); GSL-BUG */

	return 0;
}

int
ethos_emd(const struct ethos *e, double *u, double *a, double *freq,
          double *s, double eps)
{
	double *upp, *low, r;
	size_t i;

	/*
	 * upper and lower envelopes
	 */
	upp = ereallocarray(NULL, e->n, sizeof(*upp));
	low = ereallocarray(NULL, e->n, sizeof(*low));

	if (ethos_upper_envelope(e, upp, s, eps)) {
		warn("ethos_shift: Failed");
		free(upp);
		free(low);
		return 1;
	}

	/* flip signal */
	for (i = 0; i < e->n; i++) {
		s[i] = -s[i];
	}

	if (ethos_upper_envelope(e, low, s, eps)) {
		warn("ethos_shift: Failed");
		free(upp);
		free(low);
		return 1;
	}

	/* un-flip signal and lower envelope */
	for (i = 0; i < e->n; i++) {
		s[i] = -s[i];
		low[i] = -low[i];
	}

	/* 
	 * IMF and residual (stored in s)
	 */
	for (i = 0; i < e->n; i++) {
		r = (upp[i] + low[i]) / 2.0;
		u[i] = s[i] - r;
		s[i] = r;
	}

	/*
	 * instantaneous amplitude
	 */
	for (i = 0; i < e->n; i++) {
		a[i] = upp[i] - s[i];
	}

	/* cleanup */
	free(low);
	free(upp);

	/*
	 * instantaneous frequency
	 */

	/* calculate simple imf from u and a */
	for (i = 0; i < e->n; i++) {
		freq[i] = u[i];
	}
	if (ethos_divide_spline(e, freq, a)) {
		return 1;
	}

	/* extract frequency */
	if (ethos_frequency_from_simple_imf(e, freq)) {
		return 1;
	}

	return 0;
}

int
ethos_characteristic(const struct ethos *e, double mu[static 3],
                     const double *a, const double *freq)
{
	size_t m, i;
	double a1, freq0, freq1;

	mu[0] = INFINITY;
	mu[1] = mu[2] = DBL_EPSILON;

	for (m = 0; m < e->nextgrid; m++) {
		/* evaluate a', freq, freq' */
		a1 = freq0 = freq1 = 0.0;
		for (i = e->istart[m]; i <= e->iend[m]; i++) {
			a1 += a[i] * e->dB[m][i - e->istart[m]][1];
			freq0 += freq[i] * e->dB[m][i - e->istart[m]][0];
			freq1 += freq[i] * e->dB[m][i - e->istart[m]][1];
		}

		/* take the absolute value */
		a1 = fabs(a1);
		freq0 = fabs(freq0);
		freq1 = fabs(freq1);

		/* mu_0 */
		if (freq0 < mu[0]) {
			mu[0] = freq0;
		}

		/* mu_1 */
		if (a1 / freq0 > mu[1]) {
			mu[1] = a1 / freq0;
		}

		/* mu_2 */
		if (freq1 / freq0 > mu[2]) {
			mu[2] = freq1 / freq0;
		}
	}

	return 0;
}

double
ethos_norm_sup_bound_spline(const struct ethos *e, const double *c)
{
	size_t i;
	double res;

	res = -INFINITY;

	for (i = 0; i < e->n; i++) {
		if (fabs(c[i]) > res) {
			res = fabs(c[i]);
		}
	}

	return res;
}

int
ethos_relerror(const struct ethos *e, double *relerror, double *c,
               double *c_analytic)
{
	size_t m, i;
	int ret;
	double *relerror_eval, c0, c_an0;

	relerror_eval = ereallocarray(NULL, e->nextgrid,
	                              sizeof(*relerror_eval));

	for (m = 0; m < e->nextgrid; m++) {
		c0 = c_an0 = 0.0;

		for (i = e->istart[m]; i <= e->iend[m]; i++) {
			c0 += c[i] * e->dB[m][i - e->istart[m]][0];
			c_an0 += c_analytic[i] * e->dB[m][i -
			         e->istart[m]][0];
		}

		relerror_eval[m] = fabs((c0 - c_an0) / c_an0);
	}

	ret = ethos_fit(e, relerror, e->extgrid, relerror_eval,
	                e->nextgrid);

	free(relerror_eval);

	return ret;
}
\end{lstlisting}
\subsection{config.mk}
\begin{lstlisting}[language=make,label=lst:ethos-configmk]
# Customize below to fit your system

# paths
PREFIX = /usr/local
MANPREFIX = $(PREFIX)/share/man

# flags
CPPFLAGS = -D_DEFAULT_SOURCE
CFLAGS   = -std=c99 -pedantic -Wall -Wextra -Os
ARFLAGS  = rcs

# compiler, linker and archiver
CC = cc
AR = ar
\end{lstlisting}
\subsection{Makefile}
\begin{lstlisting}[language=make,label=lst:ethos-makefile]
# See LICENSE file for copyright and license details
# ETHOS - EMD Toolbox using Hybrid Operator-Based Methods and B-Splines
.POSIX:

include config.mk

all: libethos.a

ethos.o: ethos.c config.mk ethos.h

libethos.a: ethos.o
	$(AR) $(ARFLAGS) $@ $^

.c.o:
	$(CC) -c $(CPPFLAGS) $(CFLAGS) $<

clean:
	rm -f libethos.a ethos.o
\end{lstlisting}
\section{Examples}\label{sec:code-examples}
These programs expect libethos.a and ethos.h somewhere in
the environment. Set the variable \texttt{ETHOS} in config.mk
(see Listing~\ref{lst:examples-configmk}) to point to the
directory containing both files. This in turn will properly
set the \texttt{I}- and \texttt{L}-flags in the preprocessor flags
\texttt{CPPFLAGS} and the linker flags \texttt{LDFLAGS} respectively.
\subsection{emd.c}
\begin{lstlisting}[language=C,label=lst:examples-emdc]
/* See LICENSE file for copyright and license details. */
#include <ethos.h>
#include <gsl/gsl_matrix.h>
#include <limits.h>
#include <math.h>
#include <string.h>

#include "util.h"

double
a_0_0(double t)
{
	return t + 1;
}

double
phi_0_0(double t)
{
	return (15 * t + 21) * M_PI * t;
}

double
freq_0_0(double t)
{
	return 15 * M_PI * t + (15 * t + 21) * M_PI;
}

double
a_0_1(double t)
{
	return 3 * t + 1;
}

double
phi_0_1(double t)
{
	return 5 * M_PI * t;
}

double
freq_0_1(double t)
{
	(void)t;

	return 5 * M_PI;
}
	
double
s_0(double t)
{
	return a_0_0(t) * cos(phi_0_0(t)) +
	       a_0_1(t) * cos(phi_0_1(t)) +
	       20 * (t + 1);
}

double
a_1_0(double t)
{
	return 1 + pow(t, 3.0);
}

double
phi_1_0(double t)
{
	return 80 * M_PI * t;
}

double
freq_1_0(double t)
{
	(void)t;

	return 80 * M_PI;
}

double
a_1_1(double t)
{
	return 1 + 10 * pow(t, 2.0);
}

double
phi_1_1(double t)
{
	return 15 * M_PI * t;
}

double
freq_1_1(double t)
{
	(void)t;

	return 15 * M_PI;
}
	
double
s_1(double t)
{
	return a_1_0(t) * cos(phi_1_0(t)) +
	       a_1_1(t) * cos(phi_1_1(t)) +
	       (20 * pow(t, 3.0) + 3);
}

struct imf {
	double (*a)(double);
	double (*phi)(double);
	double (*freq)(double);
};

struct imf imfs_0[] = {
	{ a_0_0, phi_0_0, freq_0_0 },	
	{ a_0_1, phi_0_1, freq_0_1 },
};

struct imf imfs_1[] = {
	{ a_1_0, phi_1_0, freq_1_0 },	
	{ a_1_1, phi_1_1, freq_1_1 },
};

struct example {
	char *name;
	double a;
	double b;
	size_t N;
	int q;
	double dens;
	double (*s)(double);
	size_t ncomp;
	struct imf *comp;
	double eps;
} examples[] = {
	{
		.name  = "emd.data/0",
		.a     = 0.0,
		.b     = 1.0,
		.N     = 596,
		.q     = 4,
		.dens  = 0.3,
		.s     = s_0,
		.ncomp = 2,
          	.comp  = imfs_0,
		.eps   = 0.01,
	},
	{
		.name  = "emd.data/1",
		.a     = 0.0,
		.b     = 1.0,
		.N     = 596,
		.q     = 4,
		.dens  = 0.3,
		.s     = s_1,
		.ncomp = 2,
          	.comp  = imfs_1,
		.eps   = 0.01,
	},

};

int
process(struct example *ex)
{
	struct comp_data {
		double *U_an;
		double *A_an;
		double *Freq_an;
		double *u_an;
		double *a_an;
		double *freq_an;
		double *u;
		double *a;
		double *freq;
		double *u_abserr;
		double *a_relerr;
		double *freq_relerr;
	} *cd;
	struct ethos *e;
	double *T, *S, *s, *s_an, *tmp, mu[3];
	size_t i, k;

	e = ereallocarray(NULL, 1, sizeof(struct ethos));

	/* generate analytical input points */
	T = ereallocarray(NULL, ex->N, sizeof(*T));
	S = ereallocarray(NULL, ex->N, sizeof(*S));
	cd = ereallocarray(NULL, ex->ncomp, sizeof(*cd));
	for (i = 0; i < ex->ncomp; i++) {
		cd[i].U_an    = ereallocarray(NULL, ex->N,
		                              sizeof(*(cd[i].U_an)));
		cd[i].A_an    = ereallocarray(NULL, ex->N,
		                              sizeof(*(cd[i].A_an)));
		cd[i].Freq_an = ereallocarray(NULL, ex->N,
		                              sizeof(*(cd[i].Freq_an)));
	}
	for (i = 0; i < ex->N; i++) {
		T[i] = ex->a + (double)(ex->b - ex->a) / (ex->N - 1) * i;
		S[i] = ex->s(T[i]);
		for (k = 0; k < ex->ncomp; k++) {
			cd[k].U_an[i]   = ex->comp[k].a(T[i]) *
			                  cos(ex->comp[k].phi(T[i]));
			cd[k].A_an[i]   = ex->comp[k].a(T[i]);
			cd[k].Freq_an[i] = ex->comp[k].freq(T[i]);
		}
	}

	/* initialize ethos */
	if (ethos_init(e, T, ex->N, 4, ex->q, ex->dens,
	    ETHOS_GRID_UNIFORM)) {
		warn("ethos_init: Failed");
		return 1;
	}

	fprintf(stderr, "%s: n = %zu\n", ex->name, e->n);

	/* fit */
	s = ereallocarray(NULL, e->n, sizeof*(s));
	if (ethos_fit(e, s, T, S, ex->N)) {
		warn("ethos_fit: Failed");
		return 1;
	}
	plot_spline_fmt(e, s, ETHOS_PLOT_CSV, "%s-s.csv", ex->name);

	s_an = ereallocarray(NULL, e->n, sizeof(*s_an));
	for (i = 0; i < e->n; i++) {
		s_an[i] = s[i];
	}

	for (k = 0; k < ex->ncomp; k++) {
		/* fit analytical solutions */
		cd[k].u_an = ereallocarray(NULL, e->n,
		                           sizeof(*(cd[k].u_an)));
		cd[k].a_an = ereallocarray(NULL, e->n,
		                           sizeof(*(cd[k].a_an)));
		cd[k].freq_an = ereallocarray(NULL, e->n,
		                              sizeof(*(cd[k].freq_an)));

		if (ethos_fit(e, cd[k].u_an, T, cd[k].U_an, ex->N) ||
		    ethos_fit(e, cd[k].a_an, T, cd[k].A_an, ex->N) ||
		    ethos_fit(e, cd[k].freq_an, T, cd[k].Freq_an, ex->N)
		   ) {
			warn("ethos_fit: Failed");
			return 1;
		}

		/* plot analytical solutions */
		plot_spline_fmt(e, cd[k].u_an, ETHOS_PLOT_CSV,
		                "%s-u-%zu-analytic.csv", ex->name, k);
		plot_spline_fmt(e, cd[k].a_an, ETHOS_PLOT_CSV,
		                "%s-a-%zu-analytic.csv", ex->name, k);
		plot_spline_fmt(e, cd[k].freq_an, ETHOS_PLOT_CSV,
		                "%s-freq-%zu-analytic.csv", ex->name, k);

		/* analytical IMF characteristic */
		ethos_characteristic(e, mu, cd[k].a_an, cd[k].freq_an);
		fprintf(stderr, "%s: IMF-%zu: characteristic "
		        "(%e, %e, %e)\n", ex->name, k, mu[0], mu[1],
		        mu[2]);

		cd[k].u = ereallocarray(NULL, e->n, sizeof(*(cd[k].u)));
		cd[k].a = ereallocarray(NULL, e->n, sizeof(*(cd[k].a)));
		cd[k].freq = ereallocarray(NULL, e->n,
		                           sizeof(*(cd[k].freq)));
		if (ethos_emd(e, cd[k].u, cd[k].a, cd[k].freq, s,
		    ex->eps)) {
			return 1;
		}

		/* plot solutions */
		plot_spline_fmt(e, cd[k].u, ETHOS_PLOT_CSV,
		                "%s-u-%zu.csv", ex->name, k);
		plot_spline_fmt(e, cd[k].a, ETHOS_PLOT_CSV,
		                "%s-a-%zu.csv", ex->name, k);
		plot_spline_fmt(e, cd[k].freq, ETHOS_PLOT_CSV,
		                "%s-freq-%zu.csv", ex->name, k);

		/* plot error */
		cd[k].u_abserr = ereallocarray(NULL, e->n,
		                               sizeof(*(cd[k].
		                               u_abserr)));
		cd[k].a_relerr = ereallocarray(NULL, e->n,
		                               sizeof(*(cd[k].
		                               a_relerr)));
		cd[k].freq_relerr = ereallocarray(NULL, e->n,
		                               sizeof(*(cd[k].
		                               freq_relerr)));

		for (i = 0; i < e->n; i++) {
			cd[k].u_abserr[i] = fabs(cd[k].u[i] -
			                         cd[k].u_an[i]);
		}
		ethos_relerror(e, cd[k].a_relerr, cd[k].a, cd[k].a_an);
		ethos_relerror(e, cd[k].freq_relerr, cd[k].freq,
		               cd[k].freq_an);

		plot_spline_fmt(e, cd[k].u_abserr, ETHOS_PLOT_CSV,
		                "%s-u-%zu-abserr.csv", ex->name, k);
		plot_spline_fmt(e, cd[k].a_relerr, ETHOS_PLOT_CSV,
		                "%s-a-%zu-relerr.csv", ex->name, k);
		plot_spline_fmt(e, cd[k].freq_relerr, ETHOS_PLOT_CSV,
		                "%s-freq-%zu-relerr.csv", ex->name, k);

		/* subtract u from s in the analytical solution */
		for (i = 0; i < e->n; i++) {
			s_an[i] -= cd[k].u_an[i];
		}
		plot_spline_fmt(e, s_an, ETHOS_PLOT_CSV,
		                "%s-r-%zu-analytic.csv", ex->name,
		                k + 1);		
		plot_spline_fmt(e, s, ETHOS_PLOT_CSV,
		                "%s-r-%zu.csv", ex->name, k + 1);		

		tmp = ereallocarray(NULL, e->n, sizeof(*tmp));
		ethos_relerror(e, tmp, s, s_an);
		plot_spline_fmt(e, tmp, ETHOS_PLOT_CSV,
		                "%s-r-%zu-relerr.csv", ex->name, k + 1);
		free(tmp);
	}

	/* cleanup */
	free(T);
	free(S);
	free(s);

	return 0;
}

int
main(int argc, char *argv[])
{
	int ret;
	size_t i;

	ret = 0;

	for (i = 0; i < LEN(examples); i++) {
		if (argc <= 1 || !strcmp(examples[i].name, argv[1])) {
			ret |= process(&examples[i]);

			if (argv[1]) {
				break;
			}
		}
	}
	
	return ret;
}
\end{lstlisting}
\subsection{envelope.c}
\begin{lstlisting}[language=C,label=lst:examples-envelopec]
/* See LICENSE file for copyright and license details. */
#include <ethos.h>
#include <gsl/gsl_matrix.h>
#include <limits.h>
#include <math.h>
#include <string.h>

#include "util.h"

double
s_x(double t)
{
	return 20 * t + (2 + cos(4 * M_PI * t)) * cos(13 * M_PI * t);
}

double
m_x(double t)
{
	return 20 * t + (2 + cos(4 * M_PI * t));
}

double
s_0(double t)
{
	return 40 * t + (20 + 10 * cos(5 * M_PI * t)) *
	       cos(25 * M_PI * t);
}

double
m_0(double t)
{
	return 40 * t + (20 + 10 * cos(5 * M_PI * t));
}

double
s_1(double t)
{
	return 1.0 / 16.0 * (pow(t, 2.0) + 2) *
	       cos(M_PI * sin(8 * t) + M_PI);
}

double
m_1(double t)
{
	return 1.0 / 16.0 * (pow(t, 2.0) + 2);
}

struct example {
	char *name;
	double a;
	double b;
	size_t N;
	int q;
	double dens;
	double (*s)(double);
	double (*m)(double);
	double eps;
} examples[] = {
	{
		.name = "envelope.data/x",
		.a    = 0.0,
		.b    = 1.0,
		.N    = 596,
		.q    = 4,
		.dens = 0.3,
		.s    = s_x,
		.m    = m_x,
		.eps  = 0.01,
	},
	{
		.name = "envelope.data/0",
		.a    = 0.0,
		.b    = 1.0,
		.N    = 596,
		.q    = 4,
		.dens = 0.3,
		.s    = s_0,
		.m    = m_0,
		.eps  = 0.01,
	},
	{
		.name = "envelope.data/1",
		.a    = -4.0,
		.b    = 4.0,
		.N    = 596,
		.q    = 4,
		.dens = 0.3,
		.s    = s_1,
		.m    = m_1,
		.eps  = 0.10,
	},
};

int
process(struct example *ex)
{
	struct ethos *e;
	double *T, *S, *A, *s, *a_an, *a, *a_class, *err,
	       *relerror;
	size_t i;

	e = reallocarray(NULL, 1, sizeof(struct ethos));

	T = ereallocarray(NULL, ex->N, sizeof(*T));
	S = ereallocarray(NULL, ex->N, sizeof(*S));
	A = ereallocarray(NULL, ex->N, sizeof(*A));

	/* generate input points */
	for (i = 0; i < ex->N; i++) {
		T[i] = ex->a + (double)(ex->b - ex->a) / (ex->N - 1) * i;
		S[i] = ex->s(T[i]);
		A[i] = ex->m(T[i]);
	}

	plot_points_fmt(T, S, ex->N, ETHOS_PLOT_CSV,
	                "%s-s-points.csv", ex->name);

	/* initialize ethos */
	if (ethos_init(e, T, ex->N, 4, ex->q, ex->dens,
	    ETHOS_GRID_UNIFORM)) {
		warn("ethos_init: Failed");
		return 1;
	}

	fprintf(stderr, "%s: n = %zu\n", ex->name, e->n);

	/* fit */
	s = ereallocarray(NULL, e->n, sizeof(*s));
	if (ethos_fit(e, s, T, S, ex->N)) {
		warn("ethos_fit: Failed");
		return 1;
	}
	plot_spline_fmt(e, s, ETHOS_PLOT_CSV,
	                "%s-s.csv", ex->name);

	a_an = ereallocarray(NULL, e->n, sizeof(*a_an));
	if (ethos_fit(e, a_an, T, A, ex->N)) {
		warn("ethos_fit: Failed");
		return 1;
	}
	plot_spline_fmt(e, a_an, ETHOS_PLOT_CSV,
	                "%s-a-analytic.csv", ex->name);

	double *r = ereallocarray(NULL, e->n, sizeof(*r));
	for (i = 0; i < e->n; i++) {
		r[i] = 0.0;
	}
	a_class = ereallocarray(NULL, e->n, sizeof(*a_class));
	if (ethos_upper_envelope(e, a_class, s, INFINITY)) {
		warn("ethos_sift: Failed");
		return 1;
	}
	plot_spline_fmt(e, a_class, ETHOS_PLOT_CSV,
	                "%s-a-classic.csv", ex->name);

	a = ereallocarray(NULL, e->n, sizeof(*a));
	if (ethos_upper_envelope(e, a, s, ex->eps)) {
		warn("ethos_sift: Failed");
		return 1;
	}
	plot_spline_fmt(e, a, ETHOS_PLOT_CSV,
	                "%s-a.csv", ex->name);

	/* error */
	err = ereallocarray(NULL, e->n, sizeof(*err));
	for (i = 0; i < e->n; i++) {
		err[i] = a[i] - a_an[i];
	}
	plot_spline_fmt(e, err, ETHOS_PLOT_CSV,
	                "%s-err.csv", ex->name);

	relerror = ereallocarray(NULL, e->n, sizeof(*relerror));

	ethos_relerror(e, relerror, a, a_an);
	plot_spline_fmt(e, relerror, ETHOS_PLOT_CSV,
	                "%s-relerr.csv", ex->name);

	/* classic error */
	for (i = 0; i < e->n; i++) {
		err[i] = a_class[i] - a_an[i];
	}
	plot_spline_fmt(e, err, ETHOS_PLOT_CSV,
	                "%s-err-classic.csv", ex->name);

	ethos_relerror(e, relerror, a_class, a_an);
	plot_spline_fmt(e, relerror, ETHOS_PLOT_CSV,
	                "%s-relerr-classic.csv", ex->name);
	free(relerror);


	/* cleanup */
	free(T);
	free(S);
	free(A);
	free(s);
	free(a_an);
	free(a);
	free(err);

	return 0;
}

int
main(int argc, char *argv[])
{
	int ret;
	size_t i;

	ret = 0;

	for (i = 0; i < LEN(examples); i++) {
		if (argc <= 1 || !strcmp(examples[i].name, argv[1])) {
			ret |= process(&examples[i]);

			if (argv[1]) {
				break;
			}
		}
	}
	
	return ret;
}
\end{lstlisting}
\subsection{regop.c}
\begin{lstlisting}[language=C,label=lst:examples-regopc]
/* See LICENSE file for copyright and license details. */
#include <ethos.h>
#include <gsl/gsl_matrix.h>
#include <limits.h>
#include <math.h>
#include <string.h>

#include "util.h"

double
u_0(double t)
{
	return cos(40 * t);
}

double
freq_0(double t)
{
	(void)t;

	return 40;
}

double
u_1(double t)
{
	return cos(3 * sin(3 * M_PI * t) + 16 * M_PI * t);
}

double
freq_1(double t)
{
	return 3 * 3 * M_PI * cos(3 * M_PI * t) + 16 * M_PI;
}

double
u_2(double t)
{
	return cos(40 * t + 100.0 / 90.0 * log(1 + exp(90 * (t - 0.5))));
}

double
freq_2(double t)
{
	return 40 + 100.0 / (1 + exp(-90 * (t - 0.5)));
}

struct example {
	char *name;
	double a;
	double b;
	size_t N;
	int q;
	double dens;
	double (*u)(double);
	double (*freq)(double);
} examples[] = {
	{
		.name = "regop.data/0",
		.a    = 0.0,
		.b    = 1.0,
		.N    = 596,
		.q    = 4,
		.dens = 0.3,
		.u    = u_0,
		.freq = freq_0,
	},
	{
		.name = "regop.data/1",
		.a    = 0.0,
		.b    = 1.0,
		.N    = 596,
		.q    = 4,
		.dens = 0.3,
		.u    = u_1,
		.freq = freq_1,
	},
	{
		.name = "regop.data/2",
		.a    = 0.0,
		.b    = 1.0,
		.N    = 596,
		.q    = 4,
		.dens = 0.3,
		.u    = u_2,
		.freq = freq_2,
	},
};

int
process(struct example *ex)
{
	struct ethos *e;
	double *T, *U, *F, *u, *f_an, *f, *err, *tmp;
	size_t i;

	e = reallocarray(NULL, 1, sizeof(struct ethos));

	T = ereallocarray(NULL, ex->N, sizeof(*T));
	U = ereallocarray(NULL, ex->N, sizeof(*U));
	F = ereallocarray(NULL, ex->N, sizeof(*F));

	/* generate input points */
	for (i = 0; i < ex->N; i++) {
		T[i] = (double)(ex->b - ex->a) / (ex->N - 1) * i;
		U[i] = ex->u(T[i]);
		F[i] = ex->freq(T[i]);
	}
	plot_points_fmt(T, U, ex->N, ETHOS_PLOT_CSV, "%s-u-points.csv",
	                ex->name);

	/* initialize ethos */
	if (ethos_init(e, T, ex->N, 4, ex->q, ex->dens,
	    ETHOS_GRID_UNIFORM)) {
		warn("ethos_init: Failed");
		return 1;
	}

	fprintf(stderr, "%s: n = %zu\n", ex->name, e->n);

	/* fit */
	u = ereallocarray(NULL, e->n, sizeof(*u));
	if (ethos_fit(e, u, T, U, ex->N)) {
		warn("ethos_fit: Failed");
		return 1;
	}
	f_an = ereallocarray(NULL, e->n, sizeof(*f_an));
	if (ethos_fit(e, f_an, T, F, ex->N)) {
		warn("ethos_fit: Failed");
		return 1;
	}
	plot_spline_fmt(e, u, ETHOS_PLOT_CSV, "%s-u.csv", ex->name);

	/* calculate instantaneous frequency */
	f = ereallocarray(NULL, e->n, sizeof(*f));
	for (i = 0; i < e->n; i++) {
		f[i] = u[i];
	}
	if (ethos_frequency_from_simple_imf(e, f)) {
		return 1;
	}

	err = ereallocarray(NULL, e->n, sizeof(*err));
	for (i = 0; i < e->n; i++) {
		err[i] = f[i] - f_an[i];
	}

	plot_spline_fmt(e, err, ETHOS_PLOT_CSV, "%s-abserr.csv",
	                ex->name);
	plot_spline_fmt(e, f, ETHOS_PLOT_CSV, "%s-freq.csv", ex->name);
	plot_spline_fmt(e, f_an, ETHOS_PLOT_CSV, "%s-freq-analytic.csv",
	                ex->name);

	tmp = ereallocarray(NULL, e->n, sizeof(*tmp));
	ethos_relerror(e, tmp, f, f_an);
	plot_spline_fmt(e, tmp, ETHOS_PLOT_CSV, "%s-relerr.csv",
	                ex->name);
	free(tmp);

	/* cleanup */
	free(T);
	free(U);
	free(F);
	free(u);
	free(f_an);
	free(f);
	free(err);

	return 0;
}

int
main(int argc, char *argv[])
{
	int ret;
	size_t i;

	ret = 0;

	for (i = 0; i < LEN(examples); i++) {
		if (argc <= 1 || !strcmp(examples[i].name, argv[1])) {
			ret |= process(&examples[i]);

			if (argv[1]) {
				break;
			}
		}
	}
	
	return ret;
}
\end{lstlisting}
\subsection{util.h}
\begin{lstlisting}[language=C,label=lst:examples-utilh]
/* See LICENSE file for copyright and license details. */
#ifndef UTIL_H
#define UTIL_H

#include <gsl/gsl_vector.h>
#include <stddef.h>
#include <time.h>

#undef MIN
#define MIN(x,y)  ((x) < (y) ? (x) : (y))
#undef MAX
#define MAX(x,y)  ((x) > (y) ? (x) : (y))
#undef LEN
#define LEN(x) (sizeof (x) / sizeof *(x))

void warn(const char *, ...);
void die(const char *, ...);

long long strtonum(const char *, long long, long long, const char **);

int esnprintf(char *, size_t, const char *, ...);
void *ereallocarray(void *, size_t, size_t);

void util_array_vector_wrap(const double *, size_t, gsl_vector *);

int plot_spline_fmt(const struct ethos *, const double *,
                    enum ethos_plot_format, const char *, ...);
int plot_points_fmt(const double *, const double *, size_t,
                    enum ethos_plot_format, const char *, ...);

#endif /* UTIL_H */
\end{lstlisting}
\subsection{util.c}
\begin{lstlisting}[language=C,label=lst:examples-utilc]
/* See LICENSE file for copyright and license details. */
#include <errno.h>
#include <gsl/gsl_vector.h>
#include <limits.h>
#include <stdarg.h>
#include <stddef.h>
#include <stdint.h>
#include <stdio.h>
#include <stdlib.h>
#include <string.h>
#include <sys/types.h>
#include <time.h>

#include <ethos.h>
#include "util.h"

static char *argv0;

static void
verr(const char *fmt, va_list ap)
{
	if (argv0 && strncmp(fmt, "usage", sizeof("usage") - 1)) {
		fprintf(stderr, "%s: ", argv0);
	}

	vfprintf(stderr, fmt, ap);

	if (fmt[0] && fmt[strlen(fmt) - 1] == ':') {
		fputc(' ', stderr);
		perror(NULL);
	} else {
		fputc('\n', stderr);
	}
}

void
warn(const char *fmt, ...)
{
	va_list ap;

	va_start(ap, fmt);
	verr(fmt, ap);
	va_end(ap);
}

void
die(const char *fmt, ...)
{
	va_list ap;

	va_start(ap, fmt);
	verr(fmt, ap);
	va_end(ap);

	exit(1);
}

#define	INVALID  1
#define	TOOSMALL 2
#define	TOOLARGE 3

long long
strtonum(const char *numstr, long long minval, long long maxval,
         const char **errstrp)
{
	long long ll = 0;
	int error = 0;
	char *ep;
	struct errval {
		const char *errstr;
		int err;
	} ev[4] = {
		{ NULL,		0 },
		{ "invalid",	EINVAL },
		{ "too small",	ERANGE },
		{ "too large",	ERANGE },
	};

	ev[0].err = errno;
	errno = 0;
	if (minval > maxval) {
		error = INVALID;
	} else {
		ll = strtoll(numstr, &ep, 10);
		if (numstr == ep || *ep != '\0')
			error = INVALID;
		else if ((ll == LLONG_MIN && errno == ERANGE) ||
		          ll < minval)
			error = TOOSMALL;
		else if ((ll == LLONG_MAX && errno == ERANGE) ||
		          ll > maxval)
			error = TOOLARGE;
	}
	if (errstrp != NULL)
		*errstrp = ev[error].errstr;
	errno = ev[error].err;
	if (error)
		ll = 0;

	return ll;
}

/*
 * This is sqrt(SIZE_MAX+1), as s1*s2 <= SIZE_MAX
 * if both s1 < MUL_NO_OVERFLOW and s2 < MUL_NO_OVERFLOW
 */
#define MUL_NO_OVERFLOW	((size_t)1 << (sizeof(size_t) * 4))

static void *
_reallocarray(void *optr, size_t nmemb, size_t size)
{
	if ((nmemb >= MUL_NO_OVERFLOW || size >= MUL_NO_OVERFLOW) &&
	    nmemb > 0 && SIZE_MAX / nmemb < size) {
		errno = ENOMEM;
		return NULL;
	}
	return realloc(optr, size * nmemb);
}

int
esnprintf(char *str, size_t size, const char *fmt, ...)
{
	va_list ap;
	int ret;

	va_start(ap, fmt);
	ret = vsnprintf(str, size, fmt, ap);
	va_end(ap);

	return (ret < 0 || (size_t)ret >= size);
}

void *
ereallocarray(void *optr, size_t nmemb, size_t size)
{
	void *p;

	if (!(p = _reallocarray(optr, nmemb, size))) {
		die("reallocarray: Out of memory");
	}

	return p;
}

void
util_array_vector_wrap(const double *arr, size_t len, gsl_vector *v)
{
	v->size   = len;
	v->stride = 1;
	v->data   = (double *)arr;
	v->block  = NULL;
	v->owner  = 0;
}

int                                                                       
plot_spline_fmt(const struct ethos *e, const double *c,
                enum ethos_plot_format p, const char *fmt, ...)
{
	char fname[PATH_MAX];
	va_list ap;

	va_start(ap, fmt);
	if ((unsigned)vsnprintf(fname, LEN(fname), fmt, ap) >=
            LEN(fname)) {
		warn("vsnprintf: Truncated filename");
		return 1;
	}
	va_end(ap);

	return ethos_plot_spline(e, c, fname, p);
}

int                                                                       
plot_points_fmt(const double *T, const double *S, size_t N,
                enum ethos_plot_format p, const char *fmt, ...)
{
	char fname[PATH_MAX];
	va_list ap;

	va_start(ap, fmt);
	if ((unsigned)vsnprintf(fname, LEN(fname), fmt, ap) >=
	    LEN(fname)) {
		warn("vsnprintf: Truncated filename");
		return 1;
	}
	va_end(ap);

	return ethos_plot_points((double *)T, (double *)S, N, fname, p);
}
\end{lstlisting}
\subsection{config.mk}
\begin{lstlisting}[language=make,label=lst:examples-configmk]
# Customize below to fit your system

# paths
PREFIX    = /usr/local
ETHOS     = ../ethos/

# flags
CPPFLAGS = -D_DEFAULT_SOURCE -I $(ETHOS)
CFLAGS   = -std=c99 -pedantic -Wall -Wextra -Os
LDFLAGS  = -s -L $(ETHOS)
LDLIBS   = -lethos -lgsl -lgslcblas -lm

# compiler and linker
CC = cc
\end{lstlisting}
\subsection{Makefile}
\begin{lstlisting}[language=make,label=lst:examples-makefile]
# See LICENSE file for copyright and license details
# ETHOS - EMD Toolbox using Hybrid Operator-Based Methods and B-Splines
.POSIX:

include config.mk

EXAMPLES = emd envelope regop spline

all: $(EXAMPLES:=.data/SENTINEL)

emd.data/SENTINEL: emd
envelope.data/SENTINEL: envelope
regop.data/SENTINEL: regop
spline.data/SENTINEL: spline

emd: emd.o util.o
envelope: envelope.o util.o
regop: regop.o util.o
spline: spline.o util.o

emd.o: emd.c config.mk
envelope.o: envelope.c config.mk
regop.o: regop.c config.mk
spline.o: spline.c config.mk
util.o: util.c config.mk

.o:
	$(CC) -o $@ $(LDFLAGS) $< util.o $(LDLIBS)
	
.c.o:
	$(CC) -c $(CPPFLAGS) $(CFLAGS) $<

$(EXAMPLES:=.data/SENTINEL):
	touch $@
	./$^

clean:
	rm -f $(EXAMPLES) $(EXAMPLES:=.o) util.o
\end{lstlisting}
\section{License}
This ISC license applies to all code listings in
Chapter~\ref{ch:code}.
\begin{lstlisting}[label=lst:license]
Copyright 2019-2020 Laslo Hunhold

Permission to use, copy, modify, and/or distribute this software for any
purpose with or without fee is hereby granted, provided that the above
copyright notice and this permission notice appear in all copies.

THE SOFTWARE IS PROVIDED "AS IS" AND THE AUTHOR DISCLAIMS ALL WARRANTIES
WITH REGARD TO THIS SOFTWARE INCLUDING ALL IMPLIED WARRANTIES OF
MERCHANTABILITY AND FITNESS. IN NO EVENT SHALL THE AUTHOR BE LIABLE FOR
ANY SPECIAL, DIRECT, INDIRECT, OR CONSEQUENTIAL DAMAGES OR ANY DAMAGES
WHATSOEVER RESULTING FROM LOSS OF USE, DATA OR PROFITS, WHETHER IN AN
ACTION OF CONTRACT, NEGLIGENCE OR OTHER TORTIOUS ACTION, ARISING OUT OF
OR IN CONNECTION WITH THE USE OR PERFORMANCE OF THIS SOFTWARE.
\end{lstlisting}
\backmatter
\nocite{*}
\bibliography{laslo_hunhold-emd}

\newcommand{\etalchar}[1]{$^{#1}$}
\begin{thebibliography}{KTRZ{\etalchar{+}}17}
  \providebibliographyfont{name}{}%
  \providebibliographyfont{lastname}{}%
  \providebibliographyfont{title}{\emph}%
  \providebibliographyfont{jtitle}{\btxtitlefont}%
  \providebibliographyfont{etal}{\emph}%
  \providebibliographyfont{journal}{}%
  \providebibliographyfont{volume}{}%
  \providebibliographyfont{ISBN}{\MakeUppercase}%
  \providebibliographyfont{ISSN}{\MakeUppercase}%
  \providebibliographyfont{url}{\url}%
  \providebibliographyfont{numeral}{}%
  \expandafter\btxselectlanguage\expandafter {\btxfallbacklanguage}

\expandafter\btxselectlanguage\expandafter {\btxfallbacklanguage}
\bibitem [{Bou03}]{bou03}
\btxnamefont {\btxlastnamefont {Bourbaki}, Nicolas}\btxauthorcolon\
  \btxtitlefont {Topological Vector Spaces: Chapters 1–5}, \btxvolumelong
  {}~\btxvolumefont {1} \btxofserieslong {}\ \btxtitlefont {Elements of
  Mathematics}.
\newblock \btxpublisherfont {Springer-Verlag Berlin Heidelberg}, Berlin,
  Germany, \btxeditionnumlong {1}{}, 2003\ifbtxprintISBN {,
  \mbox{\btxISBN~\btxISBNfont {978-3-642-61715-7}}}.
\newblock {\latintext \btxurlfont{https://doi.org/10.1007/978-3-642-61715-7}}.

\bibitem [{Cie97}]{cie97}
\btxnamefont {\btxlastnamefont {Ciesielski}, Krzysztof}\btxauthorcolon\
  \btxtitlefont {Set Theory for the Working Mathematician}, \btxvolumelong
  {}~\btxvolumefont {39} \btxofserieslong {}\ \btxtitlefont {London
  Mathematical Society Student Texts}.
\newblock \btxpublisherfont {Cambridge University Press}, Cambridge, England,
  UK, \btxeditionnumlong {1}{}, 1997\ifbtxprintISBN {,
  \mbox{\btxISBN~\btxISBNfont {978-1-139-17313-1}}}.
\newblock {\latintext \btxurlfont{https://doi.org/10.1017/CBO9781139173131}}.

\bibitem [{Dau92}]{dau92}
\btxnamefont {\btxlastnamefont {Daubechies}, Ingrid}\btxauthorcolon\
  \btxtitlefont {Ten Lectures on Wavelets}, \btxvolumelong {}~\btxvolumefont
  {61} \btxofserieslong {}\ \btxtitlefont {CBMS-NSF Regional Conference Series
  in Applied Mathematics}.
\newblock \btxpublisherfont {Society for Industrial and Applied Mathematics
  (SIAM)}, University City, Philadelphia, PA, USA, \btxeditionnumlong {1}{},
  \btxprintmonthyear{.}{06}{1992}{long}\ifbtxprintISBN {,
  \mbox{\btxISBN~\btxISBNfont {978-1-61197-010-4}}}.
\newblock {\latintext \btxurlfont{https://doi.org/10.1137/1.9781611970104}}.

\bibitem [{dB01}]{db01}
\btxnamefont {\btxlastnamefont {Boor}, {Carl-Wilhelm
  Reinhold}~de}\btxauthorcolon\ \btxtitlefont {A Practical Guide to Splines},
  \btxvolumelong {}~\btxvolumefont {27} \btxofserieslong {}\ \btxtitlefont
  {Applied Mathematical Sciences}.
\newblock \btxpublisherfont {Springer-Verlag New York}, New York City, NY, USA,
  \btxeditionnumlong {revised}{},
  \btxprintmonthyear{.}{4}{2001}{long}\ifbtxprintISBN {,
  \mbox{\btxISBN~\btxISBNfont {978-0-387-95366-3}}}.
\newblock {\latintext
  \btxurlfont{https://www.springer.com/book/978-0-387-95366-3}}.

\bibitem [{Die95}]{di95}
\btxnamefont {\btxlastnamefont {Dierckx}, Paul}\btxauthorcolon\ \btxtitlefont
  {Curve and Surface Fitting with Splines}, \btxvolumelong {}~\btxvolumefont
  {1} \btxofserieslong {}\ \btxtitlefont {Numerical Mathematics and Scientific
  Computation}.
\newblock \btxpublisherfont {Clarendon Press}, Oxford, England, UK,
  \btxeditionnumlong {new}{},
  \btxprintmonthyear{.}{04}{1995}{long}\ifbtxprintISBN {,
  \mbox{\btxISBN~\btxISBNfont {978-0-19-853440-2}}}.
\newblock {\latintext
  \btxurlfont{https://global.oup.com/academic/product/curve-and-surface-fitting-with-splines-9780198534402}}.

\bibitem [{DLW11}]{dlw11}
\btxnamefont {\btxlastnamefont {Daubechies}, Ingrid}, \btxnamefont {Jianfeng
  \btxlastnamefont {Lu}}\btxandcomma {} \btxandlong {}\ \btxnamefont
  {Hau\btxfnamespacelong Tieng \btxlastnamefont {Wu}}\btxauthorcolon\
  \btxjtitlefont {\btxifchangecase {Synchrosqueezed wavelet transforms: An
  empirical mode decomposition-like tool}{Synchrosqueezed wavelet transforms:
  An empirical mode decomposition-like tool}}.
\newblock \btxjournalfont {Applied and Computational Harmonic Analysis},
  30(2):243--261, \btxprintmonthyear{.}{03}{2011}{long}\ifbtxprintISSN {,
  \mbox{\btxISSN~\btxISSNfont {1063-5203}}}.
\newblock {\latintext
  \btxurlfont{https://dx.doi.org/10.1016/j.acha.2010.08.002}}.

\bibitem [{Fou22}]{f22}
\btxnamefont {\btxlastnamefont {Fourier}, {Jean Baptiste
  Joseph}}\btxauthorcolon\ \btxtitlefont {Th{\'e}orie analytique de la
  chaleur}.
\newblock \btxpublisherfont {Ambroise Firmin Didot, p{\`e}re et fils}, Paris,
  France, \btxeditionnumlong {1}{}, 1822.
\newblock {\latintext \btxurlfont{https://openlibrary.org/books/OL24141486M/}}.

\bibitem [{GDT{\etalchar{+}}18}]{gdt+18}
\btxnamefont {\btxlastnamefont {Galassi}, Mark}, \btxnamefont {Jim
  \btxlastnamefont {Davies}}, \btxnamefont {James \btxlastnamefont {Theiler}},
  \btxnamefont {Brian \btxlastnamefont {Gough}}, \btxnamefont {Gerard
  \btxlastnamefont {Jungman}}, \btxnamefont {Patrick \btxlastnamefont {Alken}},
  \btxnamefont {Michael \btxlastnamefont {Booth}}, \btxnamefont {Fabrice
  \btxlastnamefont {Rossi}}\btxandcomma {} \btxandlong {}\ \btxnamefont {Rhys
  \btxlastnamefont {Ulerich}}\btxauthorcolon\ \btxtitlefont {{GNU} Scientific
  Library}.
\newblock Free Software Foundation, Boston, MA, USA, \btxeditionnumlong
  {2.5}{}, \btxprintmonthyear{.}{06}{2018}{long}.
\newblock {\latintext
  \btxurlfont{https://www.gnu.org/software/gsl/doc/latex/gsl-ref.pdf}}.

\bibitem [{GPHX17}]{gphx17}
\btxnamefont {\btxlastnamefont {Guo}, Baokui}, \btxnamefont {Silong
  \btxlastnamefont {Peng}}, \btxnamefont {Xiyuan \btxlastnamefont
  {Hu}}\btxandcomma {} \btxandlong {}\ \btxnamefont {Pengcheng \btxlastnamefont
  {Xu}}\btxauthorcolon\ \btxjtitlefont {\btxifchangecase {Complex-valued
  differential operator-based method for multi-component signal
  separation}{Complex-valued differential operator-based method for
  multi-component signal separation}}.
\newblock \btxjournalfont {Signal Processing}, 132:66--76,
  \btxprintmonthyear{.}{3}{2017}{long}\ifbtxprintISSN {,
  \mbox{\btxISSN~\btxISSNfont {0165-1684}}}.
\newblock {\latintext
  \btxurlfont{https://dx.doi.org/10.1016/j.sigpro.2016.09.015}}.

\bibitem [{Haa10}]{h10}
\btxnamefont {\btxlastnamefont {Haar}, Alfréd}\btxauthorcolon\ \btxjtitlefont
  {\btxifchangecase {Zur {T}heorie der orthogonalen {F}unktionensysteme}{Zur
  {T}heorie der orthogonalen {F}unktionensysteme}}.
\newblock \btxjournalfont {Mathematische Annalen}, 69:331--371,
  \btxprintmonthyear{.}{09}{1910}{long}\ifbtxprintISSN {,
  \mbox{\btxISSN~\btxISSNfont {1432-1807}}}.
\newblock {\latintext \btxurlfont{https://dx.doi.org/10.1007/BF01456326}}.

\bibitem [{HJ12}]{hj12}
\btxnamefont {\btxlastnamefont {Horn}, Roger\btxfnamespacelong Alan}
  \btxandlong {}\ \btxnamefont {Charles\btxfnamespacelong Royal
  \btxlastnamefont {Johnson}}\btxauthorcolon\ \btxtitlefont {Matrix Analysis}.
\newblock \btxpublisherfont {Cambridge University Press}, Cambridge, England,
  UK, \btxeditionnumlong {2}{},
  \btxprintmonthyear{.}{10}{2012}{long}\ifbtxprintISBN {,
  \mbox{\btxISBN~\btxISBNfont {978-0-521-83940-2}}}.
\newblock {\latintext \btxurlfont{https://dx.doi.org/10.1017/9781139020411}}.

\bibitem [{HK13}]{hk13}
\btxnamefont {\btxlastnamefont {Huang}, Boqiang} \btxandlong {}\ \btxnamefont
  {Angela \btxlastnamefont {Kunoth}}\btxauthorcolon\ \btxjtitlefont
  {\btxifchangecase {An optimization based empirical mode decomposition
  scheme}{An optimization based empirical mode decomposition scheme}}.
\newblock \btxjournalfont {Journal of Computational and Applied Mathematics},
  240:174--183, \btxprintmonthyear{.}{03}{2013}{long}\ifbtxprintISSN {,
  \mbox{\btxISSN~\btxISSNfont {0377-0427}}}.
\newblock {\latintext
  \btxurlfont{https://dx.doi.org/10.1016/j.cam.2012.07.012}}, MATA 2012.

\bibitem [{HPH12}]{hph12}
\btxnamefont {\btxlastnamefont {Hu}, Xiyuan}, \btxnamefont {Silong
  \btxlastnamefont {Peng}}\btxandcomma {} \btxandlong {}\ \btxnamefont
  {{Wen-Liang} \btxlastnamefont {Hwang}}\btxauthorcolon\ \btxjtitlefont
  {\btxifchangecase {{EMD} revisited: A new understanding of the envelope and
  resolving the mode-mixing problem in {AM}-{FM} signals}{{EMD} Revisited: A
  New Understanding of the Envelope and Resolving the Mode-Mixing Problem in
  {AM}-{FM} Signals}}.
\newblock \btxjournalfont {IEEE Transactions on Signal Processing},
  60(3):1075--1086, \btxprintmonthyear{.}{03}{2012}{long}\ifbtxprintISSN {,
  \mbox{\btxISSN~\btxISSNfont {1941-0476}}}.
\newblock {\latintext
  \btxurlfont{https://dx.doi.org/10.1109/TSP.2011.2179650}}.

\bibitem [{HS11}]{hs11}
\btxnamefont {\btxlastnamefont {Hou}, {Thomas Yizhao}} \btxandlong {}\
  \btxnamefont {Zuoqiang \btxlastnamefont {Shi}}\btxauthorcolon\ \btxjtitlefont
  {\btxifchangecase {Adaptive data analysis via sparse time-frequency
  representation}{Adaptive Data Analysis via Sparse Time-Frequency
  Representation}}.
\newblock \btxjournalfont {Advances in Adaptive Data Analysis}, 3(1\&2):1--28,
  \btxprintmonthyear{.}{4}{2011}{long}.
\newblock {\latintext
  \btxurlfont{https://dx.doi.org/10.1142/S1793536911000647}}.

\bibitem [{HSL{\etalchar{+}}98}]{hsl+98}
\btxnamefont {\btxlastnamefont {Huang}, {Norden Eh}}, \btxnamefont {Zheng
  \btxlastnamefont {Shen}}, \btxnamefont {{Steven R.} \btxlastnamefont {Long}},
  \btxnamefont {{Manli C.} \btxlastnamefont {Wu}}, \btxnamefont {{Hsing H.}
  \btxlastnamefont {Shih}}, \btxnamefont {Quanan \btxlastnamefont {Zheng}},
  \btxnamefont {{Nai-Chyuan} \btxlastnamefont {Yen}}, \btxnamefont {{Chi Chao}
  \btxlastnamefont {Tung}}\btxandcomma {} \btxandlong {}\ \btxnamefont {{Henry
  H.} \btxlastnamefont {Liu}}\btxauthorcolon\ \btxjtitlefont {\btxifchangecase
  {The empirical mode decomposition and the hilbert spectrum for nonlinear and
  non-stationary time series analysis}{The empirical mode decomposition and the
  Hilbert spectrum for nonlinear and non-stationary time series analysis}}.
\newblock \btxjournalfont {Proceedings of the Royal Society of London. Series
  A: Mathematical, Physical and Engineering Sciences}, 454:903--995,
  \btxprintmonthyear{.}{03}{1998}{long}\ifbtxprintISSN {,
  \mbox{\btxISSN~\btxISSNfont {1471-2946}}}.
\newblock {\latintext \btxurlfont{https://dx.doi.org/10.1098/rspa.1998.0193}}.

\bibitem [{HYY15}]{hyy15}
\btxnamefont {\btxlastnamefont {Huang}, Chao}, \btxnamefont {Lijun
  \btxlastnamefont {Yang}}\btxandcomma {} \btxandlong {}\ \btxnamefont {Lihua
  \btxlastnamefont {Yang}}\btxauthorcolon\ \btxjtitlefont {\btxifchangecase
  {$\epsilon$-{M}ono-{C}omponent: Its characterization and
  construction}{$\epsilon$-{M}ono-{C}omponent: Its Characterization and
  Construction}}.
\newblock \btxjournalfont {IEEE Transactions on Signal Processing},
  63:234--243, \btxprintmonthyear{.}{01}{2015}{long}\ifbtxprintISSN {,
  \mbox{\btxISSN~\btxISSNfont {1053-587X}}}.
\newblock {\latintext
  \btxurlfont{https://dx.doi.org/10.1109/TSP.2014.2370950}}.

\bibitem [{{ISO}99}]{iso99}
\btxnamefont {\btxlastnamefont {{ISO\slash IEC JTC 1/SC 22}}}\btxauthorcolon\
  \btxtitlefont {{ISO\slash IEC 9899:1999}: Programming Languages --- {C}}.
\newblock \btxpublisherfont {International Organization for Standardization},
  Geneva, Switzerland, \btxeditionnumlong {2}{},
  \btxprintmonthyear{.}{12}{1999}{long}.
\newblock {\latintext \btxurlfont{https://www.iso.org/standard/29237.html}}.

\bibitem [{Jah07}]{jah07}
\btxnamefont {\btxlastnamefont {Jahn}, Johannes}\btxauthorcolon\ \btxtitlefont
  {Introduction to the Theory of Nonlinear Optimization}.
\newblock \btxpublisherfont {Springer-Verlag Berlin Heidelberg}, Berlin,
  Germany, \btxeditionnumlong {3}{}, 2007\ifbtxprintISBN {,
  \mbox{\btxISBN~\btxISBNfont {978-3-540-49379-2}}}.
\newblock {\latintext
  \btxurlfont{https://dx.doi.org/10.1007/978-3-540-49379-2}}.

\bibitem [{KK00}]{kk00}
\btxnamefont {\btxlastnamefont {Küpfmüller}, Karl} \btxandlong {}\
  \btxnamefont {Gerhard \btxlastnamefont {Kohn}}\btxauthorcolon\ \btxtitlefont
  {Theoretische Elektrotechnik und Elektronik: Eine Einführung}.
\newblock Springer-Lehrbuch. \btxpublisherfont {Springer-Verlag Berlin
  Heidelberg}, Berlin, Germany, \btxeditionnumlong {15}{}, 2000\ifbtxprintISBN
  {, \mbox{\btxISBN~\btxISBNfont {978-3-662-10425-5}}}.
\newblock {\latintext
  \btxurlfont{https://dx.doi.org/10.1007/978-3-662-10425-5}}.

\bibitem [{KTRZ{\etalchar{+}}17}]{ktrz+17}
\btxnamefont {\btxlastnamefont {Kreutzer}, Moritz}, \btxnamefont {Jonas
  \btxlastnamefont {Thies}}, \btxnamefont {Melven \btxlastnamefont
  {R{\"o}hrig-Z{\"o}llner}}, \btxnamefont {Andreas \btxlastnamefont {Pieper}},
  \btxnamefont {Faisal \btxlastnamefont {Shahzad}}, \btxnamefont {Martin
  \btxlastnamefont {Galgon}}, \btxnamefont {Achim \btxlastnamefont
  {Basermann}}, \btxnamefont {Holger \btxlastnamefont {Fehske}}, \btxnamefont
  {Georg \btxlastnamefont {Hager}}\btxandcomma {} \btxandlong {}\ \btxnamefont
  {Gerhard \btxlastnamefont {Wellein}}\btxauthorcolon\ \btxjtitlefont
  {\btxifchangecase {{GHOST}: Building blocks for high performance sparse
  linear algebra on heterogeneous systems}{{GHOST}: Building Blocks for High
  Performance Sparse Linear Algebra on Heterogeneous Systems}}.
\newblock \btxjournalfont {International Journal of Parallel Programming},
  45(5):1046--1072, \btxprintmonthyear{.}{10}{2017}{long}\ifbtxprintISSN {,
  \mbox{\btxISSN~\btxISSNfont {1573-7640}}}.
\newblock {\latintext
  \btxurlfont{https://dx.doi.org/10.1007/s10766-016-0464-z}}.

\bibitem [{Kö88}]{k88}
\btxnamefont {\btxlastnamefont {Körner}, {Thomas William}}\btxauthorcolon\
  \btxtitlefont {Fourier Analysis}.
\newblock \btxpublisherfont {Cambridge University Press}, Cambridge, England,
  UK, \btxeditionnumlong {1}{}, 1988\ifbtxprintISBN {,
  \mbox{\btxISBN~\btxISBNfont {978-1-107-04994-9}}}.
\newblock {\latintext
  \btxurlfont{https://dx.doi.org/10.1017/CBO9781107049949}}.

\bibitem [{LWW13}]{lww13}
\btxnamefont {\btxlastnamefont {Liu}, Yanping}, \btxnamefont {Yong
  \btxlastnamefont {Wang}}\btxandcomma {} \btxandlong {}\ \btxnamefont {Zhen
  \btxlastnamefont {Wang}}\btxauthorcolon\ \btxtitlefont {\btxifchangecase
  {{RBF} prediction model based on {EMD} for forecasting {GPS} precipitable
  water vapor and annual precipitation}{{RBF} Prediction Model Based on {EMD}
  for Forecasting {GPS} Precipitable Water Vapor and Annual Precipitation}}.
\newblock \Btxinlong {}\ \btxnamefont {\btxlastnamefont {Luo}, Xun}\
  (\btxeditorlong {}): \btxtitlefont {2nd International Conference On Systems
  Engineering and Modeling (ICSEM-13)}, \btxvolumelong {}~\btxvolumefont {35}
  \btxofserieslong {}\ \btxtitlefont {Advances in Intelligent Systems
  Research}, \btxpageslong {}\ 51--55, Paris, France,
  \btxprintmonthyear{.}{04}{2013}{long}. \btxpublisherfont {Atlantis Press}.
\newblock {\latintext \btxurlfont{https://dx.doi.org/10.2991/icsem.2013.11}}.

\bibitem [{Mal09}]{m09}
\btxnamefont {\btxlastnamefont {Mallat}, {Stéphane Georges}}\btxauthorcolon\
  \btxtitlefont {A Wavelet Tour of Signal Processing}.
\newblock \btxpublisherfont {Academic Press}, Boston, MA, USA,
  \btxeditionnumlong {3}{}, 2009\ifbtxprintISBN {, \mbox{\btxISBN~\btxISBNfont
  {978-0-12-374370-1}}}.
\newblock {\latintext
  \btxurlfont{https://dx.doi.org/10.1016/B978-0-12-374370-1.50001-9}}.

\bibitem [{NP06}]{np06}
\btxnamefont {\btxlastnamefont {Niculescu}, Constantin\btxfnamespacelong P.}
  \btxandlong {}\ \btxnamefont {Lars\btxfnamespacelong Erik \btxlastnamefont
  {Persson}}\btxauthorcolon\ \btxtitlefont {Convex Functions and Their
  Applications: A Contemporary Approach}, \btxvolumelong {}~\btxvolumefont {24}
  \btxofserieslong {}\ \btxtitlefont {CMS Books in Mathematics}.
\newblock \btxpublisherfont {Springer-Verlag New York}, New York City, NY, USA,
  \btxeditionnumlong {1}{}, 2006\ifbtxprintISBN {, \mbox{\btxISBN~\btxISBNfont
  {978-0-387-31077-0}}}.
\newblock {\latintext \btxurlfont{https://dx.doi.org/10.1007/0-387-31077-0}}.

\bibitem [{PH08}]{ph08}
\btxnamefont {\btxlastnamefont {Peng}, Silong} \btxandlong {}\ \btxnamefont
  {{Wen-Liang} \btxlastnamefont {Hwang}}\btxauthorcolon\ \btxjtitlefont
  {\btxifchangecase {Adaptive signal decomposition based on local narrow band
  signals}{Adaptive Signal Decomposition Based on Local Narrow Band Signals}}.
\newblock \btxjournalfont {IEEE Transactions on Signal Processing},
  56(7):2669--2676, \btxprintmonthyear{.}{6}{2008}{long}\ifbtxprintISSN {,
  \mbox{\btxISSN~\btxISSNfont {1941-0476}}}.
\newblock {\latintext \btxurlfont{https://dx.doi.org/10.1109/TSP.2008.917360}}.

\bibitem [{PH10}]{ph10}
\btxnamefont {\btxlastnamefont {Peng}, Silong} \btxandlong {}\ \btxnamefont
  {{Wen-Liang} \btxlastnamefont {Hwang}}\btxauthorcolon\ \btxjtitlefont
  {\btxifchangecase {Null space pursuit: An operator-based approach to adaptive
  signal separation}{Null Space Pursuit: An Operator-based Approach to Adaptive
  Signal Separation}}.
\newblock \btxjournalfont {IEEE Transactions on Signal Processing},
  58(5):2475--2483, \btxprintmonthyear{.}{5}{2010}{long}\ifbtxprintISSN {,
  \mbox{\btxISSN~\btxISSNfont {1941-0476}}}.
\newblock {\latintext
  \btxurlfont{https://dx.doi.org/10.1109/TSP.2010.2041606}}.

\bibitem [{Sch46a}]{sch46a}
\btxnamefont {\btxlastnamefont {Schoenberg}, {Isaac Jacob}}\btxauthorcolon\
  \btxjtitlefont {\btxifchangecase {Contributions to the problem of
  approximation of equidistant data by analytic functions. {P}art {A}. {O}n the
  problem of smoothing or graduation. {A} first class of analytic approximation
  formulae}{Contributions to the problem of approximation of equidistant data
  by analytic functions. {P}art {A}. {O}n the problem of smoothing or
  graduation. {A} first class of analytic approximation formulae}}.
\newblock \btxjournalfont {Quarterly of Applied Mathematics}, 4(1):45--99,
  \btxprintmonthyear{.}{04}{1946}{long}\ifbtxprintISSN {,
  \mbox{\btxISSN~\btxISSNfont {1552-4485}}}.
\newblock {\latintext \btxurlfont{https://doi.org/10.1090/qam/15914}}.

\bibitem [{Sch46b}]{sch46b}
\btxnamefont {\btxlastnamefont {Schoenberg}, {Isaac Jacob}}\btxauthorcolon\
  \btxjtitlefont {\btxifchangecase {Contributions to the problem of
  approximation of equidistant data by analytic functions. {P}art {B}. {O}n the
  problem of osculatory interpolation. {A} second class of analytic
  approximation formulae}{Contributions to the problem of approximation of
  equidistant data by analytic functions. {P}art {B}. {O}n the problem of
  osculatory interpolation. {A} second class of analytic approximation
  formulae}}.
\newblock \btxjournalfont {Quarterly of Applied Mathematics}, 4(2):112--141,
  \btxprintmonthyear{.}{07}{1946}{long}\ifbtxprintISSN {,
  \mbox{\btxISSN~\btxISSNfont {1552-4485}}}.
\newblock {\latintext \btxurlfont{https://doi.org/10.1090/qam/16705}}.

\bibitem [{SW99}]{sw99}
\btxnamefont {\btxlastnamefont {Schaefer}, Helmut\btxfnamespacelong Heinrich}
  \btxandlong {}\ \btxnamefont {Michael\btxfnamespacelong P. \btxlastnamefont
  {Wolff}}\btxauthorcolon\ \btxtitlefont {Topological Vector Spaces},
  \btxvolumelong {}~\btxvolumefont {3} \btxofserieslong {}\ \btxtitlefont
  {Graduate Texts in Mathematics}.
\newblock \btxpublisherfont {Springer-Verlag New York}, New York City, NY, USA,
  \btxeditionnumlong {2}{},
  \btxprintmonthyear{.}{06}{1999}{long}\ifbtxprintISBN {,
  \mbox{\btxISBN~\btxISBNfont {978-1-4612-1468-7}}}.
\newblock {\latintext \btxurlfont{https://doi.org/10.1007/978-1-4612-1468-7}}.

\bibitem [{WR10}]{wr10}
\btxnamefont {\btxlastnamefont {Wu}, Qin} \btxandlong {}\ \btxnamefont
  {{Sherman Delbert} \btxlastnamefont {Riemenschneider}}\btxauthorcolon\
  \btxjtitlefont {\btxifchangecase {Boundary extension and stop criteria for
  empirical mode decomposition}{Boundary Extension and Stop Criteria for
  Empirical Mode Decomposition}}.
\newblock \btxjournalfont {Advances in Adaptive Data Analysis}, 2(2):157--169,
  \btxprintmonthyear{.}{04}{2010}{long}.
\newblock {\latintext \btxurlfont{https://doi.org/10.1142/S1793536910000434}}.

\bibitem [{YYJ12}]{yyj12}
\btxnamefont {\btxlastnamefont {Yong}, Wang}, \btxnamefont {Liu
  \btxlastnamefont {Yanping}}\btxandcomma {} \btxandlong {}\ \btxnamefont {Yang
  \btxlastnamefont {Jing}}\btxauthorcolon\ \btxjtitlefont {\btxifchangecase
  {Signal prediction based on empirical mode decomposition and artificial
  neural networks}{Signal prediction based on empirical mode decomposition and
  artificial neural networks}}.
\newblock \btxjournalfont {Geodesy and Geodynamics}, 3(1):52--56,
  \btxprintmonthyear{.}{02}{2012}{long}\ifbtxprintISSN {,
  \mbox{\btxISSN~\btxISSNfont {1674-9847}}}.
\newblock {\latintext
  \btxurlfont{https://dx.doi.org/10.3724/SP.J.1246.2012.00052}}.

\end{thebibliography}
\selectlanguage{german}
\chapter{Eigenständigkeitserklärung}
Hiermit versichere ich an Eides statt, daß ich die vorliegende Arbeit
selbstständig und ohne die Benutzung anderer als der angegebenen
Hilfsmittel angefertigt habe. Alle Stellen, die wörtlich oder sinngemäß
aus veröffentlichten und nicht veröffentlichten Schriften entnommen wurden,
sind als solche kenntlich gemacht.
\newline
Die Arbeit ist in gleicher oder ähnlicher Form oder auszugsweise im
Rahmen einer anderen Prüfung noch nicht vorgelegt worden. Ich versichere,
daß die eingereichte elektronische Fassung der eingereichten Druckfassung
vollständig entspricht.
\\[3cm]
Laslo Hunhold
\end{document}